\documentclass[11pt,oneside]{article}
\usepackage{setspace,graphicx,amssymb,amsmath,latexsym,amsfonts,amscd,amsthm,multirow,ctable,mathdots,caption,array}
\usepackage{fancyhdr,tabularx,cite,mathrsfs}
\usepackage[headings]{fullpage}
\usepackage{stmaryrd}
\usepackage{rotating}

\usepackage{hyperref}

\newcommand{\bea}{\begin{eqnarray}} 
\newcommand{\eea}{\end{eqnarray}} 
\newcommand{\bee}{\begin{eqnarray*}} 
\newcommand{\eee}{\end{eqnarray*}} 
\newcommand{\al}{\begin{align*}} 
\newcommand{\eal}{\end{align*}} 
\newcommand{\be}{\begin{equation}} 
\newcommand{\ee}{\end{equation}} 
\newcommand{\eq}[1]{(\ref{#1})} 
\newcommand{\bem}{\begin{pmatrix}} 
\newcommand{\eem}{\end{pmatrix}} 

\def\a{\alpha} 
\def\b{\beta} 
\def\c{\gamma} 
\def\d{\delta} 
 
\def\e{\epsilon}    
\def\f{\phi}

\def\h{\eta}

\def\mm{^{(m)}}
\def\inf{\infty} 
 
\def\k{\kappa}
\def\l{\lambda} 
\def\m{\mu} 
\def\n{\nu}

\def\p{\pi}    
\def\pa{\partial}

\def\s{\sigma}            
\def\t{\tau} 
\def\th{\theta} 
\def\til{\widetilde}

\def\D{\Delta}

\def\L{\Lambda} 
\def\O{\Omega}

\newcolumntype{R}{ >{$}r <{$}}
\newcolumntype{C}{ >{$}c <{$}}
\newcolumntype{L}{ >{$}l <{$}}
\newcolumntype{F}{>{\centering\arraybackslash}m{1.5cm}}

\def\ll{\ell}
\def\LL{\Lambda}

\newcommand{\mc}[1]{\mathcal{#1}}

\newcommand{\comment}[1]{}

\newcommand{\RR}{{\mathbb R}}%Reals
\newcommand{\CC}{{\mathbb C}}%Complex
\newcommand{\ZZ}{{\mathbb Z}}%Integers
\newcommand{\QQ}{{\mathbb Q}}%Rationals
\newcommand{\HH}{{\mathbb H}}%quaternions
\newcommand{\ii}{{\bf i}}%Sqrt of -1
\newcommand{\tpi}{2\pi\ii}%The fundamental constant
\newcommand{\lab}{{\langle}}    %Left angle brackets
\newcommand{\rab}{{\rangle}}    %Right angle brackets

\newcommand{\Aut}{\operatorname{Aut}}

\newcommand{\Span}{\operatorname{Span}}

\def\jac{\operatorname{j}}

\def\infm{\operatorname{inf}}
\def\supr{\operatorname{sup}}

\newcommand{\tr}{\operatorname{{tr}}}
\newcommand{\str}{\operatorname{{str}}}

\newcommand{\Sym}{{\textsl{Sym}}}
\newcommand{\Alt}{{\textsl{Alt}}}
\newcommand{\Dih}{{\textsl{Dih}}}

\newcommand{\sgn}{\operatorname{sgn}}
\newcommand{\ex}{\operatorname{e}} %Number theory exp

\newcommand{\PSL}{\operatorname{\textsl{PSL}}}    %PSL group
\newcommand{\SL}{\operatorname{\textsl{SL}}}      %SL group
\newcommand{\PGL}{\operatorname{\textsl{PGL}}}    %PGL group
\newcommand{\AGL}{{\textsl{AGL}}}    %AGL group
\newcommand{\GL}{{\textsl{GL}}}      %GL group
\newcommand{\SU}{\operatorname{\textsl{SU}}}    %SU group
\newcommand{\SO}{\operatorname{\textsl{SO}}}    %SO group

\newcommand{\G}{\Gamma}	%Gamma
\newcommand{\g}{\gamma}	%gamma

\newcommand{\rs}{{X}}	%deep hole root system

\newcommand{\MM}{\mathbb{M}}	%monster group
\newcommand{\Co}{\textsl{Co}}	%Conway group

\newtheorem{thm}{Theorem}[section]
\newtheorem{cor}[thm]{Corollary}
\newtheorem{lem}[thm]{Lemma}
\newtheorem{prop}[thm]{Proposition}
\newtheorem{conj}[thm]{Conjecture}

\theoremstyle{definition}

\theoremstyle{remark}
\newtheorem{rmk}[thm]{Remark}

\numberwithin{equation}{section}
\begin{document}

\setstretch{1.4}

\title{
\vspace{-35pt}
    \textsc{\huge{{U}mbral {M}oonshine\\ and the {N}iemeier {L}attices}
        }
    }

\renewcommand{\thefootnote}{\fnsymbol{footnote}} 
\footnotetext{\emph{MSC2010:} 11F22, 11F37, 11F46, 11F50, 20C34, 20C35}     
\renewcommand{\thefootnote}{\arabic{footnote}} 

\author{
	Miranda C. N. Cheng\footnote{
	Universit\'e Paris 7, UMR CNRS 7586, Paris, France.
	{\em E-mail:} {\tt chengm@math.jussieu.fr}
	}\\
	John F. R. Duncan\footnote{
         Department of Mathematics, Applied Mathematics and Statistics,
         Case Western Reserve University,
         Cleveland, OH 44106,
         U.S.A.
         {\em E-mail:} {\tt john.duncan@case.edu}
                  }\\
	Jeffrey A. Harvey\footnote{
	Enrico Fermi Institute and Department of Physics,
         University of Chicago,
         Chicago, IL 60637,
         U.S.A.
         {\em E-mail:} {\tt j-harvey@uchicago.edu}
          }
}

\maketitle
\abstract{

In this paper we relate umbral moonshine to the Niemeier lattices: the 23 even unimodular positive-definite lattices of rank 24 with non-trivial root systems.
To each Niemeier lattice we attach a finite group by considering a naturally defined quotient of the lattice automorphism group,
and for each conjugacy class of each of these groups we identify a vector-valued mock modular form whose components coincide with mock theta functions of Ramanujan in many cases. 
This leads to the umbral moonshine conjecture, stating that an infinite-dimensional module is assigned to each of the Niemeier lattices in such a way that the associated graded trace functions are mock modular forms of a distinguished nature.
These constructions and conjectures extend those of our earlier paper, and in particular include the Mathieu moonshine observed by Eguchi--Ooguri--Tachikawa as a special case. 
Our analysis also highlights a correspondence between genus zero groups and Niemeier lattices. As a part of this relation we recognise the Coxeter numbers of Niemeier root systems with a type A component as exactly those levels for which the corresponding classical modular curve has genus zero.

}

\clearpage

\tableofcontents

\clearpage

%------------------------------------------------------------------%
\section{Introduction}\label{sec:intro}
%------------------------------------------------------------------%

In this paper we relate umbral moonshine to the Niemeier lattices. 
This relation associates one case of umbral moonshine to each of the  23 Niemeier lattices and in particular constitutes an extension of our previous work \cite{UM}, incorporating 17 new instances. 
Moreover, this prescription displays interesting connections to genus zero groups (subgroups $\G <\SL_2(\mathbb R)$ that define a genus zero quotient of the upper-half plane) and extended Dynkin diagrams via McKay's correspondence.

To explain this moonshine relation,  let us first recall what Niemeier lattices are. 
In 1973 Niemeier classified the even unimodular positive-definite lattices of rank 24 \cite{Nie_DefQdtFrm24}. 
There are 24 of them, including the so-called {\em Leech lattice} discovered almost a decade earlier in the context of optimal lattice sphere packing in 24 dimensions \cite{Lee_SphPkgs}. Niemeier proved that the Leech lattice is the unique even, unimodular and positive-definite lattice of rank 24 with no root vectors (lattice vectors with norm-square 2), while the other 23 even unimodular rank 24-dimensional lattices all have root systems of the full rank 24. 
Moreover, these 23 lattices are uniquely labelled by their root systems, which are in turn uniquely specified by the following two conditions: first, they are unions of simply-laced root systems with the same Coxeter numbers; second, the total rank is 24.  

We will refer to these 23 root systems as the {\em Niemeier root systems} and the 23 corresponding lattices as the {\em Niemeier lattices}. In this paper we associate a finite group and a set of vector-valued mock modular forms to each of these 23 Niemeier lattices. 
The main result of the present paper is then the {\em umbral moonshine conjecture} relating the two. 

To understand this statement let us recall what one means by moonshine. 
This term was first introduced in mathematics 
to describe the remarkable {\em monstrous moonshine} phenomenon. 
The study of monstrous moonshine was initiated by John McKay's observation that $196883+1=196884$, where the summands on the left are degrees of irreducible representations of the {\em Fischer--Griess monster} $\MM$ and the number on the right is the coefficient of $q$ in the Fourier expansion of the {\em normalised elliptic modular invariant}
\begin{gather}\label{eqn:intro:FouExpJ}
J(\t)=\sum_{m \ge -1} a(m)\, q^m= q^{-1} + 196884\, q + 21493760 \,q^2 + 864299970 \,q^3 + \cdots,
\end{gather}
where we write $q=e^{2\p i \t}$. 
Following Thompson's idea \cite{Tho_FinGpsModFns} that $J(\t)$ should be the graded dimension of an infinite-dimensional module for $\MM$, this observation was later expanded into the full {\em monstrous moonshine conjecture}  by Conway and Norton  \cite{conway_norton}, conjecturing that the graded character $T_g(\t)$ attached to the monster module
\be
V=\bigoplus_{m \ge -1} V_m
\ee
and $g\in \MM$ should be a principal modulus for a certain genus zero group $\G_g <\SL_2(\RR)$. (When a discrete group $\G<\SL_2(\RR)$ has the property that $\G\backslash\HH$ is isomorphic to the Riemann sphere minus finitely many points, there exists a holomorphic function $f$ on $\HH$ that generates the field of $\G$-invariant functions on $\HH$. Such a function $f$ is called a {\em principal modulus}, or {\em Hauptmodul}, for $\G$.) We refer to \cite{gannon} or the introduction of \cite{UM} for a more detailed account of monstrous moonshine.

In 2010 the study of a new type of moonshine was triggered by an observation of Eguchi--Ooguri--Tachikawa, which constituted an analogue of McKay's  observation in monstrous moonshine. In the work of Eguchi--Taormina and Eguchi--Ooguri--Taormina--Yang in the 1980's \cite{Eguchi1987,Eguchi1988,Eguchi1989}, these authors encountered a $q$-series
\be\label{eqn:intro:H2Fou}
H^{(2)}(\tau)= 2\, q^{-1/8}(-1 + 45\, q + 231\, q^2 + 770 \,q^3+2277\, q^4 + \cdots )
\ee
in the  
decomposition of the elliptic genus of a $K3$ surface into irreducible characters of the ${\cal N}=4$ superconformal algebra. It was later understood by Eguchi--Hikami \cite{Eguchi2008} that the above $q$-series is a mock modular form. See \S\ref{sec:forms} for the definition of mock modular forms.  Subsequently the coincidence between the numbers 45, 231, 770, 2277,\ldots and the dimensions of irreducible representations of $M_{24}$  was pointed out in \cite{Eguchi2010}. This observation was later extended into a {\em Mathieu moonshine} conjecture in \cite{Cheng2010_1,Gaberdiel2010,Gaberdiel2010a,Eguchi2010a} by providing the corresponding twisted characters, the mock modular forms $H^{(2)}_g$, and was moreover related in a more general context to the $K3$-compactification of superstring theory in \cite{Cheng2010_1}.  Very recently, the existence of an infinite-dimensional $M_{24}$-module underlying the mock modular form \eq{eqn:intro:H2Fou} and those constructed in \cite{Cheng2010_1,Gaberdiel2010,Gaberdiel2010a,Eguchi2010a}  was shown by T. Gannon \cite{Gannon:2012ck}, although the nature of this $M_{24}$-module remains mysterious. See \cite{CheDun_M24MckAutFrms} and \cite{Gaberdiel:2012um} for a review of this $M_{24}$-mock modular relation, and  
see \cite{Creutzig2012,Gaberdiel:2012gf,Gaberdiel:2013nya,Taormina:2013jza,Taormina:2013mda} for recent developments in this direction. 

Meanwhile, it was found that Mathieu moonshine  
is but one example of a more general phenomenon, 
{\em umbral moonshine}. In \cite{UM} we associated a finite group $G^{(\ell)}$ and a vector-valued mock modular form $H_g^{(\ll)}=(H_{g,r}^{(\ll)})$ with $(\ell-1)$-components for every conjugacy class $[g]$ of $G^{(\ell)}$ to each of the six positive integers $\ell$ such that $\ell-1$  divides 12, and conjectured  that there exists an infinite-dimensional $G^{(\ll)}$-module, the {\em umbral module}, with the property that its graded character coincides with the mock modular form $H_g^{(\ll)}$ for every conjugacy class $[g] \subset G^{(\ll)}$.

Despite the discovery of this more general framework of umbral moonshine, encompassing  
Mathieu moonshine as a special case 
and displaying various beautiful properties, many questions remained unanswered. For example: why these specific umbral groups $G^{(\ll)}$? Why are they labelled by divisors of the number 12? What is the  structure underlying all these instances of moonshine? 

In the present paper we provide partial answers to the above questions. We present evidence that there exists an instance of umbral moonshine naturally associated to each of the 23 Niemeier lattices. As a Niemeier lattice is uniquely determined by its root system $X$, in the main text we shall use $X$ (or equivalently the corresponding {\em lambency}; see Tables \ref{tab:Hauptmodul}-\ref{tab:mugs}) to label the instances of umbral moonshine.
In particular, we construct  in each instance an {\em umbral group} $G^X$ as the quotient of the automorphism group of the corresponding Niemeier lattice $L^X$ by the normal subgroup
generated by refections in root vectors. 
This property gives a uniform construction as well as a concrete understanding of the umbral groups. 

Similarly, we provide a prescription that attaches to each of the 23 Niemeier lattices a distinguished  
vector-valued modular form---the {\em umbral mock modular form} $H^X$---which conjecturally encodes the dimensions of the homogeneous subspaces of the corresponding umbral module.
The Niemeier lattice uniquely specifies the shadow of the mock modular form through a map which associates a unary theta series of a specific type to each of the irreducible simply-laced ADE root systems, as well as unions of such root systems where all the irreducible components have the same Coxeter number. As will be explained in \S\ref{sec:forms:ADE}, this map bears a strong resemblance to the ADE classification by Cappelli--Itzykson--Zuber of modular invariant combinations of the characters of the $\widehat{A}_1$ affine Lie algebra \cite{Cappelli:1987xt}.
When applied to the Niemeier root systems, we dictate the resulting unary theta series to be the shadow of the corresponding umbral mock modular form. 
Together with the natural requirement that $H^X$ satisfies an optimal growth condition
the specification of the shadow uniquely fixes the desired umbral form (cf. Theorem \ref{thm:uniqueness_umbral_mock_jac} and Corollary  \ref{cor:uniqueness_umbral_mock_mod}). 

By associating a case of umbral moonshine to each Niemeier lattice we extend our earlier work on umbral moonshine to include 17 more instances. 
In fact, the 6 instances discussed in the earlier paper, labelled by the 6 divisors of 12, correspond to {\em pure A-type}  Niemeier root systems containing only A-type  irreducible components. 
There are 8 pure A-type Niemeier root systems, one for each divisor $\ll-1$ of 24, and they are given simply as the union of $\frac{24}{\ll-1}$ copies of $A_{\ll-1}$. 
This new proposal relating Niemeier lattices and umbral moonshine can be regarded as a completion of our earlier work \cite{UM}, in that it includes Niemeier root systems with D- or/and E- components and sheds important light on the underlying structure of umbral moonshine. 

More properties of umbral moonshine reveal themselves as new instances are included and as the structure of umbral moonshine is examined in light of the connection to Niemeier lattices. 
Recall that in \cite{UM}  
we observed a connection between the (extended) Dynkin diagrams and some of the groups $G^{(\ll)}$  via McKay's correspondence for subgroups of $\SU(2)$.
In the present paper we observe that the same holds for many of the new instances of umbral moonshine, and the result  
presents itself as a natural sequence of extended Dynkin diagrams with decreasing rank, starting with $\widehat E_8$ and ending with $\widehat A_1$. 
Moreover, we observe an  
interesting relation between umbral moonshine and the genus zero groups $\G <\SL_2(\RR)$ through the shadows of the former and the principal moduli for the latter.
As will be discussed in \S\ref{sec:holes:gzero} and \S\ref{sec:forms:genus0}, 
our construction attaches a principal modulus  
for a genus zero group to each Niemeier lattice. In particular, we recognise the Coxeter numbers of the root systems with an  A-type component as exactly those levels for which the corresponding classical modular curve has genus zero.

The outline of this paper is as follows.
In \S\ref{sec:holes} we give some background on Niemeier lattices, define the umbral finite groups, and discuss the mysterious relation to  extended ADE diagrams and genus zero quotients of the upper-half plane. 
In \S\ref{sec:forms} we introduce various automorphic objects that play a role in umbral moonshine, including (mock) modular forms and Jacobi forms of the weak, meromorphic, and mock type. For later use we also introduce the Eichler--Zagier (or Atkin--Lehner) map on Jacobi forms, and an ADE classification of such maps. In \S\ref{sec:forms:umbral} we focus on the umbral mock modular forms, which are conjecturally the generating functions of the dimensions of the homogeneous subspaces of the umbral modules. In \S\ref{sec:forms:umbral}-\ref{sec:mckay} we give explicit formulas for these umbral mock modular forms as well as most of the umbral McKay--Thompson series. This is  achieved partially with the help of {\em multiplicative relations}, relating McKay--Thompson series in different instances of umbral moonshine corresponding to Niemeier lattices 
with one Coxeter number being the  multiple of the nother.  In \S\ref{sec:conj} we present the main results of the paper, which are the umbral moonshine conjectures relating the umbral groups and umbral mock modular forms, and a counterpart for umbral moonshine of the genus zero property of monstrous moonshine. 
We also observe certain discriminant properties relating the exponents of the powers of $q$ in the mock modular forms and the imaginary quadratic number fields over which the homogeneous submodules of the umbral modules are defined, extending the discriminant properties observed in \cite{UM}. Finally, we  
present some conclusions and discussions in \S\ref{sec:conc}. 

To provide the  data and evidence in support of our conjectures, this paper also contains four appendices. 
In Appendix \ref{sec:modforms} we describe some modular forms and Jacobi forms which are utilised in the paper. 
In Appendix \ref{sec:chars} we present tables of irreducible characters  
as well as the characters of certain naturally defined (signed) permutation representations of the 23 umbral groups. 
In Appendix \ref{sec:coeffs} we provide the first few dozen coefficients of all the umbral McKay--Thompson series. 
In Appendix \ref{sec:decompositions}, using the tables in Appendix \ref{sec:chars} and  \ref{sec:coeffs}, we explicitly present decompositions into irreducible representations for the first 10 or so homogeneous subspaces of the umbral modules for all instances of umbral moonshine. 

\subsection{Notation}

We conclude this introduction with a guide to the most important and frequently used notation, and indications as to where the relevant definitions can be found.
\begin{list}{}{
	\itemsep -1pt
	\labelwidth 23ex
	\leftmargin 7ex
	}
\item[$X$]  A root system (cf. \S\ref{sec:holes:rootsys}). Usually $X$ is a union of irreducible simply-laced root systems with the same Coxeter number; for example, a Niemeier root system (cf. \S\ref{sec:holes:lats}).
\item[${\sf m}$] The Coxeter number of an irreducible root system $X$ (cf. \S\ref{sec:holes:rootsys}), or the Coxeter number of any irreducible component of $X$ when all such numbers coincide.
\item[${\sf r}$] The rank of a root system $X$ (cf. \S\ref{sec:holes:rootsys}).
\item[$\pi^X$] The (formal) product of Frame shapes of Coxeter elements of irreducible components of a root system $X$ (cf. \S\ref{sec:holes:rootsys}).
\item[$W^X$] The Weyl group of a root system $X$ (cf. \S\ref{sec:holes:rootsys}).
\item[$L^X$] The Niemeier lattice attached to the Niemeier root system $X$ (cf. \S\ref{sec:holes:lats}).
\item[$X_A$] The union of irreducible components of type A in a Niemeier root system $X$ (cf. \S\ref{sec:holes:lats}). Similarly for $X_D$ and $X_E$. 
\item[$d^X_A$] The number of irreducible components of type A in the root system $X$ (cf. \S\ref{sec:holes:lats}). Similarly for $d^X_D$ and $d^X_E$.
\item[$d^X$] The total number of irreducible components of the root system $X$ (cf. \S\ref{sec:holes:gzero}).
\item[$\G^X$] The genus zero subgroup of $\SL_2(\RR)$ attached to the Niemeier root system $X$ (cf. \S\ref{sec:holes:gzero}).
\item[$T^X$] A certain principal modulus for $\G^X$, for $X$ a Niemeier root system (cf. \S\ref{sec:holes:gzero}).
\item[$f^X$] A certain modular form of weight $2$ for $\G^X$, for $X$ Niemeier root system (cf. \S\S\ref{sec:holes:gzero},\ref{sec:forms:genus0}).
\item[$\ll$] A lambency. A symbol that encodes a genus zero group $\G^X$ attached to a Niemeier root system $X$, and thereby also $X$ (cf. \S\ref{sec:holes:gzero}).
\item[$G^X$] The umbral group attached to the Niemeier root system $X$ (cf. \S\ref{sec:holes:gps}). Also denoted $G^{(\ll)}$ for $\ll$ the lambency corresponding to $X$.
\item[$\bar{G}^X$] A naturally defined quotient of $G^X$ (cf. \S\ref{sec:holes:gps}). Also denoted $\bar{G}^{(\ll)}$.
\item[$\widetilde{\chi}^X$] A twisted Euler character. A certain naturally defined character of $G^X$ (cf. \S\ref{sec:holes:gps}). Similarly for $\widetilde{\chi}^{X_A}$, $\chi^{X_A}$, $\bar{\chi}^{X_A}$, \&c.
\item[$D^{(\ll)}$] The finite subgroup of $\SU(2)$ attached to the umbral group $G^{(\ll)}$ for $2<\ll<11$ (cf. \S\ref{sec:grps:dyn}).
\item[$\D^{(\ll)}$] The extended Dynkin diagram of rank $11-\ll$ attached to the umbral group $G^{(\ll)}$ for $\ll$ as above (cf. \S\ref{sec:grps:dyn}).
\item[$m$] Usually the index of a Jacobi form (cf. \S\ref{sec:forms:jac}). Often this is chosen to coincide with the Coxeter number ${\sf m}$ of some root system $X$, in which case we write $m$ for both (cf. \S\ref{sec:forms:ADE}).
\item[$\theta_{m}$] The vector-valued function whose components are the standard index $m$ theta functions $\th_{m,r}$ for $r\in \ZZ/2m\ZZ$ (cf. \S\ref{sec:forms:jac}).
\item[$\Omega^X$] The $2m\times 2m$ matrix attached to a simply-laced root system $X$ with all irreducible components having Coxeter number $m$ (cf. \S\ref{sec:forms:ADE}).
\item[${\cal W}^X$] The Eichler--Zagier operator on Jacobi forms of index $m$ attached to a simply-laced root system $X$ with all irreducible components having Coxeter number $m$ (cf. \S\ref{sec:forms:ADE}).
\item[$\psi^P$] The polar part of a meromorphic Jacobi form $\psi$ (cf. \S\ref{sec:forms:meromock}).
\item[$\psi^F$] The finite part of a meromorphic Jacobi form $\psi$ (cf. \S\ref{sec:forms:meromock}).
\item[$\mu_{m,j}$] A generalised Appell--Lerch sum of index $m$ (cf. \S\ref{sec:forms:meromock}). The function $\mu_{1,0}$ is a meromorphic Jacobi form of weight $1$ and index $1$ with vanishing finite part. More generally, scalar multiples of $\mu_{m,0}$ arise as polar parts of certain meromorphic Jacobi forms of weight $1$ and index $m$.
\item[$h$]	Usually a vector-valued mock modular form, with components $h_r$ for $r\in \ZZ/2m\ZZ$, obtained from the theta expansion of the finite part of a meromorphic Jacobi form of weight $1$ and index $m$ (cf. \S\ref{sec:forms:meromock}).
\item[$S_m$]	The vector-valued cusp form of weight $3/2$ whose components are the unary theta series $S_{m,r}$ for $r\in \ZZ/2m\ZZ$ (cf. \S\ref{sec:forms:meromock}), related to $\th_m$ by $S_{m,r}(\t)=\tfrac{1}{2\pi i}\partial_z\th_{m,r}(\t,z)|_{z=0}$.
\item[$S^X$] The vector-valued cusp form of weight $3/2$ attached to a simply-laced root system $X$ with all irreducible components having the same Coxeter number (cf. \S\ref{sec:umbral shadow}). An umbral shadow in case $X$ is a Niemeier root system.
\item[$\psi^X$] The unique meromorphic Jacobi form of weight $1$ and index $m$ satisfying the conditions of Theorem \ref{thm:uniqueness_umbral_mock_jac}, if such a function exists, where $X$ is a simply-laced root system for which all irreducible components have Coxeter number $m$ (cf. \S\ref{sec:umbral shadow}).
\item[$H^X$] The unique vector-valued mock modular form with shadow $S^X$ whose components furnish the theta expansion of the finite part of $\psi^X$, if $\psi^X$ exists (cf. \S\ref{sec:umbral shadow}). An umbral mock modular form in case $X$ is a Niemeier root system (cf. \S\ref{sec:weight_zero_umbral_forms}), and, in this situation, also denoted $H^{(\ll)}$ for $\ll$ the lambency corresponding to $X$.
\item[$\sigma^X$] The skew-holomorphic Jacobi cusp form of weight $2$ and index $m$ naturally attached to $X$, where $X$ is a simply-laced root system for which all irreducible components have Coxeter number $m$ (cf. \S\ref{sec:forms:genus0}).
\item[$H^X_g$] The umbral McKay--Thompson series attached to $g\in G^X$ for $X$ a Niemeier root system (cf. \S\ref{sec:mckay}). A vector-valued mock modular form of weight $1/2$.  Also denoted $H^{(\ll)}_g$ for $\ll$ the lambency corresponding to $X$.
\item[$S^X_g$] The vector-valued cusp form conjectured to be the shadow of $H^X_g$, for $g\in G^X$ and $X$ a Niemeier root system (cf. \S\S\ref{sec:mckay:aut},\ref{sec:conj:aut}).
\item[$\O^X_g$] The $2m\times 2m$ matrix attached to $g\in G^X$ for $X$ a Niemeier root system with Coxeter number $m$ (cf. \S\ref{sec:mckay:aut}).
\item[$K^X$]	The umbral module attached to the Niemeier root system $X$. A conjectural $G^X$-module with graded-super-characters given by the $H^X_g$ (cf. \S\ref{sec:conj:mod}).
\item[$n_g$] The order of the image of an element $g\in G^X$ in the quotient group $\bar{G}^X$ (cf. \S\ref{sec:conj:aut}).
\item[$h_g$] The unique positive integer such that $n_gh_g$ is the product of the shortest and longest cycle lengths in the cycle shape $\widetilde{\Pi}^X_g$ for $g\in G^X$ and $X$ a Niemeier root system (cf. \S\ref{sec:conj:aut}).
\item[$\widetilde{\Pi}^X_g$] The cycle shape attached to $g\in G^X$ via the permutation representation of $G^X$ with twisted Euler character $\widetilde{\chi}^X$ (cf. \S\S\ref{sec:holes:gps},\ref{sec:conj:aut},\ref{sec:chars:eul}). Similarly for $\widetilde{\Pi}^{X_A}_g$, $\widetilde{\Pi}^{X_D}_g$, \&c.
\item[$\nu^X_g$] The multiplier system of $H^X_g$ (cf. \S\ref{sec:conj:aut}).
\end{list}

%-----------------------------------------------------------------------------------%
\section{Groups}\label{sec:holes}
%-----------------------------------------------------------------------------------%

\subsection{Root Systems}\label{sec:holes:rootsys}

In this subsection we give a brief summary of  simply-laced root systems and their corresponding {\em Dynkin diagrams}. Standard references for this
material include \cite{MR1890629} and \cite{MR0323842}.

Let $V$ be a finite-dimensional vector space of rank $\sf{r}$ over $\RR$ equipped with an  inner product $\langle \cdot, \cdot \rangle$. For $v \in V$ define the hyperplane $H_v $ to be the set of elements of $V$ orthogonal to $v$ and the reflection in the hyperplane $H_v$ to be the linear map
$r_v: V \rightarrow V$ defined by
\be
r_v(v') = v'- 2 \frac{\langle v,v' \rangle}{\langle v,v \rangle}v .
\ee

A finite subset $X \subset V$ of non-zero vectors is a rank ${\sf r}$  crystallographic {\em root system } if
\begin{itemize}
\item $X$ spans $V$,
\item  $r_\alpha(X)\in X$ for all $\alpha \in X$, 
\item  $X \cap \RR \alpha = \{\alpha,-\alpha \}$ for all $\alpha \in X$,
\item $2 \langle \alpha,\beta \rangle/\langle \alpha,\alpha \rangle \in \ZZ$ for all $\alpha, \beta \in X$.
\end{itemize}

Given a root system $X$ we say that $X$ is {\em irreducible} provided that it can not be partitioned into proper subsets
$X= X_1 \cup X_2$ with $\langle \alpha_1, \alpha_2 \rangle=0$ for all $\alpha_1 \in X_1$ and $\alpha_2 \in X_2$.
If $X$ is an irreducible root system then there are at most two values for the length $\langle \alpha,\alpha \rangle^{1/2}$ that occur. If all roots have
the same length then the irreducible root system is called {\em simply-laced}.

It is possible to choose a subset of roots in $X$ that form a basis of $V$. We define a subset $\Phi=   
	\{f_1,f_2,\cdots,f_{{\sf r}}\} \subset X$ to be a set of {\em simple roots} provided that
\begin{itemize}
\item $\Phi$ is a basis for $V$,
\item each root $\alpha \in X$ can be written as a linear combination of the  
$f_i$ with integer coefficients
\be \label{simprotdecom}
\alpha = \sum_{i=1}^{\sf r}  n_i 
	f_i
\ee
and with either all $n_i \le 0$ or all $n_i \ge 0$. 
\end{itemize}
Given a choice of simple roots we define the positive roots of $X$ to be those $\alpha$ for which all $n_i \ge 0$ in (\ref{simprotdecom}). The negative roots are those for which all $n_i \le 0$. 
We also define the {\em height}
of $\alpha$ as in (\ref{simprotdecom})  
by setting
\be
{\rm ht}(\alpha) = \sum_{i=1}^{\sf r} n_i.
\ee

To each irreducible root system we can associate a connected Dynkin diagram as follows. We associate a node to each simple root. The nodes associated to two distinct simple roots  
$f_i,f_j$ are then either not connected if $\langle 
f_i,f_j\rangle =0$ or connected by $N_{i j}$ lines with
\be
N_{ij}= \frac{2 \langle f_i, f_j \rangle}{\langle f_i, f_i \rangle} \frac{2 \langle f_j, f_i \rangle}{\langle f_j, f_j \rangle}~ \in \{1,2,3 \}.
\ee
The Dynkin diagrams associated to simply-laced irreducible root systems all have $N_{ij} =\{0,1 \}$ and 
are of type $A_n, D_{n}, E_6, E_7, E_8$ as shown in Figure \ref{fig:dynkin}. Here the subscript indicates the rank of the associated root system, and in the figure we choose a specific enumeration of simple roots for later use in \S\ref{sec:holes:gps}.

\begin{figure}[h]
\begin{center}
\includegraphics[scale=0.3]{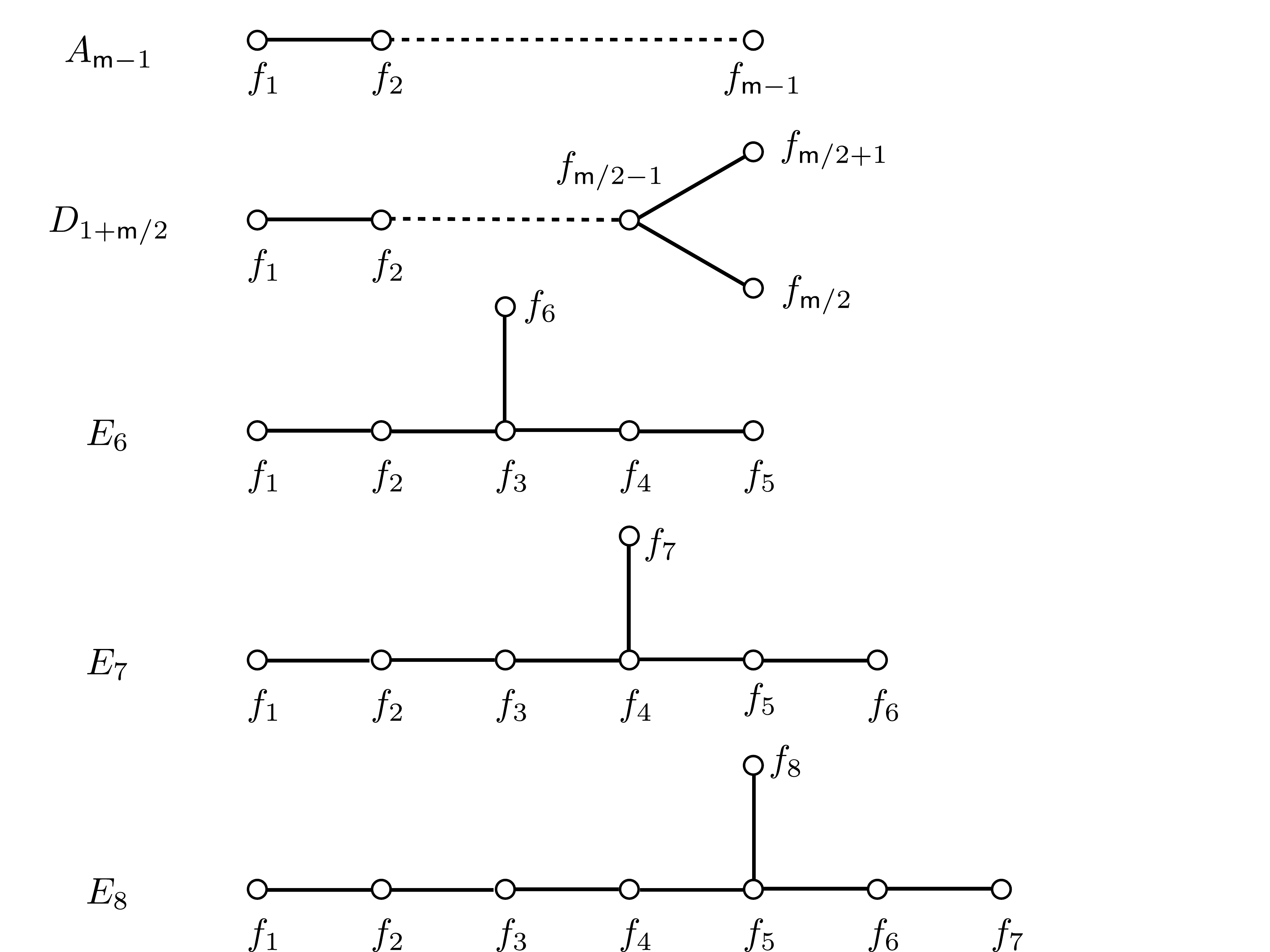}
\caption{The ADE Dynkin diagrams\label{fig:dynkin}}
\end{center}
\end{figure}

The height function defines a $\ZZ$-gradation on the set of roots.  Every irreducible root system has a unique root  $\theta$ of largest height with respect to a given set of simple roots $\Phi=\{f_i\}$ with an expansion
\be
\theta = \sum_{i=1}^{\sf r} a_i 	f_i
\ee
where the $a_i$ are a set of integers known as the {\em Coxeter labels} of the root system or Dynkin diagram. 
If we append the negative of this highest root (the lowest root) to the simple roots of the simply-laced root system, we obtain the extended Dynkin diagrams of type $\widehat A_n, \widehat D_{n}, \widehat E_6, \widehat E_7, \widehat E_8$.   These are shown
in Figure \ref{fig:edynkin}, where we indicate the lowest root with a filled in circle and the simple roots with empty circles.

\begin{figure}[h]
\begin{center}
\includegraphics[scale=0.35]{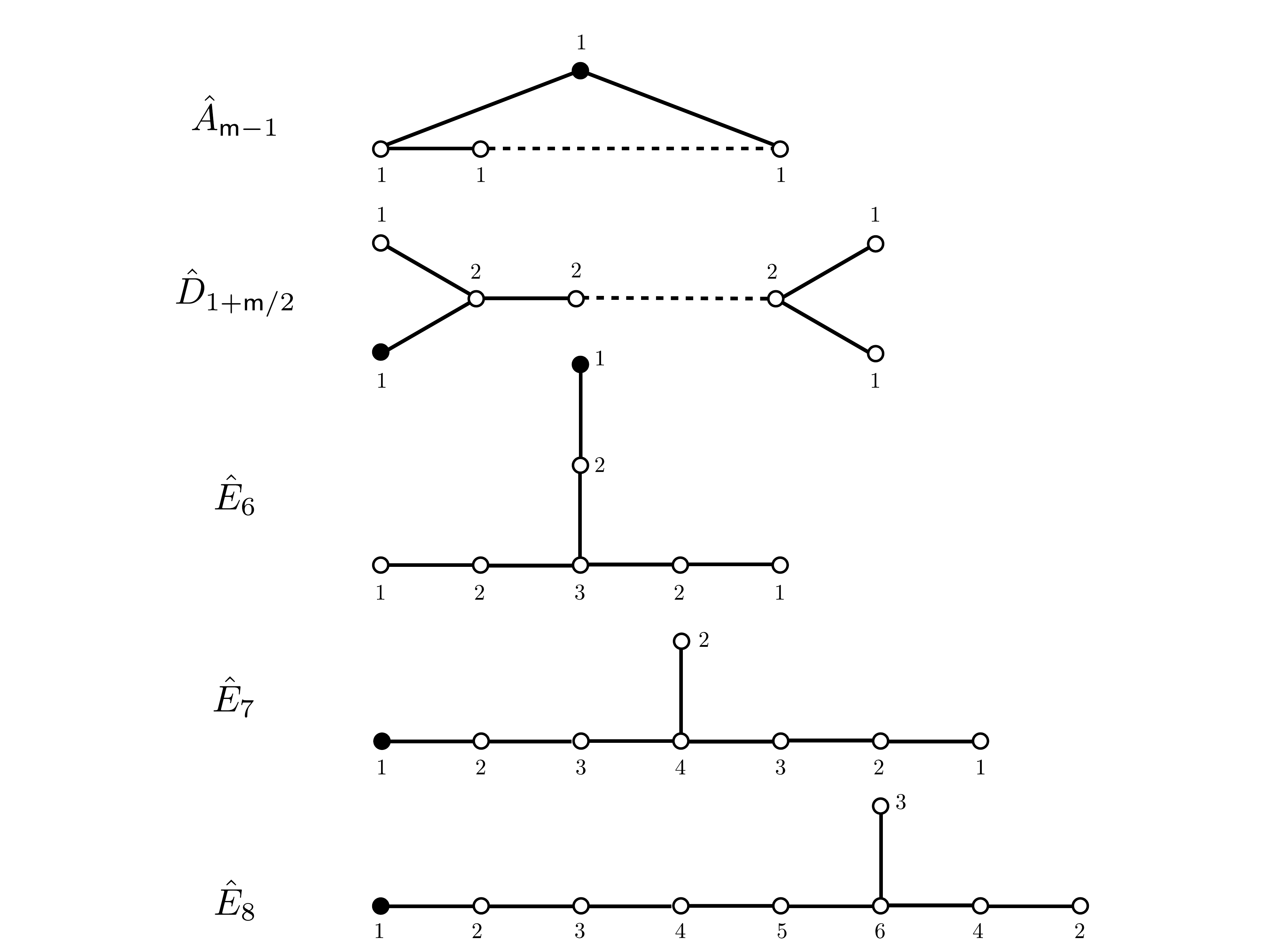}
\caption{The extended ADE Dynkin diagrams\label{fig:edynkin}}
\end{center}
\end{figure}

Given  
an irreducible root system $X$, its {\em Coxeter number} ${\sf m}={\sf m}(X)$ 
is the sum
\be\label{def_coxeter_number}
{\sf{m} }= 1 + \sum_{i=1}^{\sf r} a_i.
\ee
An equivalent definition of the Coxeter number may be given in terms of the Weyl group of $X$. The {\em Weyl group} $W^X$ is the group generated by the 
reflections $r_{\alpha}$ for $\alpha\in X$. The product $w = r_1 r_2 \dots r_{\sf r}$ of reflections $r_i:=r_{f_i}$ in simple roots $f_i \in \Phi$ is called a {\em Coxeter element} of $W^X$ and is uniquely determined up to conjugacy in $W^X$, meaning that different choices of simple roots and different orderings of the simple roots chosen lead to conjugate elements of $W^X$. The Coxeter number ${\sf m}={\sf m}(X)$ is then the order of any Coxeter element of $X$. 

We obtain a finer invariant of $X$ by considering the eigenvalues of a Coxeter element of $W^X$. Say $u_1,\dots, u_{\sf r}$ are the {\em Coxeter exponents} of $X$ if a Coxeter element $w$ has eigenvalues $e^{2\p i u_1/\sf{m}},\dots,e^{2\p i u_{\sf r}/\sf{m}}$ (counting multiplicity). This data is conveniently recorded using the notion of {\em Frame shape}, whereby a formal product $\prod_in_i^{k_i}$ (with $n_i,k_i\in\ZZ$ and $n_i>0$) serves as a shorthand for the rational polynomial $\prod_i(x^{n_i}-1)^{k_i}$. For each Coxeter element there is a Frame shape $\pi^X$---the {\em Coxeter Frame shape} of $X$---such that the corresponding polynomial function coincides with the characteristic polynomial $\prod_{i=1}^{\sf r}(x-e^{2\pi i u_i/{\sf m}})$ of $w$. These Frame shapes will play a prominent role in what follows. They are given along with the corresponding Coxeter numbers in Table \ref{tab:CoxNum}.

\begin{table}[h]
\captionsetup{font=small}
\centering
\begin{tabular}{c|cccccc}
\toprule
	&$A_{\sf m-1}$&$D_{1+{\sf m}/2}$&$E_6$&$E_7$&$E_8$\\
	\midrule
	Coxeter &\multirow{2}*{${\sf m}$}&\multirow{2}*{${\sf m}$}&\multirow{2}*{$12$}&\multirow{2}*{$18$}&\multirow{2}*{$30$} \\ 
	number\\\midrule
 Coxeter & \multirow{2}*{$1,2,3,\dots,{\sf m-1}$} & $1,3,5,\dots,{\sf m}-1,$ & 1,4,5,& 1,5,7,9,&1,7,11,13,\\
 exponents&&${\sf m}/{2}$&$7,8,11$&11,13,17 &17,19,23,29\\ \midrule
	Coxeter &\multirow{2}*{$\frac{{\sf m}}{1}$}&\multirow{2}*{$\frac{2.{\sf m}}{1.({\sf m}/{2})}$}&\multirow{2}*{$\frac{2.3.12}{1.4.6}$}&\multirow{2}*{$\frac{2.3.18}{1.6.9}$}&\multirow{2}*{$\frac{2.3.5.30}{1.6.10.15}$} \\ 
	Frame shapes\\
 \bottomrule
\end{tabular}
\caption{\label{tab:CoxNum}{Coxeter numbers, exponents, and Frame shapes}}
\end{table}

\subsection{Lattices}\label{sec:holes:lats}

A {\em lattice} is a free $\ZZ$-module 
equipped with a symmetric bilinear form $\lab\cdot\,,\cdot\rab$. We say that a lattice $L$ is {positive-definite} if $\lab\cdot\,,\cdot\rab$ induces a positive-definite inner product on the vector space $L_{\RR}=L\otimes_{\ZZ}\RR$. 
Since $L$ is a free $\ZZ$-module the natural map $L\to L_{\RR}$ is an embedding and we may identify $L$ with its image in $L_{\RR}$. Say that $L$ is {integral} if we have $\lab\l,\m\rab\in\ZZ$ for all $\l,\m\in L$ and say that $L$ is {even} if we have $\lab\l,\l\rab\in 2\ZZ$ for each $\l\in\LL$. (An even lattice is necessarily integral.) The {dual} of $L$ is the lattice $L^*\subset L_{\RR}$ defined by setting 
\begin{gather}
	L^*=\{\l\in L_{\RR}\mid 	
							\lab\l,L\rab\subset\ZZ\}.
\end{gather}
Clearly, if $L$ is integral then $L^*$ contains $L$. In the case that $L^*$ coincides with (the image of) $L$ (in $L_{\RR}$) we say that $L$ is unimodular. For an even lattice $L$ we call $L_2=\{\l\in L\mid \lab\l,\l\rab=2\}$ the set of {\em roots} of $L$.

The {Leech lattice} is the unique (up to isomorphism) even, unimodular,  positive-definite lattice of rank $24$ with no roots \cite{Con_ChrLeeLat}, and is named for its discoverer, John Leech \cite{Lee_SphPkgs,Lee_SphPkgHgrSpc}. Shortly after Leech's work, the unimodular even positive-definite lattices of rank $24$ were classified by Niemeier \cite{Nie_DefQdtFrm24}; we refer to those with non-empty root sets as the  Niemeier lattices. There are exactly $23$ Niemeier lattices up to isomorphism, and if $L$ is such a lattice then its isomorphism type is determined by its root set $L_2$, which is a union of irreducible simply-laced root systems (cf. \S\ref{sec:holes:rootsys}). Say a root system $\rs$ is a {Niemeier root system} if it occurs as $L_2$ for some Niemeier lattice $L$. The {Niemeier root systems} are  precisely the 23 root systems satisfying the two conditions that first, they are unions of simply-laced root systems with the same Coxeter numbers, and second, the total rank is 24.  Explicitly, they are
\begin{gather}
A_1^{24},\,A_2^{12},\,A_3^{8},\,A_4^6,\,A_5^4D_4,\,A_6^4,\,A_7^2D_5^2,\,A_8^3,\,A_9^2D_6,\,A_{11}D_7E_6,\,A_{12}^2,\,A_{15}D_9,\,A_{17}E_7,\,A_{24},
\label{eqn:holes:NieRoot_A}\\
D_4^6,\,D_6^4,\,D_8^3,\,D_{10}E_7^2,\,D_{12}^2,\,D_{16}E_8,\,D_{24},
\label{eqn:holes:NieRoot_D}\\
E_6^4,\,E_8^3.
\label{eqn:holes:NieRoot_E}
\end{gather}
In (\ref{eqn:holes:NieRoot_A}) we list the Niemeier root systems containing a type $A$ component, in (\ref{eqn:holes:NieRoot_D}) we list the root systems containing a type $D$ component but no type $A$ component, and the remaining two root systems, having only type $E$ components, appear in (\ref{eqn:holes:NieRoot_E}). We will call them the $A${\em -type}, $D${\em -type}, and the $E${\em -type} Niemeier root systems, respectively. 
We say that a Niemeier root system $\rs$ has {\em Coxeter number} $\sf{m}$ if $\sf{m}$ is the Coxeter number of any simple component of $\rs$.

Since all the simple components of a Niemeier root system have the same Coxeter number all the type $A$ components appearing have the same rank, and similarly for components of type $D$ and $E$. So we can write 
\begin{gather}\label{eqn:holes:lats:ADEdecomp}
	\rs=
	\rs_A\rs_D\rs_E
\end{gather}
where $\rs_A=A_{{\sf m}-1}^{d^{\rs}_A}$ for some non-negative integer $d^{\rs}_A$ (or $\rs_A=\emptyset$), and ${\sf m}$ the Coxeter number of $\rs$, and similarly for $\rs_D$ and $\rs_E$. For example,
\begin{gather}\label{eqn:holes:lats:ADEtypeofrs}
\text{if $\rs=A_7^2D_5^2$ then ${\sf m}=8$, $\rs_A=A_7^2$, $\rs_D=D_5^2$, $d^{\rs}_A=d^{\rs}_D=2$ and $\rs_E=\emptyset$.}
\end{gather}

Before finishing this subsection we will comment on the relation between the Niemeier lattices and the Leech lattice. The {covering radius} of the Leech lattice is $\sqrt{2}$ according to \cite{ConParSlo_CvgRadLeeLat}, meaning that $R=\sqrt{2}$ is the minimal positive $R$ such that the $24$-dimensional vector space $\LL_{\RR}=\LL\otimes_{\ZZ}\RR$ is covered by placing a closed ball of radius $R$ at each point of $\LL$,
\begin{gather}
	\sqrt{2}=
	\supr_{x\in\LL_{\RR}} \infm_{\l\in\LL}\| x-\l \|.
	\end{gather}
A point $x\in \LL_{\RR}$ that realizes the maximum value $\sqrt{2}= \infm_{\l\in\LL}\|x-\l\|$ is called a {\em deep hole} of $\LL$. Let $x\in\LL_{\RR}$ be a deep hole and let $V_x$ be the set of {vertices of $x$},
\begin{gather}
	V_x=\left\{\l\in\LL\mid \|x-\l\|=\sqrt{2}\right\}.
\end{gather}
It is shown in \cite{ConParSlo_CvgRadLeeLat} that if $\l,\l'\in V_x$ with $\l\neq \l'$ then $\|\l-\l'\|^2\in\{4,6,8\}$. Following \cite{ConParSlo_CvgRadLeeLat} define the {hole diagram} attached to $x$ by joining vertices $\l,\l'\in V_x$ with a single edge if $\|\l-\l'\|^2=6$, and by joining them with a double edge if $\|\l-\l'\|^2=8$. The vertices $\l$ and $\l'$ are disjoined in case $\|\l-\l'\|^2=4$. Then the diagram so obtained is the extended Dynkin diagram corresponding to a Niemeier root system, and all Niemeier root systems arise in this way \cite{ConParSlo_CvgRadLeeLat}.
Conversely, from each Niemeier lattice one can obtain a different ``holy" construction of the Leech lattice \cite{MR661720}. 

\subsection{Genus Zero Groups}\label{sec:holes:gzero}

In this section we attach a genus zero subgroup of $\SL_2(\RR)$ to each of the $23$ Niemeier root systems.

If $\G$ is a discrete subgroup of $\SL_2(\RR)$ that is commensurable with the modular group $\SL_2(\ZZ)$ then its natural action on the boundary $\widehat{\RR}=\RR\cup\{\ii\infty\}$ of the upper half plane $\HH$ restricts to $\widehat{\QQ}=\QQ\cup\{\ii\infty\}$. The orbits of $\G$ on $\widehat{\QQ}$ are called the {\em cusps} of $\G$, and the quotient space
\begin{gather}\label{eqn:sums:XG}
	X_{\G}=\G\backslash\HH\cup\widehat{\QQ}
\end{gather}
is naturally a compact Riemann surface (cf. e.g. \cite[\S1.5]{Shi_IntThyAutFns}). We adopt the common practice of saying that $\G$ has {\em genus zero} in case $X_{\G}$ is a genus zero surface.

For $n$ a positive integer the {\em Hecke congruence group} of level $n$, denoted $\Gamma_0(n)$, is defined by setting
\begin{gather}\label{def:hecke_congruence}
     \Gamma_0(n)
     =\left\{
     \left.
  \bem
         a & b \\
         cn & d \\
       \eem\right|
     a,b,c,d\in \ZZ,\,ad-bcn=1
     \right\}.
\end{gather}
Say $e$ is an {\em exact divisor} of $n$, and write $e\|n$, if $e|n$ and $(e,n/e)=1$. According to \cite{conway_norton} the normaliser $N(\Gamma_0(n))$ of $\Gamma_0(n)$ in $\SL_2(\RR)$ is commensurable with $\SL_2(\ZZ)$ and admits
the description
\begin{gather}\label{eqn:conven:groups:Normalizer_Gamma0(n)}
     N(\Gamma_0(n))
     =\left\{\left.
     \frac{1}{\sqrt{e}}
     \left(
       \begin{array}{cc}
         ae & b/h \\
         cn/h & de \\
       \end{array}
     \right)\right|
     a,b,c,d\in \ZZ,\,e\|n/h,\,ade-bcn/eh^2=1
     \right\}
\end{gather}
where $h$ is the largest divisor of $24$ such that $h^2$ divides
$n$. So if $e\|n$ then we obtain a coset $W_n(e)$ for $\G_0(n)$ in its normaliser by setting
\begin{gather}\label{def:AT_inv}
	W_n(e)
	     =\left\{\left.
	\frac{1}{\sqrt{e}}
     \left(
       \begin{array}{cc}
         ae & b \\
         cn & de \\
       \end{array}
     \right)\right|
     a,b,c,d\in \ZZ,\,e\|n,\,ade-bcn/e=1
     \right\}.
\end{gather}
Observe that the product of any two elements of $W_n(e)$ lies in $W_n(1)=\G_0(n)$. More generally, the operation $e*f=ef/(e,f)^2$ equips the set of exact divisors of $n$ with a group structure isomorphic to that obtained by multiplication of Atkin--Lehner involutions, $W_n(e) W_n(f)=W_n(e*f)$. So for $S$ a subgroup of the group of exact divisors of $n$ we may define a group $\G_0(n)+S$, containing and normalizing $\G_0(n)$, by taking the union of the Atkin--Lehner cosets $W_n(e)$ for $e\in S$. It is traditional \cite{conway_norton} to simplify this notation by writing $\G_0(n)+e,f,\ldots$ in place of $\G_0(n)+\{1,e,f,\ldots\}$. Thus we have
\begin{gather}\label{eqn:conven:groups:Gamma0(n|h)+S}
     \Gamma_0(n)+S
     =\left\{\left.
     \frac{1}{\sqrt{e}}
     \left(
       \begin{array}{cc}
         ae & b \\
         cn & de \\
       \end{array}
     \right)
     \right|
     a,b,c,d\in \ZZ,\,
     e\in S,\,
     ade-bcn/e=1
     \right\}.
\end{gather}

The positive integers occurring as Coxeter numbers of the A-type  Niemeier root systems (cf. (\ref{eqn:holes:NieRoot_A})) are 
\begin{gather}\label{a_type_coxeter_list}
2,3,4,5,6,7,8,9,10,12,13,16,18,25.
\end{gather}
Observe that these are exactly the positive integers $n>1$ for which the {Hecke congruence group} $\G_0(n)$ has genus zero (cf. e.g. \cite{Fer_Genus0prob}). The Coxeter numbers of the D-type Niemeier root systems (cf. (\ref{eqn:holes:NieRoot_D})) are 
\begin{gather}
6,10,14,18,22,30,46,
\end{gather}
and these are exactly the even integers $2n$ such that the group $\G_0(2n)+n$ has genus zero \cite{Fer_Genus0prob}. We will demonstrate momentarily that the root system $E_6^4$, having Coxeter number $12$, is naturally attached to the genus zero group $\G_0(12)+4$, and $E_8^3$, having Coxeter number $30$, is naturally attached to the genus zero group $\G_0(30)+6,10,15$. 
As such, we obtain a correspondence between the $23$ Niemeier root systems and the genus zero groups of the form 
\begin{gather}
\G_0(n), \quad\G_0(2n)+n,\quad \G_0(12)+4,\quad \G_0(30)+6,10,15.
\end{gather}
Write $\G^X$ for the genus zero subgroup of $\SL_2(\RR)$ associated in this way to a Niemeier root system $X$. Write $T^X$ for the unique principal modulus for $\G^X$ that has an expansion 
\begin{gather}
	T^X=q^{-1}-d^X+O(q)
\end{gather} 
about the infinite cusp, where $d^X$ denotes the number of irreducible components of $X$. Then we may recover $T^X$, and hence also $\G^X$, directly, as follows, from the Coxeter Frame shapes (cf. \S\ref{sec:holes:rootsys}) of the irreducible components of $X$.

\begin{table}[h]
\captionsetup{font=small}
\begin{center}
\caption{Niemeier Root Systems and Principal Moduli}\label{tab:Hauptmodul}
\medskip

\begin{tabular}{ccccccccccc}
\multicolumn{1}{c|}{$\rs$}&$A_1^{24}$&$A_2^{12}$&$A_3^8$&$A_4^6$&$A_5^4D_4$&$A_6^4$&$A_7^2D_5^2$\\
	\cmidrule{1-8}
\multicolumn{1}{c|}{$\ll$}&	2&	3&	4&	5&	6&	7&	8\\
	\cmidrule{1-8}
\multicolumn{1}{c|}{$\pi^{\rs}$}&			$\frac{2^{24}}{1^{24}}$&	$\frac{3^{12}}{1^{12}}$&	$\frac{4^{8}}{1^{8}}$&	$\frac{5^{6}}{1^{6}}$&	$\frac{2^16^{5}}{1^{5}3^1}$&	$\frac{7^{4}}{1^{4}}$&$\frac{2^{2} 8^4}{1^{4}4^2}$\vspace{0.2em}\\
	\cmidrule{1-8}
\multicolumn{1}{c|}{$\G^{\rs}$}&			$2B$&	$3B$&	$4C$&	$5B$&	$6E$&	$7B$&$8E$\\
\\
\multicolumn{1}{c|}{$\rs$}&$A_8^3$&$A_9^2D_6$&$A_{11}D_7E_6$&$A_{12}^2$&$A_{15}D_9$&$A_{17}E_7$&$A_{24}$\\
	\cmidrule{1-8}
\multicolumn{1}{c|}{$\ll$}&	9&	10& 12&	13&	16&	18&	25\\
	\cmidrule{1-8}
\multicolumn{1}{c|}{$\pi^{\rs}$}&$\frac{9^{3}}{1^{3}}$&$\frac{2^1 10^{3}}{1^{3} 5^1}$&$\frac{2^{2} 3^112^3 }{1^{3} 4^1 6^2}$&$\frac{13^{2}}{1^2}$& $\frac{2^116^{2}}{1^{2}8^1}$&$\frac{2^1 3^1 18^{2}}{1^{2} 6^1 9^1}$&$\frac{25^1}{1^1}$\vspace{0.2em}\\
	\cmidrule{1-8}
\multicolumn{1}{c|}{$\G^{\rs}$}&$9B$&$10E$&$12I$&	$13B$& $16B$&$18D$&$(25Z)$\\
\\
\multicolumn{1}{c|}{$\rs$}&$D_4^{6}$&$D_6^{4}$&$D_8^3$&$D_{10}E_7^2$&$D_{12}^2$&$D_{16}E_8$&$D_{24}$\\
	\cmidrule{1-8}
\multicolumn{1}{c|}{$\ll$}& 6+3&	10+5&	14+7&	18+9&	22+11&	30+15&	46+23\\
	\cmidrule{1-8}
\multicolumn{1}{c|}{$\pi^{\rs}$}&$\frac{2^6 6^6 }{1^6 3^6}$&	$\frac{2^4 10^4 }{1^4 5^4}$&	$\frac{2^3 14^3 }{1^3 7^3}$&	$\frac{2^3 3^2 18^3 }{1^3 6^2 9^3}$&	$\frac{2^2 22^2 }{1^2 11^2}$&	$\frac{2^2 3^1 5^1  30^2 }{1^2 6^1 10^1  15^2}$&$\frac{2^1 46^1 }{1^1 23^1}$\vspace{0.2em}\\
	\cmidrule{1-8}
\multicolumn{1}{c|}{$\G^{\rs}$}&$6C$&	$10B$&	$14B$&	$18C$&	$22B$&	$30G$&$46AB$\\
 \\
\multicolumn{1}{c|}{$\rs$}&$E_6^4$&$E_8^3$\\
	\cmidrule{1-3}
\multicolumn{1}{c|}{$\ll$}	&12+4&	30+6,10,15 &
\\
	\cmidrule{1-3}
\multicolumn{1}{c|}{$\pi^{\rs}$}&$\frac{2^{4} 3^4 12^4}{1^{4} 4^4 6^4}$&$\frac{2^{3} 3^3 5^3 30^3}{1^{3} 6^3 10^3 15^3}$\vspace{0.2em}\\
	\cmidrule{1-3}
\multicolumn{1}{c|}{$\G^{\rs}$}&$12B$&$30A$\\
\end{tabular}
\end{center}
\end{table}

Define the Coxeter Frame shape $\pi^X$ of an arbitrary root system $X$ to be the product of Coxeter Frame shapes of the irreducible components of $X$. Next, for a Frame shape $\pi=\prod_in_i^{k_i}$ define the associated {\em eta product} $\eta_{\pi}$ by setting
\begin{gather}
	\eta_{\pi}(\t)=\prod_i \eta(n_i\t)^{k_i},
\end{gather}
and observe that if $X$ is simply laced and $\pi^X$ is the Coxeter Frame shape of $X$ then 
\begin{gather}
\frac{1}{\eta_{\pi^X}(\t)}=q^{-{\sf r}/24}(1-d^Xq+O(q^2))
\end{gather} 
where ${\sf r}$ denotes the rank of $X$ and $d^X$ is the number of irreducible components. We may also consider the {\em lambda sum} $\l_{\pi}$ attached to a Frame shape $\pi$, which is the function
\begin{gather}
	\lambda_{\pi}(\t)=\sum_i k_i\l_{n_i}(\t)
\end{gather}
where $\l_n(\t)$ is defined in (\ref{Eisenstein_form}). Observe that if $\pi=\prod_i n_i^{k_i}$ is such that $\sum_i k_i=0$ then $\l_{\pi}=q\partial_q\log\eta_{\pi}$.

The Coxeter Frame shapes of the Niemeier root systems are given in Table \ref{tab:Hauptmodul}. By inspection we obtain the following result.
\begin{prop}\label{prop:TX}
If $X$ is a Niemeier root system and $\pi^X$ is the Coxeter Frame shape of $X$ then 
\begin{gather}
	T^X=\frac{1}{\eta_{\pi^X}}
\end{gather}
is the unique principal modulus for $\G^X$ satisfying $T^X=q^{-1}-d^X+O(q)$ as $\t\to i\infty$.
\end{prop}

\begin{rmk}
Niemeier's classification of even unimodular positive-definite lattices of rank $24$ together with Proposition \ref{prop:TX} imply that if $X$ is the root system of an even unimodular positive-definite lattice of rank $24$ then the eta product of the Coxeter Frame shape of $X$ is a principal modulus for a genus zero subgroup of $\SL_2(\RR)$. It would be desirable to have a conceptual proof of this fact.
\end{rmk}

The relation between the $T^X$ and umbral moonshine will be discussed in \S\ref{sec:forms:genus0}, where the weight two Eisenstein forms 
\begin{gather}
f^X=\l_{\pi^X}
\end{gather}
will play a prominent role. (Cf. Table \ref{tab:ADEfX}.) We have $\sum_ik_i=0$ when $\pi^X=\prod_in_i^{k_i}$ for every Niemeier root system $X$, so the functions $f^X$ and $T^X$ are related by 
\begin{gather}
	f^X=\l_{\pi}=q\partial_q\log\eta_{\pi}=-q\partial_q\log T^X
\end{gather}
for $\pi=\pi^X$.

It is interesting to note that all of the $\G^X$, except for $X=A_{24}$, appear in monstrous moonshine as genus zero groups for whom monstrous McKay--Thompson series serve as principal moduli. 
Indeed, all of the Frame shapes $\pi^X$, except for $X=A_{24}$, are Frame shapes of elements of Conway's group $\Co_0$, the automorphism group of the Leech lattice (cf. \cite[\S7]{conway_norton}). We observe that for the cases that $\pi^X$ is the Frame shape of an element in $\Co_0$ the corresponding centralizer in $\Co_0$ typically contains a subgroup isomorphic to $G^X$.

We include the ATLAS names \cite{atlas} (see also \cite{conway_norton}) for the monstrous conjugacy classes corresponding to the groups $\G^X$ via monstrous moonshine in the rows labelled $\G^X$ in Table \ref{tab:Hauptmodul}. Extending the notation utilised in \cite{UM} we assign {\em lambencies} $\ll$---now symbols rather than integers---to each Niemeier system $X$ according to the prescription of Table \ref{tab:Hauptmodul}. The lambencies then serve to name the groups $\G^X$ also, according to the convention that $n$ corresponds to $\G_0(n)$, and $12+4$ corresponds to $\G_0(12)+4$, \&c. It will be convenient in what follows to sometimes use ${(\ll)}$ in place of $X$, writing $G^{(\ll)}$, $H^{(\ll)}$, \&c., to 
label the finite groups and mock modular forms associated to the corresponding Niemeier root system.

\subsection{Umbral Groups}\label{sec:holes:gps}

Given a Niemeier root system $\rs$ we may consider the automorphism group of the associated Niemeier lattice $L^{\rs}$. The reflections in roots of $L^{\rs}$ generate a normal subgroup of the full automorphism group of $L^{\rs}$---the Weyl group of $X$---which we denote $W^{\rs}$. We define $G^{\rs}$ to be the corresponding quotient,
\begin{gather}\label{def:umbral_group}
	G^{\rs}=\Aut(L^{\rs})/W^{\rs}.
\end{gather}
The particular groups $G^{\rs}$ arising in this way are displayed in Table \ref{tab:mugs}. Observe\footnote{We are grateful to George Glauberman for first alerting us to this fact.} that the group $G^{(\ll)}$ of \cite{UM} appears here as $G^{\rs}$ for $\rs$ the unique root system with a component $A_{\ll-1}$. In fact, the $G^{(\ll)}$ of \cite{UM} are exactly those $G^{\rs}$ for which $\rs$ 
is of the form $\rs=A_{\ll-1}^d$ with even $d$. It will develop that, for every Niemeier root system $X$, the representation theory of $G^{\rs}$ is intimately related to a set of vector-valued mock modular forms $H^{\rs}_g$, to be introduced in \S\ref{sec:forms:umbral}-\ref{sec:mckay}. 

\begin{table}[h]
\captionsetup{font=small}
\begin{center}
\caption{Umbral Groups}\label{tab:mugs}
\medskip

\begin{tabular}{ccccccccccc}
\multicolumn{1}{c|}{$\rs$}&$A_1^{24}$&$A_2^{12}$&$A_3^8$&$A_4^6$&$A_5^4D_4$&$A_6^4$&$A_7^2D_5^2$\\
	\cmidrule{1-8}
\multicolumn{1}{c|}{$\ll$}&	2&	3&	4&	5&	6&	7&	8\\
	\cmidrule{1-8}
\multicolumn{1}{c|}{$G^{\rs}$}&			$M_{24}$&	$2.M_{12}$&	$2.\AGL_3(2)$&	$\GL_2(5)/2$&	$\GL_2(3)$&	$\SL_2(3)$&$\Dih_4$\\
\multicolumn{1}{c|}{$\bar{G}^{\rs}$}&		$M_{24}$&	$M_{12}$&	$\AGL_3(2)$&		$\PGL_2(5)$&	$\PGL_2(3)$&	$\PSL_2(3)$&$2^2$\\
\\
\multicolumn{1}{c|}{$\rs$}&$A_8^3$&$A_9^2D_6$&$A_{11}D_7E_6$&$A_{12}^2$&$A_{15}D_9$&$A_{17}E_7$&$A_{24}$\\
	\cmidrule{1-8}
\multicolumn{1}{c|}{$\ll$}&	9&	10& 12&	13&	16&	18&	25\\
	\cmidrule{1-8}
\multicolumn{1}{c|}{$G^{\rs}$}&$\Dih_6$&$4$&			$2$&	$4$& $2$&$2$&$2$\\
\multicolumn{1}{c|}{$\bar{G}^{\rs}$}&$\Sym_3$&$2$&		$1$&$2$& $1$&$1$&$1$\\
\\
\multicolumn{1}{c|}{$\rs$}&$D_4^{6}$&$D_6^{4}$&$D_8^3$&$D_{10}E_7^2$&$D_{12}^2$&$D_{16}E_8$&$D_{24}$\\
	\cmidrule{1-8}
\multicolumn{1}{c|}{$\ll$}& 6+3&	10+5&	14+7&	18+9&	22+11&	30+15&	46+23\\
	\cmidrule{1-8}
\multicolumn{1}{c|}{$G^{\rs}$}&			$3.\Sym_6$&	$\Sym_4$&	$\Sym_3$&	$2$&	$2$&	$1$&$1$\\
\multicolumn{1}{c|}{$\bar{G}^{\rs}$}&		$\Sym_6$&	$\Sym_4$&	$\Sym_3$&	$2$&	$2$&	$1$&$1$\\\\
\multicolumn{1}{c|}{$\rs$}&$E_6^4$&$E_8^3$\\
	\cmidrule{1-3}
\multicolumn{1}{c|}{$\ll$}	&12+4&	30+6,10,15 &
\\
	\cmidrule{1-3}
\multicolumn{1}{c|}{$G^{\rs}$}&	$\GL_2(3)$&$\Sym_3$&\\
\multicolumn{1}{c|}{$\bar{G}^{\rs}$}&	$\PGL_2(3)$&$\Sym_3$&
\end{tabular}
\end{center}
\end{table}

As mentioned in \S\ref{sec:holes:gzero}, it will often be useful to use the lambencies to label the groups and mock modular forms associated to a given Niemeier root system. 
To this end we define $G^{(n)}=G^{\rs}$ in case $\G_0(n)$ has genus zero and $\rs$ is the unique A-type  Niemeier root system with Coxeter number $n$ (cf. (\ref{eqn:holes:NieRoot_A})). We define $G^{(2n+n)}=G^{\rs}$ when $\G_0(2n)+n$ has genus zero and $\rs$ is the unique D-type Niemeier root system with Coxeter number $2n$ (cf. (\ref{eqn:holes:NieRoot_D})). We write $G^{(12+4)}$ for $G^{\rs}$ when $\rs=E_6^4$ and we write $G^{(30+6,10,15)}$ for $G^{\rs}$ when $\rs=E_8^3$. 

Observe that the subgroup $\widehat{W}^{\rs}<\Aut(L^{\rs})$ consisting of automorphisms of $L^{\rs}$ that stabilize the irreducible components of $\rs$ is also normal in $\Aut(L^{\rs})$. Define $\bar{G}^{\rs}$ to be the corresponding quotient,
\begin{gather}
	\bar{G}^{\rs}=\Aut(L^{\rs})/\widehat{W}^{\rs},
\end{gather}
so that $\bar{G}^{\rs}$ is precisely the group of permutations of the irreducible components of $\rs$ induced by automorphisms of $L^{\rs}$, and is a quotient of $G^{\rs}$ (viz., the quotient by $\widehat{W}^{\rs}/W^{\rs}$) since $W^{\rs}<\widehat{W}^{\rs}$. It turns out that $\widehat{W}^{\rs}/W^{\rs}$ has order $2$ when $\rs_A\neq\emptyset$ or $\rs=\rs_E=E_6^4$, has order $3$ when $\rs=\rs_{D}=D_4^6$, and is trivial otherwise.

\begin{rmk}
In terms of the notation of \cite{MR661720} we have $W^{\rs}=G_0$, $\widehat{W}^{\rs}/W^{\rs}\simeq G_1$, $\bar{G}^{\rs}\simeq G_2$ and $G^{\rs}\simeq G_1G_2$.
\end{rmk}

The groups $G^{\rs}$ and $\bar{G}^{\rs}$ come naturally equipped with various permutation representations. To see this choose a set $\Phi$ of simple roots for $L^{\rs}$, meaning a set which is the union 
of sets of simple roots for each irreducible root sublattice of $L^{\rs}$. Then $\Phi$ constitutes a basis for the $24$-dimensional space $L^{\rs}_{\RR}$, and for each $g\in G^{\rs}$ there is a unique element in the pre-image of $g$ under $\Aut(L^{\rs})\to G^{\rs}$ that belongs to the subgroup $\Aut(L^{\rs},\Phi)$, consisting of automorphisms of $L^X$ that stabilize $\Phi$ as a set (i.e. act as permutations of the irreducible root subsystems followed by permutations---corresponding to Dynkin diagram automorphisms---of simple roots within irreducible root subsystems). Thus we obtain a section $G^{\rs}\to\Aut(L^{\rs})$ of the projection $\Aut(L^{\rs})\to G^{\rs}$ whose image is $\Aut(L^{\rs},\Phi)$, and composition with the natural map $\Aut(L^{\rs},\Phi)\to\Sym_{\Phi}$ defines a permutation representation of $G^{\rs}$ on $\Phi$. Write $\Phi=\Phi_A\cup\Phi_D\cup\Phi_E$ where $\Phi_A$ contains the roots in $\Phi$ belonging to type $A$ components of $\rs$, and similarly for $\Phi_D$ and $\Phi_E$. Then the decomposition $\Phi=\Phi_A\cup\Phi_D\cup\Phi_E$ is stable under $G^{\rs}$, since $\Aut(L^{\rs})$ cannot mix roots that belong to non-isomorphic root systems, so we obtain maps $G^{\rs}\to \Sym_{\Phi_A}$, $G^{\rs}\to \Sym_{\Phi_D}$ and $G^{\rs}\to\Sym_{\Phi_E}$. Write $g\mapsto\widetilde{\chi}^{\rs}_g$ for the character of $G^{\rs}$ attached to the representation $G^{\rs}\to\Sym_{\Phi}$, write $g\mapsto\widetilde{\chi}^{\rs_A}_g$ for that attached to $G^{\rs}\to\Sym_{\Phi_A}$, and interpret $\widetilde{\chi}^{\rs_D}$ and $\widetilde{\chi}^{\rs_E}$ similarly. Observe that $\widetilde{\chi}^{\rs}$ (and hence also $\widetilde{\chi}^{\rs_A}$, $\widetilde{\chi}^{\rs_D}$ and $\widetilde{\chi}^{\rs_E}$) are independent of the choice of $\Phi$. We set $\widetilde{\chi}^{\rs_A}=0$ in case $\Phi_A$ is empty, and similarly for $\widetilde{\chi}^{\rs_D}$ and $\widetilde{\chi}^{\rs_E}$. We have 
\begin{gather}
\widetilde{\chi}^{\rs}=\widetilde{\chi}^{\rs_A}+\widetilde{\chi}^{\rs_D}+\widetilde{\chi}^{\rs_E}.
\end{gather}

The characters $\widetilde{\chi}^{\rs_A}$, \&c., are naturally decomposed further as follows. Suppose that $\rs_A\neq \emptyset$. Then $\rs_A=A_{{\sf m}-1}^{d_A}$ for $d_A=d^{\rs}_A$ (cf. (\ref{eqn:holes:lats:ADEdecomp})) and we may write 
\begin{gather}
\Phi_A=\left\{f^i_j\mid 1\leq i\leq d_A,\,1\leq j\leq {\sf m}-1\right\}
\end{gather}
where the superscript indicates the $A_{{\sf m}-1}$ component to which the simple root $f^i_j$ belongs, and the inner products between the $f^i_j$ for varying $j$ are as described by the labeling in Figure \ref{fig:dynkin} (so that $\lab f^i_j,f^i_k\rab$ is $-1$ or $0$ according as the nodes labelled $f_j$ and $f_k$ are joined by an edge or not). Then for fixed $j$ the vectors $\{f^i_j+f^i_{{\sf m}-j}\mid 1\leq i\leq d_A\}$ define a permutation representation of degree $d_A$ for $G^{\rs}$. We denote the corresponding character $g\mapsto \bar{\chi}^{\rs_A}_g$ since the isomorphism type of the representation is evidently independent of the choice of $j$. Observe that $\bar{\chi}^{\rs_A}$ is generally not a faithful character since permutations of $\Phi_A$ arising from diagram automorphisms, exchanging $f^i_j$ with $f^i_{{\sf m}-j}$ for some $i$, act trivially. The vectors $\{f^i_{j}-f^i_{{\sf m}-j}\mid 1\leq i\leq d_A\}$ also span $G^{\rs}$-invariant subspaces of $\Span_{\RR} \Phi_A<L^{\rs}_{\RR}$, with different $j$ in the range $0< j<{\sf m}/2$ furnishing isomorphic (signed permutation) representations; we denote the corresponding character $g\mapsto \chi^{\rs_A}_g$. Since the $f^i_j$ are linearly independent we can conclude that 
\begin{gather}
	\widetilde{\chi}^{\rs_A}=\left\lceil \frac{{\sf m}-1}{2}\right\rceil \bar{\chi}^{\rs_A}+\left\lfloor \frac{{\sf m}-1}{2}\right\rfloor\chi^{\rs_A}
\end{gather}
by counting the possibilities for $j$ in each case.

If $\Phi_D$ is non-empty then ${\sf m}$ is even and $\rs_D=D_{{\sf m}/2+1}^{d_D}$ for $d_D=d^{\rs}_D$ (cf. (\ref{eqn:holes:lats:ADEdecomp})). Write now
\begin{gather}
\Phi_D=\left\{f^i_j\mid 1\leq i\leq d_D,\,1\leq j\leq {\sf m}/2+1\right\}
\end{gather}
where, similar to the above, the superscript indicates the $D_{{\sf m}/2+1}$ component to which the simple root $f^i_j$ belongs, and the inner products between the $f^i_j$ for varying $j$ are as described in Figure \ref{fig:dynkin}. Suppose first that ${\sf m}>6$. Then ${\sf m}/2+1>4$ and the only non-trivial diagram automorphism of $D_{{\sf m}/2+1}$ has order $2$ and interchanges $f^i_{{\sf m}/2}$ and $f^i_{{\sf m}/2+1}$. So we find that for $1\leq j<{\sf m}/2$ the sets $\{f^i_j\mid 1\leq i\leq d_D\}$ serve as bases for isomorphic permutation representations of degree $d_D$ for $G^{\rs}$, as does $\{f^i_{{\sf m}/2}+f^i_{{\sf m}/2 +1}\mid 1\leq i\leq d_D\}$; we denote the character of this (i.e. any one of these) permutation representation(s) by $g\mapsto \bar{\chi}^{\rs_D}_g$. We define $\chi^{\rs_D}$ to be the (signed permutation) character of the representation spanned by the vectors $\{f^i_{{\sf m}/2}-f^i_{{\sf m}/2 +1}\mid 1\leq i\leq d_D\}$, and we have
\begin{gather}
	\widetilde{\chi}^{\rs_D}=\frac{{\sf m}}{2} \bar{\chi}^{\rs_D}+\chi^{\rs_D}
\end{gather}
when $\rs_D\neq\emptyset$ and ${\sf m}>6$. In case ${\sf m}=6$ the group of diagram automorphisms of $D_{{\sf m}/2+1}=D_4$ is a copy of $S_3$, acting transitively on the sets $\{f^i_1,f^i_3,f^i_4\}$ (for fixed $i$), so we define $\bar{\chi}^{\rs_D}$ to be the character attached to the (permutation) representation of $G^{\rs}$ spanned by $\{f^i_2\mid 1\leq i\leq d_D\}$ (or equivalently, $\{f^i_1+f^i_3+f^i_4\mid 1\leq i\leq d_D\}$) and define $\check{\chi}^{\rs_D}$ to be the character of the representation spanned by the vectors $\{f^i_1-f^i_3,f^i_1-f^i_4\mid 1\leq i\leq d_D\}$. Evidently
\begin{gather}
	\widetilde{\chi}^{\rs_D}=2\bar{\chi}^{\rs_D}+\check{\chi}^{\rs_D}
\end{gather}
in case ${\sf m}=6$. In preparation for \S\ref{sec:mckay:aut}, where the characters defined here will be used to specify certain vector-valued cusp forms of weight $3/2$, we define $\chi^{\rs_D}_g={\sgn}^{\rs_D}_g\bar{\chi}^{\rs_D}_g$ for $g\in G^{\rs}$ when $\rs=A_4^5D_4$ or $\rs=D_4^6$---the two cases for which $\rs$ involves $D_4$---where ${\sgn}^{\rs_D}_g=\pm 1$ is the function defined as follows. Write the image of $g\in G^{\rs}$ in $\Sym_{\Phi_D}$ as a product $g_d\circ g_p$ where $g_p\cdot f^i_j=f^{\pi(i)}_j$ for all $j\in\{1,2,3,4\}$, for some permutation $\pi\in \Sym_{d^{\rs}_D}$, and $g_d\cdot f^i_j=f^i_{\sigma_i(j)}$ for some permutations $\sigma_i\in\Sym_4$. Then set $\sgn^{\rs_D}_g=\prod_{i=1}^{d^{\rs}_D}\sgn(\sigma_i)$.

If $\Phi_E=E_n^{d_E}$ for $d_E=d^{\rs}_E>0$ then we may identify $f^i_j\in\Phi_E$ such that
\begin{gather}
\Phi_E=\left\{f^i_j\mid 1\leq i\leq d_E,\,1\leq j\leq n\right\}
\end{gather}
and, as above, the superscripts enumerate simple components of $X_E$ and the subscripts indicate inner products for simple vectors within a component as per Figure \ref{fig:dynkin}. Define $\bar{\chi}^{\rs_E}$ to be the character of $G^{\rs}$ attached to the permutation representation spanned by the set $\{f^i_3\mid 1\leq i\leq d_E\}$ (for example). In case $n=6$ write $\chi^{\rs_E}$ for the character of $G^{\rs}$ attached to the representation afforded by ${\rm Span}_{\RR}\{f^i_1-f^i_5\mid 1\leq i\leq d_E\}$.
We have 
\begin{gather}
\widetilde{\chi}^{\rs_E}=n\bar{\chi}^{\rs_E}
\end{gather}
when $n\in\{7,8\}$ since for each $1\leq j\leq n$ the set $\{f^i_j\mid 1\leq i\leq d_E\}$ spans a representation with character $\bar{\chi}^{\rs_E}$ in these cases, and 
\begin{gather}
\widetilde{\chi}^{\rs_E}=4\bar{\chi}^{\rs_E}+2\chi^{\rs_E}
\end{gather}
when $n=6$, the invariant subspace with character $2\chi^{\rs_E}$ being spanned by the vectors $f^i_1-f^i_5$ and $f^i_2-f^i_4$ for $1\leq i\leq d_E$. 

We call the functions $\bar{\chi}^{\rs_A}_g$, $\chi^{\rs_A}_g$, $\bar{\chi}^{\rs_D}_g$, $\chi^{\rs_D}_g$, \&c., the {\em twisted Euler characters} attached to $G^X$. They are given explicitly in the tables of \S\ref{sec:chars:eul}. As mentioned above, we will use them to attach a vector-valued cusp form $S_g^X$ to each $g\in G^X$ for $X$ a Niemeier root system in \S\ref{sec:mckay:aut}.

\subsection{McKay Correspondence}\label{sec:grps:dyn}

The {\em McKay correspondence} \cite{McKay_Corr} relates finite subgroups of $\SU(2)$ to the extended Dynkin diagrams of ADE type by associating irreducible representations of the finite groups to nodes of the corresponding diagrams. A beautiful explanation for this can be given in terms of resolutions of {simple singularities} $\CC^2/G$ for $G<\SU(2)$ \cite{Slo_SmpSngSmpAlgGps,GonVer_GeomCnstMcKCorr}. In \S3.5 of \cite{UM} we observed a curious connection between the umbral groups $G^{(\ll)}$ and certain finite subgroups $D^{(\ll)}<\SU(2)$, for the cases $\ll\in\{3,4,5,7\}$, such that the lambency $\ll$ and the rank ${\sf r}$ of the Dynkin diagram attached to $D^{(\ll)}$ via McKay's correspondence are related by $\ll+{\sf r}=11$. In this section we describe an extension of this observation, relating the umbral group $G^{(\ll)}$ to a finite subgroup $D^{(\ll)}<\SU(2)$, for each $\ll$ in $\{3,4,5,6,7,8,9,10\}$.

In \cite{UM} it was observed that a Dynkin diagram of rank $11-\ll$ may be attached to each $G^{(\ll)}$ for $\ll\in \{3,4,5,7\}$ in the following manner. If $p=(25-\ll)/(\ll-1)$ then $p$ is a prime and there is a unique (up to conjugacy) subgroup $\bar{L}^{(\ll)}<\bar{G}^{(\ll)}$ such that $\bar{L}^{(\ll)}$ is isomorphic to $\PSL_2(p)$ and acts transitively in the degree $24/(\ll-1)$ permutation representation of $G^{(\ll)}$ defined in \cite[\S3.3]{UM}. Now $\bar{L}^{(\ll)}$ has a unique (up to isomorphism) subgroup $\bar{D}^{(\ll)}$ of index $p$ in $\bar{L}^{(\ll)}$---a fact which is peculiar to the particular $p$ arising---and $\bar{D}^{(\ll)}$ is a finite subgroup of $\SO(3)$ whose preimage $D^{(\ll)}$ in $\SU(2)$ realises the extended diagram $\Delta^{(\ll)}$ corresponding (cf. Figure \ref{fig:edynkin}) to a Dynkin diagram of rank $11-\ll$ via McKay's correspondence. In the present setting, with groups $G^{(\ll)}$ defined for all $\ll$ such that $\G_0(\ll)$ has genus zero, and in particular for $3\leq \ll\leq 10$, it is possible to extend this correspondence as follows. 

\begin{table}[h]
\captionsetup{font=small}
\begin{center}
\caption{The McKay Correspondence in Umbral Moonshine}\label{tab:dyntab}

\medskip

\begin{tabular}{l|cccccccc}
$\rs$&$A_2^{12}$&$A_3^8$&$A_4^6$&$A_5^4D_4$&$A_6^4$&$A_7^2D_5^2$&$A_8^3$&$A_9^2D_6$\\
	\midrule
$\ll$&			3&	4&	5&	6&	7&	8&	9&	10\\
$p$&			11&	7&	5&	4&	3&	2&	2&	2\\
	\midrule
$G^{(\ll)}$&	$2.M_{12}$&	$2.\AGL_3(2)$&$\GL_2(5)/2$&$\GL_2(3)$&$\SL_2(3)$&	$\Dih_4$&$\Dih_6$	&$4$\\
${D}^{(\ll)}$&	$2.\Alt_5$&		$2.\Sym_4$&		&$\Dih_6$&		$Q_8$&	$4$&	$3$&	$2$\\
	\midrule
$\bar{G}^{(\ll)}$&	$M_{12}$&	$\AGL_3(2)$&	$\PGL_2(5)$&$\PGL_2(3)$&	$\PSL_2(3)$&$2^2$	&$\PSL_2(2)$	&$2$\\
$\bar{L}^{(\ll)}$&	$\PSL_2({11})$&	$\PSL_2(7)$&	$\PSL_2(5)$&$\PGL_2(3)$&	$\PSL_2(3)$	&$2^2$	&$\PSL_2(2)$&	$2$	\\
$\bar{D}^{(\ll)}$&	$\Alt_5$&		$\Sym_4$&		$\Alt_4$&	$\Sym_3$&	$2^2$&	$2$&	$3$&	$1$\\
	\midrule
$\Delta^{(\ll)}$& 	$\widehat{E}_8$&	$\widehat{E}_7$&	$\widehat{E}_6$&	$\widehat{D}_5$&	$\widehat{D}_4$&	$\widehat{A}_3$&	$\widehat{A}_2$&	$\widehat{A}_1$\\	
\end{tabular}
\end{center}
\end{table}
Since $(25-\ll)/(\ll-1)$ is not an integer for $\ll\in\{6,8,10\}$ we seek a new definition of $p$. Armed with the Niemeier root systems attached to each $G^{(\ll)}$ we set $p=d_A-1$ in case $\rs=\rs_A=A_{\ll-1}^{d_A}$ has only A-type  components, and set $p=d_A$ otherwise. This definition yields values coincident with the former one when $(25-\ll)/(\ll-1)$ is an integer. Next we seek a subgroup $\bar{L}^{(\ll)}<\bar{G}^{(\ll)}$ acting transitively on the irreducible components of $X_A$ and $X_D$ that has a unique up to isomorphism index $p$ subgroup $\bar{D}^{(\ll)}$,
\begin{gather}
	[\bar{L}^{(\ll)}:\bar{D}^{(\ll)}]=p.
\end{gather}
Such $\bar{L}^{(\ll)}$ and $\bar{D}^{(\ll)}$ exist for each $3\leq \ll\leq 10$ and are given explicitly in Table \ref{tab:dyntab}. In the new cases $\ll\in\{6,8,9,10\}$ the groups $\bar{L}^{(\ll)}$ and $\bar{G}^{(\ll)}$ coincide. 

The main observation of this section is the following.
\begin{quote}
{\em For every $3\leq \ll\leq 10$ the group $\bar{D}^{(\ll)}$ is the image in $\SO(3)$ of a finite subgroup $D^{(\ll)}<\SU(2)$ that is attached, via McKay's correspondence, to the extended diagram $\Delta^{(\ll)}$ corresponding to a Dynkin diagram of rank $11-\ll$.} 
\end{quote}
The group $D^{(\ll)}$ is even a subgroup of $G^{(\ll)}$---the pre image under the natural map $G^{\rs}\to\bar{G}^{\rs}$---except in the case that $\ll=5$. (We refer to \cite[\S3.4]{UM} for a discussion of this exceptional case.) To aid in the reading of Table \ref{tab:dyntab} we note here the exceptional isomorphisms 
\begin{gather}
	\PGL_2(5)\simeq \Sym_5,\; \PSL_2(5)\simeq \Alt_5,\\
	\PGL_2(3)\simeq \Sym_4,\; \PSL_2(3)\simeq \Alt_4,\\
	\PGL_2(2)\simeq\PSL_2(2)\simeq \Sym_3.
\end{gather}
In \cite{UM} we used the common abbreviation $L_n(q)$ for $\PSL_n(q)$.
	
Recall from \S3.5 of \cite{UM} the following procedure for obtaining a length $8$ sequence of Dynkin diagrams. Start with the (finite type) $E_8$ Dynkin diagram, being star shaped with three {branches}, and construct a sequence of diagrams iteratively by removing the end node from a branch of maximal length at each iteration. In this way we obtain ${E}_8$, ${E}_7$, $E_6$, $D_5$, $D_4$, $A_3$, $A_2$, $A_1$, and it is striking to observe that our list $\D^{(\ll)}$, obtained by applying the McKay correspondence to distinguished subgroups of the $G^{(\ll)}$, is exactly the sequence obtained from this by replacing (finite type) Dynkin diagrams with their corresponding extended diagrams.

%-----------------------------------------------------------------------------------%
\section{Automorphic Forms}\label{sec:forms}
%-----------------------------------------------------------------------------------%

In this section we discuss  modular objects that play a role in the moonshine relation between mock modular forms and finite groups that is the main focus of this paper.

In what follows we take $\tau$ in the upper half-plane $\HH$ and $z \in \CC$, and adopt the shorthand notation $e(x) = e^{2\p ix}$. We also define $q=e(\tau)$ and $y=e(z)$ and write 
\be
 \g\t = \frac{a\t+b}{c\t+d}, \quad \g= 
 	\begin{pmatrix}
	a&b\\
	c&d
	\end{pmatrix}
	\in\SL_2(\RR)
\ee
for the natural action of $\SL_2(\RR)$ on $\HH$,
and write
\be
 \g (\t,z) = \left(\frac{a\t+b}{c\t+d},\frac{z}{c\t+d} \right)
\ee
for the action of $\SL_2(\ZZ)$ on $\HH \times \CC$. We set
\be
\jac(\g,\t) = (c\t+d)^{-1}
\ee
and choose the principal branch of the logarithm (i.e. $x^s=|x|^se^{\ii\th s}$ when $x=|x|e^{\ii\th}$ and $-\pi<\th\leq \pi$) to define non-integer exponentials.

%---------------------------------------------------------%
\subsection{Mock Modular Forms}\label{sec:forms:mock}
%---------------------------------------------------------%

We briefly recall modular forms, mock modular forms, and their vector-valued generalisations. 

Let $\G$ be a discrete subgroup of the  group $\SL_2(\RR)$ that is commensurable with the modular group $\SL_2(\ZZ)$. 
For $w\in \frac{1}{2}\ZZ$ say that a non-zero function $\psi\colon\G\to \CC$ is a {\em multiplier system} for $\G$ with weight $w$ if
\begin{gather}\label{eqn:sums:mult}
	\psi(\g_1)\psi(\g_2)\jac(\g_1,\g_2\t)^{w}\jac(\g_2,\t)^{w}
	=
	\psi(\g_1\g_2)\jac(\g_1\g_2,\t)^{w}
\end{gather}
for all $\g_1,\g_2\in \G$. 
Given such a multiplier system $\psi$ for $\G$  
we may define the {\em $(\psi,w)$-action} of $\G$ on the space $\mc{O}(\HH)$ of holomorphic functions on the upper half-plane by setting
\begin{gather}\label{eqn:sums:psiw_actn}
	(f|_{\psi,w}\g)(\t)=f(\g\t)\psi(\g)\jac(\g,\t)^{w}
\end{gather}
for $f\in \mc{O}(\HH)$ and $\g\in \G$. We then say that $f\in \mc{O}(\HH)$ is an {\em (unrestricted) modular form} with multiplier $\psi$ and weight $w$ for $\G$ in the case that $f$ is invariant for this action; i.e. $f|_{\psi,w}\g=f$ for all $\g\in\G$. We say that an unrestricted modular form $f$ for $\G$ with multiplier $\psi$ and weight $w$ is a {\em weakly holomorphic modular form} in case $f$ has at most exponential growth at the cusps of $\G$. We say that $f$ is a {\em modular form} if $(f|_{\widetilde{\psi},w}\s)(\t)$ remains bounded as $\Im(\t)\to \inf$ for any $\s\in\SL_2(\ZZ)$, and we say $f$ is a {\em cusp form} if $(f|_{\widetilde{\psi},w}\s)(\t)\to 0$ as $\Im(\t)\to\inf$ for any $\s\in\SL_2(\ZZ)$.

Suppose that $\psi$ is a multiplier system for $\G$ with weight $w$, and $g$ is a modular form for $\G$ with the {conjugate multiplier system} $\bar{\psi}\colon\g\mapsto \overline{\psi(\g)}$ and {dual} weight $2-w$. Then we may use $g$ to twist the $(\psi,w)$-action of $\G$ on $\mc{O}(\HH)$ by setting
\begin{gather}\label{eqn:sums:gtwact}
	\left(f|_{\psi,w,g}\g\right)(\t)
	=
	f(\g\t)\psi(\g)\jac(\g,\t)^{w}
	+e(\tfrac{w-1}{4})
	\int_{-\g^{-1}\infty}^{\infty}(\t'+\t)^{-w}\overline{g(-\bar{\t}')}{\rm d}\t'.
\end{gather}
With this definition, we say that $f\in \mc{O}(\HH)$ is an {\em (unrestricted) mock modular form} with multiplier $\psi$, weight $w$ and shadow $g$ for $\G$ if $f$ is invariant for this action; i.e. $f|_{\psi,w,g}\g=f$ for all $\g\in\G$.
We say that an unrestricted mock modular form $f$ for $\G$ with multiplier $\psi$, weight $w$ and shadow $g$ is a {\em weakly holomorphic mock modular form} in case $f$ has at most linear exponential growth at the cusps of $\G$. 
From this point of view a (weakly holomorphic) modular form is a (weakly holomorphic) mock modular form with vanishing shadow. This notion of mock modular form developed from the Maass form theory due to Bruinier--Funke \cite{BruFun_TwoGmtThtLfts}, and from Zwegers' 
work \cite{zwegers} on Ramanujan's mock theta functions. 

In this paper we will consider the generalisation of the above definition to {\em vector-valued (weakly holomorphic) mock modular forms} with $n$ components, where the multiplier  $\psi \colon \Gamma \to \GL_n(\CC)$ is a (projective) representation of $\G$. From the definition \eq{eqn:sums:gtwact}, it is not hard to see that the multiplier $\psi$ of a (vector-valued) mock modular form is necessarily the inverse of that of its shadow. 
To avoid clutter, we omit the adjective ``weakly holomorphic" in the rest of the paper when there is no room for confusion.

Following Zwegers \cite{zwegers} and Zagier \cite{zagier_mock} we define a {\em mock theta function} to be a $q$-series $h=\sum_n a_n q^n$ such that for some $\lambda \in \QQ$ the assignment $\t\mapsto q^{\lambda}h|_{q=e(\t)}$ 
defines a mock
modular form of weight $1/2$ whose shadow is a unary (i.e. attached to a quadratic form in one variable) theta series of weight ${3}/{2}$.
In \S\ref{sec:conj} we conjecture that specific sets of mock theta functions appear as McKay--Thompson series associated to infinite-dimensional modules for the groups $G^{X}$ (cf. \S\ref{sec:holes:gps}), where $X$ is a Niemeier root system.

%---------------------------------------------------------%
\subsection{Jacobi Forms}\label{sec:forms:jac}
%---------------------------------------------------------%

We first discuss Jacobi forms following \cite{eichler_zagier}. For every pair of integers $k$ and $m$, we define the $m$-action of the group $\ZZ^2$ and the $(k,m)$-action of the group $\SL_2(\ZZ)$  on the space of holomorphic functions $\f\colon \HH \times \CC \to \CC$ as
\begin{align} \label{elliptic}
(\f\lvert_{m} (\l,\m) )(\t,z) &= e( m(\l^2 \t + 2\l z)) \, \f(\t, z+\l \t +\m)  \\\label{modular}
(\f\lvert_{k,m}\g )(\t,z) &= e(-m \tfrac{c z^2}{c\t+d})\, \jac(\g,\t)^{{k}} \f(\g(\t, z))
\end{align}
where $\g \in \SL_2(\ZZ)$ and $ \l,\m\in \ZZ$. 
We say a holomorphic function $\f\colon \HH \times \CC \to \CC$ is an {\em (unrestricted) Jacobi form} of weight $k$ and index $m$ for the Jacobi group $\SL_2(\ZZ)\ltimes \ZZ^2$ if it is invariant under the above actions, $\f= \f\lvert_{k,m}\g$ and $\f= \f\lvert_{m} (\l,\m)$, for all $\g \in \SL_2(\ZZ)$ and for all $ (\l,\m)\in \ZZ^2$.
 In what follows we refer to the transformations  \eqref{elliptic} and \eqref{modular}  as the {\em  elliptic} and {\em modular} transformations, respectively.

The invariance of $\f(\tau,z)$ under $\tau \rightarrow \tau+1$ and $z \rightarrow z+1$ implies a Fourier expansion
\be\label{eqn:forms:jac:FouExp}
\f(\t,z) = \sum_{n,r \in \ZZ} c(n,r) q^n y^r 
\ee
and the elliptic transformation can be used to show that $c(n,r)$ depends only on the {\em discriminant} $D=r^2-4mn$ and on $r ~{\rm mod}~ 2m$. In other words, we have $c(n,r)=C_{\til r} (r^2-4mn)$ where  $\til r\in \ZZ/2m\ZZ$ and $r=\til r$ mod $2m$, for some appropriate function $D\mapsto C_{\til r}(D)$. An unrestricted Jacobi form is called a {\em weak Jacobi form}, a {\em (strong) Jacobi form}, or a {\em Jacobi cusp form}  when the Fourier coefficients satisfy $c(n,r)=0$ whenever $n< 0$, $C_{\til r}(D)=0$ whenever $D > 0$, or $C_{\til r}(D)=0$ whenever $D \ge0$, respectively. In a slight departure from the notation in \cite{eichler_zagier} we denote the space of weak Jacobi forms of weight $k$ and index $m$ by $J_{k,m}$.

In what follows we will need two further generalisations of the above definitions. The first is relatively straightforward and replaces  $\SL_2(\ZZ)$  by a finite index subgroup $\Gamma \subset \SL_2(\ZZ)$ in the modular transformation law. (One has to consider Fourier expansions (\ref{eqn:forms:jac:FouExp}) for each cusp of $\G$.) The second is more subtle and leads to {\em meromorphic Jacobi forms} which obey the modular and elliptic transformation laws but are such that the functions $z\mapsto \phi(\t,z)$ are allowed to have poles lying at values of $z\in\CC$ corresponding to torsion points of the elliptic curve $\CC/(\ZZ \tau + \ZZ)$. Our treatment of meromorphic Jacobi forms (cf. \S\ref{sec:forms:meromock}) mostly follows \cite{zwegers} and \cite{Dabholkar:2012nd}. We will only consider functions with simple poles in $z$.

The elliptic transformation \eq{elliptic} implies (cf. \cite{eichler_zagier}) that a (weak) Jacobi form of weight $k$ and index $m$ admits an expansion 
\be\label{eqn:forms:jac:thetaxpn}
\phi(\tau,z)= \sum_{r \,({\rm mod}~2m)}  \til h_{m,r}(\tau) \th_{m,r}(\tau,z)
\ee
in terms of the {\em index $m$ theta functions},
\be\label{thetexp}
\theta_{m,r}(\tau,z) =\sum_{\substack{k\in \ZZ \\ k= r~{\text{mod}}~2m}} q^{k^2/4m} y^{k}. 
\ee
Recall that the vector-valued function $\th_{m} =(\th_{m,r}) $ 
satisfies
\begin{gather}
\th_{m}(-\frac{1}{\t},-\frac{z}{\t}) = \sqrt{-i \t}\, e\left(\frac{mz^2}{\t}\right) \, {\bf S} \,\th_{m}(\t,z),\qquad\label{transf_theta}
\th_{m}({\t}+1,{z}) = \,{\bf T} \,\th_{m}(\t,z),
\end{gather}
where ${\bf S}$ and ${\bf T}$ are the $2m \times 2m$ unitary matrices with entries
\begin{gather}
{\bf S}_{rr'} = \frac{1}{\sqrt{2m}} e\left(\frac{rr'}{2m}\right),\qquad \label{transf_theta_2}
{\bf T}_{rr'} = e\left(\frac{r^2}{4m}\right)\,\d_{r,r'}.
\end{gather}
From this we can see that $\til h= ( \til h_{m,r})$ is a $2m$-component vector transforming as a weight $k-1/2$ modular form for $\SL_2(\ZZ)$,  with a multiplier system represented by 
the matrices ${\bf S}^\dag$ and ${\bf T}^\dag$, satisfying ${\bf S}{\bf S}^\dag={\bf T}{\bf T}^\dag =I_{2m}$, and corresponding to the modular transformations $S$ and $T$, respectively. (See \cite{MR0332663}.)
Moreover, the invariance under the modular transformation \eqref{modular} with $\gamma=-I_2$ implies that $\phi(\tau,-z)=(-1)^k \phi(\tau,z)$. Combining this with the identity
$\th_{m,-r}(\tau,z)= \th_{m,r}(\tau,-z)$ we see that 
\be\label{halving_theta_coeff}  \til h_{m,r} = (-1)^k \til h_{m,-r}.\ee

%---------------------------------------------------------------------------------------%
\subsection{The Eichler--Zagier Operators and an ADE Classification}\label{sec:forms:ADE}
%---------------------------------------------------------------------------------------%

We now turn to a  discussion of the Eichler--Zagier operators on Jacobi forms and establish an ADE classification of 
maps satisfying a certain positivity condition.  

Recall from \S\ref{sec:forms:jac} that a Jacobi form of weight $k$ and index $m$ admits a decomposition (\ref{eqn:forms:jac:thetaxpn}) into a combination of theta functions $\th_{m,r}$ and the $2m$ components $\til h_{m,r}$ of a vector-valued modular form $\til h_m$ of weight $k-1/2$. On the other hand, one can also consider the following question: for a given vector-valued modular form $\til h_m$, is the expression  in (\ref{eqn:forms:jac:thetaxpn}) the only combination of $\th_{m,r}$ and $\widetilde h_{m,r}$ that has the right transformation property to be a weight $k$,  index $m$ Jacobi form? 
In other words, we would like to consider all  
$2m\times 2m$ matrices $\Omega$ such that 
\begin{gather}
\til h_m^T\cdot \O \cdot\th_m =\sum_{r,r' \,({\rm mod}~2m)}\, \widetilde h_{m,r} \, \O_{r,r'} \th_{m,r'}
\end{gather}
is again a weight $k$ index $m$ Jacobi form.

From (\ref{transf_theta})-(\ref{transf_theta_2}) as well as the transformation under $\g=-I_2$,  we see that this condition amounts to considering the commutants $\O$ of ${\bf S}$ and  ${\bf T}$ satisfying
\be\label{commutants}
{\bf S}^\dag \Omega {\bf S} = {\bf T}^\dag \Omega {\bf T} = \Omega .
\ee
In particular, as $\Omega$ commutes with ${\bf S}^2$ we see that it has the reflection symmetry 
\be\label{prop_reflection_Omega}
\O_{r,r'} = \O_{-r,-r'}. 
\ee
Such commutants have been classified in \cite{Gepner:1986hr}. 
For each positive integer $m$, the space of $2m\times 2m$ matrices satisfying \eq{commutants} has dimension given by the number of divisors of $m$, $\s_0(m)=\sum_{d|m} 1$, and is spanned by the set of linearly independent matrices $\{\Omega_{m}(n_1), \Omega_{m}(n_2), \dots ,\Omega_{m}(n_{\s_0(m)})\}$ whose entries are given by 
\begin{align}\label{def:OmegaMatrices}
\Omega_{m}(n_i)_{r,r'} = \begin{cases} 1 &\text{if $r+r' = 0$ mod $2 n_i$ and $r-r' = 0$ mod ${2m}/{n_i}$,} \\ 
0 &{\rm otherwise},
\end{cases}
\end{align}
where  $1=n_1<n_2<\dots< m=n_{\s_0(m)}$ are the divisors of $m$. It is easy to check that these matrices automatically satisfy \eq{prop_reflection_Omega}.

Note that 
\be\label{reflection_AL}
 \widetilde h_{m}^T \cdot \Omega_m(n)\cdot \th_m= (-1)^k\,  \widetilde h_{m}^T \cdot \Omega_m(m/n)\cdot \th_m ,
\ee
as is evident from the definition \eq{def:OmegaMatrices} of $\O_m(n)$  as well as the reflection property \eq{halving_theta_coeff} of the components $\til h_{m,r}$ of $\til h_m$.

In fact, as we will now show, for a given vector-valued modular form $\widetilde h_m= (\widetilde h_{m,r})$ and any given divisor $n$ of $m$, the new Jacobi form 
$\til h_m^T\cdot \O_m(n) \cdot\th_m $ can be obtained from the original one $\til h_m^T\cdot \th_m $
via a natural operator---the so-called Eichler--Zagier operator \cite{eichler_zagier}---on Jacobi forms. 

Given positive integers $n$, $m$ such that $n|m$, we define an {\em Eichler--Zagier operator} ${\cal W}_m{(n)}$ acting on a function $f\colon \HH \times \CC \to \CC$ by setting
\be\label{Atkin--Lehner1}
( f\lvert {\cal W}_m{(n)}) \,(\t,z) = \frac{1}{n} \sum_{a,b = 0}^{n-1} e\left(m\left(\tfrac{a^2}{n^2} \t + 2 \tfrac{a}{n}z +\tfrac{ab}{n^2}\right)\right) f\left(\t,z+\tfrac{a}{n}\t+\tfrac{b}{n}\right). 
\ee
It is easy to see that the operator ${\cal W}_m{(n)}$ commutes with the index $m$ elliptic transformation (\ref{elliptic}) 
\be
f\lvert {\cal W}_m{(n)}\,\lvert_{m}(\l,\m) = f\lvert_{m}(\l,\m)\,\lvert {\cal W}_m{(n)} 
\ee
for all $\m,\l \in \ZZ$ and in particular preserves the invariance under elliptic transformations. 
Moreover, one can easily check that the modular invariance (\ref{modular}) is also preserved. 
As a result ${\cal W}_m{(n)}$ maps an unrestricted Jacobi form of weight $k$ and index $m$ to another unrestricted Jacobi form of the same weight and index. Moreover, when $n\| m$ this operation is an involution on the space of strong Jacobi forms. This involution is sometimes referred to as an {\em Atkin--Lehner involution} for Jacobi forms due to its intimate relation to Atkin--Lehner involutions for modular forms \cite{MR958592,MR1074485}. We will explain and utilise some aspects of this relation in \S\ref{sec:forms:genus0}. 

For later use, we define the more general operator ${\cal W} = \sum_{n_i|m} c_i {\cal W}_m(n_i)$ by setting
\be
f\lvert {\cal W}  = \sum_{n_i|m} c_i  \,\left( f\lvert {\cal W}_m(n_i)\right) .
\ee

The relation between the Eichler--Zagier operators ${\cal W}$ and the transformation on Jacobi forms 
\begin{gather}
\til h_m^T\cdot \th_m \mapsto \til h_m^T\cdot \O\cdot \th_m,
\end{gather}
where $\O$ is a linear combination of the matrices $\O_m(n)$ in \eq{def:OmegaMatrices}, 
 can be seen via the action of the former on the theta functions $\th_{m,r}$.
Notice that 
\begin{align}
\th_{m,r} \lvert {\cal W}_m{(n)} 
=\sum_{r'  \,({\rm mod}~2m)} \O_m(n)_{r,r'} \th_{m,r'} .
\end{align}
In terms of the $2m$-component vector $\th_{m} = (\th_{m,r})$,  we have 
\be\label{EZ_on_theta}
 \th_{m} \lvert {\cal W}_m{(n)}= \O_m{(n)} \cdot\th_{m} \;,
\ee 
which immediately leads to 
\be\label{relation_EZ_Omega}
  \til h_{m}^T\cdot \th_{m}\lvert {\cal W}_m{(n)}
=
 \til h_{m}^T\cdot\O_m(n)\cdot \th_{m}.
\ee
In other words, the Jacobi forms we discussed above in terms of the matrices $ \Omega_{m}(n)$ are simply the images of the original Jacobi forms under the corresponding Eichler--Zagier operators. This property makes it  obvious that $\til h_{m}^T\cdot\O_m(n)\cdot \th_{m}$ is also a Jacobi form since ${\cal W}_m{(n)}$ preserves the transformation under the Jacobi group.  
This relation will be important in the discussion in \S\ref{sec:forms:umbral}.

Apart from the modularity (Jacobi form) condition, it is also natural to impose a certain positivity condition.  As we will see, this additional condition leads to an ADE classification of the matrices $\O$. 
To explain this positivity condition, first recall that all the entries of the matrices $ \Omega_{m}(n)$ for any divisor $n$ of $m$ are non-negative integers and it might seem that any positivity condition would be redundant. 
However, we have also seen that the description of the theta-coefficients $\widetilde h_{m,r}$ of a weight $k$, index $m$ Jacobi form as a vector with $2m$ components has some redundancy since different components are related to each other by $\til h_{m,r}= (-1)^k \til h_{m,-r}$ (cf. (\ref{halving_theta_coeff})). 
For the purpose of the present paper we will from now on consider only the case of odd $k$, where there are at most $m-1$ independent components in $(\widetilde h_{m,r})$.
In this case, using the property \eq{prop_reflection_Omega}
we can rewrite the Jacobi form $\til h^T_m \cdot \O \cdot \th_m$ as
\be
\til h^T_m \cdot \O \cdot \th_m=\sum_{r,r' =1}^{m-1} \widetilde h_{m,r}\, ( \Omega_{r,r'}-\Omega_{r,-r'} )\, ( \th_{m,r'}- \th_{m,-r'}).
\ee
As a result, it is natural to consider the $2m\times 2m$ matrices $\O=\sum_{i=1}^{\s_0(m)} c_{i} \O_m({n_i})$ where $n_1,n_2,\dots $ are the (distinct) divisors of $m$ that satisfy the corresponding positivity and integrality condition  
\be\label{positivity_ADE}
 \text{$\O_{r,r'}-\O_{r,-r'} \in \ZZ_{\geq 0}$ for all $r,r'=1,\dots, m-1$,} 
\ee
with a natural normalisation
\be\label{positivity_ADE2}
\O_{1,1}-\O_{1,-1} =1.
\ee
Evidently, they are in one-to-one correspondence with the non-negative integer combinations of $(\til h_{m,r})$ and  $( \th_{m,r'}- \th_{m,-r'})$, with $r,r'=1,2,\dots, m-1$, with the coefficient of the term $\til h_{m,1} \th_{m,1}$ equal to $1$. From a conformal field theory point of view this is precisely the requirement of having a unique ground state in the theory. 

\begin{table}
\captionsetup{font=small}
\centering
\begin{tabular}{CCCC}\toprule
 X & m(X) & \pi^X & \O^{X} \\\midrule
A_{m-1} & m &\frac{m}{1} & \O_m{(1)} 		\vspace{0.2em}\\
D_{{m}/{2}+1} & m &\frac{2.m}{1.(m/2)} &  \O_m{(1)}+ \O_m{(m/2)}	\vspace{0.2em}\\
E_6 & 12&\frac{2.3.12}{1.4.6} & \O_{12}{(1)}+ \O_{12}{(4)}+ \O_{12}{(6)} 	\vspace{0.2em}\\

E_7& 18&\frac{2.3.18}{1.6.9} & \O_{18}{(1)}+\O_{18}{(6)}+\O_{18}{(9)}   	\vspace{0.2em}\\

E_8 & 30&\frac{2.3.5.30}{1.6.10.15} & \O_{30}{(1)}+\O_{30}{(6)}+\O_{30}{(10)}+\O_{30}{(15)}
\vspace{0.1em}\\
\bottomrule
\end{tabular}
\caption{\label{ADE1}{The ADE classification of matrices $\Omega$ producing Jacobi forms $\til h_m^T\cdot \O \cdot\th_m $ 	}}
\end{table}

It turns out that this problem has been studied by Cappelli--Itzykson--Zuber \cite{Cappelli:1987xt}. They found a beautiful ADE classification of such  
$2m\times 2m$ matrices $\O$  (see Proposition 2 of \cite{Cappelli:1987xt}) and we present these matrices in Table \ref{ADE1}, denoting by $\O^X$ the matrix  corresponding to the irreducible simply-laced root system $X$. 
The motivation of \cite{Cappelli:1987xt} was very different from ours: these authors were interested in classifying the modular invariant combinations of chiral and anti-chiral characters of the affine Lie algebra $\widehat{A}_1$ (the $\SU(2)$ current algebra). However, as the modular transformation of the $\widehat A_1$ characters at level $m-2$ is very closely related to that of the index $m$ theta functions $\th_{m,r}$, the relevant matrices are also the commutants of the same ${\bf S}$ and ${\bf T}$ matrices satisfying \eq{commutants}.

The relation between the Eichler--Zagier operators and the $\O_m(n)$ matrices discussed earlier makes it straightforward to extend the above ADE classification to an ADE classification of Eichler--Zagier operators.
Combining the results of the above discussion, we arrive at the following theorem.

\begin{thm}\label{thm_ADE}
For any integer $m$ and any odd integer $k$, and any vector-valued modular form $\widetilde h_m=(\widetilde h_{m,r})$ 
such that $\til h_m^T\cdot \th_m$ 
is an (unrestricted) weight $k$, index $m$ Jacobi form, suppose $\O$ coincides with a matrix $\O^X$ corresponding to an irreducible simply-laced root system $X$ with Coxeter number $m$ via Table \ref{ADE1}. Then the combination $\til h_m^T\cdot\O\cdot \th_m$ is also a weight $k$, index $m$ (unrestricted)  Jacobi form which moreover satisfies the positivity condition 
\be\label{thm1_jacform}
\til h_m^T\cdot\O\cdot \th_m = \sum_{r,r' =1}^{m-1} c_{r,r'} \widetilde h_{m,r} (\th_{m,r'}-\th_{m,-r'}),~ c_{r,r'}\in\ZZ_{\geq 0}, ~c_{1,1}=1.
\ee
Conversely, any $2m\times 2m$ matrix $\O$ for which  the above statement is true necessarily coincides with a matrix $\O^X$ corresponding to an irreducible simply-laced root system $X$ with Coxeter number $m$. 
Moreover, the resulting  (unrestricted)  Jacobi form is  the image of the original Jacobi form $\til h_m^T\cdot \th_m$ under the Eichler--Zagier operator ${\cal W}^X$ defined by replacing $\O_m{(n)}$ in $\O^X$ with ${\cal W}_m{(n)}$ (cf. \eq{EZ_relation}). 
 \end{thm}

\begin{proof}
First, the Jacobi form condition on $\til h_m^T\cdot\O\cdot \th_m$ requires that $\O$ satisfies the commutant condition \eq{commutants}. 
It was shown in \cite{Gepner:1986hr} (cf. Proposition 1 of \cite{Cappelli:1987xt}) that the space of such $2m\times 2m$ matrices are spanned by $\{\Omega_{m}(n_1), \Omega_{m}(n_2), \dots ,\Omega_{m}(n_{\s_0(m)})\}$ given in \eq{def:OmegaMatrices}. 

Next, the positivity and integrality conditions on $c_{r,r'}$ are equivalent to those on the entries $\O_{r,r'}-\O_{r,-r'}$ given in (\ref{positivity_ADE}-\ref{positivity_ADE2}). 
The linear combinations of $\Omega_{m}(n_i)$ satisfying (\ref{positivity_ADE}-\ref{positivity_ADE2}) were shown in \cite{Cappelli:1987xt} to correspond to ADE root systems via Table \ref{ADE1}. 
Finally, the equality \eq{EZ_relation} follows from the equality \eq{EZ_on_theta}. 
\end{proof}

The relation between $\O^X$ and the ADE root system $X$ lies in the following two facts. 
First, $\O^X$ is a $2m\times 2m$ matrix where $m$ is the Coxeter number of $X$. 
Moreover,  $ \O^X_{r,r}-\O^X_{r,-r}=\a_r^X$ for $r=1,\dots,m-1$ coincides with the multiplicity of $r$ as a Coxeter exponent of $X$ (cf. Table \ref{tab:CoxNum}). 
Note the striking similarity between the expression for $\O^X$ and the denominator of the Coxeter Frame shape $\pi^X$ (cf. \S\ref{sec:holes:rootsys}).
For instance, $\O^X=\O_m(1)$ for $X=A_{m-1}$ is nothing but the $2m\times 2m$ identity matrix.

More generally, for a union  $X=\bigcup_i\! X_i$ of simply-laced root systems with the same Coxeter number, we let
\be\label{linear_combination_Omega_EZ}
\O^{X} = \sum_i \O^{X_i},\text{ and similarly }~{\cal W}^{X} = \sum_i {\cal W}^{X_i}.
\ee
Then we have the relation 
\be\label{EZ_relation}
 (\til h^T_m \cdot  \th_m) \lvert {\cal W}^X
 = \til h^T_m \cdot \O^X\cdot \th_m 
\ee
among different (odd) weight $k$ and index $m$ Jacobi forms corresponding to the same vector-valued modular form $\til h_m=(\til h_{m,r})$. Note that the operator ${\cal W}^X$ is in general no longer an involution and often not even invertible. The Eichler--Zagier operators ${\cal W}^X$ corresponding to Niemeier root systems (cf. \S\ref{sec:holes:lats}) will play a central role in \S\ref{sec:mckay}. 

The $\widehat A_1$ characters, which have led to the ADE classification of  Cappelli--Itzykson--Zuber, are rather ubiquitous in two-dimensional conformal field theory. They can be viewed as the building blocks of, for instance, the characters of $N=4$ super-conformal algebra and the partition functions of $N=2$ minimal models (cf. e.g. \cite{Gepner:1986hr,Eguchi1989}). Moreover, the positivity and the integrality condition (\ref{positivity_ADE}-\ref{positivity_ADE2}) that are necessary to obtain the ADE classification are completely natural from the point of view of the conformal field theories. This might be seen as suggesting a relationship between umbral moonshine and two-dimensional conformal field theories. The concrete realisation of such a relationship is beyond the scope of the present paper.

%---------------------------------------------------------%
\subsection{From Meromorphic Jacobi Forms to Mock Modular Forms}\label{sec:forms:meromock}
%---------------------------------------------------------%

In \S\ref{sec:forms:mock} we have seen the definition of mock modular form and its vector-valued generalisation. 
One of the natural places where such vector-valued mock modular forms occur is in the theta expansion of meromorphic Jacobi forms.
To be more precise, following \cite{zwegers} and \cite{Dabholkar:2012nd} we will establish a uniform way to separate a meromorphic Jacobi form $\psi$ into its {\em polar} and  {\em finite} parts
\be
\psi^{}(\t,z) = \psi^{P}(\t,z)+\psi^{F}(\t,z). 
\ee
The finite part will turn out to be a mock Jacobi form, admitting a theta expansion as in (\ref{eqn:forms:jac:thetaxpn}), whose theta-coefficients are the components of a vector-valued mock modular form.
Up to this point, all the mock modular forms playing a role in umbral moonshine, as well as  
other interesting examples including many of Ramunanjan's mock theta functions, can be obtained in this way as theta-coefficients of finite parts of meromorphic Jacobi forms.

For the purpose of this paper we will focus on the case of weight $1$ Jacobi forms with simple poles as a function of $z$. 
Consider such a Jacobi form $\psi$ with a pole at $z=z_s$, where $z_s$ is a point inside the fundamental cell $\a\t+\b$, $\a,\b \in (-1,0]$.   
The elliptic transformation (\ref{elliptic}) then forces $\psi$ to have poles at all  $z \in z_s+ \mathbb Z+\tau \mathbb Z$. 
With this property in mind, in the rest of the paper we will only write down the location of the poles inside the fundamental cell. 
To capture this translation property of the poles, following \cite{Dabholkar:2012nd} we define an averaging operator 
\be
{\rm Av}_{m}\left[F(y)\right] = \sum_{k\in \ZZ} q^{m k^2} y^{2m k} F(q^k y )
\ee
which takes a function of $y=e(z)$ with polynomial growth and returns a function of $z$ which is invariant under the index $m$ elliptic transformations \eqref{elliptic}. 
For a given pole $z=z_s$ of a weight $1$ index $m$ meromorphic Jacobi form $\psi$, we will consider  the image $\psi_{z_s}^P$ under ${\rm Av}_{m}$ of a suitably chosen meromorphic function $F_{z_s}(y)$ that has a pole at $z=z_s$, such that $\psi- \psi_{z_s}^P$ is regular at all  $z \in z_s+ \mathbb Z+\tau \mathbb Z$. 

In the remaining part of this subsection we will first review (following \cite{zwegers,Dabholkar:2012nd}) this construction of mock modular forms in more detail, 
and then extend the discussion of the Eichler--Zagier operators to 
meromorphic Jacobi forms and study how they act on the polar and the finite part separately. 
This will allow us to establish an ADE classification of mock Jacobi forms of a specific type in the next section and constitutes a crucial element in the construction of the umbral mock modular forms $H^X$.

\vspace{15pt}
\noindent{\em A Simple Pole at $z=0$}
\vspace{5pt}

To start with, consider a meromorphic Jacobi form $\psi^{}(\t,z)$ of weight $1$ and index $m$, with a simple pole at $z=0$ and no other poles.
Define the polar part of $\psi$ to be
\be \label{polar_simplest}
\psi^{P}(\tau,z)=  \chi(\t)
	{\rm Av}_{m}\left[ \frac{y+1}{y-1} \right] 
\ee
where $\chi(\t)/\p i$ is the residue of $\psi(\t,z)$ at $z=0$. 
For the applications in the present paper we need only consider the case that $\chi(\t)=\chi$ is a constant.
With this definition, one can easily check that $\psi^F = \psi-\psi^P$ is indeed a holomorphic function with no poles in $z$. 

Note that  
\be
\m_{m,0} (\t,z) = {\rm Av}_{m}\left[\frac{y+1}{y-1}\right],
\ee
where we define the generalised Appell--Lerch sum
\be\label{gAPsum}
\m_{m,j} (\t,z) =(-1)^{1+2j}  \sum_{k\in \ZZ}  q^{m k^2} y^{2 m k} \frac{(yq^{k})^{-2j}+(yq^{k})^{-2j+1}+\dots+(yq^k)^{1+2j} }{ 1-yq^k}\ee
for $j\in\tfrac{1}{2}\ZZ$. The function $\m_{m,0}$ enjoys the following relation to the modular group $\SL_2(\ZZ)$. Define the {\em completion} of $\m_{m,0} (\t,z)$ by setting
\be\label{pole_completion}
\widehat \m_{m} (\t,\bar \t,z) = \m_{m,0} (\t,z) - e(-\tfrac{1}{8}) \frac{1}{\sqrt{2m}}   \sum_{r \,({\rm mod } \,2m)} \th_{m,r}(\t,z)  \int^{i\inf}_{-\bar \t}  (\t'+\t)^{-1/2} \overline{S_{m,r}(-\bar \t')} \, {\rm d}\t', 
\ee
where $S_{m,r} (\t) $ denotes the unary theta series 
\be\label{def:S}
S_{m,r}  (\t) =-S_{m,-r}  (\t) = \frac{1}{2\p i} \frac{\pa}{\pa z} \th_{m,r} (\t,z)\big\lvert_{z=0} ,
\ee
then $\widehat\m_{m}$ transforms like a Jacobi form of weight $1$ and index $m$ for $\SL_2(\ZZ)$ but is clearly no longer holomorphic whenever $m>1$. 
From the above definition and the transformation \eq{transf_theta_2} of the theta functions, we see that $S_m =(S_{m,r})$ is a (vector-valued) weight 3/2 cusp form for $\SL_2(\ZZ)$.

Returning to our weight $1$ index $m$ Jacobi form $\psi$, assumed to have a simple pole at $z=0$, it is now straightforward to see that the finite part $\psi^F=\psi-\psi^P$, a holomorphic function on $\mathbb H\times \mathbb C$, has a completion given by
\be
\widehat\psi^F = \psi^F +\chi\,  e(-\tfrac{1}{8})  \frac{1}{\sqrt{2m}} \sum_{r \,({\rm mod } \,2m)} \th_{m,r}(\t,z)  \int^{i\inf}_{-\bar \t}  (\t'+\t)^{-1/2} \overline{S_{m,r}(-\bar \t')} \, {\rm d}\t'
\ee
that transforms like a Jacobi form of weight $1$ and index $m$ for $\SL_2(\ZZ)$. As such, $ \psi^F$ is an example of a {\em mock Jacobi form} (cf. \cite[\S7.2]{Dabholkar:2012nd}).
Since both $\psi$ and $\psi^{P}$ are invariant under the index $m$  elliptic transformation, so is the finite part $\psi^{F}$. 
This fact guarantees a theta expansion of $\psi^{F}$ analogous to that of a (weak) Jacobi form \eq{eqn:forms:jac:thetaxpn}
\be\label{mock_from_decomp}
\psi^{F}(\t,z) =\sum_{r \,(\rm{ mod }\,2m )} h_r(\t) \,  \th_{m,r}(\t,z), 
\ee
where $h=(h_r)$ is a weight 1/2 vector-valued holomorphic function on $\mathbb H$ whose completion
\be
\widehat h_r(\t) =h_r(\t)  +\chi\,  e(-\tfrac{1}{8}) \frac{1}{\sqrt{2m}} \int^{i\inf}_{-\bar \t}  (\t'+\t)^{-1/2} \overline{S_{m,r}(-\bar \t')} \, {\rm d}\t'
\ee
transforms as a weight $1/2$ vector-valued modular form with $2m$ components. 
As such, we conclude that $h=(h_r)$ is a vector-valued mock modular form for the modular group $\SL_2(\ZZ)$  with shadow $\chi \,S_m= (\chi\,S_{m,r})$. 

Note that $S_1$ vanishes identically. This is a reflection of the fact that $\mu_{1,0}$, in contrast to the $\mu_{m,0}$ for $m>1$, coincides with its completion, and is thus (already) a meromorphic Jacobi form, of weight $1$ and index $1$. By construction it has simple poles at $z\in\ZZ\t+\ZZ$ and nowhere else, and we also have the explicit formula
\be\label{CoverA}
\mu_{1,0}(\t,z)= -i \frac{\theta_1(\tau,2z)\, \eta(\tau)^3}{\theta_1(\tau,z)^2}= \frac{y+1}{y-1}- (y^2-y^{-2})\,q+ \cdots.
\ee
(See \S\ref{sec:JacTheta} for $\th_1(\t,z)$.) The function $\mu_{1,0}$ is further distinguished by being a meromorphic Jacobi form with vanishing finite part; a ``Cheshire cat'' in the language of \cite[\S8.5]{Dabholkar:2012nd}. It will play a distinguished role in \S\ref{sec:forms:umbral}, where it will serve as a device for producing meromorphic Jacobi forms of weight $1$ from (weak, holomorphic) Jacobi forms of weight $0$.

\vspace{15pt}
\noindent{\em Simple Poles at $n$-Torsion Points}
\vspace{5pt}

Next we would like to consider the more general situation in which we have a weight $1$ index $m$ meromorphic Jacobi form $\psi$ with simple poles at more general torsion points 
$z\in \QQ \t+ \QQ $. We introduce the row vector with two elements $s= (\a~\b)$ to label the pole at $z_s =\a\t+\b$ and write $y_s=e(z_s)$.
For the purpose of this paper we will restrict our attention to the $n$-torsion points satisfying $nz \in \ZZ \tau +\ZZ$, where $n$ is a divisor of the index $m$. 
Focus on a pole located at say $z_s=\a\t+\b$ with $\a,\b \in \frac{1}{n}\ZZ$. Again following \cite{Dabholkar:2012nd}, we require the corresponding polar term to be given by the formula
\be\label{polar_part_from_individual_poles}
\psi^P_{z_s}(\t,z) = \p i\, {\text{Res}}_{z= z_s} (\psi(\t,z) )\,{\rm Av}_{m}\left[\left(\frac{y}{y_s}\right)^{-2m\a}\,\frac{y/y_s+1}{y/y_s-1}\right],
\ee
generalising \eq{polar_simplest}. One can easily check that $ \psi-\psi^P_{z_s}$ has no  pole at  $z\in z_s + \ZZ \t + \ZZ$. 

As before, the above polar part is invariant under the elliptic transformation by construction. 
To discuss its variance under the modular group, first notice that the transformation $(\t,z) \mapsto \g(\t,z)$ maps the pole at $z_s$ to a different pole according to $s\mapsto s\g$. As a result, to obtain a mock Jacobi form for $\SL_2(\ZZ)$ from a meromorphic Jacobi form with poles at $n$-torsion points $z_s$ (where $n$ is the smallest integer such that $z_s\in \frac{\t}{n}\ZZ+ \frac{1}{n}\ZZ$), we should consider  meromorphic Jacobi forms that have poles at {all} the $n$-torsion points. 
Moreover, the modular transformation of $\psi$ dictates that the residues of the poles satisfy $D_s(\g\t) = D_{s\g}(\t)$, where we have defined, after \cite{Dabholkar:2012nd},
\be
D_s(\t)  = e(m\a z_s) \,{\text{Res}}_{z= z_s} (\psi(\t,z)) . 
\ee

More specifically, we would like to consider the situation where $\psi$ satisfies 
\be\label{pole_residue_ntorsion}
{\text{Res}}_{z= -\frac{a}{n}\t-\frac{b}{n}} \psi (\t,z) = \chi e(-ma(a\t+b)/n^2)/n \p i,\text{ for $a,b=0,1,\dots,n-1$,} 
\ee
corresponding to the simplest case where the function $D_s(\t)$ is just a constant. 
Without loss of generality we will also assume for the moment that $\psi$ has no other poles, as the more general situation can be obtained by taking linear combinations. 
In this case, using \eq{polar_part_from_individual_poles} it is not hard to see that the polar parts contributed by the poles at these $n$-torsion points  are given by  
the images under the Eichler--Zagier operator ${\cal W}_m{(n)}$  (cf. \eq{Atkin--Lehner1}) of the polar term contributed by the simple pole at the origin, so that
\be\label{pol_n_torsion}
\psi^P = \sum_{a,b=0}^{n-1}\psi^P_{-\frac{a}{n}\t-\frac{b}{n}} = \chi \,\m_{m,0}\big\lvert{\cal W}_m{(n)} .
\ee

From \eq{pole_completion} and the fact that the  Eichler--Zagier operators preserve the Jacobi transformations, we immediately see how considering Jacobi forms with simple poles at torsion points leads us to vector-valued mock modular forms with more general shadows. 
In this case, from \eq{EZ_on_theta}, \eq{pole_completion} and  \eq{pol_n_torsion} it is straightforward to see that the completion of the polar part
\be\label{completion_polar_n_torsion}
\widehat \psi^P   = \psi^P - \chi\, e(-\tfrac{1}{8})\frac{1}{\sqrt{2m}}   \sum_{r,r' \,({\rm mod } \,2m)} \th_{m,r}(\t,z)\, \O_m(n)_{r,r'}  \int^{i\inf}_{-\bar \t}  (\t'+\t)^{-1/2} \overline{S_{m,r'}(-\bar \t')} \, {\rm d}\t'
\ee
again transforms like a Jacobi form of weight $1$ and index $m$ for $\SL_2(\ZZ)$.

Following the same argument as before, we conclude that the theta-coefficients of the finite part
\be
\psi^F = \psi - \psi^P =\sum_{r \,(\rm{ mod }\,2m )} h_r(\t) \,  \th_{m,r}(\t,z),
\ee
define a vector-valued mock modular form $h=(h_r)$, whose completion is given by 
\be
\widehat h_r(\t) =h_r(\t)  +\chi \,e(-\tfrac{1}{8})\, \frac{1}{\sqrt{2m}} \sum_{r' \,(\rm{ mod }\,2m )}\O_m(n)_{r,r'}\int^{i\inf}_{-\bar \t}  (\t'+\t)^{-1/2}  \overline{S_{m,r'}(-\bar \t')} \, {\rm d}\t',
\ee
and whose shadow is hence given by a vector of unary theta series whose $r$-th component equals
\be\label{torsion_shadow}
\sum_{r' \,(\rm{ mod }\,2m )} \Omega_m{(n)}_{r,r'} \, S_{m,r'} 
\ee
where $r\in \ZZ/2m\ZZ$.
In particular, this means that the vector-valued mock modular forms arising from meromorphic Jacobi forms in this way are closely related to mock theta functions, as their shadows are always given by unary theta series.

Finally, we also note that the 
Eichler--Zagier operators and the operations of extracting polar and finite parts are 
commutative in the following sense. 
\begin{prop}\label{proposition_EZ_commutes}
Suppose $\psi$ is a weight $1$ index $m$ meromorphic Jacobi form with simple poles at $\til n$-torsion points with $\til n|m$ and with no poles elsewhere. 
Then for any positive integer $n$ such that $n|m$ and $(n,\til n)=1$ we have
\be
\left(\psi\lvert {\cal W}_m(n)\right)^P = \psi^P\lvert {\cal W}_m(n) . 
\ee
\end{prop}

\begin{proof}
Denote the set of poles of $\psi$ in the unit cell by $S$, and focus on the pole of $\psi$ at $z_\ast=-\til a/\til n\t -\til b/\til n\in S$. From the action of ${\cal W}_m(n)$ we see that $\psi\lvert {\cal W}_m(n)$ has poles at all $z\in z_\ast + \frac{1}{n}\ZZ + \frac{\t}{n}\ZZ$. Focussing on the pole at $z=z_s=z_\ast-(a\t/n+b/n)$, from \eq{Atkin--Lehner1} we get 
\be
{\text{Res}}_{z= z_s} (\psi\lvert {\cal W}_m(n) )(\t,z)  = \frac{1}{n} \,e(\tfrac{2maz_s}{n})e(\tfrac{ma}{n^2}(a\t+b)) {\text{Res}}_{z= z_\ast} (\psi (\t,z)),
\ee
which leads to 
\begin{align}\notag
(\psi\lvert {\cal W}_m(n) )^P (\t,z) &= \frac{\p i}{n} \sum_{z_\ast=-\frac{\til a}{\til n}\t -\frac{\til b}{\til n}\in S} {\text{Res}}_{z= z_\ast} (\psi (\t,z))\\\notag& \times \sum_{a,b=0}^{n-1} e(\tfrac{2maz_s}{n})e(\tfrac{ma}{n^2}(a\t+b)) 
 \sum_{k\in \ZZ} q^{mk^2} y^{2mk} (q^ky/y_s)^{2m(\frac{a}{n}+\frac{\til a}{\til n})} \frac{q^ky/y_s+1}{q^ky/y_s-1} 
\end{align}
where $y_s$ denotes $y_s = e(z_s) = e(-(\til a/\til n+a/n)\t -(\til b/\til n+b/n))$ in the second line. 

By direct comparison  using \eq{Atkin--Lehner1} and (\ref{polar_part_from_individual_poles}), this is exactly $\psi^P\lvert {\cal W}_m(n) $ and this finishes the proof. 
\end{proof}

Since all the operations involved are linear,  we also have the following corollary.
\begin{cor}\label{cor:commute_EZ_polar}
Consider $\psi$ as defined as in Proposition \ref{proposition_EZ_commutes} and let $ {\cal W} =\sum_{i}c_i  {\cal W}_m(n_i)$ where the $n_i$ are divisors of $m$ satisfying $(n_i,\widetilde{n})=1$. Then 
\be
\left(\psi\lvert {\cal W}\right)^P = \psi^P\lvert {\cal W} 
\ee
and 
\be
\left(\psi\lvert {\cal W}\right)^F = \psi^F\lvert {\cal W}. 
\ee
Moreover, if we denote the theta-coefficients of $\psi^F$ by $h_\psi=((h_\psi)_r)$ and its shadow by $S_{\psi}=((S_{\psi})_r)$ with $r\in \ZZ/2m\ZZ$, then the theta-coefficients of $\left(\psi\lvert {\cal W}\right)^F$ form a vector-valued mock modular form $h_{\psi\lvert {\cal W}}$ satisfying 
\be
h_{\psi\lvert {\cal W}}= \O\, h_{\psi} , 
\ee
with shadow given by 
\be
S_{\psi\lvert {\cal W}}= \O\, S_{\psi}
\ee
where $\O = \sum_{i}c_i  {\O}_m(n_i)$ (cf. \eq{def:OmegaMatrices}).
\end{cor}

As a result, the relations \eq{EZ_relation} between different Jacobi forms of the same index can be applied separately to the polar and the finite part. 
In the present paper we will mostly be concerned with the application of the above to the case that $\til n=1$. 
For later use it will be useful to note the following property.

\begin{lem}\label{inversion_EZ}
For $\psi$ and $n$ as defined as in Proposition \ref{proposition_EZ_commutes}, we have  
\be
\psi^F\lvert {\cal W}(n) = -\psi^F\lvert {\cal W}(m/n),\quad \widehat \psi^F\lvert {\cal W}(n) = -\widehat \psi^F\lvert {\cal W}(m/n).
\ee
\end{lem}
\begin{proof}
From the property \eq{pol_n_torsion} of $\psi^P$  and the elliptic transformation of $\m_{m,0}$, it follows that 
$\psi^P(\t,z) = -\psi^P(\t,-z)$ and therefore $\psi^F(\t,z) = -\psi^F(\t,-z)$. As such, the vector-valued mock modular form $h=(h_r)$ arising from the theta expansion
\eq{mock_from_decomp} of $\psi^F$ satisfies $h_r=-h_{-r}$. Together with $S_{m,r}=-S_{m,-r}$, the lemma follows from the action of ${\cal W}_m{(n)}$ \eq{EZ_on_theta} on $\th_m=(\th_{m,r})$ and the definition \eq{def:OmegaMatrices} of the matrix $\O_m{(n)}$.
\end{proof}

%---------------------------------------------------------------------------------------%
\section{The Umbral Mock Modular Forms}\label{sec:forms:umbral}
%---------------------------------------------------------------------------------------%

Following the general discussion of the relevant automorphic objects in the previous section, in this section we will start specifying concretely the vector-valued mock modular forms which encode, according to our conjecture, the graded dimensions of certain infinite-dimensional modules for the umbral groups defined in \S\ref{sec:holes}. We will specify the shadows of these functions---the {\em umbral mock modular forms}---in \S\ref{sec:umbral shadow}. 
Subsequently, in \S\ref{sec:forms:genus0} we will show how these shadows distinguish the Niemeier root systems 
through a relation to genus zero groups and their principal moduli. Afterwards we will provide explicit expressions for the umbral forms of the A-type  Niemeier root systems by specifying a set of weight $0$ weak Jacobi forms. The umbral forms of D- and E-type Niemeier root systems will be specified in the next section.

\subsection{The Umbral Shadows} \label{sec:umbral shadow}

In \S\ref{sec:forms:meromock} we have seen how the theta expansion of the finite part of a meromorphic Jacobi form gives rise to a vector-valued mock modular form, and how different configurations of poles lead to different shadows. The shadows of the mock modular forms obtained in this way are always given by unary theta series. 
In this subsection we will see how the ADE classification discussed in \S\ref{sec:forms:ADE} leads to particular cases of the above construction.
Moreover, by combining the ADE classification and the construction of mock modular forms from meromorphic Jacobi forms discussed in \S\ref{sec:forms:meromock}, we will associate a specific shadow $S^X$, or equivalently a  pole structure of the corresponding meromorphic Jacobi form $\psi^X$, to each of the Niemeier root systems $X$.

Consider a meromorphic Jacobi form $\psi$ with weight $1$ and index $m$. 
Recall from \eq{pol_n_torsion} that 
 the contribution to its polar part $\psi^P$ from the simple poles at the $n$-torsion points with residues satisfying \eq{pole_residue_ntorsion}, is given by
 \be\label{eqn:umbshad:Wmn}
\sum_{a,b=0}^{n-1}\psi^P_{-\frac{a}{n}\t-\frac{b}{n}} = \chi \,\m_{m,0}\big\lvert{\cal W}_m{(n)}. 
\ee

Clearly, one may consider a linear combination of expressions as in (\ref{eqn:umbshad:Wmn}). 
Consider a weight $1$ index $m$ meromorphic Jacobi form $\psi$ with poles at $n_1$-,\dots,$n_{\k}$-torsion points where $n_j|m$, 
and where each $n_j$ contributes $c_j e(-ma (a\t+b)/n_j^2)/n_j \p i$ to the residue of the pole located at $ -\frac{a}{n_j}\t-\frac{b}{n_j}$.
From the above discussion it follows that its polar part is given by 
\be
 \psi^P =  \m_{m,0}\big\lvert{\cal W},\;{\text{where}}\;{\cal W} = \sum_{i=1}^\k c_i {\cal W}_m(n_i). 
 \ee 
By taking a linear combination of expressions as in \eq{completion_polar_n_torsion} we see that its completion, given by 
\be
\widehat \psi^P   = \psi^P -   e(-\tfrac{1}{8}) \frac{1}{\sqrt{2m}}  \sum_{r,r' \,({\rm mod } \,2m)} \th_{m,r}(\t,z)\, \O_{r,r'}  \int^{i\inf}_{-\bar \t}  (\t'+\t)^{-1/2} \overline{S_{m,r'}(-\bar \t')} \, {\rm d}\t', 
\ee
where ${\O} = \sum_{i=1}^\k c_i {\O}_m(n_i)$, transforms like a Jacobi form of weight $1$ and index $m$ for $\SL_2(\ZZ)$. 
Immediately we conclude that the theta-coefficients of the finite part $\psi^F=\psi-\psi^P$ constitute a vector-valued mock modular form with (the $r$-th component of) the shadow given by
\begin{gather}
 \sum_{r' \,({\rm mod } \,2m)}\O_{r,r'}  S_{m,r'}.
\end{gather}

Now, recall that in \S\ref{sec:forms:ADE} we used the reflection property $\til h_{m,r} = - \til h_{m,-r}$ to impose a positivity condition which then led to an ADE classification (cf. Theorem \ref{thm_ADE}). Analogously, in the context of meromorphic Jacobi forms we also have a natural positivity condition that we want to impose. 
Using the reflection property $S_{m,r} = - S_{m,-r}$, for $r \in \ZZ/2m\ZZ$ (cf. \eq{def:S}), we may instead consider  a $(m-1)$-component vector with the $r$-th component given by 
\begin{gather}
 \sum_{r'=1}^{m-1} (\O_{r,r'}-\O_{r,-r'})  S_{m,r'}, \;\; r=1,\dots, m-1. 
\end{gather}
Requiring that each component of this $(m-1)$-component vector is a non-negative linear combination of the unary theta series $S_{m,r}$, for $r=1,2,\dots,m-1$, with the normalisation $ \O_{1,1}-\O_{1,-1} =1$,
from Theorem \ref{thm_ADE} we immediately see that such $\O$ (and, equivalently, ${\cal W}$) are classified by ADE root systems. 
More precisely, to each irreducible simply-laced root system $X$ with Coxeter number $m$ we associate a $2m$-vector-valued cusp form $S^X$, of weight $3/2$ for $\SL_2(\ZZ)$, with $r$-th component given by
\be
 S^X_r = \sum_{r' ~{\text{mod}}~ 2m }^{} \O_{r,r'}^X S_{m,r'}, 
\ee
where $m$ denotes the Coxeter number of $X$ and the matrix $\O^X$ is defined as in Table \ref{ADE1}. 
For instance, we have 
\be S^{A_{m-1}}_r = S_{m,r}.\ee
For the D-series root systems with even rank we have 
\begin{gather}\label{eqn:shadow_even_D}
	S^{D_{2n}}_r
	=
	\begin{cases}
	S_{4n-2,r}+S_{4n-2,4n-2-r}&\text{ if $r$ is odd,}\\
	0&\text{ if $r$ is even,}
	\end{cases}
\end{gather}
and in the case that $n$ is odd we have
\begin{gather}
	S^{D_n}_r
	=
	\begin{cases}
	S_{2n-2,r}&\text{ if $r$ is odd,}\\
	S_{2n-2,2n-2-r}&\text{ if $r$ is even.}
	\end{cases}
\end{gather}
For $E_6$, $E_7$ and $E_8$ we have
\begin{gather}\label{eqn:E6_shadow}
	S^{E_6}_r
	=
	\begin{cases}
	S_{12,1}+S_{12,7}&\text{ if $r=1$ or $r=7$,}\\
	S_{12,4}+S_{12,8}&\text{ if $r=4$ or $r=8$,}\\
	S_{12,5}+S_{12,11}&\text{ if $r=5$ or $r=11$,}\\
	0&\text{ otherwise;}
	\end{cases}
\end{gather}
\begin{gather}
	S^{E_7}_r
	=
	\begin{cases}
	S_{18,1}+S_{18,17}&\text{ if $r=1$ or $r=17$,}\\
	S_{18,9}&\text{ if $r=3$ or $r=15$,}\\
	S_{18,5}+S_{18,13}&\text{ if $r=5$ or $r=13$,}\\
	S_{18,7}+S_{18,11}&\text{ if $r=7$ or $r=11$,}\\
	S_{18,3}+ S_{18,9}+S_{18,15}&\text{ if $r=9$,}\\
	0&\text{ otherwise;}
	\end{cases}
\end{gather}
\begin{gather}\label{eqn:E8_shadow}
	S^{E_8}_r
	=
	\begin{cases}
	S_{30,1}+S_{30,11}+S_{30,19}+S_{30,29}&\text{ if $r\in\{1,11,19,29\}$,}\\
	S_{30,7}+S_{30,13}+S_{30,17}+S_{30,23}&\text{ if $r\in\{7,13,17,23\}$,}\\
	0&\text{ otherwise;}
	\end{cases}
\end{gather}
where for simplicity we have only specified $S^{E_n}_r$ for $r\in\{1,\cdots,m-1\}\subset \ZZ/2m\ZZ$. The remaining components are determined by the rule $S^{E_n}_{-r}=-S^{E_n}_r$. 
 
Following \eq{linear_combination_Omega_EZ}, more generally for a union  $X=\bigcup_i\! X_i$ of simply-laced root systems with the same Coxeter number we have
 $S^{X} = \sum_i S^{X_i}$. With this definition, $S^{X}$ is given by the matrix $\O^X$ as
 \be\label{def:um_shadow}
 S^X = (S^X_r) \quad, \quad S^X_r = \sum_{r' ~{\text{mod}}~ 2m }^{} \O_{r,r'}^X S_{m,r'}.
 \ee

From the above discussion we see that 
the cusp form $S^X$ arises naturally as the shadow of a vector-valued mock modular form obtained from the theta expansion of a meromorphic Jacobi form $\psi^X$ with simple poles at $z\in \frac{\t}{n_i} \ZZ + \frac{1}{n_i} \ZZ $ for all $n_i|m$ such that $c_i>0$ in $\O^X = \sum_{i} c_i \O_m{(n_i)}$,
whose polar part is given by 
\be\label{eqn:polar_part_umbral_mock_jac}
(\psi^X)^P= \m_{m,0}(\t,z)\lvert {\cal W}^X . 
\ee

It is not hard to see that such meromorphic Jacobi forms exist for any simply-laced root system $X$ where all the irreducible components have the same Coxeter number.  Choose an arbitrary weight 0, index $m-1$ weak Jacobi form $\phi$ with $\phi(\t,0)\neq 0$ and assume, without loss of generality, that $\phi(\t,0) =1$. 
Recall from the definition of weak Jacobi forms in \S\ref{sec:forms:jac} that, if $\f(\t,z)$ is a weight $0$ weak Jacobi form then $\f(\t,0)$ is a weight $0$ modular form, which is necessarily a constant. 
As such, by multiplying with the meromorphic Jacobi form $\mu_{1,0}$ of weight $1$ and index $1$ (cf. (\ref{CoverA})) 
we obtain a weight 1 index $m$ meromorphic Jacobi form with simple poles at $z\in \mathbb Z \t+  \mathbb Z$ and nowhere else, and with the polar part given by $\m_{m,0}$. (See \cite[\S2.3]{UM} for a more detailed discussion, but note that $\mu_{1,0}$ is denoted there also by $\Psi_{1,1}$.) 
The Corollary \ref{cor:commute_EZ_polar} then shows that its image $
-\mu_{1,0}\phi\lvert {\cal W}^X$ under the corresponding Eichler--Zagier operator
is a weight $1$ index $m$ meromorphic Jacobi with polar part coinciding with \eq{eqn:polar_part_umbral_mock_jac}. 

Besides specifying the poles of the meromorphic Jacobi form, 
in what follows we will also require  an {\em optimal growth} condition (equivalent to the optimal growth condition formulated in \cite{Dabholkar:2012nd}, with its name derived from the fact that it guarantees the slowest possible growth of the coefficients of the corresponding mock Jacobi form). It  turns out that there does not always exist a meromorphic Jacobi form with polar part given by \eq{eqn:polar_part_umbral_mock_jac} that moreover satisfies this optimal growth condition for an arbitrary simply-laced root system $X$, but when it does exist it is unique. 

\begin{thm}\label{thm:uniqueness_umbral_mock_jac}
Let $X$ be a simply-laced root system with all irreducible components having the same Coxeter number $m$.  There exists at most one weight $1$ index $m$ meromorphic Jacobi form $\psi^X$ satisfying the following two conditions. First, its polar part $(\psi^X)^P$ is given by \eq{eqn:polar_part_umbral_mock_jac}. Second, its finite part 
\be
(\psi^X)^F = \psi^X-(\psi^X)^P = \sum_{r } h_r^X\th_{m,r}
\ee
satisfies the optimal growth condition
\be\label{optimal_growth}
q^{1/4m} h_r^X(\t) = O(1) 
\ee
as $\t\to i\inf$ for all $r\in \ZZ/2m\ZZ$.  
\end{thm}

\begin{proof}
If there are two distinct meromorphic Jacobi forms satisfying the above conditions, then their difference is necessarily a weight $1$ index $m$ weak Jacobi form of optimal growth in the sense of \cite{Dabholkar:2012nd}. But Theorem 9.7 of \cite{Dabholkar:2012nd} is exactly the statement that no such weak Jacobi form exists.
\end{proof}

\begin{cor}\label{cor:uniqueness_umbral_mock_mod}
Let $X$ be a simply-laced root system with all irreducible components having the same Coxeter number $m$.  
Then there exists at most one vector-valued mock modular form $h^X$ for $\SL_2(\ZZ)$ with shadow  $S^X$ that satisfies the optimal growth condition \eq{optimal_growth}.
\end{cor}

\begin{proof}
Let $h^X$ be a vector-valued mock modular form  satisfying the above conditions. 
Consider $\psi'= (\psi^X)^P + \sum_{r } h_r^X \th_{m,r}$ with $(\psi^X)^P$  given by \eq{eqn:polar_part_umbral_mock_jac}. 
From the fact that the multiplier of a mock modular form is the inverse of that of its shadow and from the
discussion in \S\ref{sec:forms:meromock},  $\psi'$ is a weight $1$ index $m$ meromorphic Jacobi form satisfying the conditions of Theorem \ref{thm:uniqueness_umbral_mock_jac}.  It then follows from Theorem \ref{thm:uniqueness_umbral_mock_jac} that such $h^X$ is unique if it exists.
\end{proof}
So far our discussion has been very general, applicable to any simply-laced root system with all irreducible components having the same Coxeter number. 
In the next subsection we will see how the cusp forms $S^X$ with $X$ given by one of the 23 Niemeier root systems play a distinguished role. The $S^X$ for $X$ a Niemeier root system are called the {\em umbral shadows},
\be\label{eqn:umbralshadows}
S^X_r = \sum_{r' \in \ZZ/2m\ZZ} \O^X_{r,r'} S_{m,r'}. 
\ee

We will demonstrate the existence of meromorphic Jacobi forms $\psi^X$ satisfying the conditions of Theorem \ref{thm:uniqueness_umbral_mock_jac} for $X$ an A-type Niemeier root system in Proposition \ref{prop:A_UmbralForms}, and for D- and E-type Niemeier root systems in Proposition \ref{prop:DE_UmbralForms}. The resulting vector-valued functions defined by the theta-coefficients of the finite parts of these $\psi^X$ are the {umbral mock modular forms}, to be denoted $H^X=(H^X_r)$. The first few dozen Fourier coefficients of the umbral mock modular forms are tabulated in Appendix \ref{sec:coeffs}. For some of the Niemeier root systems $X$, the meromorphic Jacobi forms $\psi^X$ are closely related to some of the meromorphic Jacobi forms analysed in  \cite[\S9.5]{Dabholkar:2012nd}; for other Niemeier root systems $X$, the corresponding shadows $S^X$ fall outside the range of analysis in \cite{Dabholkar:2012nd} and this is why we find mock Jacobi forms of optimal growth for values of $m$ other than those appearing in the work of Dabholkar--Murthy--Zagier.

%---------------------------------------------------------------------------------------%
\subsection{From Niemeier Lattices to Principal Moduli}\label{sec:forms:genus0}
%---------------------------------------------------------------------------------------%

In the last subsection we have seen how ADE root systems have an intimate relation to  meromorphic Jacobi forms. 
More precisely, to a simply-laced root system with all irreducible components having the same Coxeter number $m$, we associate a pole structure for meromorphic Jacobi forms of weight $1$ and index $m$. 
Equivalently, we associate a weight $3/2$ vector-valued cusp form to every such root system, which plays the role of the shadow of the mock modular form arising from the meromorphic Jacobi form via the relation discussed in \S\ref{sec:forms:meromock}.
In this subsection we see how the shadows $S^X$ attached to Niemeier root systems are distinguished, and in particular, how they are related to the genus zero groups $\G^X$ of \S\ref{sec:holes:gzero}.

Recall \cite{MR1074485} that a {\em skew-holomorphic Jacobi form} of weight 2 and index $m$ is a smooth function $\til\f(\t,z)$ on $\HH\times \CC$ which is periodic in both $\t$ and $z$ with period $1$, transforms under the $S$-transformation as
\be
\til \f (-\tfrac{1}{\t},\tfrac{z}{\t})  \,e(-m\tfrac{z^2}{\t}) = \bar \t |\t| \til \f(\t,z)
\ee
and has a Fourier expansion
\be
\til \f (\t,z)   = \sum_{\substack{\Delta,r \in \mathbb Z \\ \Delta = r^2 {\text{ mod }} 4m}} C_{\til \f}(\D,r) \,\ex\!\left(\frac{r^2-\D}{4m}\Re (\t)+\frac{r^2+|\D|}{4m} i \Im(\t) + rz\right)
\ee
where $C_{\til \f}(\D,r)=0$ for $\D<0$. We denote the space of such functions by $J_{2,m}^+$. Recall that an integer $\D$ is called a {fundamental discriminant} if $\D=1$ or $\D$ is the discriminant of a quadratic number field.
Following Skoruppa \cite{MR1074485} (see also \cite{MR1116103}), given a pair $(\D_0,r_0)$ where $\D_0$ is a positive fundamental discriminant that is a square modulo $4m$ and $\Delta_0 =r_0^2 {\text{ mod }} 4m$, 
we may associate a weight 2 modular form 
\be
{\mathscr S}_{\Delta_0,r_0} (\til \f)=  c_{\til \f}(\Delta_0,r_0) + \sum_{n\geq 1} q^n \sum_{a|n} \left(\frac{\D_0}{a}\right) C_{\til \f}\,(\D_0\tfrac{n^2}{a^2},r_0\tfrac{n}{a}) 
\ee
for $\G_0(m)$
to each $\til \f\in J_{2,m}^{+}$, where $\left(\frac{\D_0}{a}\right)$ denotes the Jacobi symbol and $c_{\til \f}(\Delta_0,r_0)$ denotes a suitably chosen constant term. 

From the discussion in \S\ref{sec:forms:ADE}, it is not difficult to see that given a simply-laced root system $X$ with each of its irreducible components having the same Coxeter number $m$, we may consider the following skew-holomorphic form
\be\label{eqn:sigmaX}
\sigma^X = (\overline{S^X})^T \cdot \th_m = \sum_{r \, ({\rm mod}\,2m) } \overline{S^X_r (\t)}\th_{m,r}(\t,z) = \sum_{r ,r'\, ({\rm mod}\,2m)} \overline{S_{m,r} (\t)} \,\O^X_{r,r'} \th_{m,r'}(\t,z) 
\ee
of weight $2$ and index $m$.

Applying ${\mathscr S}_{\D_0,r_0}$ with the simplest choice $(\D_0,r_0)=(1,1)$ to the skew-holomorphic Jacobi form $\sigma^X$, and using the fact that the Jacobi symbol $\left(\tfrac{1}{a}\right)=1$ for all positive integers $a$, we arrive at a weight 2 form on $\G_0(m)$
\begin{align}
f^{X}(\t) &
= {\mathscr S}_{1,1}(\sigma^X)(\t)
=\frac{{\sf{r}}}{24} +\sum_{r=1}^{m-1} \a^{ X}_{r} \sum_{\varepsilon=\pm1}\sum_{k=-\lfloor \frac{\varepsilon r}{2m} \rfloor}^\infty (2km+\varepsilon r) \frac{q^{2km+\varepsilon r}}{1-q^{2km+\varepsilon r}}\label{def:weight2}
\end{align}
where $\sf{r}$ is the rank of the root system $X$ and $\a^{ X}_{r}$ is the multiplicity of the multiplicity of $r$ as a Coxeter exponent of $X$, which 
also coincides with the ``diagonal'' coefficient $\O^X_{r,r}-\O^X_{r,-r}$ of $S_{m,r}$ in the $r$-th component  of the vector-valued cusp form $S^X=(S^X_r)$ (cf.  Table \ref{tab:CoxNum} and \S\ref{sec:forms:ADE}). 

\begin{table}
\captionsetup{font=small}
\centering
\begin{tabular}{CCCC}\toprule
 X & m(X)&\pi^X & f^X \\\midrule
A_{m-1} & m& \frac{m}{1}  & \l_m	\vspace{0.2em}\\
D_{{m}/{2}+1} &m&  \frac{2.m}{1.(m/2)} & \l_2+\l_m-\l_{m/2}	\vspace{0.2em}\\
E_6 & 12&\frac{2.3.12}{1.4.6} & \l_2+\l_3+\l_{12} -\l_4-\l_6	\vspace{0.2em}\\ 
E_7& 18&\frac{2.3.18}{1.6.9} & \l_2+\l_3+\l_{18} -\l_6-\l_9	\vspace{0.2em}\\
E_8 & 30&\frac{2.3.5.30}{1.6.10.15} &\l_2+\l_3+\l_5+\l_{30} -\l_{6}-\l_{10}-\l_{15}\vspace{0.1em}\\
\bottomrule
\end{tabular}
\caption{\label{tab:ADEfX}{The weight 2 modular forms $f^X$ 
	}}
\end{table}

For $X=A_{m-1}$, one can easily see from the fact that $ \a^{ X}_{r}=1$ for all $r\in\{1,2,\dots,m-1\}$ that the associated weight--two form is nothing but the following Eisenstein form at level $m$ (cf. \eq{Eisenstein_form})
\bea\label{weight2:Aseries}
f^{A_{m-1}}(\t)= \l_m(\t)=\, q\pa_q\log\left(\frac{\eta(m\tau)}{\eta(\tau)}\right)=\frac{m-1}{24}+\sum_{k>0}\s(k) (q^k -m q^{mk}).
\eea
One can compute the function $f^X$ in a similar way for the D- and E-series and arrive at the result in Table \ref{tab:ADEfX}. 
From this table  one observes that the weight two form $f^X$ has a close relation to the Coxeter Frame shape $\pi^X$, and hence also to the matrix $\O^X$ according to Table \ref{ADE1}. We now discuss this further. 

From $\O^X$ where $X$ has Coxeter number $m$, we can obtain a map on the space spanned by the weight two modular forms $\{\l_n(\t), n|m\}$ of level $m$ by replacing each $\O_m{(m_1)}$ in $\O^X$ with the operator $w_m(m_1)$ which acts on Eisenstein forms according to
 \be\label{def:weigh2preAL}
 \l_{m_2}|w_m(m_1) = \l_{m_1\ast m_2} - \l_{m_1}, 
 \ee
where $m_1$ and $m_2$ are assumed to divide $m$, and $m_1\ast m_2=m_1m_2/(m_1,m_2)^2$ (cf. \S\ref{sec:holes:gzero}).
 Then the weight two form corresponding to a simply-laced root system $X$ with Coxeter number $m$ is nothing but 
\be\label{weight2_ALrel} 
f^X =  \l_m \lvert w^X=\l_{\pi^X}
\ee
where $\pi^X$ is the Coxeter Frame shape of $X$ (cf. \S\ref{sec:holes:rootsys}) and $\l_{\pi}$, for $\pi$ an arbitrary Frame shape, is defined in \S\ref{sec:holes:gzero}.

The reader will notice that the map $ w_m(e) $ acts in the same way as the Atkin--Lehner  involution ${W}_m(e)$  \eq{def:AT_inv} on modular forms for $\G_0(m)$ in the cases where $e \| m$.
(Notice that the set ${W}_m(e)$ is empty if $e$ is not an exact divisor of $m$.)
Indeed, from the definition of the Eisenstein form \eq{Eisenstein_form} one can compute that 
\be
(\l_f\lvert {W}_m(e)) (\t) = (\l_f\lvert_{\psi=1,w=2} \g) (\t)=  e (cm\t+de)^{-2} \l_f(\tfrac{ae\t+b}{cm\t+de}) = \l_{e\ast f}(\t)  - \l_e(\t) ,
\ee 
where $\g = \frac{1}{\sqrt{e}}\left( \begin{smallmatrix} ae & b\\ c m  & d e\end{smallmatrix}\right) \in\SL_2(\RR)$  (cf. \eq{eqn:sums:psiw_actn}). 
On the other hand, from \eq{EZ_on_theta} we see that the skew-holomorphic form $\sigma^X$ can be obtained as the image of the Eichler--Zagier operator
\be
\sigma^X  = \sigma^{A_{m-1}}\lvert  {\cal W}^X.  
\ee
Taken together, at a given $m$ and for a given union $X$ of simply-laced root systems with Coxeter number $m$ we have the equality
\be\label{weight2_EZ}
f^X={\mathscr S}_{1,1}\left( \sigma^{A_{m-1}}\lvert  {\cal W}^X\right) =\sum_i {\mathscr S}_{1,1}\left( \sigma^{A_{m-1}}\lvert  {\cal W}_m(e_i) \right),
\ee
where we have written ${\cal W}^X=\sum_i{\cal W}_m(e_i)$ explicitly in terms of its different components. 
At the same time, for the cases that all $e_i||m$, we also have 
\be\label{weight2_AT}
f^X=\sum_i f^{A_{m-1}}\lvert W_m(e_i)  .
\ee
For these cases, we note that the equality $f^{A_{m-1}}\lvert W_m(e_i) = {\mathscr S}_{1,1}\left( \sigma^{A_{m-1}}\lvert  {\cal W}_m(e_i) \right)$  can be viewed as a consequence of the relation between the Eichler--Zagier operators and the Atkin--Lehner involutions observed in \cite{MR958592,MR1074485}.  
Due to this relation, the Eichler--Zagier operators that define involutions are sometimes referred to as Atkin--Lehner involutions on Jacobi forms in the literature.

We conclude this section by observing that if $X$ is a Niemeier root system then 
\begin{gather}
{\mathscr S}_{1,1}(\sigma^X)=f^X=-q\partial_q\log T^X
\end{gather}
where $\sigma^X$ is the skew-holomorphic Jacobi form defined by $S^X$ in (\ref{eqn:sigmaX}) and $T^X$ is the principal modulus for $\G^X$ defined in \S\ref{sec:holes:gzero}. In this way we obtain a direct connection between the umbral shadows $S^X$ and the genus zero groups $\G^X$ attached to Niemeier root systems in \S\ref{sec:holes:gzero}.

\subsection{From Weight Zero Jacobi Forms to Umbral Mock Modular Forms } \label{sec:weight_zero_umbral_forms}

The goal of this subsection is to construct the umbral mock modular form $H^X$ which (conjecturally) encodes the graded dimension of the umbral module $K^X$ (cf. \S\ref{sec:conj:mod}) for every A-type  Niemeier root system $X$. 
In \S\ref{sec:umbral shadow} we have seen how to associate an umbral shadow $S^X$ to a Niemeier root system $X$.
Equivalently, we can associate a pole structure, which together with the optimal growth condition (\ref{optimal_growth}) determines a(t most one) meromorphic weight $1$ index $m$ Jacobi form $\psi^X$ according to Theorem \ref{thm:uniqueness_umbral_mock_jac}.
In this subsection we will explicitly construct meromorphic Jacobi forms $\psi^X$ satisfying the conditions of Theorem \ref{thm:uniqueness_umbral_mock_jac} for $X$ an A-type root system via certain weight $0$ Jacobi forms $\f^X$ (cf. Proposition \ref{prop:A_UmbralForms}).  
After obtaining these $\psi^X$, the procedure discussed in \S\ref{sec:forms:meromock} then immediately leads to the umbral forms $H^X$.
It is also possible to specify the D- and E-type Niemeier root systems in a similar way. However, the discussion would become somewhat less illuminating and we will instead specify these in an arguably more elegant way in \S\ref{sec:mckay}. 

Our strategy in the present subsection is the following. As mentioned in \S\ref{sec:umbral shadow}, if we take a weight $0$ index $m-1$ weak Jacobi form $\phi$ with $\phi(\t,0) =1$, then $
-\mu_{1,0}\phi\lvert {\cal W}^X$ is a weight $1$ index $m$ meromorphic Jacobi whose pole structure is automatically of the desired form (cf. \eq{CoverA}). Namely, it always leads to a vector-valued mock modular form whose shadow is given by $S^X$. In this subsection we will see how to select (uniquely) a weight $0$ form $\phi^X$ such that the resulting weight $1$ Jacobi form 
\be\label{eqn:weight_one_from_weight_zero}
\psi^X=
	-\mu_{1,0}\phi^X\lvert {\cal W}^X\ee 
satisfies the optimal growth condition (\ref{optimal_growth}).

For the simplest cases this optimal growth condition can be rephrased in terms of weight $0$ Jacobi forms  using the language of characters of the ${\cal N}=4$ superconformal  algebra.
Recall from \cite{Eguchi1987,Eguchi1988} that this algebra contains subalgebras isomorphic to the affine Lie algebra $\widehat{A}_1$ as well as the Virasoro algebra, and in a unitary representation the former of these acts with level $m-1$, for some integer $m>1$, and the latter with central charge $c=6(m-1)$. The unitary irreducible highest weight representations $V^{(m)}_{h,j}$ are labelled by the two ``quantum numbers" $h$ and $j$ which are the eigenvalues of $L_0$ and $\frac{1}{2}J_0^3$, respectively, when acting on the highest weight state. (We adopt a normalisation of the {$\SU(2)$ current} $J^3$ such that the zero mode $J^3_0$ has integer eigenvalues.) In the Ramond sector of the
superconformal algebra there are two types of highest weight representations: the {\em short} (or {\em BPS}, {\em supersymmetric}) ones with $h=\frac{m-1}{4}$ and $j\in\{ 0,\frac{1}{2},\cdots,\frac{m-1}{2}\}$, and the {\em long} (or {\em non-BPS}, {\em non-supersymmetric}) ones with $h > \frac{m-1}{4}$ and $j\in\{ \frac{1}{2},1,\cdots, \frac{m-1}{2}\}$. Their {\em (Ramond) characters}, defined as 
\be
{\rm ch}^{(m)}_{h,j}(\t,z) = \tr_{V^{(m)}_{h,j}} \left( (-1)^{J_0^3}y^{J_0^3} q^{L_0-c/24}\right),
\ee
are given by 
\be \label{masslesschar}
	{\rm ch}^{(m)}_{\frac{m-1}{4},j} (\t,z)= 
	\frac{\m_{m,j}  (\t,z)}{\m_{1,0}(\t,z)}
\ee
and 
\be \label{massivechar}
	{\rm ch}^{(m)}_{h,j} (\t,z) =
		(-1)^{2j}q^{h-\frac{m-1}{4}-\frac{j^2}{m}} \,  \frac{\th_{m,2j}(\t,z)-\th_{m,-2j}(\t,z)}{\mu_{1,0}(\t,z)} 
\ee
in the short and long cases, respectively, \cite{Eguchi1988}, where the function $ \m_{m,j}(\t,z)$ is defined as in  \eq{gAPsum}. 

\vspace{15pt}
\noindent{\em Lambencies $2,3,4,5,7,13$}
\vspace{10pt}

It turns out that for the pure A-type Niemeier root systems given as the union of $24/(\ll-1)$ copies of $A_{\ll-1}$ for $(\ll-1)|12$, 
the relevant criterion for $\phi^{(\ll)}$ is that of an {\em extremal Jacobi form} \cite[\S2.5]{UM}. The idea of an extremal Jacobi form can be viewed as a generalisation of the concept of an {extremal Virasoro character}, a notion that was introduced  in \cite{Hoehn2007} and discussed in 
\cite{Witten2007} in the context of pure $AdS_3$ gravity.

With the above definitions, following \cite{UM}, 
for $m$ a positive integer and $\phi$ a weak Jacobi form with weight $0$ and index $m-1$ we say $\phi$ is {\em extremal} if it admits an expression
\begin{gather}\label{eqn:forms:wtzero:ext}
	\f
	=	
	a_{\frac{m-1}{4},0}{\rm ch}^{(m)}_{\frac{m-1}{4},0} 
	+  
	a_{\frac{m-1}{4},\frac{1}{2}}{\rm ch}^{(m)}_{\frac{m-1}{4},\frac{1}{2}} 
			+ \sum_{0<r<m}
			\sum_{\substack{n>0\\ r^2-4mn< 0}}
			a_{\frac{m-1}{4}+n,\frac{r}{2}} {\rm ch}^{(m)}_{\frac{m-1}{4}+n,\frac{r}{2}}
\end{gather}
for some $a_{h,j}\in \CC$. Note the restriction on $n$ in the last summation in (\ref{eqn:forms:wtzero:ext}). Write $J^{\rm ext}_{0,m-1}$ for the subspace of $J_{0,m-1}$ consisting of extremal weak Jacobi forms.

We recall here that the extremal condition has a very natural interpretation in terms of the mock modular forms $h=(h_r)$ of weight $1/2$ via the relation 
\be
\psi = 
	-\mu_{1,0}\phi,\qquad \psi^F =\sum_{\text{$r$ mod $2m$}}h_r \th_{m,r},
\ee
discussed in \S\ref{sec:umbral shadow}. More precisely, the extremal condition is equivalent to the condition that 
\be
h_r = r\,\d_{r^2,1}\,a_{\frac{m-1}{4},\frac{1}{2}} q^{-\frac{1}{4m}} + { O}\left(q^{\frac{1}{4m}}\right)
\ee
as $\t\to i\inf$, which clearly implies the optimal growth condition $q^{1/4m}h_r(\t)=O(1)$ of Theorem \ref{thm:uniqueness_umbral_mock_jac}.

In \cite{UM} we proved that $\dim J^{\rm ext}_{0,m-1}=1$ in case $m-1$ divides $12$, and vanishes otherwise, at least when $m\leq 25$. (Cf. \cite[\S2.5]{UM}.) 
Explicitly, if $m-1$ divides $12$ then the one-dimensional space $J^{\rm ext}_{0,m-1}$ of extremal Jacobi forms with index $m-1$ is spanned by $\varphi\mm_1$, using the basis defined in \S\ref{subsec:basis_Jac}. 
These weight $0$ forms $\varphi^{(m)}_1$ arising from the extremal condition \eq{eqn:forms:wtzero:ext} will determine the unique weight $1$ meromorphic Jacobi forms $\psi^X$ satisfying the conditions of Theorem \ref{thm:uniqueness_umbral_mock_jac} for $X\in \{A_{1}^{24}, A_2^{12}, A_3^{8}, A_4^{6}, A_6^4, A_{12}^2\}$ according to (\ref{eqn:weight_one_from_weight_zero}) where $m$ is the Coxeter number of $X$.

\vspace{15pt}
\noindent{\em Lambencies $9,25$}
\vspace{10pt}

In order to include the other two pure A-type  Niemeier root systems ($X=A_8^3$ and $X=A_{24}$), it is sufficient to weaken the extremal condition slightly and consider 
weight $0$ and index $m-1$ Jacobi forms admitting an expression
\begin{gather}\label{condition_phi_9_25}
	\f
	=	
	a_{\frac{m-1}{4},0}{\rm ch}^{(m)}_{\frac{m-1}{4},0} 
	+  
	a_{\frac{m-1}{4},\frac{1}{2}}{\rm ch}^{(m)}_{\frac{m-1}{4},\frac{1}{2}} 
			+ \sum_{0<r<m}
			\sum_{\substack{n>0\\ r^2-4mn\leq  0}}
			a_{\frac{m-1}{4}+n,\frac{r}{2}} {\rm ch}^{(m)}_{\frac{m-1}{4}+n,\frac{r}{2}}
\end{gather}
for some $a_{h,j}\in \CC$. Notice that we have weakened the bound from $r^2-4mn<0$ to  $r^2-4mn\leq0$ in the last summand. 
This is directly related to the fact that for $m=9,25$ there exists $0<r<m$ such that $r^2=0$ mod $4m$. 

There exists at most one solution to (\ref{condition_phi_9_25}) up to rescaling (look ahead to Lemma \ref{prop:lemma_uniqueness_optimal}), and inspection reveals that at $m=9,25$ we have the non-zero solutions $\varphi^{(9)}_1$ and $\varphi^{(25)}_1$, in the notation of \S\ref{subsec:basis_Jac}. These Jacobi forms will determine, by way of (\ref{eqn:weight_one_from_weight_zero}), the unique weight $1$ meromorphic Jacobi forms $\psi^X$ satisfying the conditions of Theorem \ref{thm:uniqueness_umbral_mock_jac} for $X\in \{A_{8}^{3}, A_{24}\}$.

\vspace{15pt}
\noindent{\em Lambencies $6,10$}
\vspace{10pt}

In order to capture $X=A_5^4 D_4$ and $X=A_9^2 D_6$ at $m=6$ and $m=10$ (the cases with $m$ given by the product of two distinct primes), 
we will relax the extremal condition further and consider weight $0$ and index $m-1$ Jacobi forms admitting an expression
\begin{gather}\label{eqn:optimal_forms}
	\f
	=	
	a_{\frac{m-1}{4},0}{\rm ch}^{(m)}_{\frac{m-1}{4},0} +a_{\frac{m-1}{4},\frac{1}{2}} {\rm ch}^{(m)}_{\frac{m-1}{4},\frac{1}{2}}
				+ \sum_{0<r<m}
			\sum_{\substack{n>0\\ r^2-4mn\leq  1}}
			a_{\frac{m-1}{4}+n,\frac{r}{2}} {\rm ch}^{(m)}_{\frac{m-1}{4}+n,\frac{r}{2}}
\end{gather}
for some $a_{h,j}\in \CC$. 
For this more general condition we also have a uniqueness property. 
\begin{prop}\label{prop:lemma_uniqueness_optimal}
For a given index $m-1$ 
the dimension of the space of Jacobi forms satisfying (\ref{eqn:optimal_forms}) for some $a_{h,j}\in\CC$ is at most $1$.
\end{prop}
\begin{proof}
If $\f_1$ and $\f_2$ are two weight $0$ Jacobi forms that can be written as  in (\ref{eqn:optimal_forms}), there exists a linear combination $\f$ of $\f_1$ and $\f_2$ 
satisfying 
\be
\f
	=	
	aq^{-\frac{1}{4m}}\frac{\th_{m,1}-\th_{m,-1}}{\mu_{1,0}}+
	\sum_{0<r<m}
			\sum_{\substack{n>0\\ r^2-4mn\leq  1}}
			a_{\frac{m-1}{4}+n,\frac{r}{2}} {\rm ch}^{(m)}_{\frac{m-1}{4}+n,\frac{r}{2}}
\ee
for some $a$, $a_{h,j}\in \CC$, where we have used that 
\be
2\,{\rm ch}^{(m)}_{\frac{m-1}{4},0} + {\rm ch}^{(m)}_{\frac{m-1}{4},\frac{1}{2}} 
	+q^{-\frac{1}{4m}}\frac{\th_{m,1}-\th_{m,-1}}{\mu_{1,0}}=0.
\ee
Equivalently, $\psi=
\mu_{1,0}\f$ is a weight $1$ index $m$ weak Jacobi form with Fourier expansion $\psi(\t,z)=\sum_{n,\ll}c(n,\ll) q^n y^\ll$, where $c(n,\ll) =0$ for $\ll^2-4mn>1$. But such a weight $1$ index $m$ weak Jacobi form does not exist according to Theorem 9.7 of \cite{Dabholkar:2012nd}, based on the fact that there is no (strong) Jacobi form of weight 1 and index $m$ for any positive integer $m$, as shown earlier by Skoruppa \cite{Sko_Thesis}.
 We therefore conclude that $\f=0$, and $\f_1$ and $\f_2$ are linearly dependent. 
\end{proof}

This more general condition \eq{eqn:optimal_forms} singles out the weight $0$ Jacobi forms $\varphi^{(6)}_1$ and $\varphi^{(10)}_1$ (cf. \S\ref{subsec:basis_Jac}) at $m\in \{6,10\}$ in addition to those already mentioned, for which $m$ is prime or the square of a prime. As above, these weight $0$ forms will determine, by way of (\ref{eqn:weight_one_from_weight_zero}), the unique weight $1$ meromorphic Jacobi forms $\psi^X$ satisfying the conditions of Theorem \ref{thm:uniqueness_umbral_mock_jac} for $X\in \{A_5^4 D_6, A_9^2D_6 \}$.

\begin{rmk}
We expect that the space of solutions to (\ref{eqn:optimal_forms}) is zero-dimensional for all but finitely many $m$.
The discussion of weight $1$ Jacobi forms in \cite[\S9]{Dabholkar:2012nd} suggests that Jacobi forms satisfying \eq{eqn:optimal_forms} might only exist for these values $m\in\{2,3,4,5,6,7,9,10,13, 25\}$ which are among those of relevance to umbral moonshine.
\end{rmk}

\newpage
\vspace{15pt}
\noindent{\em Lambencies $8,12,16,18$}
\vspace{10pt}

We are left with the A-type  Niemeier root systems with Coxeter numbers that are not square-free and not  squares of primes: they are $X=A_7^2 D_5^2, A_{11}D_7E_6, A_{15}D_9, A_{17}E_7$ with $m=8,12,16,18$ respectively. To discuss these cases, let us first point out a subtlety in our procedure for determining the weight $1$ meromorphic Jacobi form $\psi^X$ from a weight $0$ meromorphic Jacobi form $\f^X$ using \eq{eqn:weight_one_from_weight_zero}. Although the resulting weight $1$ form $\psi^X$ is unique following Theorem \ref{thm:uniqueness_umbral_mock_jac}, in general the corresponding weight $0$ form $\f^X$ is not. In other words,  there could be more than one $\f^X$ satisfying \eq{eqn:weight_one_from_weight_zero} for a given $\psi^X$.
For the A-type cases with $m\in \{2,3,4,5,6,7,9,10,13,25\}$ discussed above, there is no such ambiguity since the matrix $\O^X$ corresponding to the Eichler--Zagier operator ${\cal W}^X$ is invertible. 
On the other hand, the matrix $\O^X$ is not invertible for $X\in\{A_7^2 D_5^2, A_{11}D_7E_6, A_{15}D_9, A_{17}E_7\}$, corresponding to $m\in\{8,12,16,18\}$. 
At the same time, in these cases there exists a unique $d>1$ such that $d^2$ is a proper divisor of $m$ and correspondingly, there is  an interesting feature in the space of Jacobi forms. 
This is due to the fact (cf. \cite[\S4.4]{Dabholkar:2012nd}) that if $\psi_{k,m/d^2}(\t,z)$ transforms as a weight $k$ index $m/d^2$ Jacobi form then $\psi_{k,m/d^2}(\t,dz)$ transforms as a weight $k$ index $m$ Jacobi form. It turns out that the umbral forms discussed above, in particular those with lambency $\ll=2,3,4$, help to determine the weight $0$ forms $\phi^X$ at lambency $\ll=8,12,16,18$ by requiring the ``square relation"
\be\label{square_relation}
-\mu_{1,0}\f^{(\ll)}\lvert {\cal W}_\ll(d)(\t,z) = \frac{\ll/d^2 -1 }{24}\, \psi^{(\ll/d^2)} (\t,dz) ,
\ee
where $d$ is the unique integer such that  $d^2$ is a proper divisor of $\ll$ different from 1, and  $\psi^{(\ll/d^2)}$ is the weight $1$ meromorphic Jacobi form with index $\ll/d^2$ that we have constructed above via \eq{eqn:weight_one_from_weight_zero}. This extra condition \eq{square_relation} eliminates the kernel of the Eichler--Zagier operator ${\cal W}^X$ and renders our choice for $\f^X$ unique. 
Notice that, following Table \ref{tab:mugs} we have used the lambency $\ll$ to denote the Niemeier root system $X$, and the former simply coincides with the Coxeter number for the  A-type  cases discussed in this subsection. 

To specify this particular choice of $\f^X$, let us impose the following condition. 
For a non-square-free $m$ which is not a square of a prime, we consider the weight $0$ index $m-1$ Jacobi forms $\f$,  such that the finite part of the weight $1$ index $m$ Jacobi form $-\mu_{1,0}\f \lvert ({\bf 1}+{\cal W}_m(m/d))$ is a mock Jacobi form with expansion $\sum_{n,\ll} c(n,\ll) q^n y^\ll$ and $c(n,r)=0$ whenever $r^2-4mn>1$, i.e.
\be\label{projection_nonsquarefree}
\left(
-\mu_{1,0}\f \lvert ({\bf 1}+{\cal W}_m(m/d))\right)^F(\t,z) = \sum_{\substack{n,r\in \ZZ, n\geq0\\ r^2-4mn\leq1}} c(n,r) q^n y^r. 
\ee

From the above discussion we arrive at the following uniqueness property for such Jacobi forms. 
\begin{prop}\label{lemma_uniqueness_optimal_nonsquarefree}
Consider integers $m,\til m,d$ satisfying $m=\til m d^2$ and $\til m, d>1$. 
Given any meromorphic weight $1$ index $n$ Jacobi form $\psi^{(\til m)}(\t,z)$, 
there exists at most one weight $0$ index $m-1$ weak Jacobi form $\f$ satisfying \eq{projection_nonsquarefree}  and the square relation 
\be
-\mu_{1,0}\f\lvert {\cal W}_m(d ) (\t,z)=  \psi^{(\til m)} (\t,dz) .
\ee
\end{prop}
\begin{proof}
Assume that there are two weight $0$ index $m-1$ weak Jacobi forms $\f_1$ and $\f_2$  satisfying  
\begin{gather}
\mu_{1,0}\f_1\lvert {\cal W}_m(d ) (\t,z)=  
\mu_{1,0}\f_2\lvert {\cal W}_m(d )(\t,z)=  \psi^{(\til m)} (\t,dz) 
\end{gather}
and therefore $\left(
\mu_{1,0}\f_1\lvert {\cal W}_m(d )\right)^F =  \left(
\mu_{1,0}\f_2\lvert {\cal W}_m(d )\right)^F$.
It then follows from Lemma \ref{inversion_EZ} that
\be\label{identify_finite_part}
\left(
\mu_{1,0}\f_1\lvert {\cal W}_m(m/d )\right)^F =  \left(
\mu_{1,0}\f_2\lvert {\cal W}_m(m/d )\right)^F.
\ee
Note that we also have 
\begin{gather}
\begin{split}
\f_1(\t,0) \m_{m,0}\lvert {\cal W}_m(d ) 
&=\left(
\mu_{1,0}\f_1\lvert {\cal W}_m(d )\right)^P \\
&=\left(\mu_{1,0}\f_2\lvert {\cal W}_m(d )\right)^P\\
&= 
\f_2(\t,0) \m_{m,0}\lvert {\cal W}_m(d ),
\end{split}
\end{gather}
which leads to the equality between the two constants $\f_1(\t,0)=\f_2(\t,0)$, and hence
\be \left(
\mu_{1,0}\f_1 \lvert ({\bf 1}+{\cal W}_m(m/d))\right)^P = \left(
\mu_{1,0}\f_2 \lvert ({\bf 1}+{\cal W}_m(m/d))\right)^P.
\ee

Next, assume that $\f_1$ and $\f_2$ both satisfy \eq{projection_nonsquarefree}. 
Then $
\mu_{1,0}(\f_1-\f_2) \lvert ({\bf 1}+{\cal W}_m(m/d))$  is a weight $1$ index $m$ weak Jacobi form with expansion $\sum_{n,\ll}c(n,\ll) q^n y^\ll$, where $c(n,\ll) =0$ for $\ll^2-4mn>1$, which can only be identically zero (cf.  \cite{Dabholkar:2012nd}, Theorem 9.7) and 
 we arrive at 
\begin{gather}
\mu_{1,0}\f_1 \lvert ({\bf 1}+{\cal W}_m(m/d))=
\mu_{1,0}\f_2 \lvert ({\bf 1}+{\cal W}_m(m/d)).
\end{gather}

Combining with \eq{identify_finite_part} and using Corollary \ref{cor:commute_EZ_polar} we obtain $(
\mu_{1,0}\f_1 )^F=(
\mu_{1,0}\f_2)^F$.
Again, the polar part $(
\mu_{1,0}\f)^P$ of the meromorphic Jacobi form $
\mu_{1,0}\f$ for any weight $0$ index $m-1$ weak Jacobi form $\f$ is given by $
\f(\t,0)\m_{m,0}$. This proves that $\mu_{1,0}\f_1=
\mu_{1,0}\f_2$ and hence  $\f_1=\f_2$.
\end{proof}
 At $m\in\{8,12,16,18\}$, applying the above Proposition and choosing 
\begin{gather}
\psi^{(\til m)} = 
-\mu_{1,0}\f^{(\til m)},\quad  \til m=2,3,4,
\end{gather}
with $\f^{(\til m)}$ as in Table \ref{tab:weight_zero_form} 
 gives the weight $0$ forms we need in order to specify the remaining umbral mock modular forms $H^X$ for $X$ of A-type. 

The explicit expressions for $\f^X$ for all A-type  Niemeier root systems $X$ are given in Table \ref{tab:weight_zero_form}, where the basis we use for weight $0$ weak Jacobi forms is summarised in \S\ref{subsec:basis_Jac} and $\varphi_0^{(\ll)}$ denotes the constant $\varphi_0^{(\ll)}=\varphi^{(\ll)}_1(\t,0)$.

 \begin{table}[h]
\captionsetup{font=small}
\begin{center}
\caption{Weight Zero Jacobi Forms}\label{tab:weight_zero_form}
\medskip

\begin{tabular}{ccccccccccc}
\multicolumn{1}{c|}{$\rs$}&$A_1^{24}$&$A_2^{12}$&$A_3^8$&$A_4^6$&$A_5^4D_4$&$A_6^4$&$A_8^3$&$A_9^2D_6$&$A_{12}^2$&$A_{24}$\\
	\midrule
\multicolumn{1}{c|}{$\ll$}&	2&	3&	4&	5&	6&	7&	9&	10&	13&	25\\
	\midrule
	\multicolumn{1}{c|}{$\varphi_0^{(\ll)}  \phi^{(\ll)}$}&			 $\varphi^{(2)}_1$&	 $\varphi^{(3)}_1$&	 $\varphi^{(4)}_1$&	 $\varphi^{(5)}_1$&	 $\varphi^{(6)}_1$& $\varphi^{(7)}_1$&$\varphi^{(9)}_1$&$\varphi^{(10)}_1$&	$\varphi^{(13)}_1$&$\varphi^{(25)}_1$\\
\end{tabular}
\end{center}

\begin{center}
\begin{tabular}{cccc}
\multicolumn{1}{c|}{$\rs$}&$A_7^2D_5^2$&$A_{11}D_7E_6$&$A_{15}D_9$\\
	\midrule
\multicolumn{1}{c|}{$\ll$}&	8& 12&	16\\
	\midrule
\multicolumn{1}{c|}{$\varphi_0^{(\ll)}  \phi^{(\ll)}$}&$\varphi^{(8)}_1+\tfrac{1}{2}\varphi^{(8)}_2$&$\varphi^{(12)}_1+\varphi^{(12)}_2$&$\varphi^{(16)}_1+\tfrac{1}{2}\varphi^{(16)}_2$\\
\end{tabular}
\end{center}
\begin{center}
\begin{tabular}{cc}
\multicolumn{1}{c|}{$\rs$}&$A_{17}E_7$\\
	\midrule
\multicolumn{1}{c|}{$\ll$}&	18\\
	\midrule
\multicolumn{1}{c|}{$\varphi_0^{(\ll)}  \phi^{(\ll)}$}&$ \varphi^{(18)}_1+\tfrac{1}{3}\varphi^{(18)}_3 + 4\xi (\varphi^{(12)}_1+2\varphi^{(12)}_2+\tfrac{1}{3}\varphi^{(12)}_3)$\\
\end{tabular}
\end{center}
\end{table}

Given $\f^X$ and using the Eichler--Zagier operator ${\cal W}^X$ defined in \S\ref{sec:forms:ADE}, the formula \eq{eqn:weight_one_from_weight_zero} gives the weight $1$ meromorphic Jacobi form $\psi^X$. From there, using the method described in \S\ref{sec:forms:meromock} we can separate it into the polar and the finite part $\psi^X=(\psi^X)^P+(\psi^X)^F$ in a canonical way.  As can be verified by inspection, the choices of $\f^X$ specified in this subsection determine solutions to the hypotheses of Theorem \ref{thm:uniqueness_umbral_mock_jac}. 

\begin{prop}\label{prop:A_UmbralForms}
Let $X$ be one of the $14$ Niemeier root systems with an A-type component (cf. (\ref{eqn:holes:NieRoot_A})) and let $\phi^X$ be as specified in Table \ref{tab:weight_zero_form}. Then the meromorphic Jacobi form $\psi^X$ determined by (\ref{eqn:weight_one_from_weight_zero}) is the unique such function satisfying the conditions of Theorem \ref{thm:uniqueness_umbral_mock_jac}. Write 
\be
(\psi^X)^F = \sum_{r \,({\rm mod}\,2m)} H_r^X \th_{m,r} , 
\ee
then $H^X=(H_r^X)$ is the unique vector-valued mock modular form with shadow $S^X$ satisfying the optimal growth condition (cf. Corollary \ref{cor:uniqueness_umbral_mock_mod}).  
\end{prop}

As a result, the weight $0$ weak Jacobi forms constructed in this subsection define the umbral mock modular forms $H^X$ for each of the A-type  Niemeier root systems (cf. \eq{eqn:holes:NieRoot_A}). 
The first few dozen coefficients of the components $H^X_r$ are given in Appendix \ref{sec:coeffs}. It is a reflection of the close relationship between the notions of optimal growth formulated here and in \cite{Dabholkar:2012nd} that the umbral forms $H^X$ attached to A-type Niemeier root systems are closely related to the mock modular forms of weight $1$ that appear in \S A.2 of  \cite{Dabholkar:2012nd}.

%------------------------------------------------------------------%
\section{The Umbral McKay--Thompson Series}\label{sec:mckay}
%------------------------------------------------------------------%

The purpose of this section is to discuss the {\em umbral McKay--Thompson series} $H^X_g$ conjecturally defining the graded character of the group $G^X$ attached to the umbral module $K^X$ (cf. \S\ref{sec:conj:mod}). In \S\ref{sec:mckay:aut} we specify their mock modular properties. 
In \S\ref{sec:forms:low_lambencies} we discuss the McKay--Thompson series $H^X_g$ attached to the five A-type  Niemeier root systems $X$ with prime Coxeter numbers. 
Subsequently, in \S\ref{sec:forms:mult} we discuss the {\em multiplicative relations} relating the McKay--Thompson series $H^X_g$ and $H^{X'}_{g'}$ attached to different Niemeier root systems $X$ and $X'$ with the Coxeter number of one of the root systems being an integer multiple of the other. 
In \S\ref{sec:mckay:mocktheta} we collect the relations between certain McKay--Thompson series and known mock theta series. 
As we discuss in  \S\ref{sec:mckay_spec} in detail, these relations, together with the constructions presented in \S\ref{sec:weight_zero_umbral_forms} and \S\ref{sec:forms:low_lambencies}, are sufficient to determine most of the umbral McKay--Thompson series completely. More specifically, we determine all of the umbral McKay--Thompson series $H^X_g$ attached to all conjugacy classes $[g]$ of the umbral group $G^X$ corresponding to all the A-type  Niemeier root systems, except for $X=A_6^4, A_{12}^2$, corresponding to lambencies $7,13$ (cf. (\ref{sec:holes:gzero})), for which we provide partial specifications. We also determine all of the umbral McKay--Thompson series $H^X_g$ for all conjugacy classes $[g]$ of the umbral group $G^X$ corresponding to all the D-type Niemeier root systems, except for the lambencies $10+5$ and $22+11$. 
For these D-type Niemeier root systems  $10+5$ and $22+11$ we
specify all the McKay--Thompson series $H^X_g$ except for $[g]=4A\subset G^{(10+5)}$ and $[g]=2A\subset G^{(22+11)}$. For $X=E_6^4$ we specify $H^X_g$ except for $[g]\in \{6A,8AB\}$, while for $X=E_8^3$ we omit the cases $[g]\in \{2A,3A\}$. 
We also provide the first few dozen coefficients of all the umbral McKay--Thompson series in Appendix \ref{sec:coeffs}. The conjugacy class names are defined in \S\ref{sec:chars:irr}.

\subsection{Shadows}\label{sec:mckay:aut}

In \S\ref{sec:holes:gps} we described the umbral groups $G^{\rs}$ and attached twisted Euler characters $\bar{\chi}^{\rs_A}$, ${\chi}^{\rs_A}$, $\bar{\chi}^{\rs_D}$, ${\chi}^{\rs_D}$, \&c., to the A-, D-, and E-components of each Niemeier root system. In this section we will explain how to use these characters to define a function $S^{\rs}_g$, for each $g\in G^{\rs}$, which turns out to be the shadow of the vector-valued mock modular form $H^{\rs}_g$. 

Let $\rs$ be a Niemeier root system and suppose that ${ m}$ is the Coxeter number of $\rs$.
Then given $g\in G^{\rs}$ we define $2m\times 2m$ matrices $\Omega^{X_A}_g$, $\Omega^{X_D}_g$, and $\Omega^{X_E}_g$, with entries indexed by $\ZZ/2m\ZZ\times \ZZ/2m\ZZ$, as follows. 
We define $\Omega^{X_A}_g$ by setting
\begin{gather}
\Omega^{X_A}_g=\chi^{X_A}_gP_m^0+\bar{\chi}^{X_A}_gP_m^1	
\end{gather}
where $P_m^s$ is the diagonal matrix (of size $2m\times 2m$, with entries indexed by $\ZZ/2m\ZZ\times\ZZ/2m\ZZ$) with $r$-th diagonal equal to $1$ or $0$ according as $r=s$ mod $2$ or not,
\begin{gather}\label{eqn:mckay:aut:Pdefn}
	(P^s_m)_{r,r'}=\delta_{r,s(2)}\delta_{r,r'(2m)}.
\end{gather}
In (\ref{eqn:mckay:aut:Pdefn}) we write $\delta_{i,j (n)}$ for the function that is $1$ when $i=j\pmod{n}$ and $0$ otherwise. Note that $P_m^0+P_m^1=\Omega_m(1)$ is the $2m\times 2m$ identity matrix. According to the convention that $\chi^{X_A}_g=\bar{\chi}^{X_A}=0$ if $X_A=\emptyset$ we have $\Omega^{X_A}_g=0$ for all $g\in G^X$ in case $X_A$ is empty; i.e. in case there are no type A components in the Niemeier root system $X$. 

If $X_D\neq \emptyset$ then $m$ is even. For $m>6$ we define $\Omega^{X_D}_g$ by setting
\begin{gather}
\Omega^{X_D}_g=\chi^{X_D}_gP_m^0+\bar{\chi}^{X_D}_gP_m^1+\chi^{X_D}\Omega_m(m/2),
\end{gather}
whilst for $m=6$---an exceptional case due to triality for $D_4$---we define $\Omega=\Omega^{X_D}_g$ so that
\begin{gather}
	\Omega_{r,r'}=
		\begin{cases}
		\check{\chi}^{X_D}_g\delta_{r,r'(12)},&\text{if $r=0,3$ mod $6$,}\\
		\bar{\chi}^{X_D}_g\delta_{r,r'(12)}+\chi^{X_D}_g\delta_{r,-r'(6)}\delta_{r,r'(4)},&\text{if $r=1,5$ mod $6$,}\\
		{\chi}^{X_D}_g\delta_{r,r'(12)}+\chi^{X_D}_g\delta_{r,-r'(6)}\delta_{r,r'(4)},&\text{if $r=2,4$ mod $6$}.
		\end{cases}
\end{gather}
The matrices $\Omega^{X_E}_g$ are defined by setting
\begin{gather}
	\Omega^{X_E}_g=
	\begin{cases}
		(\chi^{X_E}_gP_m^0+\bar{\chi}^{X_E}_gP_m^1)(\Omega_m(1)+\Omega_m(4))+\chi^{X_E}_g\Omega_m(6)&\text{if $m=12$,}\\
		\bar{\chi}^{X_E}_g(\Omega_m(1)+\Omega_m(6)+\Omega_m(9)),&\text{if $m=18$,}\\
		\bar{\chi}^{X_E}_g(\Omega_m(1)+\Omega_m(6)+\Omega_m(10)+\Omega_m(15)),&\text{if $m=30$.}
	\end{cases}
\end{gather}
Now for $X$ a Niemeier root system we set $\Omega^X_g=\Omega^{X_A}_g+\Omega^{X_D}_g+\Omega^{X_E}_g$, and we define $S^X_g$ by setting
\begin{gather}
	S^X_g=\Omega^X_g\cdot S_m.
\end{gather}
This generalises the construction (\ref{def:um_shadow}). We conjecture (cf. \S\ref{sec:conj:aut}) that the vector-valued function $S^{\rs}_g$ is the shadow of the mock modular form $H^{\rs}_g$ attached to $g\in G^X$. We will specify (most of) the $H^X_g$ explicitly in the remainder of \S\ref{sec:mckay}. 

\begin{rmk}
The matrices $\Omega^X$, corresponding to the case where $[g]$ is the identity class,  admit an ADE classification as explained in \S\ref{sec:forms:ADE}. It is natural to ask what the criteria are that characterise these matrices $\Omega^X_g$, attached as above to elements $g\in G^X$ via the twisted Euler characters defined in \S\ref{sec:holes:gps}.
\end{rmk}

%---------------------------------------------------------------------------------------%
\subsection{Prime Lambencies }\label{sec:forms:low_lambencies}
%---------------------------------------------------------------------------------------%

In this subsection we review the mock modular forms $H^X_g=(H^X_{g,r})$ conjecturally encoding the graded characters of the umbral module $K^X$ (cf. Conjecture \ref{conj:conj:mod:Kell}) of the umbral group $G^X$ for the five Niemeier root systems $X$ with prime Coxeter numbers. Explicitly, these are the root systems $X=A_1^{24}, A_2^{12}, A_4^{6}, A_6^{4}, A_{12}^{2}$ with $\ll=2,3,5,7,13$, and the corresponding McKay--Thompson series are denoted $H^{(\ll)}_g=(H^{(\ll)}_{g,r})$ with $r=1,2,\dots,\ll-1$, using the notation given in Table \ref{tab:mugs}. 
In the next subsection we will see that they determine many of the umbral McKay--Thompson series attached to the other Niemeier root systems with non-prime Coxeter numbers. 
The discussion of this subsection follows that of \cite{UM}. The McKay--Thompson series for $\ll=2$, $X=A_1^{24}$ were first computed in \cite{Cheng2010_1,Gaberdiel2010,Gaberdiel2010a,Eguchi2010a}.

In order to give explicit formulas for the mock modular forms $H^{(\ll)}_g=(H^{(\ll)}_{g,r})$, we consider a slightly different function
\be\label{def:Hhat}
\widehat H^{(\ll)}_{g,r} = H^{(\ll)}_{g,r} - \frac{\chi^{(\ll)}_r(g)}{\chi^{(\ll)}}H^{(\ll)}_r,
\ee
where $H^{(\ll)}=(H^{(\ll)}_r)$ is the umbral form specified in \S\ref{sec:weight_zero_umbral_forms} corresponding to the identity class $1A$ and whose Fourier coefficients are given in Appendix \ref{sec:coeffs} with $H^{(\ll)}_r=H^{(\ll)}_{1A,r}$. We also let
\be
\chi^{(\ll)}_r(g) =\chi^{X_A}_{g} \; {\text{for }} r= 0 \;{\text{mod }}2, \quad\chi^{(\ll)}_r(g) =\bar \chi^{X_A}_{g} \; {\text{for }} r= 1 \;{\text{mod }}2,
\ee
and
\be
 \chi^{(\ll)} =\chi^{X_A}_{1A} =\bar \chi^{X_A}_{1A} = \frac{24}{\ll-1}
\ee
for 
\be X_A  = X = A_{\ll-1}^{24/(\ll-1)}
\ee
where the characters $\chi^{X_A}_g$ and $\bar{\chi}^{X_A}_g$ are defined in \S\ref{sec:holes:gps} and given explicitly in \S\ref{sec:chars:eul}. Following the discussion in \S\ref{sec:mckay:aut} we note that the combination \eq{def:Hhat} of  $H^{(\ll)}_{g,r} $ and $H^{(\ll)}_r$ has the property of being a modular form rather than a mock modular form, as the shadows of the two contributions to  $\widehat H^{(\ll)}_{g,r}$ cancel.

Subsequently, we define  weight two modular forms 
\be\label{def:F_weight2form}
F^{(\ll)}_g = \sum_{r=1}^{\ll-1} \widehat H^{(\ll)}_{g,r} S_{\ll,r}
\ee
where $S_{\ll,r} $ is again the unary theta series given in \eq{def:S}. 
For $\ll>3$ we also specify further weight two modular forms  by setting
\be
F^{(\ll),2}_g = \sum_{r=1}^{\ll-1} \widehat H^{(\ll)}_{g,r} S_{\ll,\ll-r}.
\ee
As explained in detail in \cite{UM}, specifying $F^{(\ll)}_g$ and $F^{(\ll),2}_g$ is sufficient to determine $H^{(\ll)}_g$ uniquely for $\ll=2,3,5$. 
In the case $\ll=2$ there is only one term in the sum \eq{def:F_weight2form} and it is straightforward to obtain $\widehat H^{(\ll)}_{g,r}$ from the weight 2 form  $F^{(\ll)}_g$. In the case $\ll=3,5$, we also utilise the following fact obeyed by the conjugacy classes of $G^{(\ll)}$. 

For any umbral group $G^{(\ll)}$ corresponding to an A-type  Niemeier root lattice with Coxeter number ${\ll}>2$, for a given conjugacy class $[g]$ with $\chi^{(\ll)}_1(g)>0$, there exists a (not necessarily different)  conjugacy class $[g']$ with the property
\be\label{def_paired_classes}
\chi^{(\ll)}_1(g) =\chi^{(\ll)}_1(g'),\quad \chi^{(\ll)}_2(g) = -\chi^{(\ll)}_2(g') 
\ee
and the order of $g$ and $g'$ are either the same or related by a factor of $2$ or $1/2$. 
For such paired classes we have
\be\label{relation_paired_classes2}
H^{(\ll)}_{g,r} + (-1)^{r} H^{(\ll)}_{g',r} =0 .
\ee
In particular, $H^{(\ll)}_{g,even}=0$  for the self-paired classes.
For $\ll=7,13$, this serves to constrain the function $H^{(\ll)}_g$ and supports the claims regarding their modular properties discussed in  \S\ref{sec:conj:aut}. We refer the readers to \S4 of \cite{UM} for explicit expressions for the weight two forms $F^{(\ll)}_g$ and $F^{(\ll),2}_g$. 

Note that, from the discussion in \S\ref{sec:mckay:aut}, the relation \eq{def_paired_classes} implies that the shadows attached to such paired classes satisfy 
 \be
S^{(\ll)}_{g,r} + (-1)^{r} S^{(\ll)}_{g',r} =0 .
\ee
Therefore, the paired relation \eq{relation_paired_classes2} can be viewed as a consequence of (the validity of) Conjecture  \ref{conj:conj:moon:opt}.

%---------------------------------------------------------------------------------------%
\subsection{Multiplicative Relations}\label{sec:forms:mult}
%---------------------------------------------------------------------------------------%
In this subsection we will describe a web of relations, the {multiplicative relations}, among the  mock modular forms attached to different Niemeier lattices. The nomenclature comes from the fact that these relations occur only among mock modular forms attached to  Niemeier root systems with one Coxeter number being an integer multiple of the other. 
To simplify the discussion we will distinguish the following two types. 
The  {\em horizontal relations}  relate umbral McKay--Thompson series $H^X_g$ and $H^{X'}_g$ attached to different Niemeier root systems with the same Coxeter number, {\em i.e.} ${  m}(X)={  m}(X')$. 
The {\em vertical relations}  connect umbral McKay--Thompson series $H^X_g$ and $H^{X'}_g$ attached to different Niemeier root systems $X$ and $X'$ with ${  m}(X)|{  m}(X')$ and ${  m}(X)\neq {  m}(X')$.

First we will discuss the horizontal relations, summarised in Table  \ref{fig:horizontal_relations}. 
Note that there are five pairs of Niemeier root systems that share the same Coxeter number $m\in \{6,10,12,18,30\}$. 
Let us choose such a pair $(X',X)$. 
From the definition of the umbral shadows  \eq{def:um_shadow} and the generalisation in \S\ref{sec:mckay:aut} to the non-identity conjugacy classes, we see that it can happen that the $r'$-th component $S^{X'}_{g',r'}$ of the shadow  attached to the conjugacy class $[g']$ of the umbral group $G^{X'}$ is expressible as a linear combination of the components $S^{X}_{g,r}$ of the shadow attached to the conjugacy class $[g]$ of the umbral group $G^{X}$. It turns out that for all five equal-Coxeter-number pairs of Niemeier root systems, this indeed happens for various pairs $([g'],[g])$ of conjugacy classes.

\begin{table}[h!]
\captionsetup{font=small}
\centering
\begin{tabular}{CCLC}\toprule
X'&X& ([g'],[g])&{\rm Relations}\\\midrule
\multirow{7}*{6+3}&\multirow{7}*{6}&(1A,1A)& H^{X'}_{g',r}= H^{X}_{g,r}+H^{X}_{g,6-r}\\
&&  (2A,2B)&r {\text{ odd }}\\ &&(3B,3A)\\ && (4A,8AB)\\\cmidrule(l){3-4}
&&(2C,4A)&H^{X'}_{g',r}= H^{X}_{g,r}-H^{X}_{g,6-r}\\
&&(4B,8AB)&r {\text{ odd }}\\
&&(6B,6A)\\
\midrule
\multirow{3}*{10+5}&\multirow{3}*{10} &(1A,1A)& H^{X'}_{g',r}= H^{X}_{g,r}+H^{X}_{g,10-r}\\
&&(2A,4AB)&\\\cmidrule(l){3-4}
&&(2B,4AB)& H^{X'}_{g',r}= H^{X}_{g,r}-H^{X}_{g,10-r}\\
\midrule
\multirow{2}*{12+4}&\multirow{2}*{12}&(1A,1A)& 
H^{X'}_{g',r}= H^{X}_{g,r}+H^{X}_{g,6+r}
\;,r=1,5 \\
&&(2A,2A)& H^{X'}_{g',4}= H^{X}_{g,4}+H^{X}_{g,8}\\\midrule
\multirow{2}*{18+9}&\multirow{2}*{18}&\multirow{1}*{(1A,1A)}& 
H^{X'}_{g',r}= H^{X}_{g,r}+H^{X}_{g,18-r}\\\cmidrule(l){3-4}
&&\multirow{1}*{(2A,1A)}&H^{X'}_{g',r}= H^{X}_{g,r}-H^{X}_{g,18-r}\\
\midrule
\multirow{2}*{30+6,10,15}&\multirow{2}*{30+15}&\multirow{2}*{(1A,1A)}&
H^{X'}_{g',1}= H^{X}_{g,1}+H^{X}_{g,11}
 \\ 
&&&H^{X'}_{g',7}= H^{X}_{g,7}+H^{X}_{g,13}
\\ 
\bottomrule
\end{tabular}
\caption{Horizontal Relations. 
\label{fig:horizontal_relations}}
\end{table}

Note that the relation between the shadows
\be\label{horizontal_relation_shadow}
S^{X'}_{g',r'} = \sum_r c_{r',r}\, S^{X}_{g,r}
\ee
is a necessary but insufficient condition for the linear relation 
\be\label{horizontal_relation_mmf}
H^{X'}_{g',r'} = \sum_r c_{r',r} H^{X}_{g,r}
\ee
between the corresponding McKay--Thompson series to hold, since the coincidence of the shadow guarantees the coincidence of the corresponding mock modular form only up to the addition of a modular form. Nevertheless, it turns out that in umbral moonshine  the relation between the McKay--Thompson series \eq{horizontal_relation_mmf} holds whenever the relation between the shadow \eq{horizontal_relation_shadow} holds non-trivially with $S^{X'}_{g',r'}, S^{X}_{g,r}\neq 0$. This fact, together with the more general multiplicative relations \eq{vertical_relation_mmf}, can again be viewed as the consequence of the conjectured uniqueness of such mock modular forms (Conjecture \ref{conj:conj:moon:opt}).
See Table  \ref{fig:horizontal_relations} for the list of such horizontal relations. 
In particular, for the identity element a relation $H^{X'}_{r'} = \sum_r c_{r',r} H^{X}_{r}$ holds for some $c_{r',r}\in \ZZ$ for all the five pairs $(X',X)$ of Niemeier root systems with the same Coxeter numbers.

Note that, together with the discussion in \S\ref{sec:weight_zero_umbral_forms} and Proposition \ref{proposition_EZ_commutes}, this implies more specifically that the umbral mock modular forms $H^{X'}_r$ can be obtained as the theta-coefficients of the finite part of the meromorphic weight $1$ Jacobi form
\be
\psi^{X'} = 
			-\mu_{1,0}\f^{X'}\lvert {\cal W}^{X'}\;,\;\;\f^{X'}=\f^{X}
\ee
for the four pairs $(X',X)\in\{(6,6+3),(10,10+5),(12,12+4),(18,18+9)\}$ with A-type  root systems $X$ and with the weight $0$ Jacobi forms given in Table \ref{tab:weight_zero_form}.

\begin{table}
\centering
\scalebox{.85}{
\begin{tabular}{CCCL}\toprule
X'&X&([g'],[g])&{\text{ Relations}} \\\midrule
\multirow{4}*{4}&\multirow{4}*{2}& (1A,2A), (2B,4A)&{ (H^{X'}_{g',1}-H^{X'}_{g',3} )(2\t) = H^{X}_{g,1} (\t)}\\ 
&&(2C,4B),(3A,6A) \\
&&(4A,4C),(4C,8A) \\ 
&& (6BC,12A), (7AB,14AB) \\
\midrule
\multirow{6}*{6}&\multirow{2}*{2}&(1A,3A),(2B,6A)& \sum_{n=0}^2 (-1)^n H^{X'}_{g',1+2n} (3\t) = H^{X}_{g,1}(\t) \\ 
&&(8AB,12A)& \\  \cmidrule(l){2-4}
&\multirow{4}*{3} &(1A,2B),(2A,2C)& (H^{X'}_{g',r} -H^{X'}_{g',6-r} ) (2\t) =  H^{X}_{g,r}(\t)\\
&&(2B,4C),(4A,4B)&\\
&&(3A,6C),(6A,6D)&\\
&&(8AB,8CD)\\\midrule
6+3&2&  (5A,15AB)& \sum_{n=0}^2 (-1)^n H^{X'}_{g',1+2n} (3\t) = H^{X}_{g,1} (\t)\\ 
\midrule
\multirow{2}*{8}&\multirow{2}*{4} & (1A,2C),(2BC,4C) & { (H^{X'}_{g',r}-H^{X'}_{g',8-r} )(2\t) = H^{X}_{g,r} (\t) }\\
&&(4A,4B)\\\midrule
\multirow{1}*{9}&\multirow{1}*{3}& (1A,3A) ,(2B,6C)& (H^{X'}_{g',r}+H^{X'}_{g',r+6}-H^{X'}_{g',6-r})(3\t) = H^{X}_{g,r}(\t)\\ 
\midrule
  \multirow{2}*{10}&\multirow{1}*{2} &(1A,5A),(4AB,10A) &  \sum_{n=0}^4 (-1)^n H^{X'}_{g',1+2n} (5\t) = H^{X}_{g,1}(\t)\\ 
 \cmidrule(l){2-4}
 &\multirow{1}*{5} & (1A,2C),(4AB,4CD)& (H^{X'}_{g',r} -H^{X'}_{g',10-r} ) (2\t) =  H^{X}_{g,r}(\t)\\
 \midrule
 10+5& 2&(3A,15AB) & (H^{X'}_{g',1}-H^{X'}_{g',3}+H^{X'}_{g',5}/2)(5\t) = H^{X}_{g,1}(\t) \\
 \midrule
  \multirow{2}*{12} & 4 & (1A,3A) & (H^{X'}_{g',r}+H^{X'}_{g',r+8}-H^{X'}_{g',8-r})(3\t) = H^{X}_{g,r}(\t) \\\cmidrule(l){2-4}
 &6& (1A,2B) & (H^{X'}_{g',r} -H^{X'}_{g',12-r} ) (2\t) =  H^{X}_{g,r}(\t) \\\midrule
 \multirow{2}*{12+4}& \multirow{2}*{6}&(3A,3A) &  (H^{X'}_{g',1} -H^{X'}_{g',5} ) (2\t) =  (H^{X}_{g,1} -H^{X}_{g,5} )(\t)\\
 &&(2B,8AB)& H^{X'}_{g,4} (2\t)= (H^{X}_{g,2} +H^{X}_{g,4} )(\t)\\
\midrule
 \multirow{2}*{14+7} & \multirow{2}*{2} & (1A,7AB),(2A,14AB) &( -\tfrac{1}{2} H^{X'}_{g',7} +\sum_{n=0}^2 (-1)^n H^{X'}_{g',1+2n})(7\t) = H^{X}_{g,1}(\t) \\ 
  &&(3A,21AB) &\\\midrule
  16&8& (1A,2BC) &  (H^{X'}_{g',r} -H^{X'}_{g',16-r} ) (2\t) =  H^{X}_{g,r}(\t)\\
\midrule
  \multirow{2}*{18} &  \multirow{1}*{6}& (1A,3A) & (H^{X'}_{g',r} -H^{X'}_{g',12-r}+H^{X'}_{g',12+r} ) (3\t) =  H^{X}_{g,r}(\t)  \\ 
  \cmidrule(l){2-4}
  & \multirow{1}*{9}& (1A,2B) &3 (H^{X'}_{g',r} -H^{X'}_{g',18-r})(2\t) = H^{X}_{g,r}(\t) \\ 
  \midrule
  22+11&2 &(1A,11A)& ( -\tfrac{1}{2} H^{X'}_{g',11} +\sum_{n=0}^4 (-1)^n H^{X'}_{g',1+2n})(11\t) = H^{X}_{g,1}(\t)\\\midrule
   \multirow{1}*{25}& \multirow{1}*{5}& (1A,5A) &    (\sum_{n=0}^2  H^{X'}_{g',10n+r}-\sum_{n=1}^2  H^{X'}_{g',10n-r})(5\t) = H^{X}_{g,r}(\t) \\
   \midrule
   \multirow{2}*{30+15}& 2 & (1A,15AB) &  ( -\tfrac{1}{2} H^{X'}_{g',15} +\sum_{n=0}^6 (-1)^n H^{X'}_{g',1+2n})(15\t) = H^{X}_{g,1}(\t)\\\cmidrule(l){2-4}
   &10+5 & (1A,3A) & (H^{X'}_{g',r}+H^{X'}_{g',10-r}-H^{X'}_{g',10+r})(3\t) = H^{X}_{g,r}(\t)
   \\\midrule
   46+23&2& (1A,23AB) & ( -\tfrac{1}{2} H^{X'}_{g',23} +\sum_{n=0}^{10} (-1)^n H^{X'}_{g',1+2n})(23\t) = H^{X}_{g,1}(\t) \\
 \bottomrule
\end{tabular}
}
\caption{Vertical Relations. 
\label{tab:vertical_relations}}
\end{table}

In fact, a linear relation between the shadows attached to different Niemeier root systems can happen more generally and not just among those with the same Coxeter numbers. 
The first indication that non-trivial relations might exist across different Coxeter numbers is the following property of the building blocks of the umbral shadow.  As one can easily check, the unary theta function $S_{m}=(S_{m,r})$ defined in \eq{def:S} at a given index $m$ can be re-expressed in terms of those at a higher index as 
\begin{align}\label{shadow_relation}
S_{m,r}(\t) = \sum_{\ll=0}^{\lceil \frac{n}{2} \rceil -1} S_{nm,r+2m\ell}(n\t) - \sum_{\ll=1}^{\lfloor \frac{n}{2} \rfloor } S_{nm,2m\ell-r}(n\t) ,
\end{align}
for any positive integer $n$.
The above equality makes it possible to have the relation
\be\label{vertical_relation_shadow}
\sum_{r'} c_{r,r'} S^{X'}_{g',r'}(k\t) =  S^{X}_{g,r}(\t)
\ee
for some $k\in \ZZ_{>0}$. 
We will see that this relation between the umbral shadows does occur for many pairs of Niemeier root systems $(X',X)$ with ${  m}(X') = k {  m}(X)$. Moreover, whenever this relation holds non-trivially with $S^{X'}_{g',r'}, S^{X}_{g,r}\neq 0$, the corresponding relation among the McKay--Thompson series
\be\label{vertical_relation_mmf}
\sum_{r'} c_{r,r'}  H^{X'}_{g',r'} (\tfrac{{  m}(X')}{{  m}(X)}\t) =  H^{X}_{g,r}(\t) 
\ee
also holds. 

We summarise a minimal set of such relations in Table \ref{tab:vertical_relations}. In Tables \ref{fig:horizontal_relations} and  \ref{tab:vertical_relations}, when it is not explicitly specified, the relation holds for all values of $r$ such that all  $H^{X'}_{g',r_1}$ and $H^X_{g,r_2}$ appearing on both sides of the equation have $1\leq r_1\leq m(X')$ and $1\leq r_2\leq m(X)$. 
 From the relations recorded in these tables  as well as the paired relations \eq{relation_paired_classes} and \eq{relation_paired_classes_Z3}, many further relations can be derived. 
For example, combining the relations between the $(1A,2BC)$ classes for $(X',X)=(16,8)$, the $(2BC,4C)$ classes for $(X',X)=(8,4)$ and the $(4C,8A)$ classes at $(X',X)=(4,2)$, one can deduce that there is a multiplicative relation $\sum_{n=0}^7 (-1)^n H^{(16)}_{1A,2n+1}(8 \tau) = H^{(2)}_{8A,1}(\tau)$. 
See \S\ref{sec:mckay_spec} for more detailed information.

\begin{figure}[h]
\begin{center}
\includegraphics[scale=0.35]{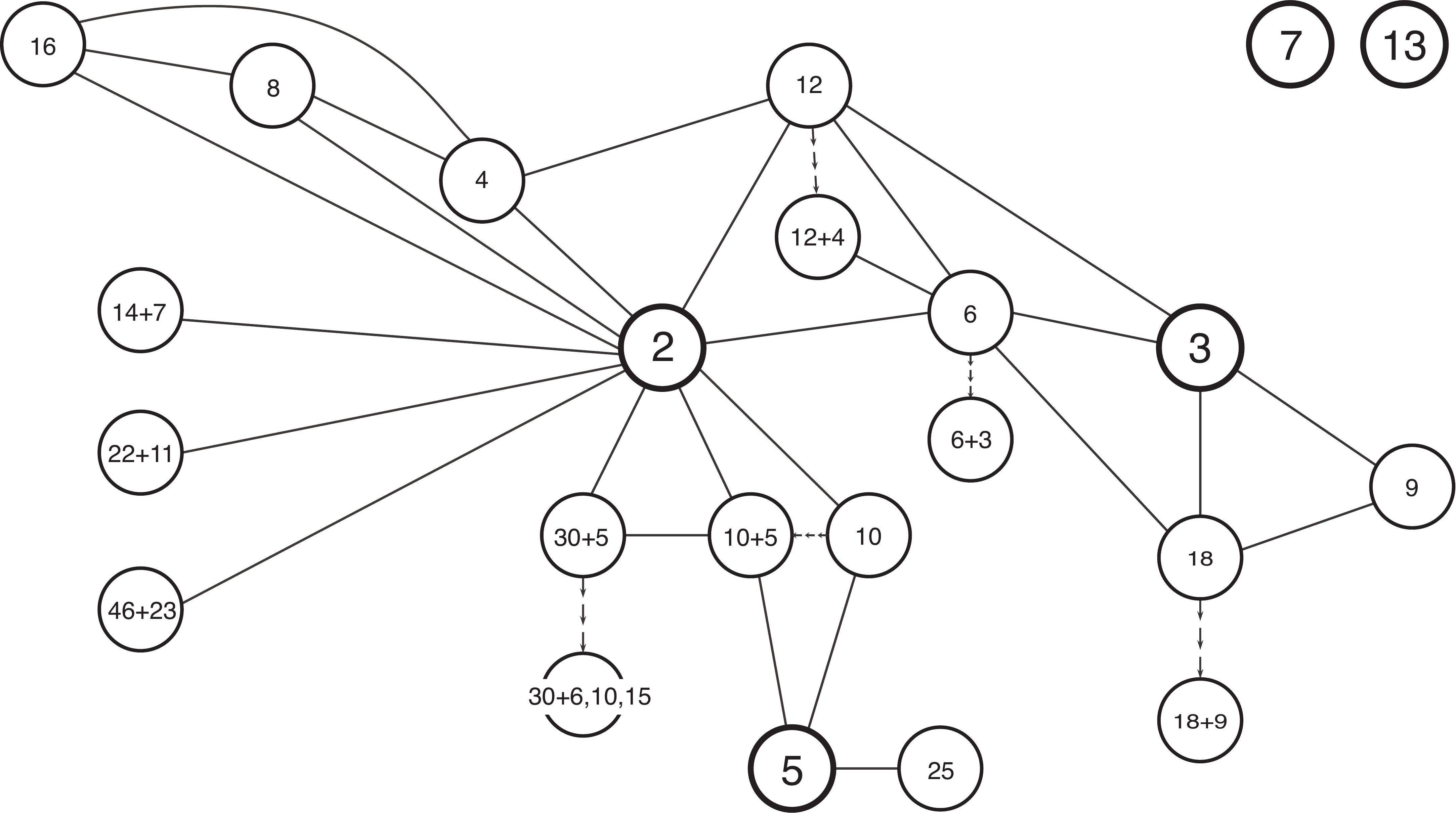}
\caption{ Multiplicative Relations.\label{fig:diag}}
\label{diagram}
\end{center}
\end{figure}

These multiplicative relations form an intricate web relating umbral moonshine at different lambencies. We summarise this web in Figure \ref{fig:diag}, where the horizontal relations are indicated by arrowed dashed lines and the vertical relations are indicated by solid  lines. 

We will finish this section with a discussion of a curious property of the multiplicative relations. Note that these relations occur among the McKay--Thomson series attached to $[g]\subset G^X $ and $[g']\subset G^{X'}$ with $ord(g') \, {  m}(X')$ coinciding with $ord(g) \, {  m}(X) $ up to a factor of 2.  
This property can be understood as a consequence of the relation between the level $n_g$ of the automorphy group $\G_0(n_g)$ of the McKay--Thomson series $H^X_{g}$ and the order $ord(g)$ of the group element (cf. \S\ref{sec:conj:aut}). The extra factor of 2 can be understood as a consequence of the structure $G^X= 2.\bar G^X$ of the associated umbral groups (cf. \S\ref{sec:holes:gps}).

\subsection{Mock Theta Functions}\label{sec:mckay:mocktheta}

In this subsection we record relations between the McKay--Thompson series of umbral moonshine and known mock theta functions. 
Many of the mock theta functions arising appear either in Ramanujan's last letter to Hardy or in his lost notebook \cite{Ramanujan_lost}. 
In what follows we will give explicit expressions for the mock theta series
using the {\em $q$-Pochhammer symbol} 
\be
(a;q)_n =\prod_{k=0}^{n-1} (1-a\,q^k) .
\ee

For lambency $2$, two of the functions $H^{(2)}_{g}(\t)$ are related to  Ramanujan's mock theta functions of orders $2$ and $8$ through
\begin{gather}
	\begin{split}
H^{(2)}_{4B}(\t) & = -2\, q^{-1/8} \m(q)=-2\,q^{-1/8}\sum_{n\geq 0}  \frac{(-1)^n\,q^{n^2}(q;q^2)_n}{(-q^2;q^2)_n^2} 
\\
H^{(2)}_{8A}(\t) & = -2\, q^{-1/8} U_0(q)=  -2\, q^{-1/8}= \sum_{n\geq 0}  \frac{q^{n^2}(-q;q^2)_n}{(-q^4;q^4)_n}. 
	\end{split}
\end{gather}

For $\ll=3$ we encounter the following  order $3$ mock theta functions of  Ramanujan: 
\begin{gather}
	\begin{split}
H^{(3)}_{2B,1}(\t) &=H^{(3)}_{2C,1}(\t)=H^{(3)}_{4C,1}(\t)=-2q^{-1/12} f(q^2)\\
H^{(3)}_{6C,1}(\t) &= H^{(3)}_{6D,1}(\t) = -2 q^{-1/12} \chi(q^2)
 \\
H^{(3)}_{8C,1}(\t) &= H^{(3)}_{8D,1}(\t) = -2 q^{-1/12} \phi(-q^2) \\
H^{(3)}_{2B,2}(\t) &= - H^{(3)}_{2C,2}(\t) = -4 q^{2/3} \omega(-q) 
\\\label{L3mocktheta}
H^{(3)}_{6C,2}(\t) &= - H^{(3)}_{6D,2}(\t) = 2 q^{2/3} \rho(-q),
	\end{split}
\end{gather}
where
\begin{align}\notag
f(q) &=  1+ \sum_{n=1}^\infty \frac{q^{n^2}}{(1+q)^2(1+q^2)^2 \cdots (1+q^n)^2} 
\\\notag
\phi(q) &=  1+\sum_{n=1}^\infty \frac{q^{n^2}}{(1+q^2)(1+q^4) \cdots (1+q^{2n})} \,  \\\notag
\chi(q) &= 1+ \sum_{n=1}^\infty \frac{q^{n^2}}{(1-q+q^2)(1-q^2+q^4) \cdots (1-q^n+q^{2n})} \, \\\notag
\omega(q) &=  \sum_{n=0}^\infty \frac{q^{2n(n+1)}}{(1-q)^2(1-q^3)^2 \cdots (1-q^{2n+1})^2} \, \\\label{order3mock}
\rho(q) &=  \sum_{n=0}^\infty \frac{q^{2n(n+1)}}{(1+q+q^2)(1+q^3+q^6) \cdots (1+q^{2n+1}+q^{4n+2})} \;.
\end{align}

For $\ll=4$ and $\ll=16$ we have the relations 
\begin{gather}
	\begin{split}
{H}^{(4)}_{2C,1}(\t)&=q^{-\frac{1}{16}} \left(-2 S_0(q) + 4T_0(q)\right), \\
{H}^{(4)}_{2C,3}(\t)&=q^{\frac{7}{16}} \left(2 S_1(q) - 4T_1(q)\right), \\ 
{H}^{(4)}_{4C,1}(\t)&= -2\, q^{-\frac{1}{16}} S_0(q), \\
{H}^{(4)}_{4C,3}(\t)&=2 \,q^{\frac{7}{16}} \, S_1(q) 
	\end{split}
\end{gather}
and 
\begin{align}
H^{(16)}_{1A,2}(\tau) &= H^{(16)}_{1A,14}=2 q^{-1/16} T_0(-q) =\frac{1}{2} \left({H}^{(4)}_{2C,1}(\t)-{H}^{(4)}_{4C,1}(\t)\right) \\\notag
H^{(16)}_{1A,4}(\tau) &= H^{(16)}_{1A,12}=2 q^{-1/4} V_1(q) \notag \\
H^{(16)}_{1A,6}(\tau) &= H^{(16)}_{1A,10}= 2 q^{7/16} T_1(-q)  
\end{align}
to the order 8 mock theta functions 
\begin{align}
S_0(\t) &= \sum_{n\geq 0} \frac{q^{n^2} (-q;q^2)_n}{(-q^2;q^2)_n} \\
S_1(\t) &= \sum_{n\geq 0} \frac{q^{n(n+2)} (-q;q^2)_n}{(-q^2;q^2)_n} \\
T_0(\t) &= \sum_{n\geq 0} \frac{q^{(n+1)(n+2)} (-q^2;q^2)_n}{(-q;q^2)_{n+1}} \\
T_1(\t) &= \sum_{n\geq 0} \frac{q^{n(n+1)} (-q^2;q^2)_n}{(-q;q^2)_{n+1}}\\
V_1 (\t) & =  \sum_{n\geq 0} \frac{q^{(n+1)^2} (-q;q^2)_n}{(q;q^2)_{n+1}}
\end{align}
discussed in  \cite{Gordon_Mcintosh}. We also have
$H^{(16)}_{2A,r}= - H^{(16)}_{1A,r}$ for $r=2,4,6,10,12,14$.

For $\ll=5$ we encounter four of  Ramanujan's order $10$ mock theta functions:
\begin{gather}
	\begin{split}
&H^{(5)}_{2BC,1}(\t) =H^{(5)}_{4CD,1}(\t) = -2 q^{{-\frac{1}{20}}} \,X(q^2)
\\ 
&H^{(5)}_{2BC,3}(\t) =H^{(5)}_{4CD,3}(\t) = -2 q^{{-\frac{9}{20}}} \,\chi_{10}(q^2)
\\ 
&H^{(5)}_{2C,2}(\t) =-H^{(5)}_{2B,2}(\t)  = 2q^{-\frac{1}{5}} \,\psi_{10}(-q)
\\ 
&H^{(5)}_{2C,4}(\t) =-H^{(5)}_{2B,4}(\t)  = -2q^{\frac{1}{5}} \,\f_{10}(-q)
	\end{split}
\end{gather}
given by 
\begin{align}
\phi_{10}(q) &= \sum_{n=0}^\infty \frac{q^{n(n+1)/2}}{(q;q^2)_{n+1}} \; \\
\psi_{10}(q) &= \sum_{n=0}^\infty \frac{q^{(n+1)(n+2)/2}}{(q;q^2)_{n+1}} \; \\
X(q) &= \sum_{n=0}^\infty \frac{(-1)^n q^{n^2}}{(-q;q)_{2n}} \; \\
\chi_{10}(q) &= \sum_{n=0}^\infty \frac{(-1)^n q^{(n+1)^2}}{(-q;q)_{2n+1}} \;. 
\end{align} 

At $\ll=6$ and $\ll=6+3$ we find
\begin{align}
H^{(6)}_{2B,3}(\tau)&=-2 \,q^{-3/8} \psi_6(q) \; \\
H^{(6+3)}_{2A,1}(\tau)&=-2 \,q^{-1/24} \phi_6(q) \; \\
H^{(6+3)}_{3B,1}(\tau)&=-2 \,q^{-1/24} \gamma_6(q) \; \\
H^{(6+3)}_{2B,1}(\tau)&=-2 \,q^{-1/24} f(q) \; \\
H^{(6+3)}_{4B,1}(\tau)&=-2 \,q^{-1/24} \phi(-q) \; \\
H^{(6+3)}_{6B,1}(\tau)&=-2 \,q^{-1/24} \chi(q) \; 
\end{align}
where $f(q), \phi(q),\chi(q)$ are third order mock theta functions given earlier and
\begin{align}
\psi_6(q)&= \sum_{n=0}^\infty \frac{(-1)^n q^{(n+1)^2}(q;q^2)_n}{(-q;q)_{2n+1}}  \; \\
\phi_6(q)&=\sum_{n=0}^\infty \frac{(-1)^n q^{n^2}(q;q^2)_n}{(-q;q)_{2n}} \; \\
\gamma_6(q)&=\sum_{n=0}^\infty \frac{ q^{n^2}(q;q)_n}{(q^3;q^3)_n} \;
\end{align}
are sixth order mock theta functions. 

For $\ll=8$
\begin{align}
H^{(8)}_{1A,2}(\tau)&=H^{(8)}_{1A,6}= 4 \,q^{-1/4} A(q)=4 \,q^{-1/4} \sum_{n\geq 0} \frac{q^{(n+1)} (-q^2;q^2)_n}{(q;q^2)_{n+1}}\\ 
H^{(8)}_{1A,4}(\tau)&=4 \,q^{1/2} B(q)=4 \,q^{1/2}  \sum_{n\geq 0} \frac{q^{n} (-q;q^2)_n}{(q;q^2)_{n+1}} 
\end{align}
where $A(q)$, $B(q)$ are  both second order mock theta functions 
discussed in \cite{MR2317449}.

For $\ll=12$ and $\ll=12+4$ we have
\begin{align}
H^{(12)}_{1A,2}(\tau)&=H^{(12)}_{1A,10}(\tau)= -2 \,q^{-4/48} \sigma(q) \; \\
H^{(12)}_{1A,4}(\tau) &= H^{(12)}_{1A,8}(\tau) = 2 \,q^{2/3} \omega(q) \; \\
H^{(12+4)}_{1A,1}(\tau) &=- q^{-1/48}\left( f(q^{1/2})-f(-q^{1/2}) \right) \;
\end{align}
where $\sigma(q)$ is the order $6$ mock theta function
\be
\sigma(q) =  \sum_{n=0}^\infty  \frac{q^{(n+1)(n+2)/2} (-q;q)_n}{(q;q^2)_{n+1}}
\ee
 and $\omega(q)$  and $f(q)$ are  third order mock theta functions given in \eq{order3mock}.

For $\ll=30+6,10,15$ we find four of Ramanujan's mock theta functions of order $5$: 
\begin{align}
H^{(30+6,10,15)}_{1A,1}(\tau) &=q^{-1/120} \left( 2\, \chi_0(q) - 4 \right) \notag\\ 
H^{(30+6,10,15)}_{1A,7}(\tau) &= 2\, q^{71/120} \chi_1(q) \notag\\
H^{(30+6,10,15)}_{2A,1}(\tau) &=-2 q^{-1/120}  \phi_0(-q)  \notag\\
H^{(30+6,10,15)}_{2A,7}(\tau) &=2 q^{-49/120} \phi_1(-q) 
\end{align}
where 
\begin{align}
\chi_0(\t)&=   \sum_{n\geq 0} \frac{q^{n} }{(q^{n+1};q)_n}\notag\\
\chi_1(\t)&= \sum_{n\geq 0} \frac{q^{n} }{(q^{n+1};q)_{n+1}} \notag\\
\phi_0(\t)&=   \sum_{n\geq 0} q^{n^2} {(-q;q^2)_n}\notag\\
\phi_1(\t)&=   \sum_{n\geq 0} q^{(n+1)^2} {(-q;q^2)_n}.
\end{align}

\subsection{Specification}\label{sec:mckay_spec}

In this subsection we will combine the different types of data on the McKay--Thompson series discussed in \S\ref{sec:weight_zero_umbral_forms} and \S\ref{sec:forms:low_lambencies}-\ref{sec:mckay:mocktheta}, and explain how they lead to explicit expressions for the umbral McKay--Thompson series. 

First we will note one more relation among umbral McKay--Thompson series $H^X_g$ and $H^X_{g'}$  attached to different conjugacy classes of the same umbral group $G^X$. 
Notice that for all the A-type  Niemeier root systems as well as $X= E_6^4$, the corresponding umbral group takes the form
\be
G^X = 2.\bar G^X 
\ee
and the corresponding McKay--Thompson series display the following paired relation. 
For the A-type  Niemeier root system $X$,  for every conjugacy class $[g]\subset G^X$ with $\bar\chi^{X_A}_g>0$ there is a unique conjugacy class $[g']$  with 
\be
\chi^{X_A}_g = -\chi^{X_A}_{g'},\quad \bar\chi^{X_A}_g = \bar \chi^{X_A}_{g'},
\ee
that we say to be paired with $[g]$. For such paired classes, the corresponding McKay--Thompson series satisfy the relation 
\be\label{relation_paired_classes} 
H^{X}_{g,r} + (-1)^{r} H^{X}_{g',r} =0 .
\ee
This generalises the paired property for the pure A-series discussed in \S\ref{sec:forms:low_lambencies}. 
Similarly, for $X=E_6^4$ we have for every conjugacy class $[g]\subset G^X$ with $\chi^{X_E}_g>0$ a conjugacy class $[g']$  with 
\be
\chi^{X_E}_g = -\chi^{X_E}_{g'},\quad \bar\chi^{X_E}_g = \bar \chi^{X_E}_{g'},
\ee
that is paired with $[g]$. For such paired classes, the corresponding McKay--Thompson series again satisfy the relation  \eq{relation_paired_classes}. 

For the lambency $\ll=6+3$, corresponding to $X=D_4^6$, the corresponding umbral group has the form
\be
G^X = 3.\bar G^X 
\ee
and the conjugacy classes with $\bar\chi^{X_D}(g)>0$ form pairs satisfying
\be
\chi^{X_D}_g = \chi^{X_D}_{g'},\quad \check \chi^{X_D}_g +2\check\chi^{X_D}_{g'}  = 0 . 
\ee
For these paired conjugacy classes, the McKay--Thompson series have the property
\be\label{relation_paired_classes_Z3}
H^{X}_{g,1} - H^{X}_{g',1} =0,\quad H^{X}_{g,3} +2 H^{X}_{g',3} =0.
\ee
In particular,  we have $H^{(6+3)}_{g,3}=0$  for all the self-paired classes.
From the discussion in \S\ref{sec:mckay:aut} we see that the relation between the twisted Euler characters and the paired relations is implied by the shadows of the umbral McKay--Thompson series, just as the multiplicative relations discussed in the \S\ref{sec:forms:mult}. 
As a result, one can view the relations \eq{relation_paired_classes} and \eq{relation_paired_classes_Z3} as a consequence of the apparently general phenomenon that the umbral McKay--Thompson series are determined by their mock modular properties together with the optimal growth condition, cf. Conjecture \ref{conj:conj:moon:opt}.

In the rest of this subsection we will tie these relations together and discuss each lambency separately. 
First of all, for all the A-type  Niemeier root systems 
\begin{gather}
	\ll\in\{2,3,4,5,6,7,8,9,10,12,13,16,18,25\},
\end{gather}
the discussion in \S\ref{sec:weight_zero_umbral_forms} specifies all $H^{(\ll)} = (H^{(\ll)}_{1A,r})$, $r=1,\dots,\ll-1$ and it remains only to specify $H^{(\ll)}_g$ with $[g] \subset G^X$ different from the identity class.  
For $\ll=2,3,5$, the discussion in  \S\ref{sec:forms:low_lambencies} specifies all the McKay--Thompson series. 
For $\ll=7,13$, the discussion in  \S\ref{sec:forms:low_lambencies} gives partial information on all the McKay--Thompson series. 
For $\ll=4$, the vertical relations in Table \ref{tab:vertical_relations} specify all the odd components of $H^{(4)}_{g,r}$ for all conjugacy classes $[g]\in G^{(4)}$ except for $[g]=4B, 8A$. 
The odd components for $[g]=4B, 8A$ can be specified by the identities
\begin{gather}\label{lambency4_1}
	\begin{split}
(H^{(4)}_{4B,1}-H^{(4)}_{4B,3})( \tau)&= -2 \frac{\eta(\t/2)\eta(\t)^4}{\eta(\t)^2\eta(4 \tau)^2}, \\
(H^{(4)}_{8A,1}-H^{(4)}_{8A,3})(\t)&= -2 \frac{ \eta( \tau)^3}{\eta(\tau/2) \eta(4 \tau)}.
	\end{split}
\end{gather}
We are left to determine the second components $H^{(4)}_{g,2}(\t)$, which are given by 
\begin{gather}\label{lambency4_2}
	\begin{split}
{ H}^{(4)}_{\;3A,2}(\t) 
& = \frac{1}{4}H^{(4)}_{2}(\t) +\frac{1}{2\h(2\t)^3}\Big(-3\L_2(\t) -4\L_3(\t) + \L_6(\t)\Big),\\
{ H}^{(4)}_{\;7AB,2}(\t) 
& =  \frac{1}{8}H^{(4)}_{2}(\t) +\frac{1}{12\,\h(2\t)^3}\Big(-{7} \L_2(\t)  -4 \L_7(\t) + \L_{14}(\t)  +{28} f_{14}(\t)\Big).
	\end{split}
\end{gather}
Together with the paired relation \eq{relation_paired_classes}, Table \ref{tab:vertical_relations}, and 
\be
{H}^{(4)}_{g,2}(\t) =0 \quad\text{for all}\; [g]\not\in\{1A,2A,3A,6A,7AB,14AB\},
\ee
\eq{lambency4_1}-\eq{lambency4_2}  completely specifies all  $H^{(4)}_g = (H^{(4)}_{g,r})$  
for all $[g]\subset G^{(4)}$.

For $\ll=6$, the vertical relations in Table \ref{tab:vertical_relations} to the McKay--Thompson series of lambencies $\ll=2,3$ suffice to specify all $H^{(6)}_{g} = (H^{(6)}_{g,r})$ except for the components $r=1,5$ of the classes $3A$ and $6A$. Subsequently, the relation to $H^{(18)}_{1A,r}$ in Table \ref{tab:vertical_relations} and the paired relation \eq{relation_paired_classes} determines 
$H^{(18)}_{3A}$ and $H^{(18)}_{6A}$.

For $\ll=8$, the vertical relations to $\ll=4$ recorded in Table \ref{tab:vertical_relations}, together with the paired relation \eq{relation_paired_classes} and $H_{2BC,2n}^{(8)}=H_{4A,2n}^{(8)}=0$ for all $n=1,2,3$ specify all the $H^{(8)}_g=(H^{(8)}_{g,r})$. 

For $\ll=9$, the vertical relations to $\ll=3$ recorded in Table \ref{tab:vertical_relations} together with the paired relation \eq{relation_paired_classes} determine all $H^{(9)}_{g}$ except for $[g]=3A,6A$. 
To determine $H^{(9)}_{3A}$ and $H^{(9)}_{6A}$, we have 
\begin{align*}
(H^{(9)}_{3A,1}-H^{(9)}_{3A,5}+H^{(9)}_{3A,7})(3\t) &= -6 \,\frac{\h(\t) \h(12\t) \h(18\t)^2 }{\h(6\t) \h(9\t) \h(36\t)}  \\
(H^{(9)}_{3A,2}-H^{(9)}_{3A,4}+H^{(9)}_{3A,8})(3\t) &= -3\left( \frac{\h(\t) \h(2\t) \h(3\t)^2 }{\h(4\t)^2 \h(9\t)}- \frac{\h(2\t)^6 \h(12\t) \h(18\t)^2 }{\h(\t) \h(4\t)^4 \h(6\t)\h(9\t)\h(36\t)}\right) \\
H^{(9)}_{3A,3}(\tau) = H^{(9)}_{3A,6}(\tau) & =H^{(9)}_{6A,3}(\tau) = H^{(9)}_{6A,6}(\tau)=0 .
\end{align*}
Together with the paired relation this determines all $H^{(9)}_{3A,r}$ and $H^{(9)}_{6A,r}$ and finishes the specification for $\ll=9$.

For $\ll=10$, the vertical relations to $\ll=2$ and $\ll=5$ recorded in Table \ref{tab:vertical_relations} specify all $H^{(10)}_g$. 
For $\ll=12, 16, 18, 25$, there is only one conjugacy class $2A$ except for the identity class. 
The paired relation  \eq{relation_paired_classes} relating the McKay--Thompson series for the $2A$ class to that of the identity class then determines all $H^{(\ll)}_g$. 

Next we turn to the D-type Niemeier root systems. 
For $\ll=6+3$, the horizontal relations in Table \ref{fig:horizontal_relations} determine all McKay--Thompson series $H^{(6+3)}_g$, except for $[g]\in\{3C,6C\}$ that are given by
\begin{align}\notag
H^{(6+3)}_{3C,1}(\t) &= -2 \frac{\eta^2(\t)}{\eta(3\t)} \;,\;H^{(6+3)}_{6C,1}(\t) = -2 \frac{\eta(2\t)\,\eta(3\t)}{\eta(6\t)},\\
H^{(6+3)}_{3C,3}&=H^{(6+3)}_{6C,3} =0 .
\end{align}
 For $\ll=10+5$, the horizontal relations in Table \ref{fig:horizontal_relations} and the vertical relations in Table \ref{tab:vertical_relations} determine all McKay--Thompson series $H^{(10+5)}_g$, except for $[g] =4A$. For $\ll=14+7$, the vertical relations in Table \ref{tab:vertical_relations} determine all McKay--Thompson series $H^{(14+7)}_g$. For $\ll=18+9$, the horizontal relations in Table \ref{fig:horizontal_relations} determine all McKay--Thompson series $H^{(18+9)}_g$. For $\ll=22+11$, the vertical relations in Table \ref{tab:vertical_relations} determine the umbral form $H^{(22+11)}_{1A}$. 
 For $\ll=30+15$ and $\ll=46+23$, the vertical relations in Table \ref{tab:vertical_relations} determine the only McKay--Thompson series $H^{(\ll)}_{1A}$ attached to the corresponding Niemeier root system.

 Finally we discuss the McKay--Thompson series attached to the two E-type Niemeier root systems. 
 For $\ll=12+4$, the horizontal relations in Table \ref{fig:horizontal_relations} and the vertical relations in Table \ref{tab:vertical_relations} determine all the McKay--Thompson series $H^{(12+4)}_{g}$ except for $[g]\in\{6A,8AB\}$. For $\ll=30+6,10,15$, the horizontal relations in Table \ref{fig:horizontal_relations} and \ref{tab:vertical_relations} determine the  umbral mock modular form $H^{(30+6,10,15)}_{1A}$.

Focusing on the umbral mock modular form $H^X=H^X_{1A}$ corresponding to the identity class, one can explicitly check that the above specification determines the unique vector-valued mock modular form satisfying the conditions of Corollary \ref{cor:uniqueness_umbral_mock_mod}. 
 
 \begin{prop}\label{prop:DE_UmbralForms}
Let $X$ be one of the D- or E-type Niemeier root systems (cf. (\ref{eqn:holes:NieRoot_D}), (\ref{eqn:holes:NieRoot_E})).  
The vector-valued mock modular form $H^X=H^X_{1A}$ specified above is the unique vector-valued mock modular form with shadow $S^X$ satisfying the optimal growth condition (cf. Corollary \ref{cor:uniqueness_umbral_mock_mod}).
\end{prop}

Together with Proposition \ref{prop:A_UmbralForms}, this proposition establishes our construction of the unique vector-valued mock modular form $H^X$ corresponding to all 23 Niemeier root systems $X$.

%------------------------------------------------------------------%
\section{Conjectures}\label{sec:conj}
%------------------------------------------------------------------%

In this section we pose the umbral moonshine conjectures connecting the 
umbral groups $G^X$ and the mock modular forms $H^X_g$ discussed in the previous sections. 

%---------------------------------------------------------------------------------------%
\subsection{Modules}\label{sec:conj:mod}
%---------------------------------------------------------------------------------------%

In this section we formulate a conjecture that relates the umbral McKay--Thompson series $H^{\rs}_g$ to an infinite-dimensional $G^{\rs}$-module $K^{\rs}$. 

Recall that a {\em super-space} $V$ is a $\ZZ/2\ZZ$-graded vector space $V=V_{\bar{0}}\oplus V_{\bar{1}}$.  
If $T:V\to V$ is a linear operator preserving the grading then the {\em super-trace} of $T$ is given by $\str_VT=\tr_{V_{\bar{0}}}T-\tr_{V_{\bar{1}}}T$ where $\tr_WT$ denotes the usual trace of $T$ on $W$. 
We say that $V$ is purely even (odd) when $V=V _{\bar{0}}$ ($V _{\bar{1}}$). If $V$ is a $G$-module, with $G$-action preserving the $\ZZ/2\ZZ$-grading, then the function $g\mapsto \str_Vg$ is called the {\em super-character} of $G$ determined by $V$.

\begin{conj}\label{conj:conj:mod:Kell}
Let $\rs$ be a Niemeier root system and let ${ m}$ be the Coxeter number of $\rs$. 

There exists a naturally defined $\ZZ/2m\ZZ\times\QQ$-graded super-module
\begin{gather}
	K^{\rs}=\bigoplus_{r\text{ mod }2m} K^{\rs}_r=\bigoplus_{r\text{ mod }2m}
	\bigoplus_{\substack{D\in\ZZ\\D=r^2\text{ mod }4m}}
	K^{\rs}_{r,-D/4m}
\end{gather}
for ${G}^{\rs}$ such that 
the graded super-character attached to an element $g\in G^{\rs}$ coincides with the vector-valued mock modular form  
\begin{gather}\label{eqn:conj:mod:str}
	c^\rs H^{\rs}_{g,r}(\tau)=
	\sum_{\substack{D\in\ZZ\\D=r^2\text{ mod }4m}}\str_{K^{\rs}_{r,-D/4m}}(g)\,q^{-D/4m},
\end{gather}
where $c^X=1$ except for  $X=A_8^3$, for which $c^X=3$. 
Moreover, the homogeneous component $K^{\rs}_{r,d}$ of $K^{\rs}$ is purely even if $d>0$. 
\end{conj}

The reason for the exceptional value $c^X=3$ for $X=A_8^3$ is the curious fact that there are no integer combinations of irreducible characters of $G^X$ that coincide with the coefficients $q^{-D/36}$, $D=-27\l^2$ for some integer $\l$, of $H^{\rs}_{g,r}(\tau)$ (cf. Conjecture \ref{conj:conj:disc:doub}). For example, the minimal positive integer $c$ for which $g\mapsto cH^X_{g,6}$ is a graded virtual super-character of $G^X$ is $c=c^X=3$.

Combining the above conjecture and the paired relations \eq{relation_paired_classes} and \eq{relation_paired_classes_Z3} of the McKay--Thompson series we arrive at the following. 
\begin{conj}\label{conj:conj:mod:factoring_through}
Let ${\rs}$ be a Niemeier root system and set $c=
\#G^{\rs}/\#\bar{G}^{\rs}$. Then 
the $\QQ$-graded $G^X$-module $K^X_{r}$ is a faithful representation of $G^X$ when $r\equiv 0\pmod{c}$ and factors through $\bar G^X$ otherwise. 
\end{conj}
As discussed in \S\ref{sec:holes}, we have $c=2$ for $\ll\in \{3,4,5,6,7,8,9,10,12,13,16,18,25,12+4\}$ and $c=3$ for $\ll=6+3$, and $c=1$ in the remaining $8$ cases.

As mentioned in \S\ref{sec:intro}, for the special case $X=A_1^{24}$, Conjecture \ref{conj:conj:mod:Kell} has been shown to be true by T. Gannon in\cite{Gannon:2012ck}, although the construction of $K^X$ is still absent. In this case, we have $c=1$ and hence Conjecture \ref{conj:conj:mod:factoring_through} is automatically true. It should be possible to apply the techniques similar to that in \cite{Gannon:2012ck} to prove Conjecture \ref{conj:conj:mod:Kell} for other Niemeier root systems $X$. 

%---------------------------------------------------------------------------------------%
\subsection{Modularity}\label{sec:conj:aut}
%---------------------------------------------------------------------------------------%

We have attached a cusp form $S^{\rs}_g$ to each $g\in G^{\rs}$ in \S\ref{sec:mckay:aut} by utilising the naturally defined permutation representations of $G^{\rs}$, and the corresponding twisted Euler characters, that are described in \S\ref{sec:holes:gps}. We begin this section with an explicit formulation of the conjecture that these cusp forms describe the shadows of the super-characters attached to the conjectural $G^{\rs}$-module $K^{\rs}$.

In preparation for the statement define $n_g$ and $h_g$ for $g\in G^{\rs}$ as follows. Take $n_g$ to be the order of the image of $g$ in $\bar{G}^{\rs}$ (cf. \S\ref{sec:holes:gps}), and set $h_g=N_g/n_g$ where $N_g$ denotes the product of the shortest and longest cycle lengths of the permutation which is the image of $g$ under $G^{\rs}\to\Sym_{\Phi}$. These values are on display in the tables of \S\ref{sec:chars:eul}. They may also be read off from the cycle shapes $\widetilde{\Pi}^{\rs}_g$ and $\bar{\Pi}^{\rs_A}_g$, $\bar{\Pi}^{\rs_D}_g$, \&c., attached to the permutation representations constructed in \S\ref{sec:holes:gps}, for $n_g$ is the maximum of the cycle lengths appearing in $\bar{\Pi}^{\rs_A}_g$, $\bar{\Pi}^{\rs_D}_g$ and $\bar{\Pi}^{\rs_E}_g$, and if $\widetilde{\Pi}^{\rs}_g=j_1^{m_1}\cdots j_k^{m_k}$ with $j_1^{m_1}<\dots<j_1^{m_k}$ and $m_i>0$ then $h_g=j_1j_k/n_g$.

\begin{conj}\label{conj:conj:aut:shad}
The graded super-characters (\ref{eqn:conj:mod:str}) for fixed $\rs$ and $g\in G^{\rs}$ and varying $r\in\ZZ/2m\ZZ$ define the components of a vector-valued mock modular form $H^{\rs}_{g}$ of weight $1/2$ on $\Gamma_0(n_g)$ with shadow function $S^{\rs}_g$.  
\end{conj}

Let $\nu^X_g$ denote the multiplier system of $H^X_g$. Since the multiplier system of a mock modular form is the inverse of the multiplier system of its shadow, Conjecture \ref{conj:conj:aut:shad} completely determines the modular properties of $H^X_g$---i.e. the matrix-valued function $\nu^X_g$---when $S^X_g$ is non-vanishing. However, it may happen that $S^X_g$ vanishes identically and $H^X_g$ is a(n honest) modular form. The following conjecture puts a strong restriction on $\nu^X_g$ even in the case of vanishing shadow.
\begin{conj}\label{conj:conj:aut:aut}
The multiplier system $\nu^X_g$ for $H^X_g$ coincides with the inverse of the multiplier system for $S^X$ when restricted to $\G_0(n_gh_g)$.
\end{conj}

%---------------------------------------------------------------------------------------%
\subsection{Moonshine}\label{sec:conj:moon}
%---------------------------------------------------------------------------------------%

We now formulate a conjecture which may be regarded as the analogue of the principal modulus property (often referred to as the genus zero property) of monstrous moonshine. 

The monstrous McKay--Thompson series $T_g$, for $g$ an element of the monster, are distinguished in that each one is a principal modulus with pole at infinity for a genus zero group $\G_g$, meaning that $T_g$ is a $\G_g$-invariant holomorphic function on the upper-half plane having a simple pole at the infinite cusp of $\G_g$, but having no poles at any other cusps of $\G_g$. Equivalently, $T_g$ satisfies the conditions
\begin{gather}\label{eqn:conj:moon:pmod}
	\begin{split}
	&\text{(i)}\quad T_g|_{1,0}\g=T_g\text{ for all $\g\in \G_g$},\\
	&\text{(ii)}\quad qT_g(\t)=O(1)\text{ as $\t\to i\infty$},\\
	&\text{(iii)}\quad T_g(\t)=O(1)\text{ as $\t\to \alpha\in \QQ$ whenever $\infty\notin \G_g\alpha$},
	\end{split}
\end{gather}
for some group $\G_g$, where $T_g|_{1,0}\g$ is the function $\t\mapsto T_g(\g\t)$ by definition (cf. (\ref{eqn:sums:psiw_actn})). 
Note that the existence of a non-constant function $T_g$ satisfying the conditions (\ref{eqn:conj:moon:pmod}) implies that the group $\G_g$ has genus zero, for such a function necessarily induces an isomorphism from $X_{\G_g}$ (cf. (\ref{eqn:sums:XG})) to the Riemann sphere.
As such, if we assume that both $T_g$ and $T'_g$ satisfy these conditions and both $qT_g$ and $qT'_g$ have the expansion $1+O(q)$ near $\t\to i\inf$, then $T_g$ and $T'_g$ differ by an additive constant; i.e. $T_g'=T_g+C$ for some $C\in \CC$. 
In other words, the space of solutions to (\ref{eqn:conj:moon:pmod}) is $1$ or $2$ dimensional, according as the genus of $X_{\G_g}$ is positive or $0$.

Observe the similarity between condition (ii) of (\ref{eqn:conj:moon:pmod}) and the optimal growth condition (\ref{optimal_growth}). Since  $q^{-1}$ is the minimal polar term possible for a non-constant $\G_g$-invariant function on the upper-half plane, assuming that the stabiliser of infinity in $\G_g$ is generated by $\pm \left(\begin{smallmatrix} 1&1\\0&1\end{smallmatrix}\right)$, the condition (ii) is an optimal growth condition on modular forms of weight $0$; the coefficients of a form with higher order poles will grow more quickly. The condition (iii) naturally extends this to the situation that $\G_g$ has more than one cusp.

Accordingly, we now formulate an analogue of (\ref{eqn:conj:moon:pmod})---and an extension of the optimal growth condition (\ref{optimal_growth}) to vector-valued mock modular forms of higher level---as follows. Suppose that $\nu$ is a (matrix-valued) multiplier system on $\G_0(n)$ with weight $1/2$, and suppose, for the sake of concreteness, that $\nu$ coincides with the inverse of the multiplier system of $S^X$, for some Niemeier root system $X$, when restricted to $\G_0(N)$ for some $N$ with $n|N$. Observe that, under these hypotheses, every component $H_r$ has a Fourier expansion in powers of $q^{1/4m}$ where $m$ is the Coxeter number of $X$, so $q^{-1/4m}$ is the smallest order pole that any component of $H$ may have. Say that a vector-valued function $H=(H_r)$ is a mock modular form of {\em optimal growth} for $\G_0(n)$ with multiplier $\nu$, weight $1/2$ and shadow $S$ if
\begin{gather}\label{eqn:conj:moon:opt}
	\begin{split}
	&\text{(i)}\quad H|_{1/2,\nu,S}\g=H\text{ for all $\g\in \G_0(n)$},\\
	&\text{(ii)}\quad q^{1/4m}H_r(\t)=O(1)\text{ as $\t\to i\infty$ for all $r$,}\\
	&\text{(iii)}\quad H_r(\t)=O(1)\text{ for all $r$ as $\t\to \alpha\in \QQ$, whenever $\infty\notin \G_g\alpha$}.
	\end{split}
\end{gather}
In condition (i) of (\ref{eqn:conj:moon:opt}) we write $|_{\nu,1/2,S}$ for the weight $1/2$ action of $\G_0(n)$ with multiplier $\nu$ and twist by $S$ (cf. (\ref{eqn:sums:gtwact})), on holomorphic vector-valued functions on the upper-half plane.  

Recall that $\nu^X_g$ denotes the multiplier system of $H^X_g$, and $S^X_g$ is its shadow. Recall also that $n_g$ denotes the order of (the image of) $g\in G^X$ in the quotient group $\bar{G}^X$. We now conjecture that the umbral McKay--Thompson series all have optimal growth in the sense of (\ref{eqn:conj:moon:opt}), and this serves as a direct analogue of the Conway--Norton conjecture of monstrous moonshine, that all the monstrous McKay--Thompson series are principal moduli for genus zero subgroups of $\SL_2(\RR)$; or equivalently, that they are all functions of optimal growth in the sense of (\ref{eqn:conj:moon:pmod}).
\begin{conj}\label{conj:conj:moon:opt}
Let $X$ be a Niemeier root system and let $g\in G^X$. Then $H^X_g$ is the unique, up to scale, mock modular form of optimal growth for $\G_0(n_g)$ with multiplier $\nu^X_g$, weight $1/2$ and shadow $S^X_g$.
\end{conj}

Conjecture \ref{conj:conj:moon:opt} should serve as an important step in obtaining a characterisation of the mock modular forms of umbral moonshine.
Note that the above conjecture has been proven in \S\ref{sec:umbral shadow} (cf. Corollary \ref{cor:uniqueness_umbral_mock_mod})  for the identity class of $G^X$ for all  Niemeier root systems $X$. 
Note that in the case of the identity class we have $\G_0(n) =\SL_2(\ZZ)$ which has  the cusp (representative) at $i\inf$ as the only cusp. 
As a result, the more general conditions in \eq{eqn:conj:moon:opt} reduce to the condition \eq{optimal_growth} discussed in \S\ref{sec:forms:umbral}.
For $X=A_1^{24}$,  this conjecture was proven for all conjugacy classes of $G^X = M_{24}$ in \cite{Cheng2011}.
 See also \cite{2012arXiv1212.0906C} for related results in this case.

%---------------------------------------------------------------------------------------%
\subsection{Discriminants}\label{sec:conj:disc}
%---------------------------------------------------------------------------------------%

One of the most striking features of umbral moonshine is the apparently intimate relation between the number fields on which the irreducible representations of $G^X$ are defined and the discriminants of the vector-valued mock modular form $H^X$. 
In this subsection we will discuss this ``discriminant property", extending the discussion in \cite{UM}.

First, for a Niemeier root system with Coxeter number $m$ we observe that the discriminants of the components $H^X_r$ of the mock modular form $H^{X}=H^{X}_{1A}$ determine some important properties of the representations of $G^X$. Here we say that an integer $D$ is a {\em discriminant of $H^{X}$} if there exists a term $q^d=q^{-\frac{D}{4m}}$ with non-vanishing Fourier coefficient in at least one of the components. The following result can be verified explicitly using the tables in \S\S\ref{sec:chars},\ref{sec:decompositions}.
\begin{prop}\label{discri1}
Let $X$ be one of the 23 Niemeier root systems. If $n>1$ is an integer satisfying 
\begin{enumerate}
\item{there exists an element of $G^X$ of order $n$}, and
\item{there exists an integer $\l$ that satisfies at least one of the following conditions and such that $D = -n \l^2$ is a discriminant of $H^{X}$. First, 
$(n,\lambda)=1$, and second,  $\l^2$ is a proper divisor of $n$,}
\end{enumerate}
then there exists at least one pair of irreducible representations $\varrho$ and $\varrho^*$ of $G^X$ and at least one element $g \in G^X$ such that $\tr_{\varrho}(g)$ is not rational but
\be\label{n_type}
{\tr}_{\varrho} (g), {\tr}_{ \varrho^*} (g) \in \QQ(\sqrt{-n})
\ee
and $n$ divides $ord(g)$.
\end{prop}
The list of integers $n$ satisfying the two conditions of Proposition \ref{discri1} is given in Table \ref{list_discriminant}. We omit from the table lambencies of Niemeier root systems for which there exists no integer $n$ satisfying these conditions.   
 
From now on we say that an irreducible representation $\varrho$ of the umbral group $G^X$ is of {\em type $n$} if $n$ is an integer satisfying the two conditions of Proposition \ref{discri1} and the character values of $\varrho$ generate the field $\QQ(\sqrt{-n})$. Evidently, irreducible representations of {type $n$} come in pairs $(\varrho,\varrho^*)$ with ${\tr}_{\varrho^*} (g)$ the complex conjugate of ${\tr}_{ \varrho} (g)$ for all $g\in G^X$. The list of all irreducible representations of type $n$ is also given in Table \ref{list_discriminant}. (See \S\ref{sec:chars:irr} for the character tables of the $G^X$ and our notation for irreducible representations.) 
 
Recall that the {\em Frobenius--Schur indicator} of an irreducible ordinary representation of a finite group is $1$, $-1$ or $0$ according as the representation admits an invariant symmetric bilinear form, an invariant skew-symmetric bilinear form, or no invariant bilinear form, respectively. The representations admitting no invariant bilinear form are precisely those whose character values are not all real. We can now state the next observation.
\begin{prop}\label{FS_indicator}
For each Niemeier root system $X$, an irreducible representation $\varrho$ of $G^X$ has Frobenius--Schur indicator $0$ if and only if it is of type $n$ for some $n$ defined in Proposition \ref{discri1}. 
\end{prop}

The {\em Schur index} of an irreducible representation $\varrho$ of a finite group $G$ is the smallest positive integer $s$ such that there exists a degree $s$ extension $k$ of the field generated by the character values ${\tr}_{\varrho}( g)$ for $g\in G$ such that $\varrho$ can be realised over $k$. Inspired by Proposition \ref{FS_indicator} we make the following conjecture.
\begin{conj}\label{conj:conj:disc:sch}
If $\varrho$ is an irreducible representation of $G^X$ of type $n$ then the Schur index of $\varrho$ is equal to $1$.
\end{conj}
In other words, we conjecture that the irreducible $G^X$-representations of type $n$ can be realised over $\QQ(\sqrt{-n})$. For 
$X=A_1^{24}$ this speculation is in fact a theorem, since it is known \cite{Ben_SchInd} that the Schur indices for $M_{24}$ are always $1$. For 
$X=A_2^{12}$ it is also known \cite{Ben_SchInd} that the Schur indices for $\bar G^{(3)}= M_{12}$ are also always $1$. Moreover,  the representations of $G^{(3)}\simeq 2.\bar{G}^{(3)}$ with characters $\chi_{16}$ and $\chi_{17}$ in the notation of Table \ref{tab:chars:irr:3} have been constructed explicitly over $\QQ(\sqrt{-2})$ in \cite{Mar_M12}. Finally, Proposition \ref{FS_indicator} constitutes a non-trivial consistency check for Conjecture \ref{conj:conj:disc:sch} since the Schur index is at least $2$ for a representation with Frobenius--Schur indicator equal to $-1$.

\vspace{18pt}
\begin{table}[h!] 
\centering  
\begin{tabular}{CCC}
\toprule
X & n &(\varrho,\varrho^*)\\\midrule
A_1^{24}& 7,15,23 & (\chi_{3},\chi_{4}),(\chi_{5},\chi_{6}),(\chi_{10},\chi_{11}),(\chi_{12},\chi_{13}),(\chi_{15},\chi_{16})\\
A_2^{12}& 5,8,11,20 & (\chi_{4},\chi_{5}),(\chi_{16},\chi_{17}),(\chi_{20},\chi_{21}),(\chi_{22},\chi_{23}),(\chi_{25},\chi_{26})\\
A_3^8& 3,7&(\chi_{2},\chi_{3}),(\chi_{13},\chi_{14}),(\chi_{15},\chi_{16})\\ 
A_4^6& 4&(\chi_{8},\chi_{9}),(\chi_{10},\chi_{11}),(\chi_{12},\chi_{13})\\
A_5^4D_4&8&(\chi_{6},\chi_{7})\\
A_6^4&3&(\chi_{2},\chi_{3}),(\chi_{6},\chi_{7})\\
A_9^2D_6&4&(\chi_3,\chi_4)\\
A_{12}^{2}&4&(\chi_3,\chi_4)\\
D_4^6&15&(\chi_{12},\chi_{13})\\
E_6^4&8&(\chi_{6},\chi_{7})\\
\bottomrule
\end{tabular}
\caption{\label{list_discriminant}  
	The irreducible representations of type $n$.}
\end{table}

Equipped with the preceding discussion we are now ready to state our main observation for the discriminant property of umbral moonshine. \begin{prop}
Let $X$ be a Niemeier root system with Coxeter number $m$.
Let $n$ be one of the integers in Table \ref{list_discriminant} and let $\l_{n}$ be the smallest positive integer such that $D = -n \l_{n}^2$ is a discriminant of $H^X$. Then $K^X_{r,-D/4m} = \varrho_{n} \oplus \varrho_{n}^*$ where $\varrho_{n}$ and $ \varrho_{n}^*$ are dual irreducible representations of type $n$. Conversely, if $\varrho$ is an irreducible representation of type $n$ and $-D$ is the smallest positive integer such that $K^X_{r,-D/4m}$ has $\varrho$ as an irreducible constituent then there exists an integer $\l$ such that $D = - n \l^2$. 
\end{prop} 
 
Extending this we make the following conjecture.
\begin{conj}\label{conj:conj:disc:dualpair}
Let $X$ be a Niemeier root system with Coxeter number $m$.
If $D$ is a discriminant of $H^X$ which satisfies $D = -n \l^2$  for some integer $\l$ then the representation $K^X_{r,-D/4m}$ has at least one dual pair of irreducible representations of type $n$ arising as irreducible constituents. 
\end{conj}

Conjecture (\ref{conj:conj:disc:dualpair}) has been verified for the case $X=A_1^{24}$ in \cite{2012arXiv1211.3703C}.

We conclude this section with conjectures arising from the observation (cf. \S\ref{sec:decompositions}) that the conjectural $G^X$-module $K^X_{r,d}$ is typically  isomorphic to several copies of a single  representation. We say a $G$-module $V$ is a {\em doublet} if it is isomorphic to the direct sum of two copies of a single representation of $G$, and interpret the term {\em sextet} similarly.

\begin{conj}\label{conj:conj:disc:doub}
Let $X$ be a Niemeier root system and let $m$ be the Coxeter number of $X$.
Then the representation $K^{X}_{r,-D/4m}$ is a doublet  
if and only if $D \neq0$ and $D \neq -n\l^2$ for any integer $\l$ and for any $n$ listed in Table \ref{list_discriminant} corresponding to $X$. 
If $X=A_8^3$ 
then the representation $K^{X}_{r,-D/4m}$ is a sextet if and only if $D \neq -27 \l^2$ for some integer $\l$.  
\end{conj}
In particular, for the nine Niemeier root systems 
\begin{gather}
	A_7^2D_5^2,\; A_{11}D_7E_6,\; D_6^4,\; D_8^3,\; D_{10}E_7^2,\; D_{12}^2,\; D_{16}E_8,\; D_{24},\; E_8^3,
\end{gather}
that have no irreducible representation with vanishing Forbenius--Schur indicator and have no terms with zero discriminant, we conjecture that all the representations $K^{X}_{r,-D/4m}$ corresponding to $H^X_{g}$ are doublets. 

%------------------------------------------------------------------%
\section{Conclusions and Discussion}\label{sec:conc}
%------------------------------------------------------------------%

Let us start by recapitulating the main results of this paper. 
Taking the 23 Niemeier lattices as the starting point, in \S\ref{sec:holes:gps} we identify a finite group $G^X$---the umbral group---for each Niemeier root system $X$. 
On the other hand, using the ADE classification discussed in \S\ref{sec:forms:ADE} and Theorem \ref{thm:uniqueness_umbral_mock_jac}, we identify a distinguished vector-valued mock modular form $H^X$---the umbral form---for each Niemeier root system $X$. We then conjecture (among other things) 
\begin{enumerate}
\item (Conjecture \ref{conj:conj:mod:Kell}) that the mock modular form $H^X$ encodes the graded super-dimension of a certain infinite-dimensional, $\ZZ/2m\ZZ\times \QQ$-graded module $K^X$ for $G^X$,
\item (Conjecture \ref{conj:conj:aut:shad}) that the graded super-characters $H^X_g$ arising from the action of $G^X$ on $K^X$ are vector-valued mock modular forms with concretely specified shadows $S^X_g$, and 
\item (Conjecture \ref{conj:conj:moon:opt}) that the umbral McKay--Thompson series $H^X_g$ are uniquely determined by an optimal growth property which is directly analogous to the genus zero property of monstrous moonshine.
\end{enumerate}
To lend evidence in support of these conjectures we explicitly identify (almost all of) the umbral McKay--Thompson series $H^X_g$.

In this way, from the 23 Niemeier root systems we obtain 23 instances of umbral moonshine, encompassing all the 6 instances previously discussed in \cite{UM} and in particular the case with $G^X=M_{24}$ first discussed in the context of the $K3$ elliptic genus \cite{Eguchi2010}. 
Apart from uncovering 17 new instances, we believe that the relation to Niemeier lattices sheds important light on the underlying structure of umbral moonshine. 
First, the construction of the umbral group $G^X$ is now completely uniform: $G^X$ is the outer-automorphism group of the corresponding Niemeier lattice (cf. \eq{def:umbral_group}).   Second, it provides an explanation for why the 6 instances discussed in \cite{UM} are naturally labelled by the divisors of 12: they correspond to Niemeier root systems given by evenly many copies (viz., $24/(\ll-1)$) of an A-type root system $A_{\ll-1}$.
Third, it also sheds light on the relation between umbral moonshine and meromorphic weight 1 Jacobi forms as well as weight $0$ Jacobi forms. For as we have seen in \S\ref{sec:forms:meromock}, the umbral forms $H^X$ can be constructed uniformly by taking theta-coefficients of finite parts of certain weight $1$ meromorphic Jacobi forms, but in general the relevant meromorphic Jacobi form has simple poles not only at the origin but also at non-trivial torsion points 
whenever the corresponding root system has a D- or E-type root system as an irreducible component. As a result, in those cases the relation to weight $0$ Jacobi forms is less direct as the Eichler--Zagier operator ${\cal W}^X$ of \eq{eqn:weight_one_from_weight_zero} is no longer proportional to the identity. In particular, in these cases the umbral mock modular form $H^X$ does not arise in a direct way from the decomposition of a weight $0$ (weak) Jacobi form into irreducible characters for the ${\cal N}=4$ superconformal algebra. 

Recall that the relevant weight $0$ Jacobi form in the construction described in \S\ref{sec:weight_zero_umbral_forms} coincides with the elliptic genus of a $K3$ surface in the case of the Niemeier root system $X=A_1^{24}$ ($\ll=2$, $G^X=M_{24}$). 
As the relation to weight $0$ forms becomes less straightforward in the more general cases, the relation between umbral moonshine and sigma models, or in fact any kind of conformal field theory, also becomes more opaque. An interesting question is therefore the following. What, if any, kind of physical theory or geometry should attach to the more general instances of umbral moonshine? 

To add to this puzzle, the Borcherds lift of the $K3$ elliptic genus is a Siegel modular form which also plays an important role in type II as well as heterotic string theory compactified on $K3\times T^2$ \cite{DVV,Kawai:1995hy,LopesCardoso:1996nc}. 
As pointed out in \cite{Cheng2010_1} and refined in \cite{CheDun_M24MckAutFrms}, Mathieu moonshine in this context (corresponding here to $X=A_1^{24}$) leads to predictions regarding Siegel modular forms which have been partially proven in \cite{Raum}. 
Furthermore, this Siegel modular form also serves 
as the square of the denominator function of a generalised Kac--Moody algebra developed by Gritsenko--Nikulin in the context of mirror symmetry for $K3$ surfaces\cite{GriNik_K3SrfsLorKMAlgsMrrSym,GriNik_ArthMrrSymCYMfds,GriNik_AutFrmLorKMAlgs_I,GriNik_AutFrmLorKMAlgs_II,GrNik_SieAutFrmCorrLorKMAlgs}. 
As discussed in detail in \cite[\S5.5]{UM}, many of these relations to string theory and $K3$ geometry extend to some of the other 5 instances of umbral moonshine discussed in that paper. 
Since the relation between umbral forms and weight $0$ modular forms is modified when D- or E-type root systems are involved, it would be extremely interesting to determine how the above-mentioned relations to string theory, $K3$ surfaces, and generalised Kac--Moody algebras manifest in the more general cases.  

Regarding $K3$ surfaces, note that Niemeier lattices have a long history of application to this field, and the study of the symmetries of $K3$ surfaces  in particular. 
See for instance \cite{Kondo_old}. See \cite{Nikulin:2011} for an analysis involving all of the Niemeier lattices. It would be interesting to explore the extent to which recent work \cite{Taormina:2013mda,Taormina:2013jza} applying the Niemeier lattice $L^X$ to the $X=A_1^{24}$ case of umbral moonshine can be extended to other Niemeier root systems in light of \cite{Nikulin:2011}.

In another direction, the physical context of  Mathieu moonshine has been extended recently to $K3$ compactifications of heterotic string theory with 8 supercharges \cite{Cheng:2013kpa}. As the structure of theories with 8 supercharges is much less rigid than those with 16 supercharges, one might speculate that a suitable generalisation of \cite{Cheng:2013kpa} could provide  physical realisations of more instances of umbral moonshine.

 Apart from posing the umbral moonshine conjecture, in this paper we have also noted  various intricate and mysterious properties of this new moonshine story. 
 An important example is the close relation between Niemeier lattices and the genus zero groups discussed in \S\ref{sec:holes:gzero} and \S\ref{sec:forms:genus0}. 
 Another is the multiplicative relations discussed in \S\ref{sec:forms:mult} between McKay--Thompson series attached to different Niemeier root systems that plays an important role in the explicit specification of the umbral McKay--Thompson series in \S\ref{sec:mckay_spec}. 
While we have a concrete description of these two properties in terms of mock modular forms and related structures, their origin is still unclear.
We also observe empirically discriminant relations between number fields underlying irreducible representations of $G^X$ and the discriminants of the vector-valued mock modular form $H^X$ (\S\ref{sec:conj:disc}) extending the observations made in \cite[\S5.4]{UM}. It would be extremely interesting to have a 
representation theoretic explanation of these relations.  
  
Last but not least, the construction of the umbral modules $K^X$ is clearly an important next step in unraveling the mystery of umbral moonshine. 
  
%---------------------------------------------------------------------------------------%
\section*{Acknowledgements}
%---------------------------------------------------------------------------------------%

We owe particular thanks to George Glauberman for alerting us to the connection between the umbral groups of our previous paper and the stabilisers of deep holes in the Leech lattice. This observation was a crucial catalyst for the present work. We also thank Daniel Allcock, Kathrin Bringmann, Scott Carnahan, Noam Elkies, Igor Frenkel, Terry Gannon, Ching Hung Lam,  Geoffrey Mason, Atsushi Matsuo, Sameer Murthy, Ken Ono, Erik Verlinde and Don Zagier for helpful comments and discussions. We thank the referees for numerous helpful suggestions and comments. MC would like to thank the Simons Center for Geometry and Physics, Stanford University and the Haussdorff Center for Mathematics for hospitality. The work of JH was supported by NSF grant 1214409.

\clearpage

%------------------------------------------------------------------------%
\appendix
%------------------------------------------------------------------------%

%------------------------------------------------------------------------%
\section{Special Functions}\label{sec:modforms}
%------------------------------------------------------------------------%

\subsection{Dedekind Eta Function}\label{sec:mdlrfrms:dedeta}
The {\em Dedekind eta function}, denoted $\eta(\t)$, is a holomorphic function on the upper half-plane defined by the infinite product 
\be\eta(\t)=q^{1/24}\prod_{n\geq 1}(1-q^n)\ee
where $q=\ex(\t)=e^{\tpi \t}$. It is a modular form of weight $1/2$ for the modular group $\SL_2(\ZZ)$ with multiplier $\e:\SL_2(\ZZ)\to\CC^*$, which means that 
\be\e(\g)\eta(\g\t){\rm jac}(\g,\t)^{1/4}=\eta(\t)\ee
for all $\g = \left(\begin{smallmatrix} a&b\\ c&d \end{smallmatrix}\right) \in\SL_2(\ZZ)$, where ${\rm jac}(\g,\t)=(c\t+d)^{-2}$. The { multiplier system} $\e$ may be described explicitly as 
\be\label{Dedmult}
\e\bem a&b\\ c&d\eem 
=
\begin{cases}
	\ex(-b/24),&c=0,\,d=1\\
	\ex(-(a+d)/24c+s(d,c)/2+1/8),&c>0
\end{cases}
\ee
where $s(d,c)=\sum_{m=1}^{c-1}(d/c)((md/c))$ and $((x))$ is $0$ for $x\in\ZZ$ and $x-\lfloor x\rfloor-1/2$ otherwise. We can deduce the values $\e(a,b,c,d)$ for $c<0$, or for $c=0$ and $d=-1$, by observing that $\e(-\g)=\e(\g)\ex(1/4)$ for $\g\in\SL_2(\ZZ)$.

\subsection{Jacobi Theta Functions}\label{sec:JacTheta}
We define the {\em Jacboi theta functions} $\th_i(\t,z)$ as follows for $q=e(\t)$ and $y=e(z)$.
\begin{align}	\th_1(\t,z)
	&= -i q^{1/8} y^{1/2} \prod_{n=1}^\infty (1-q^n) (1-y q^n) (1-y^{-1} q^{n-1})\\
	\th_2(\t,z)
	&=  q^{1/8} y^{1/2} \prod_{n=1}^\infty (1-q^n) (1+y q^n) (1+y^{-1} q^{n-1})\\
	\th_3(\t,z)
	&=  \prod_{n=1}^\infty (1-q^n) (1+y \,q^{n-1/2}) (1+y^{-1} q^{n-1/2})\\
	\th_4(\t,z) 
	&=  \prod_{n=1}^\infty (1-q^n) (1-y \,q^{n-1/2}) (1-y^{-1} q^{n-1/2})
\end{align}
Note that there are competing conventions for $\th_1(\t,z)$ in the literature and our normalisation may differ from another by a factor of $-1$ (or possibly $\pm i$).

\subsection{Higher Level Modular Forms}

The congruence subgroups of the modular group $\SL_2(\ZZ)$ that are most relevant for this paper are the {\em Hecke congruence groups}
\bea\label{congruence1}
\G_0(N) &=& \bigg\{\bigg[\begin{array}{cc} a&b\\c&d\end{array}\bigg]\in \SL_2(\ZZ)  , c= 0 \text{ mod }N\;
\bigg\}.
\eea
A modular form for $\G_0(N)$ is said to have {\em level} $N$. For $N$ a positive integer a modular form of weight $2$ for $\G_0(N)$ is given by
\bea\label{Eisenstein_form}
\l_N(\t)&=&\, q\pa_q\log\left(\frac{\eta(N\tau)}{\eta(\tau)}\right)\\\notag&=&\frac{N-1}{24}+\sum_{k>0}\s(k) (q^k -N q^{Nk})
\eea
where $\s(k)$ is the divisor function $\s(k)=\sum_{d\lvert k}d$. The function $\l_N$ is, of course, only non-zero when $N>1$.

Observe that a modular form on $\Gamma_0(N)$ is a modular form on $\Gamma_0(M)$ whenever $N|M$, and for some small $N$ the space of forms of weight $2$ is spanned by the $\l_d(\t)$ for $d$ a divisor of $N$. 

A discussion of the ring of weak Jacobi forms of higher level can be found in \cite{aoki}.

\subsection{Weight Zero Jacobi Forms}\label{subsec:basis_Jac}

According to \cite{Gri_EllGenCYMnflds} the graded ring $J_{0,*}= \bigoplus_{m \geq 1} J_{0,m-1}$, of weak Jacobi forms with weight $0$ and integral index (cf. \S\ref{sec:forms:jac}), is finitely generated, by $\varphi^{(2)}_{1}$, $\varphi^{(3)}_{1}$ and $\varphi^{(4)}_{1}$, where
\begin{gather}
\begin{split}\label{phiform}
\varphi^{(2)}_{1}&=4\left(f_2^2+f_3^2+f_4^2\right),\\
\varphi^{(3)}_{1}&=2\left(f_2^2f_3^2+f_3^2f_4^2+f_4^2f_2^2\right),\\
\varphi^{(4)}_{1}&=4 f_2^2f_3^2f_4^2,
\end{split}
\end{gather}
and $f_i(\tau,z)=\theta_i(\tau,z)/\theta_i(\tau,0)$ for $i\in\{2,3,4\}$ (cf. \S\ref{sec:JacTheta}). If we work over $\ZZ$ then we must include 
\be
\varphi^{(5)}_{1}=\frac{1}{4}\left(\varphi^{(4)}_{1}\varphi^{(2)}_{1}-(\varphi^{(3)}_{1})^2\right)
\ee
as a generator also, so that $J^{\ZZ}_{0,*}=\ZZ[\varphi^{(2)}_{1},\varphi^{(3)}_{1},\varphi^{(4)}_{1},\varphi^{(5)}_{1}]$. The ring $J_{0,*}$ has an ideal
\be
J_{0,*}(q)=\bigoplus_{m>1}J_{0,m-1}(q)= \left\{ \phi \in J_{0,*} \mid \phi(\tau,z)= \sum_{\substack{n, r \in \ZZ\\n >0}} c(n,r) q^n y^r \right\}
\ee
consisting of Jacobi forms that vanish in the limit as $\t\to i\inf$ (i.e. have vanishing coefficient of $q^0y^r$, for all $r$, in their Fourier expansion). This ideal is principal and generated by a weak Jacobi
form of weight $0$ and index $6$ given by
\be
\zeta(\tau,z)= \frac{\theta_1(\tau,z)^{12}}{\eta(\tau)^{12}}
\ee
(cf. \S\ref{sec:JacTheta} for $\theta_1$ and $\eta$). Gritsenko shows \cite{Gri_EllGenCYMnflds} that for any positive integer $m$ the quotient $J_{0,m-1}/J_{0,m-1}(q)$ is a vector space of dimension $m-1$ admitting a basis consisting of weight $0$ index $m-1$ weak Jacobi forms $\varphi^{(m)}_{n}$ (denoted $\psi^{(n)}_{0,m-1}$ in \cite{Gri_EllGenCYMnflds}) for $1\leq n\leq m-1$ such that the coefficient of $q^0y^k$ in $\varphi^{(m)}_{n}$ vanishes for $|k|>n$ but does not vanish for $|k|= n$. In fact Gritsenko works in the subring $J^{\ZZ}_{0,*}$ of Jacobi forms having integer Fourier coefficients and his $\varphi^{(m)}_{n}$ furnish a $\ZZ$-basis for the $\ZZ$-module $J^{\ZZ}_{0,*}/J^{\ZZ}_{0,*}(q)$. In what follows we record explicit formulas for some of the $\varphi^{(m)}_{n}$.

Following \cite{Gri_EllGenCYMnflds} we define
\begin{gather}
\begin{split}
\varphi^{(7)}_{1} &= \varphi^{(3)}_{1} \varphi^{(5)}_{1} - (\varphi^{(4)}_{1})^2, \\
\varphi^{(9)}_{1} &= \varphi^{(3)}_{1} \varphi^{(7)}_{1} - (\varphi^{(5)}_{1})^2 , \\
\varphi^{(13)}_{1} &= \varphi^{(5)}_{1} \varphi^{(9)}_{1}- 2(\varphi^{(7)}_{1})^2,
\end{split}
\end{gather}
and define $\varphi^{(m)}_{1}$ for the remaining positive integers $m$ according to the following recursive procedure. For $(12,m-1)=1$ and $ m > 5$ we set
\be
\varphi^{(m)}_{1}=(12,m-5) \varphi^{(m-4)}_{1} \varphi^{(5)}_{1}+(12,m-3) \varphi^{(m-2)}_{1} \varphi^{(3)}_{1}-2(12,m-4) \varphi^{(m-3)}_{1} \varphi^{(4)}_{1}.
\ee
For $(12,m-1)=2$ and $m > 10$ we set
\be
\varphi^{(m)}_{1}= \frac{1}{2} \bigl((12,m-5) \varphi^{(m-4)}_{1} \varphi^{(5)}_{1}+(12,m-3) \varphi^{(m-2)}_{1} \varphi^{(3)}_{1}-2(12,m-4) \varphi^{(m-3)}_{1} \varphi^{(4)}_{1} \bigr).
\ee
For $(12,m-1)=3$ and $m > 9$ we set
\be
\varphi^{(m)}_{1}= \frac{2}{3} (12,m-4) \varphi^{(m-3)}_{1} \varphi^{(4)}_{1} + \frac{1}{3} (12,m-7) \varphi^{(m-6)}_{1} \varphi^{(7)}_{1}-(12,m-5) \varphi^{(m-4)}_{1} \varphi^{(5)}_{1}.
\ee
For $(12,m-1)=4$ and $m > 16$ we set
\be
\varphi^{(m)}_{1}= \frac{1}{4} \bigl( (12,m-13) \varphi^{(m-12)}_{1} \varphi^{(13)}_{1}+(12,m-5) \varphi^{(m-4)}_{1} \varphi^{(5)}_{1}-(12,m-9) \varphi^{(m-8)}_{1} \varphi^{(9)}_{1} \bigr).
\ee
For $(12,m-1)=6$ and $m> 18$ we set
\be
\varphi^{(m)}_{1}=  \frac{1}{3} (12,m-4) \varphi^{(m-3)}_{1} \varphi^{(4)}_{1} + \frac{1}{6} (12,m-7) \varphi^{(m-6)}_{1} \varphi^{(7)}_{1}-\frac{1}{2} (12,m-5) \varphi^{(m-4)}_{1} \varphi^{(5)}_{1}.
\ee
Finally, for $(12,m-1)=12$ and $m> 24$ we set
\be
\varphi^{(m)}_{1}= \frac{1}{6}(12,m-4) \varphi^{(m-3)}_{1} \varphi^{(4)}_{1}-\frac{1}{4}(12,m-5) \varphi^{(m-4)}_{1} \varphi^{(5)}_{1} + \frac{1}{12}(12,m-7) \varphi^{(m-6)}_{1}\varphi^{(7)}_{1},
\ee
where we have set 
\be
{\varphi}^{(25)}_{1} =  \frac{1}{2} \varphi^{(21)}_{1} \varphi^{(5)}_{1}- \varphi^{(19)}_{1} \varphi^{(7)}_{1} + \frac{1}{2} (\varphi^{(13)}_{1})^2.
\ee

The $\varphi^{(m)}_{2}$ are defined by setting
\begin{gather}
	\begin{split}
\varphi^{(3)}_{2} &= (\varphi^{(2)}_{1})^2-24\, \varphi^{(3)}_{1},\\
\varphi^{(4)}_{2}  &= \varphi^{(2)}_{1} \varphi^{(3)}_{1} - 18 \,\varphi^{(4)}_{1},\\
\varphi^{(5)}_{2} &= \varphi^{(2)}_{1} \varphi^{(4)}_{1} - 16\, \varphi^{(5)}_{1},
	\end{split}
\end{gather}
and
\begin{gather}
\varphi^{(m)}_{2} = (12,m-4)\, \varphi^{(m-3)}_{1} \varphi^{(4)}_{1} - (12,m-5) \varphi^{(m-4)}_{1} \varphi^{(5)}_{1} - (12,m-1) \varphi^{(m)}_{1}
\end{gather}
for $m>5$, and the remaining $\varphi^{(m)}_n$ for $2\leq m\leq 25$ are given by
\begin{gather}\label{eqn:forms:wtzero:varphimn}
	\begin{split}
\varphi^{(m)}_{n} &= \varphi^{(m-3)}_{n-1} \varphi^{(4)}_{1},\\
\varphi^{(m)}_{m-2}&= (\varphi^{(2)}_{1})^{m-3} \varphi^{(3)}_{1}, \\
\varphi^{(m)}_{m-1} &= (\varphi^{(2)}_{1})^{m-1},  
	\end{split}
\end{gather}
where the first equation of (\ref{eqn:forms:wtzero:varphimn}) holds for $3\le n\le m-3$.

\newpage

%------------------------------------------------------------------------%
\section{Characters}\label{sec:chars}
%------------------------------------------------------------------------%

In \S\ref{sec:chars:irr} we give character tables (with power maps and Frobenius--Schur indicators) for each group $G^{\rs}$ for $X$ a Niemeier root system. These were computed with the aid of the computer algebra package GAP4 \cite{GAP4}. We use the abbreviations $a_n=\sqrt{-n}$ and $b_n=(-1+\sqrt{-n})/2$ in these tables.

The tables in \S\ref{sec:chars:eul} furnish cycle shapes and character values---the twisted Euler characters---attached to the representations of the groups $G^{\rs}$ described in \S\ref{sec:holes:gps}. Using this data we can obtain explicit expressions for the shadows $S^X_g$ of the vector-valued mock modular forms $H^{\rs}_g$ according to the prescription of \S\ref{sec:mckay:aut}.

\begin{sidewaystable}
%---------------------------------------------------------------------------------------%
\subsection{Irreducible Characters}\label{sec:chars:irr}
%---------------------------------------------------------------------------------------%
\begin{center}
\caption{Character table of $G^{\rs}\simeq M_{24}$, $\rs=A_1^{24}$}\label{tab:chars:irr:2}
\smallskip
\begin{small}
\begin{tabular}{c@{ }|c|@{ }r@{ }r@{ }r@{ }r@{ }r@{ }r@{ }r@{ }r@{ }r@{ }r@{ }r@{ }r@{ }r@{ }r@{ }r@{ }r@{ }r@{ }r@{ }r@{ }r@{ }r@{ }r@{ }r@{ }r@{ }r@{ }r} \toprule
$[g]$	&{FS}&1A	&2A	&2B	&3A	&3B	&4A	&4B	&4C	&5A	&6A	&6B	&7A	&7B	&8A	&10A	&11A	&12A	&12B	&14A	&14B	&15A	&15B	&21A	&21B	&23A	&23B	\\
	\midrule
$[g^{2}]$	&&1A	&1A	&1A	&3A	&3B	&2A	&2A	&2B	&5A	&3A	&3B	&7A	&7B	&4B	&5A	&11A	&6A	&6B	&7A	&7B	&15A	&15B	&21A	&21B	&23A	&23B	\\
$[g^{3}]$	&&1A	&2A	&2B	&1A	&1A	&4A	&4B	&4C	&5A	&2A	&2B	&7B	&7A	&8A	&10A	&11A	&4A	&4C	&14B	&14A	&5A	&5A	&7B	&7A	&23A	&23B	\\
$[g^{5}]$	&&1A	&2A	&2B	&3A	&3B	&4A	&4B	&4C	&1A	&6A	&6B	&7B	&7A	&8A	&2B	&11A	&12A	&12B	&14B	&14A	&3A	&3A	&21B	&21A	&23B	&23A	\\
$[g^{7}]$	&&1A	&2A	&2B	&3A	&3B	&4A	&4B	&4C	&5A	&6A	&6B	&1A	&1A	&8A	&10A	&11A	&12A	&12B	&2A	&2A	&15B	&15A	&3B	&3B	&23B	&23A	\\
$[g^{11}]$	&&1A	&2A	&2B	&3A	&3B	&4A	&4B	&4C	&5A	&6A	&6B	&7A	&7B	&8A	&10A	&1A	&12A	&12B	&14A	&14B	&15B	&15A	&21A	&21B	&23B	&23A	\\
$[g^{23}]$	&&1A	&2A	&2B	&3A	&3B	&4A	&4B	&4C	&5A	&6A	&6B	&7A	&7B	&8A	&10A	&11A	&12A	&12B	&14A	&14B	&15A	&15B	&21A	&21B	&1A	&1A	\\
	\midrule
$\chi_{1}$	&$+$&$1$	&$1$	&$1$	&$1$	&$1$	&$1$	&$1$	&$1$	&$1$	&$1$	&$1$	&$1$	&$1$	&$1$	&$1$	&$1$	&$1$	&$1$	&$1$	&$1$	&$1$	&$1$	&$1$	&$1$	&$1$	&$1$	\\
\bottomrule
\end{tabular}
\end{small}
\end{center}
\end{sidewaystable}

\begin{sidewaystable}

\begin{center}
\caption{Character table of $G^{\rs}\simeq 2.M_{12}$, $\rs=A_2^{12}$}\label{tab:chars:irr:3}

\smallskip

\begin{tabular}{c@{ }|c|@{\;}r@{ }r@{ }r@{ }r@{ }r@{ }r@{ }r@{ }r@{ }r@{ }r@{ }r@{ }r@{ }r@{ }r@{ }r@{ }r@{ }r@{ }r@{ }r@{ }r@{ }r@{ }r@{ }r@{ }r@{ }r@{ }r}\toprule
$[g]$&FS&   	1A&   	2A&   	4A&   	2B&   	2C&   	3A&   	6A&   	3B&   	6B&   	4B&   	4C&   	5A&   	10A&   	12A&   	6C&   	6D&   	8A&   	8B&   	 8C&   8D&   	20A&   	20B&   	11A&   	22A&   	11B&   	22B\\ 
	\midrule
$[g^2]$&&	1A&			1A&			2A&			1A&			1A&			3A&			3A&			3B&			3B&			2B&			2B&			5A&			5A&			6B&			3A&			3A&			4B&			4B&			4C&			4C&			10A&		10A&		11B&		11B&		11A&		11A\\		
$[g^3]$& &  1A&   		2A&   		4A&   		2B&   		2C&   		1A&   		2A&   		1A&   		2A&   		4B& 	  		4C&   		5A&   		10A&   		4A&   		2B&   		2C&   		8A&   		8B&   	 	8C&  		8D&   		20A&   		20B&   		11A&   		22A&   		11B&   		22B\\
$[g^5]$& &  1A&   		2A&   		4A&   		2B&   		2C&   		3A&   		6A&   		3B&   		6B&   		4B& 	  		4C&   		1A&   		2A&   		12A&   		6C&   		6D&   		8B&   		8A&   	 	8D&  		8C&   		4A&   		4A&   		11A&   		22A&   		11B&   		22B\\
$[g^{11}]$&&   1A&   		2A&   		4A&   		2B&   		2C&   		3A&   		6A&   		3B&   		6B&   		4B& 	  		4C&   		5A&   		10A&   		12A&   		6C&   		6D&   		8A&   		8B&   	 	8C&  		8D&   		20B&   		20A&   		1A&   		2A&   		1A&   		2A\\
	\midrule
$\chi_{1}$&$+$&   $1$&   $1$&   $1$&   $1$&   $1$&   $1$&   $1$&   $1$&   $1$&   $
1$&   $1$&   $1$&   $1$&   $1$&   $1$&   $1$&   $1$&   $1$&   $1$&   $1$&   $
1$&   $1$&   $1$&   $1$&   $1$&   $1$\\
$\chi_{2}$&$+$&   $11$&   $11$&   $-1$&   $3$&   $3$&   $2$&   $2$&   $-1$&   $
-1$&   $-1$&   $3$&   $1$&   $1$&   $-1$&   $0$&   $0$&   $-1$&   $-1$&   $
1$&   $1$&   $-1$&   $-1$&   $0$&   $0$&   $0$&   $0$\\
$\chi_{3}$&$+$&   $11$&   $11$&   $-1$&   $3$&   $3$&   $2$&   $2$&   $-1$&   $
-1$&   $3$&   $-1$&   $1$&   $1$&   $-1$&   $0$&   $0$&   $1$&   $1$&   $
-1$&   $-1$&   $-1$&   $-1$&   $0$&   $0$&   $0$&   $0$\\
\bottomrule
\end{tabular}
\end{center}

\end{sidewaystable}

\begin{table}
\begin{center}
\caption{Character table of ${G}^{\rs}\simeq 2.\AGL_3(2)$, $\rs=A_3^8$}\label{tab:chars:irr:4}

\smallskip
\begin{small}
\begin{tabular}{c|c|rrrrrrrrrrrrrrrrrrrrrrrrrrr} \toprule
$[g]$&FS&	   1A&   2A&   2B&   4A&   4B&   2C&   3A&   6A&   6B&   6C&   8A&   4C&   7A&   14A&   7B&   14B
	\\ \midrule
$[g^2]$&&   	1A&   1A&   	1A&   	2A&			2B&			1A&   	3A&   	3A&   		3A&   	3A&   	4A&   	2C&   	7A&		7A&   7B&   7B\\ 
$[g^3]$&&   	1A&   2A&   	2B&   	4A&			4B&			2C&   	1A&   	2A&   		2B&   	2B&   	8A&   	4C&   	7B&   	14B&   7A&   14A\\ 
$[g^7]$&&   	1A&   2A&   	2B&   	4A& 			4B&  		2C&   	3A&   	6A&   		6B&   	6C&   	8A&   	4C&   	1A&   	2A&   1A&   2A\\ 
	\midrule	
$\chi_{1}$&$+$&   $1$&   $1$&   $1$&   $1$&   $1$&   $1$&   $1$&   $1$&   $1$&   $
1$&   $1$&   $1$&   $1$&   $1$&   $1$&   $1$\\
$\chi_{2}$&$\circ$&   $3$&   $3$&   $3$&   $-1$&   $-1$&   $-1$&   $0$&   $0$&   $
0$&   $0$&   $1$&   $1$&   $b_{7}$&   $b_{7}$&   $\overline{b_{7}}$&   $\overline{b_{7}}$\\
$\chi_{3}$&$\circ$&   $3$&   $3$&   $3$&   $-1$&   $-1$&   $-1$&   $0$&   $0$&   $
0$&   $0$&   $1$&   $1$&   $\overline{b_{7}}$&   $\overline{b_{7}}$&   $b_{7}$&   $b_{7}$\\
$\chi_{4}$&$+$&   $6$&   $6$&   $6$&   $2$&   $2$&   $2$&   $0$&   $0$&   $0$&   $
0$&   $0$&   $0$&   $-1$&   $-1$&   $-1$&   $-1$\\
\bottomrule
\end{tabular}
\end{small}
\end{center}

\begin{center}
\caption{Character table of ${G}^{\rs}\simeq \GL_2(5)/2$, $\rs=A_4^6$}\label{tab:chars:irr:5}

\smallskip
\begin{small}
\begin{tabular}{c|c|rrrrrrrrrrrrrrrrrrrrrrrrrrr}\toprule
$[g]$&FS&   1A&   2A&   2B&   2C&   3A&   6A&   5A&   10A&   4A&   4B&   4C&   4D&   12A&   12B\\
	\midrule
$[g^2]$&&	1A&	1A&	1A&	1A&	3A&	3A&	5A&	5A&		2A&	2A&	2C&	2C&	6A&	6A\\
$[g^3]$&&	1A&	2A&	2B&	2C&	1A&	2A&	5A&	10A&	4B&	4A&	4D&	4C&	4B&	4A\\
$[g^5]$&&	1A&	2A&	2B&	2C&	3A&	6A&	1A&	2A&		4A&	4B&	4C&	4D&	12A&	12B\\
	\midrule
${\chi}_{1}$&$+$&   $1$&   $1$&   $1$&   $1$&   $1$&   $1$&   $1$&   $1$&   $1$&   $1$&   $1$&   $1$&   $1$&   $1$\\
${\chi}_{2}$&$+$&   $1$&   $1$&   $1$&   $1$&   $1$&   $1$&   $1$&   $1$&   $
-1$&   $-1$&   $-1$&   $-1$&   $-1$&   $-1$\\
${\chi}_{3}$&$+$&   $4$&   $4$&   $0$&   $0$&   $1$&   $1$&   $-1$&   $-1$&   $
2$&   $2$&   $0$&   $0$&   $-1$&   $-1$\\
${\chi}_{4}$&$+$&   $4$&   $4$&   $0$&   $0$&   $1$&   $1$&   $-1$&   $-1$&   $
-2$&   $-2$&   $0$&   $0$&   $1$&   $1$\\
\bottomrule
\end{tabular}
\end{small}
\end{center}
\end{table}

\begin{table}
\begin{center}
\caption{Character table of ${G}^{\rs}\simeq %G^{(12+4)}\simeq
	\GL_2(3)$, $\rs\in\{A_5^4D_4, E_6^4\}$}\label{tab:chars:irr:6}
\smallskip
\begin{tabular}{c|c|rrrrrrrr}\toprule
$[g]$	&FS&   1A&   2A&   2B&   4A&   3A&   6A&   8A&   8B\\
	\midrule
$[g^2]$ 	&&1A&	1A&	1A&	2A&	3A&	3A&	4A&	4A\\
$[g^3]$	&&1A&	2A&	2B&	4A&	1A&	2A&	8A&	8B\\
	\midrule
${\chi}_{1}$&$+$&   $1$&   $1$&   $1$&   $1$&   $1$&   $1$&   $1$&   $1$\\
${\chi}_{2}$&$+$&   $1$&   $1$&   $-1$&   $1$&   $1$&   $1$&   $-1$&   $-1$\\
${\chi}_{3}$&$+$&   $2$&   $2$&   $0$&   $2$&   $-1$&   $-1$&   $0$&   $0$\\
${\chi}_{4}$&$+$&   $3$&   $3$&   $-1$&   $-1$&   $0$&   $0$&   $1$&   $1$\\
${\chi}_{5}$&$+$&   $3$&   $3$&   $1$&   $-1$&   $0$&   $0$&   $-1$&   $-1$\\
${\chi}_{6}$&$\circ$&  $2$&   $-2$&   $0$&   $0$&   $-1$&   $1$&   $a_2$&   $\overline{a_2}$\\
${\chi}_{7}$&$\circ$&   $2$&   $-2$&   $0$&   $0$&   $-1$&   $1$&   $\overline{a_2}$&   $a_2$\\
${\chi}_{8}$&$+$&   $4$&   $-4$&   $0$&   $0$&   $1$&   $-1$&   $0$&   $0$\\
	\bottomrule
\end{tabular}
\end{center}
\end{table}

\begin{table}
\begin{center}
\caption{Character table of ${G}^{\rs}\simeq 3.\Sym_6$, $\rs=D_4^6$}\label{tab:chars:irr:6+3}
\smallskip
\begin{small}
\begin{tabular}{c|c|rrrrrrrrrrrrrrrrrrrrrrrrrrr}\toprule
$[g]$&FS&   1A&   3A&   2A&   6A&   3B&   3C&   4A&   12A&   5A&   15A&   15B&   2B&   2C&   4B&	6B&	6C\\
	\midrule
$[g^2]$&&	1A&	3A&	1A&3A&	3B&3C&2A&6A&5A&15A&15B&1A&1A&2A&3B&3C\\
$[g^3]$&&	1A&	1A&	2A&2A&	1A&1A&4A&4A&5A&5A&5A&2B&2C&4B&2B&2C\\
$[g^5]$&&	1A&	3A&	2A&	6A&3B&3C&4A&12A&1A&3A&3A&2B&2C&4B&6B&6C\\
	\midrule
${\chi}_{1}$&$+$&   $1$&   $1$&   $1$&   $1$&   $1$&   $1$&   $1$&   $1$&   $1$&   $1$&   $1$&   $1$&   $1$&   $1$&   $1$&   $1$\\
${\chi}_{2}$&$+$&   $1$&   $1$&   $1$&   $1$&   $1$&   $1$&   $1$&   $1$&   $1$&   $1$&   $1$&   $-1$&   $-1$&   $-1$&   $-1$&   $-1$\\
${\chi}_{3}$&$+$&   $5$&   $5$&   $1$&   $1$&   $2$&   $-1$&   $-1$&   $-1$&   $0$&   $0$&   $0$&   $3$&   $-1$&   $1$&   $0$&   $-1$\\
${\chi}_{4}$&$+$&   $5$&   $5$&   $1$&   $1$&   $2$&   $-1$&   $-1$&   $-1$&   $0$&   $0$&   $0$&   $-3$&   $1$&   $-1$&   $0$&   $1$\\
\bottomrule
\end{tabular}
\end{small}
\end{center}
\end{table}

\begin{table}
\begin{center}
\caption{Character table of ${G}^{\rs}\simeq\SL_2(3)$, $\rs=A_6^4$}\label{tab:chars:irr:7}
\smallskip
\begin{tabular}{c|c|rrrrrrr}\toprule
$[g]$	&FS&   1A&   2A&   4A&   3A&   6A&   3B&   6B\\
	\midrule
$[g^2]$ 	&&1A	&	1A&		2A&		3B&		3A&		3A&		3B\\
$[g^3]$	&&1A	&	2A&		4A&		1A&		2A&		1A&		2A\\
	\midrule
$\chi_1$&$+$&   $1$&   $1$&   $1$&   $1$&   $1$&   $1$&   $1$\\
$\chi_2$&$\circ$&   $1$&   $1$&   $1$&   ${b_3}$&   $\overline{b_3}$&   $\overline{b_3}$&   ${b_3}$\\
$\chi_3$&$\circ$&   $1$&   $1$&   $1$&   $\overline{b_3}$&   ${b_3}$&   ${b_3}$&   $\overline{b_3}$\\
$\chi_4$&$+$&   $3$&   $3$&   $-1$&   $0$&   $0$&   $0$&   $0$\\
$\chi_5$&$-$&   $2$&   $-2$&   $0$&   $-1$&   $1$&   $-1$&   $1$\\
$\chi_6$&$\circ$&   $2$&   $-2$&   $0$&   $-\overline{b_3}$&   ${b_3}$&   $-{b_3}$&   $\overline{b_3}$\\
$\chi_7$&$\circ$&   $2$&   $-2$&   $0$&   $-{b_3}$&   $\overline{b_3}$&   $-\overline{b_3}$&   ${b_3}$\\\bottomrule
\end{tabular}
\end{center}
\end{table}

\begin{table}
\begin{center}
\caption{Character table of ${G}^{\rs}\simeq \Dih_4$, $\rs=A_7^2D_5^2$}\label{tab:chars:irr:8}
\smallskip
\begin{tabular}{c|c|rrrrr}\toprule
$[g]$&FS&   1A&   2A&   2B&   2C&   4A\\
	\midrule
$[g^2]$&&	1A&	1A&	1A&	1A&	2A\\
	\midrule
$\chi_1$&$+$&   $1$&   $1$&   $1$&   $1$&   $1$\\
$\chi_2$&$+$&   $1$&   $1$&   $-1$&   $-1$&   $1$\\
$\chi_3$&$+$&   $1$&   $1$&   $-1$&   $1$&   $-1$\\
$\chi_4$&$+$&   $1$&   $1$&   $1$&   $-1$&   $-1$\\
$\chi_5$&$+$&   $2$&   $-2$&   $0$&   $0$&   $0$\\\bottomrule
\end{tabular}
\end{center}
\end{table}

\begin{table}
\begin{center}
\caption{Character table of ${G}^{\rs}\simeq \Dih_6$, $\rs=A_8^3$}\label{tab:chars:irr:9}
\smallskip
\begin{tabular}{c|c|rrrrrr}\toprule
$[g]$&FS&   1A&   2A&   2B&   2C&   3A&    6A\\
	\midrule
$[g^2]$&&	1A&	1A&	1A&	1A&	3A&	3A\\
$[g^3]$&&	1A&	2A&	2B&	2C&	1A&	2A\\
	\midrule
$\chi_1$&$+$&   $1$&   $1$&   $1$&   $1$&   $1$&   $1$\\
$\chi_2$&$+$&   $1$&   $1$&   $-1$&   $-1$&   $1$&   $1$\\
$\chi_3$&$+$&   $2$&   $2$&   $0$&   $0$&   $-1$&   $-1$\\
$\chi_4$&$+$&   $1$&   $-1$&   $-1$&   $1$&   $1$&   $-1$\\
$\chi_5$&$+$&   $1$&   $-1$&   $1$&   $-1$&   $1$&   $-1$\\
$\chi_6$&$+$&   $2$&   $-2$&   $0$&   $0$&   $-1$&   $1$\\\bottomrule
\end{tabular}
\end{center}
\end{table}

\begin{table}
\begin{center}
\caption{Character table of ${G}^{\rs}\simeq 4$, for $\rs\in\{A_9^2D_{6},A_{12}^2\}$}\label{tab:chars:irr:13}
\smallskip
\begin{tabular}{c|c|rrrr}\toprule
$[g]$&FS&   1A&   2A&   4A&   4B\\
	\midrule
$[g^2]$&&	1A&	1A&	2A&	2A\\
	\midrule
$\chi_1$&$+$&   $1$&   $1$&   $1$&   $1$\\
$\chi_2$&$+$&   $1$&   $1$&   $-1$&   $-1$\\
$\chi_3$&$\circ$&   $1$&   $-1$&   $a_1$&   $\overline{a_1}$\\
$\chi_4$&$\circ$&   $1$&   $-1$&   $\overline{a_1}$&   $a_1$\\\bottomrule
\end{tabular}
\end{center}
\end{table}

\begin{table}
\begin{center}
\caption{Character table of ${G}^{\rs}\simeq 	\PGL_2(3)\simeq \Sym_4$, $\rs=D_{6}^4$}\label{tab:chars:irr:10+5}
\smallskip
\begin{tabular}{c|c|rrrrrrrr}\toprule
$[g]$	&FS&   1A& 2A&   3A&  2B&    4A\\
	\midrule
$[g^2]$ 	&&1A&	1A&	3A&	1A&2A\\
$[g^3]$	&&1A&	2A&	1A&	2B&	4A\\
	\midrule
${\chi}_{1}$&$+$&   $1$&   $1$&   $1$&  $1$&   $1$\\
${\chi}_{2}$&$+$&   $1$&   $1$&   $1$&   $-1$&   $-1$\\
${\chi}_{3}$&$+$&   $2$&   $2$&   $-1$&   $0$&   $0$\\
${\chi}_{4}$&$+$&   $3$&   $-1$&   $0$&   $1$&   $-1$\\
${\chi}_{5}$&$+$&   $3$&   $-1$&   $0$&   $-1$&   $1$\\
	\bottomrule
\end{tabular}
\end{center}
\end{table}

\begin{table}
\centering
\caption{Character table of ${G}^{\rs}\simeq 2$, for $\rs\in\{A_{11}D_7E_6, A_{15}D_9, A_{17}E_7, A_{24}, D_{10}E_7^2, D_{12}^2\}$}\label{tab:chars:irr:25}
\smallskip
\begin{tabular}{c|c|rr}\toprule
$[g]$&FS&   1A&   2A\\
	\midrule
$[g^2]$&&	1A&	1A\\
	\midrule
$\chi_1$&$+$&   $1$&   $1$\\
$\chi_2$&$+$&   $1$&   $-1$\\\bottomrule
\end{tabular}
\end{table}

\begin{table}
\begin{center}
\caption{Character table of ${G}^{\rs}\simeq \Sym_3$, $\rs\in\{D_8^3, E_8^3\}$}\label{tab:chars:irr:14+7}
\smallskip
\begin{tabular}{c|c|rrrrrr}\toprule
$[g]$&FS&   1A&   2A&   3A\\
	\midrule
$[g^2]$&&	1A&	1A&	3A\\
$[g^3]$&&	1A&	2A&	1A\\
	\midrule
$\chi_1$&$+$&   $1$&   $1$&   $1$\\
$\chi_2$&$+$&   $1$&   $-1$&  $1$\\
$\chi_3$&$+$&   $2$&   $0$&   $-1$\\
\bottomrule
\end{tabular}
\end{center}
\end{table}

\clearpage

\begin{sidewaystable}
%---------------------------------------------------------------------------------------%
\subsection{Euler Characters}\label{sec:chars:eul}
%---------------------------------------------------------------------------------------%

The tables in this section describe the twisted Euler characters and associated cycle shapes attached to each group $G^{\rs}$ in \S\ref{sec:holes:gps}. According to the prescription of \S\ref{sec:mckay:aut} the character values $\bar{\chi}^{\rs_A}_g$, $\chi^{\rs_A}_g$, \&c., can be used to describe the shadows of the vector-valued mock modular forms $H^{\rs}_g$ attached to each $g\in G^{\rs}$ by umbral moonshine. We also identify symbols $n_g|h_g$ which are used in \S\ref{sec:conj:aut} to formulate conjectures about the modularity of $H^{\rs}_g$. By definition $n_g$ is the order of the image of $g\in G^{\rs}$ in $\bar{G}^{\rs}$ and $h_g=N_g/n_g$ where $N_g$ denotes the product of shortest and longest cycle lengths appearing in the cycle shape $\widetilde{\Pi}^{\rs}_g$. 

Note that we have $\widetilde{\Pi}^{\rs}_g=\widetilde{\Pi}^{\rs_A}_g=\bar{\Pi}^{\rs_A}_g$ in case $\rs=A_1^{24}$. More generally, we will have $\widetilde{\Pi}^{\rs}_g=\widetilde{\Pi}^{\rs_A}_g$ when $X=X_A$ (cf. \S\ref{sec:holes:lats}), and similarly when $X=X_D$ or $X=X_E$, so we suppress the row (that would otherwise be) labelled $\widetilde{\Pi}^{\rs}_g$ in these cases.

\begin{center}
\caption{Twisted Euler characters and Frame shapes at $\ll=2$, $\rs=A_1^{24}$}\label{tab:chars:eul:2}
\smallskip
\begin{tabular}{l@{ }|@{ }r@{ }r@{ }r@{ }r@{ }r@{ }r@{ }r@{ }r@{ }r@{ }r@{ }r@{ }r@{ }r@{ }r@{ }r}
\toprule
$[g]$	&1A	&2A	&2B	&3A	&3B	&4A	&4B	&4C	&5A	&6A	&6B	\\
	\midrule
$n_g|h_g$&$1|1$&$2|1$&${2|2}$&$3|1$&$3|3$&$4|2$&$4|1$&${4|4}$&$5|1$&$6|1$&$6|6$&\\
	\midrule
$\bar{\chi}^{\rs_A}_{g}$&     
	$24$&   $8$&   		$0$&   		$6$&   		$0$&   		$0$&   		$4$&  		$0$&   		$4$&   		$2$&   		$0$&  \\
	\midrule
$\bar{\Pi}^{\rs_A}_{g}$&
	$1^{24}$&	$1^{8}2^8$&	$2^{12}$&		$1^63^6$&	$3^8$&	$2^44^4$&	$1^42^24^4$&	$4^6$&		$1^45^4$&		$1^22^23^26^2$&	$6^4$\\\midrule 
\midrule
$[g]$	& 7AB	&8A	&10A	&11A&12A	&12B	&14AB	&15AB	&21AB	&23AB	\\
	\midrule
$n_g|h_g$&$7|1$&$8|1$&$10|2$&$11|1$&$12|2$&$12|12$&$14|1$&$15|1$&$21|3$&$23|1$\\
	\midrule
$\bar{\chi}^{\rs_A}_{g}$&
	$3$&   		$2$&   		$0$&   		$2$&   		$0$&   		$0$& 		$1$&   		$1$&   		$0$&   		$1$\\
	\midrule
$\bar{\Pi}^{\rs_A}_{g}$&	
	$1^37^3$&	$1^22^14^18^2$&	$2^210^2$&		$1^211^2$&$2^14^16^112^1$&$12^2$&	$1^12^17^114^1$&	$1^13^15^115^1$&$3^121^1$&	$1^123^1$\\\bottomrule
\end{tabular}
\end{center}

\begin{center}
\caption{Twisted Euler characters and Frame shapes at $\ll=3$, $\rs=A_2^{12}$}\label{tab:chars:eul:3}
\smallskip
\begin{tabular}{l@{ }|@{\;}r@{\,}r@{\,}r@{\,}r@{\,}r@{\,}r@{\,}r@{\,}r@{\,}r@{\,}r@{\,}r@{\,}r@{\,}r@{\,}r@{\,}r@{\,}r@{\,}r@{\,}r@{\,}r@{\,}r@{\,}r@{\,}r@{\,}r@{\,}r@{\,}r@{\,}r}\toprule
$[g]$&   	1A&   		2A&   		4A&   		2B&   		2C&   		3A&   		6A&   		3B&   		6B&   		4B& 	  		4C&   		5A&   		10A&   		12A&   		6C&   		6D&   		8AB&   	 	8CD&   		20AB&   		11AB&   		22AB\\ 
	\midrule
$n_g|h_g$&$1|1$&$1|4$&${2|8}$&$2|1$&$2|2$&$3|1$&$3|4$&${3|3}$&${3|12}$&$4|2$&$4|1$&$5|1$&$5|4$&$6|24$&$6|1$&$6|2$&$8|4$&$8|1$&${10|8}$&$11|1$&$11|4$\\
	\midrule
$\bar{\chi}^{\rs_A}_{g}$&   $12$&   $12$&   		$0$&   		$4$&   		$4$&   		$3$&   		$3$&  		$0$&   		$0$&   		$0$&   		$4$&   		$2$&   		$2$&   		$0$&   		$1$&   		$1$&   		$0$&   		$2$&   		$0$&   		$1$&   		$1$\\
$\chi^{\rs_A}_{g}$&   $12$&   $-12$&   		$0$&   		$4$&   		$-4$&   		$3$&   		$-3$&   		$0$&   		$0$&   		$0$&  		$0$&   		$2$&   		$-2$&   		$0$&   		$1$&   		$-1$&   		$0$&   		$0$&   		$0$&   		$1$&   		$-1$\\
	\midrule
$\bar{\Pi}^{\rs_A}_{g}$&$1^{12}$&	$1^{12}$&	$2^6$&		$1^42^4$&	$1^42^4$&	$1^33^3$&	$1^33^3$&	$3^4$&		$3^4$&		$2^24^2$&	$1^44^2$&	$1^25^2$&	$1^25^2$&	$6^2$&		$1^12^13^16^1$&$1^12^13^16^1$&$4^18^1$&	$1^22^18^1$&	$2^110^1$&$1^111^1$&	$1^111^1$\\
$\widetilde{\Pi}^{\rs_A}_{g}$&$1^{24}$&	${2^{12}}$&${4^6}$&	$1^82^8$&	${2^{12}}$&	$1^63^6$&	${2^36^3}$&$3^8$&	${6^4}$&	$2^44^4$&	$1^42^24^4$&	$1^45^4$&${2^210^2}$&	${12^2}$&		$1^22^23^26^2$&${2^26^2}$&$4^28^2$&$1^22^14^18^2$&${4^120^1}$&$1^211^2$&${2^122^1}$
	\\\bottomrule
\end{tabular}
\smallskip
\end{center}
\end{sidewaystable}

\begin{table}
\begin{center}
\caption{Twisted Euler characters and Frame shapes at $\ll=4$, $\rs=A_3^8$}\label{tab:chars:eul:4}
\smallskip
\begin{tabular}{l@{\, }|@{\;}r@{\, }r@{\, }r@{\, }r@{\, }r@{\, }r@{\, }r@{\, }r@{\, }r@{\, }r@{\, }r@{\, }r@{\, }r}\toprule
$[g]$&   		1A&   2A&   	2B&   	4A&			4B&			2C&   	3A&   	6A&   		6BC&   	8A&   	4C&   	7AB&   	14AB\\ 
	\midrule
$n_g|h_g$&$1|1$& $1|2$&	$2|2$&	$2|4$&			${4|{4}}$&	$2|1$& 	$3|1$& 	$3|2$&		$6|2$&	${4|{8}}$&		$4|1$&  	$7|1$&	$7|2$\\	
	\midrule
$\bar{\chi}^{\rs_A}_g$&   $8$&$8$&	$0$& 	$0$& 		$0$&		$4$&  	$2$& 	$2$&  		$0$& 	$0$& 	$2$& 	$1$& 	$1$\\
$\chi^{\rs_A}_g$&   $8$&$-8$&	$0$&	$0$& 		$0$&		$0$&  	$2$& 	$-2$& 		$0$& 	$0$& 	$0$& 	$1$& 	$-1$\\
	\midrule
$\bar{\Pi}^{\rs_A}_g$&	$1^8$&$1^8$&$2^4$&	$2^4$&		$4^2$&		$1^42^2$&$1^23^2$&$1^23^2$&	$2^16^1$&$4^2$&	$1^22^14^1$&$1^17^1$&$1^17^1$\\
$\widetilde{\Pi}^{\rs_A}_g$&	$1^{24}$&$1^8{2^8}$&$2^{12}$&${2^4}{4^4}$&	$4^6$&		$1^82^8$&	$1^63^6$&${1^22^23^26^2}$&	$2^36^3$&${4^2}{8^2}$&$1^42^24^4$	&$1^37^3$&${1^12^17^114^1}$
\\\bottomrule
\end{tabular}
\smallskip
\end{center}

\end{table}

\begin{table}
\begin{center}
\caption{Twisted Euler characters and Frame shapes at $\ll=5$, $\rs=A_4^6$}\label{tab:chars:eul:5}
\smallskip
\begin{tabular}{l@{\, }|@{\;}r@{\, }r@{\, }r@{\, }r@{\, }r@{\, }r@{\, }r@{\, }r@{\, }r@{\, }r@{\, }r@{\, }r@{\, }r}\toprule
$[g]$&   		1A&		2A&   	2B&   	2C&			3A&			6A&   	5A&   	10A&   		4AB&   	4CD&	12AB\\ 
	\midrule
$n_g|h_g$&		$1|1$&	$1|4$&	$2|2$&	$2|1$&		$3|3$&		$3|12$&	$5|1$&	$5|4$&		$2|8$&		$4|1$&	$6|24$	\\	
	\midrule
$\bar{\chi}^{\rs_A}_{g}$&   $6$&	$6$&	$2$& 	$2$& 		$0$&		$0$&  	$1$& 	$1$&  		$0$& 	$2$& 	$0$ 	\\
$\chi^{\rs_A}_{g}$&   $6$&	$-6$&	$-2$&	$2$& 		$0$&		$0$&  	$1$& 	$-1$& 		$0$&	$0$& 	$0$ 	\\
	\midrule
$\bar{\Pi}^{\rs_A}_{g}$&	$1^6$&	$1^6$&	$1^22^2$&$1^22^2$&	$3^2$&		$3^2$&	$1^15^1$&$1^15^1$&	$2^3$&	$1^24^1$&$6^1$\\
$\widetilde{\Pi}^{\rs_A}_{g}$&	$1^{24}$&	${2^{12}}$&${2^{12}}$&$1^82^8$&$3^8$&${6^4}$&$1^45^4$&${2^210^2}$&${4^6}$&$1^42^24^4$&${12^2}$\\\bottomrule
\end{tabular}
\smallskip
\end{center}
\end{table}

\begin{table}
\begin{center}
\caption{Twisted Euler characters and Frame shapes at $\ll=6$, $\rs=A_5^4D_4$}\label{tab:chars:eul:6}
\smallskip
\begin{tabular}{l|rrrrrrr}\toprule
$[g]$&   		1A&		2A&   	2B&   	4A&			3A&			6A&   	8AB\\ 
	\midrule
$n_g|h_g$&		$1|1$&	$1|2$&	$2|1$&	$2|2$&		$3|1$&		$3|2$&	$4|2$\\	
	\midrule
$\bar{\chi}^{\rs_A}_g$&
			$4$&	$4$&	$2$&	$0$&	$1$&	$1$&	$0$\\
$\chi^{\rs_A}_g$&
			$4$&	$-4$&	$0$&	$0$&	$1$&	$-1$&	$0$\\
			\midrule
$\bar{\Pi}^{\rs_A}_g$&
			$1^{4}$&	$1^4$&	$1^22^1$&	$2^2$&	$1^13^1$&	$1^13^1$&	$4^1$\\
$\widetilde{\Pi}^{\rs_A}_g$&
			$1^{20}$&	$1^42^8$&	$1^62^7$&	$2^24^4$&	$1^53^5$&	$1^12^23^16^2$&	$4^18^2$\\
			\midrule
$\bar{\chi}^{\rs_D}_g$&
			$1$&	$1$&	$1$&	$1$&	$1$&	$1$&	$1$\\
${\chi}^{\rs_D}_g$&
			$1$&	$1$&	$-1$&	$1$&	$1$&	$1$&	$-1$\\
$\check{\chi}^{\rs_D}_g$&
			$2$&	$2$&	$0$&	$2$&	$-1$&	$-1$&	$0$\\
			\midrule
$\bar{\Pi}^{\rs_D}_g$&
			$1^1$&	$1^1$&	$1^1$&	$1^1$&	$1^1$&	$1^1$&	$1^1$\\
$\widetilde{\Pi}^{\rs_D}_g$&
			$1^4$&	$1^4$&	$1^22^1$&	$1^4$&	$1^13^1$&	$1^13^1$&	$1^22^1$\\
	\midrule
$\widetilde{\Pi}_g^X$&
			$1^{24}$&	$1^82^8$&	$1^82^8$&	$1^42^24^4$&	$1^63^6$&	$1^22^23^26^2$&	$1^22^14^18^2$
	\\\bottomrule
\end{tabular}
\smallskip
\end{center}

\end{table}

\begin{table}
\begin{center}
\caption{Twisted Euler characters and Frame shapes at $\ll=6+3$, $\rs=D_4^6$}\label{tab:chars:eul:6+3}
\begin{tabular}{l@{\;}|@{\;}r@{\;}r@{\;}r@{\;}r@{\;}r@{\;}r@{\;}r@{\;}r@{\;}r@{\;}r@{\;}r@{\;}r@{\;}r@{\;}r@{\;}r}\toprule
$[g]$&   		1A&		3A&   	2A&   	6A&			3B&			6C&   	4A&	12A&	5A&	15AB&	2B&	2C&	4B&	6B&	6C\\ 
	\midrule
$n_g|h_g$&		$1|1$&	$1|3$&	$2|1$&	$2|3$&		$3|1$&		$3|3$&	$4|2$&	$4|6$&	$5|1$&	$5|3$&	$2|1$&	$2|2$&	$4|1$&	$6|1$&	$6|6$\\	
	\midrule
$\bar{\chi}^{\rs_D}_g$&
			$6$&	$6$&	$2$&	$2$&	$3$&	$0$&	$0$&	$0$&	$1$&	$1$&	$4$&	$0$&	$2$&	$1$&	$0$\\
${\chi}^{\rs_D}_g$&
			$6$&	$6$&	$2$&	$2$&	$3$&	$0$&	$0$&	$0$&	$1$&	$1$&	$-4$&	$0$&	$-2$&	$-1$&	$0$\\
$\check{\chi}^{\rs_D}_g$&
			$12$&	$-6$&	$4$&	$-2$&	$0$&	$0$&	$0$&	$0$&	$2$&	$-1$&	$0$&	$0$&	$0$&	$0$&	$0$\\
			\midrule
$\bar{\Pi}^{\rs_D}_g$&
			$1^6$&	$1^6$&	$1^22^2$&	$1^22^2$&	$1^33^1$&	$3^2$&	$2^14^1$&	$2^14^1$&	$1^15^1$&	$1^15^1$&	$1^42^1$&	$2^3$&$1^24^1$&$1^12^13^1$&	$6^1$\\
$\widetilde{\Pi}^{\rs_D}_g$&
			$1^{24}$&	$1^63^6$&	$1^82^8$&	$1^22^23^26^2$&	$1^63^6$&	$3^8$&	$2^44^4$&$2^14^16^112^1$&	$1^45^4$&$1^13^15^115^1$&$1^22^8$&$2^{12}$&$1^42^24^4$&$1^22^23^26^2$&$6^4$
	\\\bottomrule
\end{tabular}
\smallskip
\end{center}

\end{table}

\begin{table}
\begin{center}
\caption{\label{tab:FrmG7fp}
Twisted Euler characters and Frame shapes at $\ll=7$, $\rs=A_6^4$}\label{tab:chars:eul:7}
\begin{tabular}{l|rrrrr}\toprule
$[g]$&   1A&   2A&   4A&   3AB&   6AB\\ 
	\midrule
$n_g|h_g$&		$1|1$&	$1|4$&	$2|8$&	$3|1$&	$3|4$\\
	\midrule
$\bar{\chi}^{\rs_A}_{g}$&	4&	4&	0&	1&	1\\
$\chi^{\rs_A}_{g}$&	4&	-4&	0&	1&	-1\\
	\midrule
$\bar{\Pi}^{\rs_A}_{g}$&	$1^4$&	$1^4$&	$2^2$&	$1^13^1$&	$1^13^1$\\
$\widetilde{\Pi}^{\rs_A}_{g}$&	$1^{24}$&	${2^{12}}$&	${4^6}$&	$1^63^6$&	${2^36^3}$
\\
\bottomrule
\end{tabular}
\end{center}

\end{table}

\begin{table}
\begin{center}
\caption{Twisted Euler characters and Frame shapes at $\ll=8$, $\rs=A_7^2D_5^2$}\label{tab:chars:eul:8}
\begin{tabular}{l|rrrrr}\toprule
	$[g]$&	1A&	2A&	2B&2C&4A\\
		\midrule
$n_g|h_g$&		$1|1$&$1|2$&${2|1}$&$2|1$&${2|4}$\\	
		\midrule
$\bar{\chi}^{\rs_A}_g$&	2&2&0&2&0	\\
$\chi^{\rs_A}_g$&	2&-2&0&0&0	\\
		\midrule
$\bar{\Pi}^{\rs_A}_g$&	$1^2$&$1^2$&$2^1$&$1^2$&$2^1$\\
$\widetilde{\Pi}^{\rs_A}_g$&$1^{14}$&$1^22^6$&$2^7$&$1^82^3$&$2^14^3$\\
		\midrule
$\bar{\chi}^{\rs_D}_g$&	2&2&2&0&0	\\
$\chi^{\rs_D}_g$&	2&-2&0&0&0	\\
		\midrule
$\bar{\Pi}^{\rs_D}_g$&	$1^2$&$1^2$&$1^2$&$2^1$&$2^1$\\
$\widetilde{\Pi}^{\rs_D}_g$&	$1^{10}$&$1^62^2$&$1^82^1$&$2^5$&$2^34^1$\\
		\midrule
$\widetilde{\Pi}_g^X$&	$1^{24}$&	$1^82^8$&$1^82^8$&$1^82^8$&$2^44^4$
\\\bottomrule
\end{tabular}
\end{center}

\end{table}

\begin{table}
\begin{center}
\caption{Twisted Euler characters and Frame shapes at $\ll=9$, $\rs=A_8^3$}\label{tab:chars:eul:9}
\begin{tabular}{l|rrrrrr}\toprule
	$[g]$&	1A&	2A&	2B&2C&3A&6A\\
		\midrule
$n_g|h_g$&		$1|1$&$1|4$&${2|1}$&$2|2$&$3|3$&$3|12$\\	
		\midrule
	$\bar{\chi}^{\rs_A}_{g}$	&3&3&1&1&0&0\\
	$\chi^{\rs_A}_{g}$		&3&-3&1&-1&0&0\\
		\midrule
	$\bar{\Pi}^{\rs_A}_{g}$	&$1^3$&$1^3$&$1^12^1$&$1^12^1$&$3^1$&$3^1$\\
	$\widetilde{\Pi}^{\rs_A}_{g}$&$1^{24}$&${2^{12}}$&$1^82^8$&$2^{12}$&$3^8$&${6^4}$%\smallskip
	\\\bottomrule
\end{tabular}
\end{center}
\end{table}

\begin{table}
\begin{center}
\caption{Twisted Euler characters and Frame shapes at $\ll=10$, $\rs=A_9^2D_6$}\label{tab:chars:eul:10}
\begin{tabular}{l|rrr}\toprule
	$[g]$&	1A&	2A&	4AB\\
		\midrule
$n_g|h_g$&		$1|1$&$1|2$&${2|2}$\\	
		\midrule
$\bar{\chi}^{\rs_A}_g$&2&2&0\\
$\chi^{\rs_A}_g$&2&-2&0\\
		\midrule
$\bar{\Pi}^{\rs_A}_g$&$1^2$&$1^2$&$2^1$\\
$\widetilde{\Pi}^{\rs_A}_g$&$1^{18}$&$1^22^8$&$2^14^4$\\
	\midrule
$\bar{\chi}^{\rs_D}_g$&$1$&$1$&$1$\\
$\chi^{\rs_D}_g$&$1$&$1$&$-1$\\
		\midrule
$\bar{\Pi}^{\rs_D}_g$&$1^1$&$1^1$&$1^1$\\
$\widetilde{\Pi}^{\rs_D}_g$&$1^6$&$1^6$&$1^42^1$\\
	\midrule
$\widetilde{\Pi}^{\rs}_g$&	$1^{24}$&$1^82^8$&$1^42^24^4$
	\\\bottomrule
\end{tabular}
\end{center}

\end{table}

\begin{table}
\begin{center}
\caption{Twisted Euler characters and Frame shapes at $\ll=10+5$, $\rs=D_6^4$}\label{tab:chars:eul:10+5}
\begin{tabular}{l|rrrrr}\toprule
	$[g]$&	1A&	2A&	3A&2B&4A\\
		\midrule
	$n_g|h_g$&$1|1$&$2|2$&$3|1$&$2|1$&$4|4$\\
		\midrule
$\bar{\chi}^{\rs_D}_g$&$4$&$0$&$1$&$2$&$0$\\
$\chi^{\rs_D}_g$&$4$&$0$&$1$&$-2$&$0$\\
		\midrule
$\bar{\Pi}^{\rs_D}_g$&$1^4$&$2^2$&$1^13^1$&$1^22^1$&$4^1$\\
$\widetilde{\Pi}^{\rs_D}_g$&$1^{24}$&$2^{12}$&$1^63^6$&$1^82^8$&$4^6$
	\\\bottomrule
\end{tabular}
\end{center}
\end{table}

\begin{table}
\begin{center}
\caption{Twisted Euler characters and Frame shapes at $\ll=12$, $\rs=A_{11}D_7E_6$}\label{tab:chars:eul:12}
\begin{tabular}{l|rr}\toprule
	$[g]$&	1A&	2A\\
		\midrule
	$n_g|h_g$&		$1|1$&$1|2$\\	
		\midrule
$\bar{\chi}^{\rs_A}_g$&$1$&$1$\\
${\chi}^{\rs_A}_g$&$1$&$-1$\\
	\midrule
$\bar{\Pi}^{\rs_A}_g$&$1^1$&$1^1$\\
$\widetilde{\Pi}^{\rs_A}_g$&$1^{11}$&$1^12^5$\\
	\midrule
$\bar{\chi}^{\rs_D}_g$&$1$&$1$\\
${\chi}^{\rs_D}_g$&$1$&$-1$\\
	\midrule
$\bar{\Pi}^{\rs_D}_g$&$1^1$&$1^1$\\
$\widetilde{\Pi}^{\rs_D}_g$&$1^7$&$1^52^1$\\
	\midrule
$\bar{\chi}^{\rs_E}_g$&$1$&$1$\\
${\chi}^{\rs_E}_g$&$1$&$-1$\\
	\midrule
$\bar{\Pi}^{\rs_E}_g$&$1^1$&$1^1$\\
$\widetilde{\Pi}^{\rs_E}_g$&$1^6$&$1^22^2$\\
	\midrule
$\widetilde{\Pi}^{\rs}_g$&$1^{24}$&$1^82^8$
\\\bottomrule
\end{tabular}
\end{center}
\end{table}

\begin{table}
\begin{center}
\caption{Twisted Euler characters and Frame shapes at $\ll=12+4$, $\rs=E_6^4$}\label{tab:chars:eul:12+4}
\begin{tabular}{l|rrrrrrr}\toprule
$[g]$&   		1A&		2A&   	2B&   	4A&			3A&			6A&   	8AB\\ 
	\midrule
$n_g|h_g$&	$1|1$&$1|2$&$2|1$&$2|4$&$3|1$&$3|2$&$4|8$\\
	\midrule
$\bar{\chi}^{X_E}_g$&
			$4$&	$4$&	$2$&	$0$&	$1$&	$1$&	$0$\\
$\chi^{X_E}_g$&
			$4$&	$-4$&	$0$&	$0$&	$1$&	$-1$&	$0$\\
			\midrule
$\bar{\Pi}^{\rs_E}_g$&
			$1^4$&	$1^4$&	$1^22^1$&	$2^2$&	$1^13^1$&	$1^13^1$&	$4^1$\\
$\widetilde{\Pi}^{\rs_E}_g$&
			$1^{24}$&	$1^82^8$&	$1^82^8$&	$2^44^4$&	$1^63^6$&	$1^22^23^26^2$&	$4^28^2$
	\\\bottomrule
\end{tabular}
\smallskip
\end{center}
\end{table}

\begin{table}
\begin{center}
\caption{Twisted Euler characters and Frame shapes at $\ll=13$, $\rs=A_{12}^2$}\label{tab:chars:eul:13}
\begin{tabular}{r|rrr}\toprule
	$[g]$&	1A&	2A&	4AB\\
		\midrule
$n_g|h_g$&		$1|1$&$1|4$&${2|8}$\\	
		\midrule
	$\bar{\chi}^{\rs_A}_{g}$	&2&2&0\\
	$\chi^{\rs_A}_{g}$	&2&-2&0\\
		\midrule
	$\bar{\Pi}^{\rs_A}_{g}$	&$1^2$&$1^2$&$2^1$\\
	$\widetilde{\Pi}^{\rs_A}_{g}$&$1^{24}$&${2^{12}}$&${4^6}$
	\\\bottomrule
\end{tabular}
\end{center}

\end{table}

\begin{table}
\begin{center}
\caption{Twisted Euler characters and Frame shapes at $\ll=14+7$, $\rs=D_8^3$}\label{tab:chars:eul:14+7}
\begin{tabular}{l|rrr}\toprule
	$[g]$&	1A&	2A&	3A\\
		\midrule
	$n_g|h_g$&$1|1$&$2|1$&$3|3$\\
	\midrule
	$\bar{\chi}^{\rs_D}_{g}$	&3&1&0\\
	$\chi^{\rs_D}_{g}$	&3&1&0\\
		\midrule
	$\bar{\Pi}^{\rs_D}_{g}$	&$1^3$&$1^12^1$&$3^1$\\
	$\widetilde{\Pi}^{\rs_D}_{g}$&$1^{24}$&$1^8{2^{8}}$&${3^8}$
	\\\bottomrule
\end{tabular}
\end{center}
\end{table}

\begin{table}
\begin{center}
\caption{Twisted Euler characters and Frame shapes at $\ll=16$, $\rs=A_{15}D_9$}\label{tab:chars:eul:16}
\begin{tabular}{l|rr}\toprule
	$[g]$&	1A&	2A\\
		\midrule
	$n_g|h_g$&$1|1$&$1|2$\\
	\midrule
$\bar{\chi}^{\rs_A}_g$&$1$&$1$\\
${\chi}^{\rs_A}_g$&$1$&$-1$\\
	\midrule
$\bar{\Pi}^{\rs_A}_g$&$1^1$&$1^1$\\
$\widetilde{\Pi}^{\rs_A}_g$&$1^{15}$&$1^12^7$\\
	\midrule
$\bar{\chi}^{\rs_D}_g$&$1$&$1$\\
${\chi}^{\rs_D}_g$&$1$&$-1$\\
	\midrule
$\bar{\Pi}^{\rs_D}_g$&$1^1$&$1^1$\\
$\widetilde{\Pi}^{\rs_D}_g$&$1^9$&$1^72^1$\\
	\midrule
$\widetilde{\Pi}^{\rs}_g$&$1^{24}$&$1^82^8$
\\\bottomrule
\end{tabular}
\end{center}
\end{table}

\clearpage

\begin{table}
\begin{center}
\caption{Twisted Euler characters and Frame shapes at $\ll=18$, $\rs=A_{17}E_7$}\label{tab:chars:eul:18}
\begin{tabular}{l|rr}\toprule
	$[g]$&	1A&	2A\\
		\midrule
	$n_g|h_g$&$1|1$&$1|2$\\
	\midrule
$\bar{\chi}^{\rs_A}_g$&$1$&$1$\\
${\chi}^{\rs_A}_g$&$1$&$-1$\\
	\midrule
$\bar{\Pi}^{\rs_A}_g$&$1^1$&$1^1$\\
$\widetilde{\Pi}^{\rs_A}_g$&$1^{17}$&$1^12^8$\\
	\midrule
$\bar{\chi}^{\rs_E}_g$&$1$&$1$\\
	\midrule
$\bar{\Pi}^{\rs_E}_g$&$1^1$&$1^1$\\
$\widetilde{\Pi}^{\rs_E}_g$&$1^7$&$1^7$\\
	\midrule
$\widetilde{\Pi}^{\rs}_g$&$1^{24}$&$1^82^8$
\\\bottomrule
\end{tabular}
\end{center}
\end{table}

\begin{table}
\begin{center}
\caption{Twisted Euler characters and Frame shapes at $\ll=18+9$, $\rs=D_{10}E_7^2$}\label{tab:chars:eul:18+9}
\begin{tabular}{l|rr}\toprule
	$[g]$&	1A&	2A\\
		\midrule
	$n_g|h_g$&$1|1$&$2|1$\\
	\midrule
$\bar{\chi}^{\rs_D}_g$&$1$&$1$\\
${\chi}^{\rs_D}_g$&$1$&$-1$\\
	\midrule
$\bar{\Pi}^{\rs_D}_g$&$1^1$&$1^1$\\
$\widetilde{\Pi}^{\rs_D}_g$&$1^{10}$&$1^{8}2^1$\\
	\midrule
$\bar{\chi}^{\rs_E}_g$&$2$&$0$\\
	\midrule
$\bar{\Pi}^{\rs_E}_g$&$1^2$&$2^1$\\
$\widetilde{\Pi}^{\rs_E}_g$&$1^{14}$&$2^7$\\
	\midrule
$\widetilde{\Pi}^{\rs}_g$&$1^{24}$&$1^82^8$
\\\bottomrule
\end{tabular}
\end{center}
\end{table}

\clearpage

\begin{table}
\begin{center}
\caption{Twisted Euler characters and Frame shapes at $\ll=22+11$, $\rs=D_{12}^2$}\label{tab:chars:eul:22+11}
\begin{tabular}{l|rr}\toprule
	$[g]$&	1A&	2A\\
		\midrule
		$n_g|h_g$&$1|1$&$2|2$\\
		\midrule
	$\bar{\chi}^{\rs_D}_{g}$	&2&0\\
	$\chi^{\rs_D}_{g}$	&2&0\\
		\midrule
	$\bar{\Pi}^{\rs_D}_{g}$	&$1^2$&$2^1$\\
	$\widetilde{\Pi}^{\rs_D}_{g}$&$1^{24}$&${2^{12}}$
	\\\bottomrule
\end{tabular}
\end{center}
\end{table}

\begin{table}
\begin{center}
\caption{Twisted Euler characters and Frame shapes at $\ll=25$, $\rs=A_{24}$}\label{tab:chars:eul:25}
\begin{tabular}{l|rr}\toprule
	$[g]$&	1A&	2A\\
		\midrule
	$n_g|h_g$&	$1|1$&	$1|4$\\
	\midrule
$\bar{\chi}^{\rs_A}_g$&$1$&$1$\\
${\chi}^{\rs_A}_g$&$1$&$-1$\\
	\midrule
$\bar{\Pi}^{\rs_A}_g$&$1^1$&$1^1$\\
$\widetilde{\Pi}^{\rs_A}_g$&$1^{24}$&$2^{12}$
\\\bottomrule
\end{tabular}
\end{center}
\end{table}

\begin{table}
\begin{center}
\caption{Twisted Euler characters and Frame shapes at $\ll=30+6,10,15$, $\rs=E_8^3$}\label{tab:chars:eul:30+6,10,15}
\begin{tabular}{l|rrr}\toprule
$[g]$&   		1A&		2A&   	3A\\ 
	\midrule
	$n_g|h_g$&	$1|1$&	$2|1$&	$3|3$\\
	\midrule
$\bar{\chi}^{\rs_E}_g$&
			$3$&	$1$&	$0$\\
			\midrule
$\bar{\Pi}^{\rs_E}_g$&
			$1^3$&	$1^12^1$&	$3^1$\\
$\widetilde{\Pi}^{\rs_E}_g$&
			$1^{24}$&	$1^82^8$&	$3^8$
	\\\bottomrule
\end{tabular}
\smallskip
\end{center}
\end{table}

\clearpage

%---------------------------------------------------------------------------------------%
\section{Coefficient Tables}\label{sec:coeffs}
%---------------------------------------------------------------------------------------%

In this section we furnish tables of Fourier coefficients of small degree for the vector-valued mock modular forms $H^{\rs}_{g}$ attached to elements $g\in G^{\rs}$.
For each Niemeier root system $X$ and each conjugacy class $[g]\in G^X$ we give a table that displays the coefficients of $H^{\rs}_{g,r}$ for sufficiently many $r$ that any other component coincides with one of these up to sign. For instance, we always have $H^X_{g,r}=-H^X_{g,-r}$, so it suffices to list the $H^X_{g,r}$ for $0<r<m$ when the Coxeter number of $X$ is $m$. When $X$ has no A-type components there are further redundancies, so that $H^X_{g,2}=H^X_{g,4}=0$ and $H^X_{g,1}=H^X_{g,5}$, for example, when $X=D_4^6$. 
Recall that $H_g^X$ is conjectured to coincide with the graded trace function of the umbral module $K^X$ for all Niemeier root systems $X$ except for $X=A_3^8$, and the relation has an additional factor of 3 when $X=A_3^8$ (cf. Conjecture \ref{conj:conj:mod:Kell}).
 
For a Niemeier root system $X$ with Coxeter number $m$, the first row of each table labels the conjugacy classes, and the first column labels exponents of  $q^{1/4m}$, so that the entry in the row labelled $d$ and the column labelled $nZ$ in the table captioned $H^{\rs}_{g,r}$ is the coefficient of $q^{d/4m}$ in the Fourier expansion of $H^{\rs}_{g,r}$ for $[g]= nZ$.
Occasionally the functions $H^{\rs}_{g}$ and $H^{\rs}_{g'}$ coincide for non-conjugate $g$ and $g'$ and when this happens we condense information into a single column, writing $7AB$ in Table \ref{tab:coeffs:2_1}, for example, to indicate that the entries in that column are Fourier coefficients for both $H^{(2)}_{7A}$ and $H^{(2)}_{7B}$.

\clearpage

\begin{sidewaystable}
\subsection{Lambency 2}
\begin{small}
\centering
\caption{McKay--Thompson series $H^{(2)}_{g,1}=H^{\rs}_{g,1}$ for $\rs=A_1^{24}$}\label{tab:coeffs:2_1}\smallskip
\begin{tabular}{c@{ }|@{\;}r@{ }r@{ }r@{ }r@{ }r@{ }r@{ }r@{ }r@{ }r@{ }r@{ }r@{ }r@{ }r@{ }r@{ }r@{ }r@{ }r@{ }r@{ }r@{ }r@{ }r@{ }r@{ }r@{ }r@{ }r@{ }r}\toprule
$[g]$	&1A	&2A	&2B	&3A	&3B	&4A	&4B	&4C	&5A	&6A	&6B	&7AB	&8A	&10A	&11A	&12A	&12B	&14AB	&15AB	&21AB	&23AB	\\
\midrule
$\G_g$&$1$	&$2$	&$2|2$	&$3$	&$3|3$	&$4|2$	&$4$	&$4|4$	&$5$	&$6$	&$6|6$	&$7$	&$8$	&$10|2$	&$11$	&$12|2$	&$12|12$	&$14$	&$15$	&$21|3$	&$23$	\\
\midrule
-1	&-2	&-2	&-2	&-2	&-2	&-2	&-2	&-2	&-2	&-2	&-2	&-2	&-2	&-2	&-2	&-2	&-2	&-2	&-2	&-2	&-2	\\
7	&90	&-6	&10	&0	&6	&-6	&2	&2	&0	&0	&-2	&-1	&-2	&0	&2	&0	&2	&1	&0	&-1	&-2	\\
15	&462	&14	&-18	&-6	&0	&-2	&-2	&6	&2	&2	&0	&0	&-2	&2	&0	&-2	&0	&0	&-1	&0	&2	\\
23	&1540	&-28	&20	&10	&-14	&4	&-4	&-4	&0	&2	&2	&0	&0	&0	&0	&-2	&2	&0	&0	&0	&-1	\\
31	&4554	&42	&-38	&0	&12	&-6	&2	&-6	&-6	&0	&4	&4	&-2	&2	&0	&0	&0	&0	&0	&-2	&0	\\
39	&11592	&-56	&72	&-18	&0	&-8	&8	&0	&2	&-2	&0	&0	&0	&2	&-2	&-2	&0	&0	&2	&0	&0	\\
\bottomrule
\end{tabular}
\end{small}
\end{sidewaystable}

\clearpage

\begin{sidewaystable}
\subsection{Lambency 3}
\vspace{-16pt}
\centering
\caption{McKay--Thompson series $H^{(3)}_{g,1}=H^{\rs}_{g,1}$ for $\rs=A_2^{12}$}\smallskip
\begin{tabular}{c@{ }|@{\;}r@{ }r@{ }r@{ }r@{ }r@{ }r@{ }r@{ }r@{ }r@{ }r@{ }r@{ }r@{ }r@{ }r@{ }r@{ }r@{ }r@{ }r@{ }r@{ }r@{ }r@{ }r@{ }r@{ }r@{ }r@{ }r}\toprule
$[g]$&1A	&2A	&4A	&2B	&2C	&3A	&6A	&3B	&6B	&4B	&4C	&5A	&10A	&12A	&6C	&6D	&8AB	&8CD	&20AB	&11AB	&22AB	\\
	\midrule
$\G_g$&$1$&$1|4$&${2|8}$&$2$&$2|2$&$3$&$3|4$&${3|3}$&${3|12}$&$4|2$&$4$&$5$&$5|4$&$6|24$&$6$&$6|2$&$8|4$&$8$&${10|8}$&$11$&$11|4$\\
	\midrule
-1	&-2	&-2	&-2	&-2	&-2	&-2	&-2	&-2	&-2	&-2	&-2	&-2	&-2	&-2	&-2	&-2	&-2	&-2	&-2	&-2	&-2	\\
11	&32	&32	&8	&0	&0	&-4	&-4	&2	&2	&0	&0	&2	&2	&2	&0	&0	&0	&0	&-2	&-1	&-1	\\
23	&110	&110	&-10	&-2	&-2	&2	&2	&2	&2	&6	&-2	&0	&0	&2	&-2	&-2	&-2	&2	&0	&0	&0	\\
35	&288	&288	&8	&0	&0	&0	&0	&-6	&-6	&0	&0	&-2	&-2	&2	&0	&0	&0	&0	&-2	&2	&2	\\
47	&660	&660	&-20	&4	&4	&-6	&-6	&6	&6	&-4	&4	&0	&0	&-2	&-2	&-2	&4	&0	&0	&0	&0	\\
59	&1408	&1408	&32	&0	&0	&4	&4	&4	&4	&0	&0	&-2	&-2	&-4	&0	&0	&0	&0	&2	&0	&0	\\
\bottomrule
\end{tabular}
\end{sidewaystable}

\begin{sidewaystable}
\centering
\caption{McKay--Thompson series $H^{(3)}_{g,2}=H^{\rs}_{g,2}$ for $\rs=A_{2}^{12}$}\smallskip
\begin{tabular}{c@{ }|@{\;}r@{ }r@{ }r@{ }r@{ }r@{ }r@{ }r@{ }r@{ }r@{ }r@{ }r@{ }r@{ }r@{ }r@{ }r@{ }r@{ }r@{ }r@{ }r@{ }r@{ }r@{ }r@{ }r@{ }r@{ }r@{ }r}\toprule
$[g]$	&1A	&2A	&4A	&2B	&2C	&3A	&6A	&3B	&6B	&4B	&4C	&5A	&10A	&12A	&6C	&6D	&8AB	&8CD	&20AB	&11AB	&22AB	\\
	\midrule
$\G_g$&$1$&$1|4$&${2|8}$&$2$&$2|2$&$3$&$3|4$&${3|3}$&${3|12}$&$4|2$&$4$&$5$&$5|4$&$6|24$&$6$&$6|2$&$8|4$&$8$&${10|8}$&$11$&$11|4$\\
	\midrule
8	&20	&-20	&0	&-4	&4	&2	&-2	&-4	&4	&0	&0	&0	&0	&0	&2	&-2	&0	&0	&0	&-2	&2	\\
20	&88	&-88	&0	&8	&-8	&-2	&2	&4	&-4	&0	&0	&-2	&2	&0	&2	&-2	&0	&0	&0	&0	&0	\\
32	&220	&-220	&0	&-12	&12	&4	&-4	&4	&-4	&0	&0	&0	&0	&0	&0	&0	&0	&0	&0	&0	&0	\\
44	&560	&-560	&0	&16	&-16	&2	&-2	&-4	&4	&0	&0	&0	&0	&0	&-2	&2	&0	&0	&0	&-1	&1	\\
56	&1144	&-1144	&0	&-24	&24	&-8	&8	&4	&-4	&0	&0	&4	&-4	&0	&0	&0	&0	&0	&0	&0	&0	\\
68	&2400	&-2400	&0	&32	&-32	&6	&-6	&0	&0	&0	&0	&0	&0	&0	&2	&-2	&0	&0	&0	&2	&-2	\\
\bottomrule
\end{tabular}
\end{sidewaystable}

\begin{table}
\vspace{-16pt}
\subsection{Lambency 4}
\vspace{-16pt}
\centering
\caption{McKay--Thompson series $H^{(4)}_{g,1}=H^{\rs}_{g,1}$ for $\rs=A_3^8$}
\begin{tabular}{c@{ }|@{\;}r@{ }r@{ }r@{ }r@{ }r@{ }r@{ }r@{ }r@{ }r@{ }r@{ }r@{ }r@{ }r}\toprule
$[g]$	&1A	&2A	&2B	&4A	&4B	&2C	&3A	&6A	&6BC	&8A	&4C	&7AB	&14AB	\\
\midrule
$\G_g$&$1$& $1|2$&	$2|2$&	$2|4$&			${4|{4}}$&	$2$& 	$3$& 	$3|2$&		$6|2$&	${4|{8}}$&		$4$&  	$7$&	$7|2$\\	
	\midrule
-1	&-2		&-2		&-2	&-2	&-2	&-2	&-2	&-2	&-2	&-2&-2	&-2	&-2\\
15	&14	&14	&-2	&6	&-2	&-2	&2	&2	&-2	&2	&-2	&0	&0	\\
31	&42	&42	&-6	&-6	&2	&2	&0	&0	&0	&2	&-2	&0	&0	\\
47	&86	&86	&6	&6	&-2	&-2	&-4	&-4	&0	&-2	&2	&2	&2	\\
63	&188	&188	&-4	&-12	&-4	&4	&2	&2	&2	&0	&0	&-1	&-1	\\
79	&336	&336	&0	&16	&0	&-8	&0	&0	&0	&0	&-4	&0	&0	\\
\bottomrule
\end{tabular}
\end{table}

\begin{table}
\centering
\caption{McKay--Thompson series $H^{(4)}_{g,2}=H^{\rs}_{g,2}$ for $\rs=A_3^8$}\smallskip
\begin{tabular}{c@{ }|@{\;}r@{ }r@{ }r@{ }r@{ }r@{ }r@{ }r@{ }r@{ }r@{ }r@{ }r@{ }r@{ }r}\toprule
$[g]$	&1A	&2A	&2B	&4A	&4B	&2C	&3A	&6A	&6BC	&8A	&4C	&7AB	&14AB	\\
\midrule
$\G_g$&$1$& $1|2$&	$2|2$&	$2|4$&			${4|{4}}$&	$2$& 	$3$& 	$3|2$&		$6|2$&	${4|{8}}$&		$4$&  	$7$&	$7|2$\\	
\midrule
12	&16	&-16	&0	&0	&0	&0	&-2	&2	&0	&0	&0	&2	&-2	\\
28	&48	&-48	&0	&0	&0	&0	&0	&0	&0	&0	&0	&-1	&1	\\
44	&112	&-112	&0	&0	&0	&0	&4	&-4	&0	&0	&0	&0	&0	\\
60	&224	&-224	&0	&0	&0	&0	&-4	&4	&0	&0	&0	&0	&0	\\
76	&432	&-432	&0	&0	&0	&0	&0	&0	&0	&0	&0	&-2	&2	\\
92	&784	&-784	&0	&0	&0	&0	&4	&-4	&0	&0	&0	&0	&0	\\
\bottomrule
\end{tabular}
\end{table}

\begin{table}
\centering
\caption{McKay--Thompson series $H^{(4)}_{g,3}=H^{\rs}_{g,3}$ for $\rs=A_3^8$}\smallskip
\begin{tabular}{c@{ }|@{\;}r@{ }r@{ }r@{ }r@{ }r@{ }r@{ }r@{ }r@{ }r@{ }r@{ }r@{ }r@{ }r}\toprule
$[g]$	&1A	&2A	&2B	&4A	&4B	&2C	&3A	&6A	&6BC	&8A	&4C	&7AB	&14AB	\\
\midrule
$\G_g$&$1$& $1|2$&	$2|2$&	$2|4$&			${4|{4}}$&	$2$& 	$3$& 	$3|2$&		$6|2$&	${4|{8}}$&		$4$&  	$7$&	$7|2$\\	
\midrule
7	&6	&6	&6	&-2	&-2	&-2	&0	&0	&0	&2	&2	&-1	&-1	\\
23	&28	&28	&-4	&4	&-4	&4	&-2	&-2	&2	&0	&0	&0	&0	\\
39	&56	&56	&8	&0	&0	&-8	&2	&2	&2	&-4	&0	&0	&0	\\
55	&138	&138	&-6	&2	&2	&10	&0	&0	&0	&-2	&2	&-2	&-2	\\
71	&238	&238	&14	&-10	&-2	&-10	&-2	&-2	&2	&2	&2	&0	&0	\\
87	&478	&478	&-18	&6	&-2	&14	&4	&4	&0	&2	&-2	&2	&2	\\
\bottomrule
\end{tabular}
\end{table}

\begin{table}
\vspace{-16pt}
\subsection{Lambency 5}
\vspace{-16pt}
\centering
\caption{\label{tab:coeffs:5_1}McKay--Thompson series $H^{(5)}_{g,1}=H^{\rs}_{g,1}$ for $\rs=A_4^6$}\smallskip
\begin{tabular}{r|rrrrrrrrrrr}\toprule
$[g]$	&1A	&2A	&2B	&2C	&3A	&6A	&5A	&10A	&4AB	&4CD	&12AB	\\
	\midrule
$\G_g$&		$1$&	$1|4$&	$2|2$&	$2$&		$3|3$&		$3|12$&	$5$&	$5|4$&		$2|8$&		$4$&	$6|24$	\\	
	\midrule
-1	&-2	&-2	&-2	&-2	&-2	&-2	&-2	&-2	&-2	&-2	&-2	\\
19	&8	&8	&0	&0	&2	&2	&-2	&-2	&4	&0	&-2	\\
39	&18	&18	&2	&2	&0	&0	&-2	&-2	&-6	&2	&0	\\
59	&40	&40	&0	&0	&-2	&-2	&0	&0	&4	&0	&-2	\\
79	&70	&70	&-2	&-2	&4	&4	&0	&0	&-6	&-2	&0	\\
99	&120	&120	&0	&0	&0	&0	&0	&0	&12	&0	&0	\\
\bottomrule
\end{tabular}
\end{table}

\begin{table}
\centering
\caption{\label{tab:coeffs:5_2}McKay--Thompson series $H^{(5)}_{g,2}=H^{\rs}_{g,2}$ for $\rs=A_4^6$}\smallskip
\begin{tabular}{r|rrrrrrrrrrr}\toprule
$[g]$	&1A	&2A	&2B	&2C	&3A	&6A	&5A	&10A	&4AB	&4CD	&12AB	\\
	\midrule
$\G_g$&		$1$&	$1|4$&	$2|2$&	$2$&		$3|3$&		$3|12$&	$5$&	$5|4$&		$2|8$&		$4$&	$6|24$	\\	
	\midrule
16	&10	&-10	&2	&-2	&-2	&2	&0	&0	&0	&0	&0	\\
36	&30	&-30	&-2	&2	&0	&0	&0	&0	&0	&0	&0	\\
56	&52	&-52	&4	&-4	&4	&-4	&2	&-2	&0	&0	&0	\\
76	&108	&-108	&-4	&4	&0	&0	&-2	&2	&0	&0	&0	\\
96	&180	&-180	&4	&-4	&0	&0	&0	&0	&0	&0	&0	\\
116	&312	&-312	&-8	&8	&0	&0	&2	&-2	&0	&0	&0	\\
\bottomrule
\end{tabular}
\end{table}

\begin{table}
\centering
\caption{\label{tab:coeffs:5_3}McKay--Thompson series $H^{(5)}_{g,3}=H^{\rs}_{g,3}$ for $\rs=A_4^6$}\smallskip
\begin{tabular}{r|rrrrrrrrrrr}\toprule
$[g]$	&1A	&2A	&2B	&2C	&3A	&6A	&5A	&10A	&4AB	&4CD	&12AB	\\
	\midrule
$\G_g$&		$1$&	$1|4$&	$2|2$&	$2$&		$3|3$&		$3|12$&	$5$&	$5|4$&		$2|8$&		$4$&	$6|24$	\\	
	\midrule
11	&8	&8	&0	&0	&2	&2	&-2	&-2	&-4	&0	&2	\\
31	&22	&22	&-2	&-2	&-2	&-2	&2	&2	&2	&-2	&2	\\
51	&48	&48	&0	&0	&0	&0	&-2	&-2	&0	&0	&0	\\
71	&90	&90	&2	&2	&0	&0	&0	&0	&6	&2	&0	\\
91	&160	&160	&0	&0	&-2	&-2	&0	&0	&-8	&0	&-2	\\
111	&270	&270	&-2	&-2	&0	&0	&0	&0	&6	&-2	&0	\\
\bottomrule
\end{tabular}
\end{table}

\begin{table}
\centering
\caption{\label{tab:coeffs:5_4}McKay--Thompson series $H^{(5)}_{g,4}=H^{\rs}_{g,4}$ for $\rs=A_4^6$}\smallskip
\begin{tabular}{r|rrrrrrrrrrr}\toprule
$[g]$	&1A	&2A	&2B	&2C	&3A	&6A	&5A	&10A	&4AB	&4CD	&12AB	\\
	\midrule
$\G_g$&		$1$&	$1|4$&	$2|2$&	$2$&		$3|3$&		$3|12$&	$5$&	$5|4$&		$2|8$&		$4$&	$6|24$	\\	
	\midrule
4	&2	&-2	&2	&-2	&2	&-2	&2	&-2	&0	&0	&0	\\
24	&12	&-12	&-4	&4	&0	&0	&2	&-2	&0	&0	&0	\\
44	&20	&-20	&4	&-4	&-4	&4	&0	&0	&0	&0	&0	\\
64	&50	&-50	&-6	&6	&2	&-2	&0	&0	&0	&0	&0	\\
84	&72	&-72	&8	&-8	&0	&0	&2	&-2	&0	&0	&0	\\
104	&152	&-152	&-8	&8	&-4	&4	&2	&-2	&0	&0	&0	\\
\bottomrule
\end{tabular}
\end{table}

\clearpage

\begin{table}
\vspace{-16pt}
\subsection{Lambency 6}
\vspace{-16pt}
\centering
\caption{\label{tab:coeffs:6_1}McKay--Thompson series $H^{(6)}_{g,1}=H^{\rs}_{g,1}$ for $\rs=A_5^4D_4$}\smallskip
\begin{tabular}{r|rrrrrrr}\toprule
$[g]$&   		1A&		2A&   	2B&   	4A&			3A&			6A&   	8AB\\ 
	\midrule
$\G_g$&		$1|1$&	$1|1'$&	$2|1$&	$2|1'$&		$3|1$&		$3|1'$&	$4|2'$\\	
	\midrule
-1	&-2	&-2	&-2	&-2	&-2	&-2	&-2	\\
23	&4	&4	&0	&4	&-2	&-2	&0	\\
47	&12	&12	&0	&-4	&0	&0	&0	\\
71	&18	&18	&-2	&2	&0	&0	&-2	\\
95	&34	&34	&2	&-6	&-2	&-2	&-2	\\
119	&56	&56	&0	&8	&2	&2	&0	\\
143	&90	&90	&2	&-6	&0	&0	&2	\\
167	&128	&128	&-4	&8	&-4	&-4	&0	\\
191	&206	&206	&2	&-10	&2	&2	&-2	\\
215	&294	&294	&-2	&14	&0	&0	&2	\\
239	&424	&424	&4	&-16	&-2	&-2	&0	\\
263	&598	&598	&-6	&14	&4	&4	&-2	\\
287	&846	&846	&6	&-18	&0	&0	&-2	\\
311	&1152	&1152	&-4	&24	&-6	&-6	&0	\\
335	&1598	&1598	&6	&-26	&2	&2	&2	\\
359	&2138	&2138	&-10	&26	&2	&2	&-2	\\
383	&2886	&2886	&10	&-34	&-6	&-6	&-2	\\
407	&3834	&3834	&-6	&42	&6	&6	&2	\\
431	&5078	&5078	&10	&-42	&2	&2	&2	\\
455	&6622	&6622	&-18	&46	&-8	&-8	&-2	\\
479	&8682	&8682	&14	&-54	&6	&6	&-2	\\
503	&11214	&11214	&-14	&62	&0	&0	&2	\\
527	&14476	&14476	&20	&-68	&-8	&-8	&4	\\
551	&18530	&18530	&-22	&74	&8	&8	&-2	\\
575	&23684	&23684	&24	&-84	&2	&2	&-4	\\
599	&30006	&30006	&-22	&102	&-12	&-12	&2	\\
623	&38018	&38018	&26	&-110	&8	&8	&2	\\
647	&47792	&47792	&-36	&112	&2	&2	&-4	\\
671	&60012	&60012	&36	&-132	&-12	&-12	&-4	\\
695	&74942	&74942	&-34	&150	&14	&14	&2	\\
719	&93392	&93392	&44	&-160	&2	&2	&4	\\
743	&115782	&115782	&-54	&174	&-18	&-18	&-2	\\
767	&143388	&143388	&52	&-196	&12	&12	&-4	\\
791	&176716	&176716	&-52	&220	&4	&4	&4	\\
815	&217416	&217416	&64	&-240	&-18	&-18	&4	\\
839	&266496	&266496	&-76	&256	&18	&18	&-4	\\
863	&326074	&326074	&78	&-286	&4	&4	&-6	\\
887	&397548	&397548	&-80	&324	&-24	&-24	&4	\\
911	&483966	&483966	&90	&-346	&18	&18	&6	\\
935	&587236	&587236	&-108	&372	&4	&4	&-4	\\
959	&711394	&711394	&114	&-414	&-26	&-26	&-6	\\
	\bottomrule
\end{tabular}
\end{table}

\begin{table}
\vspace{-16pt}
\centering
\caption{\label{tab:coeffs:6_2}McKay--Thompson series $H^{(6)}_{g,2}=H^{\rs}_{g,2}$ for $\rs=A_5^4D_4$}
\begin{tabular}{r|rrrrrrr}\toprule
$[g]$&   		1A&		2A&   	2B&   	4A&			3A&			6A&   	8AB\\ 
	\midrule
$\G_g$&		$1|1$&	$1|1'$&	$2|1$&	$2|1'$&		$3|1$&		$3|1'$&	$4|2'$\\	
	\midrule
20	&8	&-8	&0	&0	&2	&-2	&0	\\
44	&16	&-16	&0	&0	&-2	&2	&0	\\
68	&32	&-32	&0	&0	&2	&-2	&0	\\
92	&56	&-56	&0	&0	&2	&-2	&0	\\
116	&88	&-88	&0	&0	&-2	&2	&0	\\
140	&144	&-144	&0	&0	&0	&0	&0	\\
164	&224	&-224	&0	&0	&2	&-2	&0	\\
188	&328	&-328	&0	&0	&-2	&2	&0	\\
212	&488	&-488	&0	&0	&2	&-2	&0	\\
236	&704	&-704	&0	&0	&2	&-2	&0	\\
260	&992	&-992	&0	&0	&-4	&4	&0	\\
284	&1400	&-1400	&0	&0	&2	&-2	&0	\\
308	&1936	&-1936	&0	&0	&4	&-4	&0	\\
332	&2640	&-2640	&0	&0	&-6	&6	&0	\\
356	&3584	&-3584	&0	&0	&2	&-2	&0	\\
380	&4816	&-4816	&0	&0	&4	&-4	&0	\\
404	&6408	&-6408	&0	&0	&-6	&6	&0	\\
428	&8480	&-8480	&0	&0	&2	&-2	&0	\\
452	&11136	&-11136	&0	&0	&6	&-6	&0	\\
476	&14512	&-14512	&0	&0	&-8	&8	&0	\\
500	&18832	&-18832	&0	&0	&4	&-4	&0	\\
524	&24288	&-24288	&0	&0	&6	&-6	&0	\\
548	&31136	&-31136	&0	&0	&-10	&10	&0	\\
572	&39760	&-39760	&0	&0	&4	&-4	&0	\\
596	&50520	&-50520	&0	&0	&6	&-6	&0	\\
620	&63888	&-63888	&0	&0	&-12	&12	&0	\\
644	&80512	&-80512	&0	&0	&4	&-4	&0	\\
668	&101032	&-101032	&0	&0	&10	&-10	&0	\\
692	&126280	&-126280	&0	&0	&-14	&14	&0	\\
716	&157344	&-157344	&0	&0	&6	&-6	&0	\\
740	&195360	&-195360	&0	&0	&12	&-12	&0	\\
764	&241752	&-241752	&0	&0	&-18	&18	&0	\\
788	&298296	&-298296	&0	&0	&6	&-6	&0	\\
812	&366960	&-366960	&0	&0	&12	&-12	&0	\\
836	&450112	&-450112	&0	&0	&-20	&20	&0	\\
860	&550688	&-550688	&0	&0	&8	&-8	&0	\\
884	&671968	&-671968	&0	&0	&16	&-16	&0	\\
908	&817840	&-817840	&0	&0	&-26	&26	&0	\\
932	&993088	&-993088	&0	&0	&10	&-10	&0	\\
956	&1203048	&-1203048	&0	&0	&18	&-18	&0	\\
	\bottomrule
\end{tabular}
\end{table}

\begin{table}
\vspace{-16pt}
\centering
\caption{\label{tab:coeffs:6_3}McKay--Thompson series $H^{(6)}_{g,3}=H^{\rs}_{g,3}$ for $\rs=A_5^4D_4$}\smallskip
\begin{tabular}{r|rrrrrrr}\toprule
$[g]$&   		1A&		2A&   	2B&   	4A&			3A&			6A&   	8AB\\ 
	\midrule
$\G_g$&		$1|1$&	$1|1'$&	$2|1$&	$2|1'$&		$3|1$&		$3|1'$&	$4|2'$\\		
	\midrule
15	&6	&6	&-2	&-2	&0	&0	&2	\\
39	&18	&18	&2	&2	&0	&0	&2	\\
63	&30	&30	&-2	&-2	&0	&0	&-2	\\
87	&60	&60	&4	&4	&0	&0	&0	\\
111	&90	&90	&-6	&-6	&0	&0	&2	\\
135	&156	&156	&4	&4	&0	&0	&0	\\
159	&228	&228	&-4	&-4	&0	&0	&0	\\
183	&360	&360	&8	&8	&0	&0	&0	\\
207	&510	&510	&-10	&-10	&0	&0	&2	\\
231	&762	&762	&10	&10	&0	&0	&2	\\
255	&1062	&1062	&-10	&-10	&0	&0	&-2	\\
279	&1518	&1518	&14	&14	&0	&0	&-2	\\
303	&2070	&2070	&-18	&-18	&0	&0	&2	\\
327	&2880	&2880	&16	&16	&0	&0	&0	\\
351	&3870	&3870	&-18	&-18	&0	&0	&-2	\\
375	&5256	&5256	&24	&24	&0	&0	&0	\\
399	&6948	&6948	&-28	&-28	&0	&0	&4	\\
423	&9270	&9270	&30	&30	&0	&0	&2	\\
447	&12120	&12120	&-32	&-32	&0	&0	&-4	\\
471	&15912	&15912	&40	&40	&0	&0	&0	\\
495	&20562	&20562	&-46	&-46	&0	&0	&2	\\
519	&26646	&26646	&46	&46	&0	&0	&2	\\
543	&34110	&34110	&-50	&-50	&0	&0	&-2	\\
567	&43710	&43710	&62	&62	&0	&0	&-2	\\
591	&55440	&55440	&-72	&-72	&0	&0	&4	\\
615	&70362	&70362	&74	&74	&0	&0	&2	\\
639	&88536	&88536	&-80	&-80	&0	&0	&-4	\\
663	&111390	&111390	&94	&94	&0	&0	&-2	\\
687	&139140	&139140	&-108	&-108	&0	&0	&4	\\
711	&173712	&173712	&112	&112	&0	&0	&0	\\
735	&215568	&215568	&-120	&-120	&0	&0	&-4	\\
759	&267246	&267246	&142	&142	&0	&0	&-2	\\
783	&329610	&329610	&-158	&-158	&0	&0	&6	\\
807	&406080	&406080	&168	&168	&0	&0	&4	\\
831	&498042	&498042	&-182	&-182	&0	&0	&-6	\\
855	&610062	&610062	&206	&206	&0	&0	&-2	\\
879	&744306	&744306	&-230	&-230	&0	&0	&6	\\
903	&906930	&906930	&242	&242	&0	&0	&2	\\
927	&1101210	&1101210	&-262	&-262	&0	&0	&-6	\\
951	&1335282	&1335282	&298	&298	&0	&0	&-2	\\
	\bottomrule
\end{tabular}
\end{table}

\begin{table}
\vspace{-16pt}
\centering
\caption{\label{tab:coeffs:6_4}McKay--Thompson series $H^{(6)}_{g,4}=H^{\rs}_{g,4}$ for $\rs=A_5^4D_4$}\smallskip
\begin{tabular}{r|rrrrrrr}\toprule
$[g]$&   		1A&		2A&   	2B&   	4A&			3A&			6A&   	8AB\\ 
	\midrule
$\G_g$&		$1|1$&	$1|1'$&	$2|1$&	$2|1'$&		$3|1$&		$3|1'$&	$4|2'$\\	
	\midrule
8	&4	&-4	&0	&0	&-2	&2	&0	\\
32	&12	&-12	&0	&0	&0	&0	&0	\\
56	&24	&-24	&0	&0	&0	&0	&0	\\
80	&40	&-40	&0	&0	&-2	&2	&0	\\
104	&72	&-72	&0	&0	&0	&0	&0	\\
128	&116	&-116	&0	&0	&2	&-2	&0	\\
152	&176	&-176	&0	&0	&-4	&4	&0	\\
176	&272	&-272	&0	&0	&2	&-2	&0	\\
200	&404	&-404	&0	&0	&2	&-2	&0	\\
224	&584	&-584	&0	&0	&-4	&4	&0	\\
248	&840	&-840	&0	&0	&0	&0	&0	\\
272	&1184	&-1184	&0	&0	&2	&-2	&0	\\
296	&1640	&-1640	&0	&0	&-4	&4	&0	\\
320	&2264	&-2264	&0	&0	&2	&-2	&0	\\
344	&3088	&-3088	&0	&0	&4	&-4	&0	\\
368	&4152	&-4152	&0	&0	&-6	&6	&0	\\
392	&5564	&-5564	&0	&0	&2	&-2	&0	\\
416	&7384	&-7384	&0	&0	&4	&-4	&0	\\
440	&9712	&-9712	&0	&0	&-8	&8	&0	\\
464	&12728	&-12728	&0	&0	&2	&-2	&0	\\
488	&16552	&-16552	&0	&0	&4	&-4	&0	\\
512	&21388	&-21388	&0	&0	&-8	&8	&0	\\
536	&27520	&-27520	&0	&0	&4	&-4	&0	\\
560	&35216	&-35216	&0	&0	&8	&-8	&0	\\
584	&44832	&-44832	&0	&0	&-12	&12	&0	\\
608	&56848	&-56848	&0	&0	&4	&-4	&0	\\
632	&71768	&-71768	&0	&0	&8	&-8	&0	\\
656	&90208	&-90208	&0	&0	&-14	&14	&0	\\
680	&113008	&-113008	&0	&0	&4	&-4	&0	\\
704	&141040	&-141040	&0	&0	&10	&-10	&0	\\
728	&175376	&-175376	&0	&0	&-16	&16	&0	\\
752	&217416	&-217416	&0	&0	&6	&-6	&0	\\
776	&268656	&-268656	&0	&0	&12	&-12	&0	\\
800	&330940	&-330940	&0	&0	&-20	&20	&0	\\
824	&406568	&-406568	&0	&0	&8	&-8	&0	\\
848	&498056	&-498056	&0	&0	&14	&-14	&0	\\
872	&608472	&-608472	&0	&0	&-24	&24	&0	\\
896	&741560	&-741560	&0	&0	&8	&-8	&0	\\
920	&901504	&-901504	&0	&0	&16	&-16	&0	\\
944	&1093312	&-1093312	&0	&0	&-26	&26	&0	\\
	\bottomrule
\end{tabular}
\end{table}

\begin{table}
\vspace{-16pt}
\centering
\caption{\label{tab:coeffs:6_5}McKay--Thompson series $H^{(6)}_{g,5}=H^{\rs}_{g,5}$ for $\rs=A_5^4D_4$}\smallskip
\begin{tabular}{r|rrrrrrr}\toprule
$[g]$&   		1A&		2A&   	2B&   	4A&			3A&			6A&   	8AB\\ 
	\midrule
$\G_g$&		$1|1$&	$1|1'$&	$2|1$&	$2|1'$&		$3|1$&		$3|1'$&	$4|2'$\\	
	\midrule
23	&6	&6	&2	&-2	&0	&0	&-2	\\
47	&8	&8	&-4	&0	&2	&2	&0	\\
71	&24	&24	&4	&0	&0	&0	&0	\\
95	&28	&28	&-4	&4	&-2	&-2	&0	\\
119	&62	&62	&6	&-2	&2	&2	&-2	\\
143	&80	&80	&-8	&0	&2	&2	&0	\\
167	&142	&142	&10	&-2	&-2	&-2	&2	\\
191	&194	&194	&-10	&2	&2	&2	&-2	\\
215	&306	&306	&10	&-6	&0	&0	&-2	\\
239	&404	&404	&-16	&4	&-4	&-4	&0	\\
263	&622	&622	&18	&-2	&4	&4	&2	\\
287	&824	&824	&-16	&8	&2	&2	&0	\\
311	&1178	&1178	&22	&-6	&-4	&-4	&-2	\\
335	&1564	&1564	&-28	&4	&4	&4	&0	\\
359	&2178	&2178	&30	&-6	&0	&0	&2	\\
383	&2844	&2844	&-32	&12	&-6	&-6	&0	\\
407	&3876	&3876	&36	&-12	&6	&6	&-4	\\
431	&5024	&5024	&-44	&8	&2	&2	&0	\\
455	&6690	&6690	&50	&-14	&-6	&-6	&2	\\
479	&8616	&8616	&-52	&16	&6	&6	&0	\\
503	&11286	&11286	&58	&-18	&0	&0	&-2	\\
527	&14384	&14384	&-72	&16	&-10	&-10	&0	\\
551	&18632	&18632	&80	&-16	&8	&8	&4	\\
575	&23578	&23578	&-82	&26	&4	&4	&-2	\\
599	&30122	&30122	&94	&-30	&-10	&-10	&-6	\\
623	&37882	&37882	&-110	&26	&10	&10	&2	\\
647	&47948	&47948	&120	&-28	&2	&2	&4	\\
671	&59848	&59848	&-128	&40	&-14	&-14	&0	\\
695	&75120	&75120	&144	&-40	&12	&12	&-4	\\
719	&93184	&93184	&-164	&40	&4	&4	&0	\\
743	&116018	&116018	&182	&-46	&-16	&-16	&6	\\
767	&143142	&143142	&-194	&54	&12	&12	&-2	\\
791	&176978	&176978	&210	&-62	&2	&2	&-6	\\
815	&217108	&217108	&-244	&60	&-20	&-20	&0	\\
839	&266838	&266838	&266	&-66	&18	&18	&6	\\
863	&325716	&325716	&-280	&84	&6	&6	&0	\\
887	&397942	&397942	&314	&-90	&-20	&-20	&-6	\\
911	&483524	&483524	&-352	&84	&20	&20	&0	\\
935	&587726	&587726	&382	&-98	&2	&2	&6	\\
	\bottomrule
\end{tabular}
\end{table}

\clearpage

\begin{table}
\vspace{-16pt}
\subsection{Lambency 6+3}
\vspace{-12pt}
\centering
\caption{\label{tab:coeffs:6+3_1}McKay--Thompson series $H^{(6+3)}_{g,1}=H^{\rs}_{g,1}$ for $\rs=D_4^6$}\smallskip
\begin{tabular}{r|@{ }r@{ }r@{ }r@{ }r@{ }r@{ }r@{ }r@{ }r@{ }r@{ }r@{ }r@{ }r@{ }r@{ }r@{ }r}\toprule
$[g]$	&1A	&3A	&2A	&6A	&3B	&3C	&4A	&12A	&5A	&15AB	&2B	&2C	&4B	&6B	&6C	\\
	\midrule
$\G_g$	&$1|1$	&$1|1'$	&$2|1$	&$2|1'$	&$3|1$	&$3|3$	&$4|2$	&$4|2'$	&$5|1$	&$5|1'$	&$2|1'$	&$2|4'$	&$4|1'$	&$6|1'$	&$6|6'$	\\
	\midrule
-1	&-2	&-2	&-2	&-2	&-2	&-2	&-2	&-2	&-2	&-2	&-2	&-2	&-2	&-2	&-2	\\
23	&10	&10	&2	&2	&-2	&4	&-2	&-2	&0	&0	&-2	&6	&2	&-2	&0	\\
47	&20	&20	&-4	&-4	&2	&2	&0	&0	&0	&0	&4	&-4	&0	&-2	&2	\\
71	&42	&42	&2	&2	&0	&-6	&-2	&-2	&2	&2	&-6	&2	&-2	&0	&2	\\
95	&62	&62	&-2	&-2	&-4	&2	&-2	&-2	&2	&2	&6	&-10	&-2	&0	&2	\\
119	&118	&118	&6	&6	&4	&-2	&-2	&-2	&-2	&-2	&-6	&10	&2	&0	&-2	\\
\bottomrule
\end{tabular}
\end{table}

\begin{table}
\vspace{-16pt}
\centering
\caption{\label{tab:coeffs:6+3_3}McKay--Thompson series $H^{(6+3)}_{g,3}=H^{\rs}_{g,3}$ for $\rs=D_4^6$}
\begin{tabular}{r|@{ }r@{ }r@{ }r@{ }r@{ }r@{ }r@{ }r@{ }r@{ }r@{ }r@{ }r@{ }r@{ }r@{ }r@{ }r}\toprule
$[g]$	&1A	&3A	&2A	&6A	&3B	&3C	&4A	&12A	&5A	&15AB	&2B	&2C	&4B	&6B	&6C	\\
	\midrule
$\G_g$	&$1|1$	&$1|1'$	&$2|1$	&$2|1'$	&$3|1$	&$3|3$	&$4|2$	&$4|2'$	&$5|1$	&$5|1'$	&$2|1'$	&$2|4'$	&$4|1'$	&$6|1'$	&$6|6'$	\\
	\midrule
15	&12	&-6	&-4	&2	&0	&0	&4	&-2	&2	&-1	&0	&0	&0	&0	&0	\\
39	&36	&-18	&4	&-2	&0	&0	&4	&-2	&-4	&2	&0	&0	&0	&0	&0	\\
63	&60	&-30	&-4	&2	&0	&0	&-4	&2	&0	&0	&0	&0	&0	&0	&0	\\
87	&120	&-60	&8	&-4	&0	&0	&0	&0	&0	&0	&0	&0	&0	&0	&0	\\
111	&180	&-90	&-12	&6	&0	&0	&4	&-2	&0	&0	&0	&0	&0	&0	&0	\\
135	&312	&-156	&8	&-4	&0	&0	&0	&0	&2	&-1	&0	&0	&0	&0	&0	\\
\bottomrule
\end{tabular}
\end{table}

\clearpage

\begin{table}
\vspace{-16pt}
\subsection{Lambency 7}
\vspace{-12pt}
\begin{minipage}[t]{0.49\linewidth}
\centering
\caption{\label{tab:coeffs:7_1}$H^{(7)}_{g,1}$, $\rs=A_6^4$}\smallskip
\begin{tabular}{r|rrrrrrrrrrr}\toprule
$[g]$	&1A	&2A	&4A	&3AB	&6AB	\\
\midrule
$\G_g$&		$1$&	$1|4$&	$2|8$&	$3$&	$3|4$\\
\midrule
-1	&-2	&-2	&-2	&-2	&-2	\\
27	&4	&4	&4	&1	&1	\\
55	&6	&6	&-2	&0	&0	\\
83	&10	&10	&2	&-2	&-2	\\
111	&20	&20	&-4	&2	&2	\\
139	&30	&30	&6	&0	&0	\\
167	&42	&42	&-6	&0	&0	\\
195	&68	&68	&4	&2	&2	\\
223	&96	&96	&-8	&0	&0	\\
251	&130	&130	&10	&-2	&-2	\\
279	&188	&188	&-12	&2	&2	\\
307	&258	&258	&10	&0	&0	\\
335	&350	&350	&-10	&-4	&-4	\\
363	&474	&474	&18	&3	&3	\\
391	&624	&624	&-16	&0	&0	\\
419	&826	&826	&18	&-2	&-2	\\
447	&1090	&1090	&-22	&4	&4	\\
475	&1410	&1410	&26	&0	&0	\\
503	&1814	&1814	&-26	&-4	&-4	\\
531	&2338	&2338	&26	&4	&4	\\
559	&2982	&2982	&-34	&0	&0	\\
587	&3774	&3774	&38	&-6	&-6	\\
615	&4774	&4774	&-42	&4	&4	\\
643	&5994	&5994	&42	&0	&0	\\
671	&7494	&7494	&-50	&-6	&-6	\\
699	&9348	&9348	&60	&6	&6	\\
727	&11586	&11586	&-62	&0	&0	\\
755	&14320	&14320	&64	&-8	&-8	\\
783	&17654	&17654	&-74	&8	&8	\\
811	&21654	&21654	&86	&0	&0	\\
839	&26488	&26488	&-88	&-8	&-8	\\
867	&32334	&32334	&94	&9	&9	\\
895	&39324	&39324	&-108	&0	&0	\\
923	&47680	&47680	&120	&-8	&-8	\\
951	&57688	&57688	&-128	&10	&10	\\
979	&69600	&69600	&136	&0	&0	\\
1007	&83760	&83760	&-152	&-12	&-12	\\
1035	&100596	&100596	&172	&12	&12	\\
\bottomrule
\end{tabular}
\end{minipage}
\begin{minipage}[t]{0.49\linewidth}
\centering
\caption{\label{tab:coeffs:7_2}$H^{(7)}_{g,2}$, $\rs=A_6^4$}\smallskip
\begin{tabular}{r|rrrrrrrrrrr}
\toprule
$[g]$	&1A	&2A	&4A	&3AB	&6AB	\\
\midrule
$\G_g$&		$1$&	$1|4$&	$2|8$&	$3$&	$3|4$\\
\midrule
24	&4	&-4	&0	&-2	&2	\\
52	&12	&-12	&0	&0	&0	\\
80	&20	&-20	&0	&2	&-2	\\
108	&32	&-32	&0	&-1	&1	\\
136	&48	&-48	&0	&0	&0	\\
164	&80	&-80	&0	&2	&-2	\\
192	&108	&-108	&0	&-3	&3	\\
220	&168	&-168	&0	&0	&0	\\
248	&232	&-232	&0	&4	&-4	\\
276	&328	&-328	&0	&-2	&2	\\
304	&444	&-444	&0	&0	&0	\\
332	&620	&-620	&0	&2	&-2	\\
360	&812	&-812	&0	&-4	&4	\\
388	&1104	&-1104	&0	&0	&0	\\
416	&1444	&-1444	&0	&4	&-4	\\
444	&1904	&-1904	&0	&-4	&4	\\
472	&2460	&-2460	&0	&0	&0	\\
500	&3208	&-3208	&0	&4	&-4	\\
528	&4080	&-4080	&0	&-6	&6	\\
556	&5244	&-5244	&0	&0	&0	\\
584	&6632	&-6632	&0	&8	&-8	\\
612	&8400	&-8400	&0	&-6	&6	\\
640	&10524	&-10524	&0	&0	&0	\\
668	&13224	&-13224	&0	&6	&-6	\\
696	&16408	&-16408	&0	&-8	&8	\\
724	&20436	&-20436	&0	&0	&0	\\
752	&25216	&-25216	&0	&10	&-10	\\
780	&31120	&-31120	&0	&-8	&8	\\
808	&38148	&-38148	&0	&0	&0	\\
836	&46784	&-46784	&0	&8	&-8	\\
864	&56976	&-56976	&0	&-12	&12	\\
892	&69432	&-69432	&0	&0	&0	\\
920	&84144	&-84144	&0	&12	&-12	\\
948	&101904	&-101904	&0	&-12	&12	\\
976	&122868	&-122868	&0	&0	&0	\\
1004	&148076	&-148076	&0	&14	&-14	\\
1032	&177656	&-177656	&0	&-16	&16	\\
1060	&213072	&-213072	&0	&0	&0	\\
\bottomrule
\end{tabular}
\end{minipage}
\end{table}

\begin{table}
\begin{minipage}[t]{0.49\linewidth}
\centering
\caption{\label{tab:coeffs:7_3}$H^{(7)}_{g,3}$, $\rs=A_6^4$}\smallskip
\begin{tabular}{r|rrrrrrrrrrr}\toprule
$[g]$	&1A	&2A	&4A	&3AB	&6AB	\\
\midrule
$\G_g$&		$1$&	$1|4$&	$2|8$&	$3$&	$3|4$\\
\midrule
19	&6	&6	&-2	&0	&0	\\
47	&12	&12	&4	&0	&0	\\
75	&22	&22	&-2	&1	&1	\\
103	&36	&36	&4	&0	&0	\\
131	&58	&58	&-6	&-2	&-2	\\
159	&90	&90	&2	&0	&0	\\
187	&132	&132	&-4	&0	&0	\\
215	&190	&190	&6	&-2	&-2	\\
243	&274	&274	&-6	&1	&1	\\
271	&384	&384	&8	&0	&0	\\
299	&528	&528	&-8	&0	&0	\\
327	&722	&722	&10	&2	&2	\\
355	&972	&972	&-12	&0	&0	\\
383	&1300	&1300	&12	&-2	&-2	\\
411	&1724	&1724	&-12	&2	&2	\\
439	&2256	&2256	&16	&0	&0	\\
467	&2938	&2938	&-22	&-2	&-2	\\
495	&3806	&3806	&22	&2	&2	\\
523	&4890	&4890	&-22	&0	&0	\\
551	&6244	&6244	&28	&-2	&-2	\\
579	&7940	&7940	&-28	&2	&2	\\
607	&10038	&10038	&30	&0	&0	\\
635	&12620	&12620	&-36	&-4	&-4	\\
663	&15814	&15814	&38	&4	&4	\\
691	&19722	&19722	&-46	&0	&0	\\
719	&24490	&24490	&50	&-2	&-2	\\
747	&30310	&30310	&-50	&4	&4	\\
775	&37362	&37362	&58	&0	&0	\\
803	&45908	&45908	&-68	&-4	&-4	\\
831	&56236	&56236	&68	&4	&4	\\
859	&68646	&68646	&-74	&0	&0	\\
887	&83556	&83556	&84	&-6	&-6	\\
915	&101436	&101436	&-92	&6	&6	\\
943	&122790	&122790	&102	&0	&0	\\
971	&148254	&148254	&-106	&-6	&-6	\\
999	&178566	&178566	&118	&6	&6	\\
1027	&214548	&214548	&-132	&0	&0	\\
1055	&257190	&257190	&142	&-6	&-6	\\
\bottomrule
\end{tabular}
\end{minipage}
\begin{minipage}[t]{0.49\linewidth}
\centering
\caption{\label{tab:coeffs:7_4}$H^{(7)}_{g,4}$, $\rs=A_6^4$}\smallskip
\begin{tabular}{r|rrrrrrrrrrr}\toprule
$[g]$	&1A	&2A	&4A	&3AB	&6AB	\\
\midrule
$\G_g$&		$1$&	$1|4$&	$2|8$&	$3$&	$3|4$\\
\midrule
12	&4	&-4	&0	&1	&-1	\\
40	&12	&-12	&0	&0	&0	\\
68	&16	&-16	&0	&-2	&2	\\
96	&36	&-36	&0	&0	&0	\\
124	&48	&-48	&0	&0	&0	\\
152	&84	&-84	&0	&0	&0	\\
180	&116	&-116	&0	&2	&-2	\\
208	&180	&-180	&0	&0	&0	\\
236	&244	&-244	&0	&-2	&2	\\
264	&360	&-360	&0	&0	&0	\\
292	&480	&-480	&0	&0	&0	\\
320	&676	&-676	&0	&-2	&2	\\
348	&896	&-896	&0	&2	&-2	\\
376	&1224	&-1224	&0	&0	&0	\\
404	&1588	&-1588	&0	&-2	&2	\\
432	&2128	&-2128	&0	&1	&-1	\\
460	&2736	&-2736	&0	&0	&0	\\
488	&3588	&-3588	&0	&0	&0	\\
516	&4576	&-4576	&0	&4	&-4	\\
544	&5904	&-5904	&0	&0	&0	\\
572	&7448	&-7448	&0	&-4	&4	\\
600	&9500	&-9500	&0	&2	&-2	\\
628	&11892	&-11892	&0	&0	&0	\\
656	&14992	&-14992	&0	&-2	&2	\\
684	&18628	&-18628	&0	&4	&-4	\\
712	&23256	&-23256	&0	&0	&0	\\
740	&28688	&-28688	&0	&-4	&4	\\
768	&35532	&-35532	&0	&3	&-3	\\
796	&43560	&-43560	&0	&0	&0	\\
824	&53528	&-53528	&0	&-4	&4	\\
852	&65256	&-65256	&0	&6	&-6	\\
880	&79656	&-79656	&0	&0	&0	\\
908	&96564	&-96564	&0	&-6	&6	\\
936	&117196	&-117196	&0	&4	&-4	\\
964	&141360	&-141360	&0	&0	&0	\\
992	&170600	&-170600	&0	&-4	&4	\\
1020	&204848	&-204848	&0	&8	&-8	\\
1048	&245988	&-245988	&0	&0	&0	\\
\bottomrule
\end{tabular}
\end{minipage}
\end{table}

\begin{table}
\begin{minipage}[t]{0.49\linewidth}
\centering
\caption{\label{tab:coeffs:7_5}$H^{(7)}_{g,5}$, $\rs=A_6^4$}\smallskip
\begin{tabular}{r|rrrrrrrrrrr}
\toprule
$[g]$	&1A	&2A	&4A	&3AB	&6AB	\\
\midrule
$\G_g$&		$1$&	$1|4$&	$2|8$&	$3$&	$3|4$\\
\midrule
3	&2	&2	&2	&-1	&-1	\\
31	&6	&6	&-2	&0	&0	\\
59	&14	&14	&-2	&2	&2	\\
87	&22	&22	&-2	&-2	&-2	\\
115	&36	&36	&4	&0	&0	\\
143	&56	&56	&0	&2	&2	\\
171	&82	&82	&2	&-2	&-2	\\
199	&126	&126	&-2	&0	&0	\\
227	&182	&182	&6	&2	&2	\\
255	&250	&250	&-6	&-2	&-2	\\
283	&354	&354	&2	&0	&0	\\
311	&490	&490	&-6	&4	&4	\\
339	&656	&656	&8	&-4	&-4	\\
367	&882	&882	&-6	&0	&0	\\
395	&1180	&1180	&4	&4	&4	\\
423	&1550	&1550	&-10	&-4	&-4	\\
451	&2028	&2028	&12	&0	&0	\\
479	&2638	&2638	&-10	&4	&4	\\
507	&3394	&3394	&10	&-5	&-5	\\
535	&4362	&4362	&-14	&0	&0	\\
563	&5562	&5562	&18	&6	&6	\\
591	&7032	&7032	&-16	&-6	&-6	\\
619	&8886	&8886	&14	&0	&0	\\
647	&11166	&11166	&-18	&6	&6	\\
675	&13940	&13940	&28	&-7	&-7	\\
703	&17358	&17358	&-26	&0	&0	\\
731	&21536	&21536	&24	&8	&8	\\
759	&26594	&26594	&-30	&-10	&-10	\\
787	&32742	&32742	&38	&0	&0	\\
815	&40180	&40180	&-36	&10	&10	\\
843	&49124	&49124	&36	&-10	&-10	\\
871	&59916	&59916	&-44	&0	&0	\\
899	&72852	&72852	&52	&12	&12	\\
927	&88296	&88296	&-56	&-12	&-12	\\
955	&106788	&106788	&52	&0	&0	\\
983	&128816	&128816	&-64	&14	&14	\\
1011	&154948	&154948	&76	&-14	&-14	\\
\bottomrule
\end{tabular}
\end{minipage}
\begin{minipage}[t]{0.49\linewidth}
\centering
\caption{\label{tab:coeffs:7_6}$H^{(7)}_{g,6}$, $\rs=A_6^4$}\smallskip
\begin{tabular}{r|rrrrrrrrrrr}
\toprule
$[g]$	&1A	&2A	&4A	&3AB	&6AB	\\
\midrule
$\G_g$&		$1$&	$1|4$&	$2|8$&	$3$&	$3|4$\\
\midrule
20	&4	&-4	&0	&-2	&2	\\
48	&4	&-4	&0	&1	&-1	\\
76	&12	&-12	&0	&0	&0	\\
104	&12	&-12	&0	&0	&0	\\
132	&32	&-32	&0	&2	&-2	\\
160	&36	&-36	&0	&0	&0	\\
188	&64	&-64	&0	&-2	&2	\\
216	&80	&-80	&0	&2	&-2	\\
244	&132	&-132	&0	&0	&0	\\
272	&160	&-160	&0	&-2	&2	\\
300	&252	&-252	&0	&3	&-3	\\
328	&312	&-312	&0	&0	&0	\\
356	&448	&-448	&0	&-2	&2	\\
384	&572	&-572	&0	&2	&-2	\\
412	&792	&-792	&0	&0	&0	\\
440	&992	&-992	&0	&-4	&4	\\
468	&1348	&-1348	&0	&4	&-4	\\
496	&1680	&-1680	&0	&0	&0	\\
524	&2220	&-2220	&0	&-6	&6	\\
552	&2776	&-2776	&0	&4	&-4	\\
580	&3600	&-3600	&0	&0	&0	\\
608	&4460	&-4460	&0	&-4	&4	\\
636	&5712	&-5712	&0	&6	&-6	\\
664	&7044	&-7044	&0	&0	&0	\\
692	&8892	&-8892	&0	&-6	&6	\\
720	&10932	&-10932	&0	&6	&-6	\\
748	&13656	&-13656	&0	&0	&0	\\
776	&16672	&-16672	&0	&-8	&8	\\
804	&20672	&-20672	&0	&8	&-8	\\
832	&25116	&-25116	&0	&0	&0	\\
860	&30856	&-30856	&0	&-8	&8	\\
888	&37352	&-37352	&0	&8	&-8	\\
916	&45564	&-45564	&0	&0	&0	\\
944	&54884	&-54884	&0	&-10	&10	\\
972	&66572	&-66572	&0	&11	&-11	\\
1000	&79848	&-79848	&0	&0	&0	\\
1028	&96256	&-96256	&0	&-14	&14	\\
\bottomrule
\end{tabular}
\end{minipage}
\end{table}

\clearpage

\begin{table}
\vspace{-16pt}
\subsection{Lambency 8}
\vspace{-12pt}
\begin{minipage}[t]{0.49\linewidth}
\centering
\caption{\label{tab:coeffs:8_1}$H^{(8)}_{g,1}$, $\rs=A_7^2D_5^2$}\smallskip
\begin{tabular}{r|rrrr}\toprule
$[g]$	&1A	&2A	&2BC &4A\\
\midrule
$\G_g$	&$1|1$	&$1|2'$	&$2|1$	&$2|4'$	\\
\midrule
-1	&-2	&-2	&-2	&-2	\\
31	&2	&2	&-2	&2	\\
63	&4	&4	&0	&-4	\\
95	&8	&8	&0	&0	\\
127	&12	&12	&0	&-4	\\
159	&14	&14	&-2	&6	\\
191	&28	&28	&0	&-4	\\
223	&36	&36	&0	&4	\\
255	&50	&50	&2	&-6	\\
287	&70	&70	&-2	&6	\\
319	&96	&96	&0	&-8	\\
351	&126	&126	&-2	&6	\\
383	&174	&174	&2	&-10	\\
415	&220	&220	&-4	&12	\\
447	&292	&292	&4	&-12	\\
479	&380	&380	&0	&12	\\
511	&490	&490	&2	&-14	\\
543	&610	&610	&-6	&18	\\
575	&790	&790	&2	&-18	\\
607	&986	&986	&-2	&18	\\
639	&1240	&1240	&4	&-24	\\
671	&1546	&1546	&-6	&26	\\
703	&1926	&1926	&6	&-26	\\
735	&2364	&2364	&-4	&28	\\
767	&2934	&2934	&6	&-34	\\
799	&3574	&3574	&-10	&38	\\
831	&4384	&4384	&8	&-40	\\
863	&5334	&5334	&-6	&38	\\
895	&6482	&6482	&10	&-46	\\
927	&7814	&7814	&-14	&54	\\
959	&9466	&9466	&10	&-54	\\
991	&11362	&11362	&-10	&58	\\
1023	&13646	&13646	&14	&-66	\\
1055	&16312	&16312	&-16	&72	\\
1087	&19498	&19498	&14	&-78	\\
1119	&23178	&23178	&-14	&82	\\
1151	&27598	&27598	&18	&-90	\\
1183	&32662	&32662	&-22	&102	\\
1215	&38678	&38678	&22	&-106	\\
1247	&45646	&45646	&-18	&110	\\
\bottomrule
\end{tabular}
\end{minipage}
\begin{minipage}[t]{0.49\linewidth}
\centering
\caption{\label{tab:coeffs:8_2}$H^{(8)}_{g,2}$, $\rs=A_7^2D_5^2$}\smallskip
\begin{tabular}{r|rrrr}\toprule
$[g]$	&1A	&2A	&2BC &4A\\
\midrule
$\G_g$	&$1|1$	&$1|2'$	&$2|1$	&$2|4'$	\\
\midrule
28	&4	&-4	&0	&0	\\
60	&8	&-8	&0	&0	\\
92	&12	&-12	&0	&0	\\
124	&20	&-20	&0	&0	\\
156	&32	&-32	&0	&0	\\
188	&44	&-44	&0	&0	\\
220	&64	&-64	&0	&0	\\
252	&92	&-92	&0	&0	\\
284	&124	&-124	&0	&0	\\
316	&172	&-172	&0	&0	\\
348	&232	&-232	&0	&0	\\
380	&304	&-304	&0	&0	\\
412	&404	&-404	&0	&0	\\
444	&528	&-528	&0	&0	\\
476	&680	&-680	&0	&0	\\
508	&876	&-876	&0	&0	\\
540	&1120	&-1120	&0	&0	\\
572	&1416	&-1416	&0	&0	\\
604	&1788	&-1788	&0	&0	\\
636	&2248	&-2248	&0	&0	\\
668	&2796	&-2796	&0	&0	\\
700	&3476	&-3476	&0	&0	\\
732	&4304	&-4304	&0	&0	\\
764	&5292	&-5292	&0	&0	\\
796	&6500	&-6500	&0	&0	\\
828	&7948	&-7948	&0	&0	\\
860	&9672	&-9672	&0	&0	\\
892	&11748	&-11748	&0	&0	\\
924	&14224	&-14224	&0	&0	\\
956	&17156	&-17156	&0	&0	\\
988	&20648	&-20648	&0	&0	\\
1020	&24784	&-24784	&0	&0	\\
1052	&29652	&-29652	&0	&0	\\
1084	&35412	&-35412	&0	&0	\\
1116	&42188	&-42188	&0	&0	\\
1148	&50120	&-50120	&0	&0	\\
1180	&59440	&-59440	&0	&0	\\
1212	&70344	&-70344	&0	&0	\\
1244	&83052	&-83052	&0	&0	\\
1276	&97896	&-97896	&0	&0	\\
\bottomrule
\end{tabular}
\end{minipage}
\end{table}

\begin{table}
\begin{minipage}[t]{0.49\linewidth}
\centering
\caption{\label{tab:coeffs:8_3}$H^{(8)}_{g,3}$, $\rs=A_7^2D_5^2$}\smallskip
\begin{tabular}{r|rrrr}\toprule
$[g]$	&1A	&2A	&2BC &4A\\
\midrule
$\G_g$	&$1|1$	&$1|2'$	&$2|1$	&$2|4'$	\\
\midrule
23	&4	&4	&0	&-4	\\
55	&10	&10	&2	&2	\\
87	&14	&14	&-2	&-2	\\
119	&26	&26	&2	&2	\\
151	&36	&36	&0	&-4	\\
183	&58	&58	&2	&2	\\
215	&76	&76	&-4	&-4	\\
247	&116	&116	&4	&4	\\
279	&154	&154	&-2	&-6	\\
311	&216	&216	&4	&8	\\
343	&282	&282	&-6	&-6	\\
375	&390	&390	&6	&6	\\
407	&500	&500	&-4	&-12	\\
439	&666	&666	&6	&10	\\
471	&848	&848	&-8	&-8	\\
503	&1108	&1108	&8	&12	\\
535	&1400	&1400	&-8	&-16	\\
567	&1798	&1798	&10	&14	\\
599	&2240	&2240	&-12	&-16	\\
631	&2842	&2842	&14	&18	\\
663	&3524	&3524	&-12	&-20	\\
695	&4406	&4406	&14	&22	\\
727	&5418	&5418	&-18	&-22	\\
759	&6722	&6722	&18	&26	\\
791	&8206	&8206	&-18	&-34	\\
823	&10082	&10082	&22	&34	\\
855	&12238	&12238	&-26	&-34	\\
887	&14918	&14918	&26	&38	\\
919	&18012	&18012	&-24	&-44	\\
951	&21806	&21806	&30	&46	\\
983	&26176	&26176	&-36	&-48	\\
1015	&31510	&31510	&38	&54	\\
1047	&37650	&37650	&-38	&-62	\\
1079	&45050	&45050	&42	&66	\\
1111	&53582	&53582	&-50	&-66	\\
1143	&63816	&63816	&52	&72	\\
1175	&75578	&75578	&-50	&-86	\\
1207	&89584	&89584	&56	&88	\\
1239	&105694	&105694	&-66	&-90	\\
\bottomrule
\end{tabular}
\end{minipage}
\begin{minipage}[t]{0.49\linewidth}
\centering
\caption{\label{tab:coeffs:8_4}$H^{(8)}_{g,4}$, $\rs=A_7^2D_5^2$}\smallskip
\begin{tabular}{r|rrrr}\toprule
$[g]$	&1A	&2A	&2BC &4A\\
\midrule
$\G_g$	&$1|1$	&$1|2'$	&$2|1$	&$2|4'$	\\
\midrule
16	&4	&-4	&0	&0	\\
48	&8	&-8	&0	&0	\\
80	&16	&-16	&0	&0	\\
112	&24	&-24	&0	&0	\\
144	&36	&-36	&0	&0	\\
176	&56	&-56	&0	&0	\\
208	&80	&-80	&0	&0	\\
240	&112	&-112	&0	&0	\\
272	&160	&-160	&0	&0	\\
304	&216	&-216	&0	&0	\\
336	&288	&-288	&0	&0	\\
368	&392	&-392	&0	&0	\\
400	&516	&-516	&0	&0	\\
432	&672	&-672	&0	&0	\\
464	&880	&-880	&0	&0	\\
496	&1128	&-1128	&0	&0	\\
528	&1440	&-1440	&0	&0	\\
560	&1840	&-1840	&0	&0	\\
592	&2320	&-2320	&0	&0	\\
624	&2912	&-2912	&0	&0	\\
656	&3648	&-3648	&0	&0	\\
688	&4536	&-4536	&0	&0	\\
720	&5616	&-5616	&0	&0	\\
752	&6936	&-6936	&0	&0	\\
784	&8516	&-8516	&0	&0	\\
816	&10416	&-10416	&0	&0	\\
848	&12720	&-12720	&0	&0	\\
880	&15456	&-15456	&0	&0	\\
912	&18720	&-18720	&0	&0	\\
944	&22632	&-22632	&0	&0	\\
976	&27248	&-27248	&0	&0	\\
1008	&32728	&-32728	&0	&0	\\
1040	&39232	&-39232	&0	&0	\\
1072	&46872	&-46872	&0	&0	\\
1104	&55872	&-55872	&0	&0	\\
1136	&66472	&-66472	&0	&0	\\
1168	&78880	&-78880	&0	&0	\\
1200	&93400	&-93400	&0	&0	\\
1232	&110400	&-110400	&0	&0	\\
\bottomrule
\end{tabular}
\end{minipage}
\end{table}

\begin{table}
\begin{minipage}[t]{0.49\linewidth}
\centering
\caption{\label{tab:coeffs:8_5}$H^{(8)}_{g,5}$, $\rs=A_7^2D_5^2$}\smallskip
\begin{tabular}{r|rrrr}\toprule
$[g]$	&1A	&2A	&2BC &4A\\
\midrule
$\G_g$	&$1|1$	&$1|2'$	&$2|1$	&$2|4'$	\\
\midrule
7	&2	&2	&-2	&2	\\
39	&8	&8	&0	&0	\\
71	&10	&10	&-2	&2	\\
103	&22	&22	&2	&-2	\\
135	&28	&28	&-4	&4	\\
167	&48	&48	&4	&0	\\
199	&64	&64	&-4	&0	\\
231	&100	&100	&4	&-4	\\
263	&126	&126	&-6	&6	\\
295	&188	&188	&4	&-4	\\
327	&244	&244	&-4	&4	\\
359	&340	&340	&8	&-4	\\
391	&430	&430	&-10	&6	\\
423	&588	&588	&8	&-4	\\
455	&742	&742	&-10	&6	\\
487	&984	&984	&12	&-8	\\
519	&1234	&1234	&-14	&10	\\
551	&1598	&1598	&14	&-10	\\
583	&1994	&1994	&-14	&10	\\
615	&2546	&2546	&18	&-14	\\
647	&3142	&3142	&-22	&14	\\
679	&3964	&3964	&20	&-12	\\
711	&4870	&4870	&-22	&14	\\
743	&6064	&6064	&28	&-16	\\
775	&7398	&7398	&-30	&22	\\
807	&9132	&9132	&28	&-20	\\
839	&11076	&11076	&-32	&20	\\
871	&13558	&13558	&38	&-26	\\
903	&16358	&16358	&-42	&30	\\
935	&19868	&19868	&44	&-28	\\
967	&23854	&23854	&-46	&30	\\
999	&28788	&28788	&52	&-36	\\
1031	&34384	&34384	&-60	&40	\\
1063	&41256	&41256	&60	&-40	\\
1095	&49090	&49090	&-62	&42	\\
1127	&58566	&58566	&74	&-50	\\
1159	&69382	&69382	&-82	&54	\\
1191	&82394	&82394	&82	&-54	\\
1223	&97234	&97234	&-90	&58	\\
\bottomrule
\end{tabular}
\end{minipage}
\begin{minipage}[t]{0.49\linewidth}
\centering
\caption{\label{tab:coeffs:8_6}$H^{(8)}_{g,6}$, $\rs=A_7^2D_5^2$}\smallskip
\begin{tabular}{r|rrrr}\toprule
$[g]$	&1A	&2A	&2BC &4A\\
\midrule
$\G_g$	&$1|1$	&$1|2'$	&$2|1$	&$2|4'$	\\
\midrule
28	&4	&-4	&0	&0	\\
60	&8	&-8	&0	&0	\\
92	&12	&-12	&0	&0	\\
124	&20	&-20	&0	&0	\\
156	&32	&-32	&0	&0	\\
188	&44	&-44	&0	&0	\\
220	&64	&-64	&0	&0	\\
252	&92	&-92	&0	&0	\\
284	&124	&-124	&0	&0	\\
316	&172	&-172	&0	&0	\\
348	&232	&-232	&0	&0	\\
380	&304	&-304	&0	&0	\\
412	&404	&-404	&0	&0	\\
444	&528	&-528	&0	&0	\\
476	&680	&-680	&0	&0	\\
508	&876	&-876	&0	&0	\\
540	&1120	&-1120	&0	&0	\\
572	&1416	&-1416	&0	&0	\\
604	&1788	&-1788	&0	&0	\\
636	&2248	&-2248	&0	&0	\\
668	&2796	&-2796	&0	&0	\\
700	&3476	&-3476	&0	&0	\\
732	&4304	&-4304	&0	&0	\\
764	&5292	&-5292	&0	&0	\\
796	&6500	&-6500	&0	&0	\\
828	&7948	&-7948	&0	&0	\\
860	&9672	&-9672	&0	&0	\\
892	&11748	&-11748	&0	&0	\\
924	&14224	&-14224	&0	&0	\\
956	&17156	&-17156	&0	&0	\\
988	&20648	&-20648	&0	&0	\\
1020	&24784	&-24784	&0	&0	\\
1052	&29652	&-29652	&0	&0	\\
1084	&35412	&-35412	&0	&0	\\
1116	&42188	&-42188	&0	&0	\\
1148	&50120	&-50120	&0	&0	\\
1180	&59440	&-59440	&0	&0	\\
1212	&70344	&-70344	&0	&0	\\
1244	&83052	&-83052	&0	&0	\\
\bottomrule
\end{tabular}
\end{minipage}
\end{table}

\begin{table}
\begin{minipage}[t]{0.49\linewidth}
\centering
\caption{\label{tab:coeffs:8_7}$H^{(8)}_{g,7}$, $\rs=A_7^2D_5^2$}\smallskip
\begin{tabular}{r|rrrr}\toprule
$[g]$	&1A	&2A	&2BC &4A\\
\midrule
$\G_g$	&$1|1$	&$1|2'$	&$2|1$	&$2|4'$	\\
\midrule
15	&2	&2	&2	&2	\\
47	&2	&2	&-2	&2	\\
79	&8	&8	&4	&0	\\
111	&6	&6	&-2	&-2	\\
143	&18	&18	&2	&2	\\
175	&16	&16	&-4	&0	\\
207	&34	&34	&6	&2	\\
239	&40	&40	&-4	&0	\\
271	&66	&66	&6	&2	\\
303	&72	&72	&-8	&0	\\
335	&120	&120	&8	&0	\\
367	&140	&140	&-8	&-4	\\
399	&202	&202	&10	&2	\\
431	&248	&248	&-12	&0	\\
463	&346	&346	&14	&2	\\
495	&412	&412	&-12	&-4	\\
527	&566	&566	&14	&6	\\
559	&678	&678	&-18	&-2	\\
591	&900	&900	&20	&4	\\
623	&1092	&1092	&-20	&-4	\\
655	&1406	&1406	&22	&6	\\
687	&1692	&1692	&-28	&-4	\\
719	&2170	&2170	&30	&2	\\
751	&2608	&2608	&-28	&-8	\\
783	&3266	&3266	&34	&10	\\
815	&3930	&3930	&-38	&-6	\\
847	&4878	&4878	&42	&6	\\
879	&5828	&5828	&-44	&-12	\\
911	&7178	&7178	&46	&10	\\
943	&8552	&8552	&-56	&-8	\\
975	&10420	&10420	&60	&12	\\
1007	&12404	&12404	&-60	&-12	\\
1039	&14992	&14992	&68	&16	\\
1071	&17756	&17756	&-76	&-12	\\
1103	&21358	&21358	&82	&14	\\
1135	&25212	&25212	&-84	&-20	\\
1167	&30108	&30108	&92	&20	\\
1199	&35464	&35464	&-104	&-16	\\
1231	&42140	&42140	&112	&20	\\
	\bottomrule
\end{tabular}
\end{minipage}
\end{table}

\clearpage

\begin{table}
\vspace{-16pt}
\subsection{Lambency 9}
\begin{minipage}[b]{0.49\linewidth}
\centering
\caption{\label{tab:coeffs:9_1}
$H^{(9)}_{g,1}$, $\rs=A_8^3$}\smallskip
\begin{tabular}{r|rrrrrr}\toprule
$[g]$	&1A	&2A	&2B	&2C	&3A	&6A	\\
\midrule
	$\G_g$&		$1$&$1|4$&$2$&$2|2$&	${3|3}$&	${3|12}$\\	
\midrule
-1& -2& -2& -2& -2& -2& -2\\
35& 0& 0& 0& 0& 0& 0\\
71& 4& 4& 0& 0& -2&-2\\
107& 4& 4& 0& 0& -2& -2\\
143& 6& 6& -2& -2& 0& 0\\
179& 12& 12& 0&0& 0& 0\\
215& 16& 16& 0& 0& -2& -2\\
251& 20& 20& 0& 0& 2& 2\\
287& 30&30& -2& -2& 0& 0\\
323& 40& 40& 0& 0& -2& -2\\
359& 54& 54& 2& 2& 0& 0\\
395& 72& 72& 0& 0& 0& 0\\
431& 92& 92& 0& 0& -4& -4\\
467& 116& 116& 0& 0&2& 2\\
503& 156& 156& 0& 0& 0& 0\\
539& 196& 196& 0& 0& -2& -2\\
575& 242&242& -2& -2& 2& 2\\
611& 312& 312& 0& 0& 0& 0\\
647& 388& 388& 0& 0& -2&-2\\
683& 476& 476& 0& 0& 2& 2\\
719& 594& 594& -2& -2& 0& 0\\
755& 728&728& 0& 0& -4& -4\\
791& 890& 890& 2& 2& 2& 2\\
827& 1092& 1092& 0& 0& 0&0\\
863& 1322& 1322& -2& -2& -4& -4\\
899& 1600& 1600& 0& 0& 4& 4\\
935&1938& 1938& 2& 2& 0& 0\\
971& 2324& 2324& 0& 0& -4& -4\\
1007& 2782& 2782&-2& -2& 4& 4\\
1043& 3336& 3336& 0& 0& 0& 0\\
1079& 3978& 3978& 2& 2& -6&-6\\
1115& 4720& 4720& 0& 0& 4& 4\\
1151& 5610& 5610& -2& -2& 0& 0\\
1187&6636& 6636& 0& 0& -6& -6\\
1223& 7830& 7830& 2& 2& 6& 6\\
1259& 9236& 9236&0& 0& 2& 2\\
1295& 10846& 10846& -2& -2& -8& -8\\
1331& 12720& 12720& 0& 0&6& 6\\
\bottomrule
\end{tabular}
\end{minipage}
\begin{minipage}[b]{0.49\linewidth}
\centering
\caption{\label{tab:coeffs:9_2}$H^{(9)}_{g,2}$, $\rs=A_8^3$}\smallskip
\begin{tabular}{r|rrrrrr}
\toprule
$[g]$	&1A	&2A	&2B	&2C	&3A	&6A	\\
\midrule
	$\G_g$&		$1$&$1|4$&$2$&$2|2$&	${3|3}$&	${3|12}$\\	
\midrule
32& 4& -4& 0& 0& -2& 2\\     
68& 6& -6& 2& -2& 0& 0\\     
104& 8& -8& 0& 0& 2& -2\\     
140& 16& -16& 0& 0& -2& 2\\     
176& 18& -18& -2& 2& 0& 0\\     
212& 30& -30&2& -2& 0& 0\\     
248& 40& -40& 0& 0& -2& 2\\    
284& 54& -54& 2& -2& 0& 0\\     
320& 74& -74& -2& 2& 2& -2\\     
356& 102& -102& 2& -2& 0& 0\\     
392& 126& -126& -2& 2& 0& 0\\     
428& 174& -174& 2& -2& 0& 0\\     
464& 218& -218& -2& 2& -4& 4\\     
500& 284& -284& 4& -4& 2& -2\\     
536& 356& -356& -4& 4& 2& -2\\     
572& 460& -460& 4& -4& -2& 2\\    
 608& 564& -564& -4& 4& 0& 0\\     
644& 716& -716& 4& -4& 2& -2\\     
680& 884& -884& -4& 4& -4& 4\\     
716& 1098& -1098& 6& -6& 0& 0\\     
752& 1342& -1342& -6& 6& 4& -4\\     
788& 1658& -1658& 6& -6& -4& 4\\     
824& 2000& -2000& -8& 8& 2& -2\\     
860& 2456& -2456& 8& -8& 2& -2\\     
896& 2960& -2960& -8& 8& -4& 4\\     
932& 3582& -3582& 10& -10& 0& 0\\    
 968& 4294& -4294& -10& 10& 4& -4\\     
1004& 5174& -5174& 10& -10& -4& 4\\     
1040& 6156& -6156& -12& 12& 0& 0\\     
1076& 7378& -7378& 14& -14& 4& -4\\     
1112& 8748& -8748& -12& 12& -6& 6\\     
1148& 10400& -10400& 16& -16& 2& -2\\     
1184& 12280& -12280& -16& 16& 4& -4\\     
1220& 14544& -14544& 16& -16& -6& 6\\     
1256& 17084& -17084& -20& 20& 2& -2\\     
1292& 20140& -20140& 20& -20& 4& -4\\     
1328& 23590& -23590& -22& 22& -8& 8\\     
1364& 27656& -27656& 24& -24& 2& -2\\\bottomrule
\end{tabular}
\end{minipage}
\end{table}

\begin{table}
\begin{minipage}[b]{0.49\linewidth}
\centering
\caption{\label{tab:coeffs:9_3}$H^{(9)}_{g,3}$, $\rs=A_8^3$}\smallskip
\begin{tabular}{r|@{ }r@{ }r@{ }r@{ }r@{ }r@{ }r}
\toprule
$[g]$	&1A	&2A	&2B	&2C	&3A	&6A	\\
\midrule
	$\G_g$&		$1$&$1|4$&$2$&$2|2$&	${3|3}$&	${3|12}$\\	
\midrule
27& 4& 4& 0& 0& 0& 0\\63& 6& 6& -2& -2& 0& 0\\99& 12& 12& 0& 0& 0&0\\
135& 18& 18& 2& 2& 0& 0\\171& 24& 24& 0& 0& 0& 0\\207& 36& 36& 0&0& 0& 0\\243& 52& 52& 0& 0& 0& 0\\279& 72& 72& 0& 0& 0& 0\\
315& 96&96& 0& 0& 0& 0\\351& 126& 126& -2& -2& 0& 0\\387& 168& 168& 0& 0& 0& 0\\
423& 222& 222& 2& 2& 0& 0\\459& 288& 288& 0& 0& 0& 0\\495& 366&366& -2& -2& 0& 0\\531& 468& 468& 0& 0& 0& 0\\
567& 594& 594& 2& 2& 0& 0\\603& 744& 744& 0& 0& 0& 0\\639& 930& 930& -2& -2& 0& 0\\
675& 1156&1156& 0& 0& 0& 0\\711& 1434& 1434& 2& 2& 0& 0\\
747& 1764& 1764& 0& 0& 0& 0\\783& 2160& 2160& 0& 0& 0& 0\\819& 2640& 2640& 0& 0& 0& 0\\
855& 3210& 3210& 2& 2& 0& 0\\891& 3888& 3888& 0& 0& 0& 0\\
927& 4692& 4692& -4& -4& 0& 0\\963& 5652& 5652& 0& 0& 0& 0\\
999& 6786& 6786& 2& 2&0& 0\\1035& 8112& 8112& 0& 0& 0& 0\\1071& 9678& 9678& -2& -2& 0& 0\\
1107& 11520& 11520& 0& 0& 0& 0\\1143& 13674& 13674& 2& 2& 0& 0\\
1179& 16188& 16188& 0& 0& 0& 0\\1215& 19116& 19116& -4& -4& 0&  0\\
1251& 22536& 22536& 0& 0& 0& 0\\1287& 26508& 26508& 4& 4& 0& 0\\1323& 31108& 31108& 0& 0& 0& 0\\1359& 36438& 36438& -2& -2& 0& 0\\
\bottomrule
\end{tabular}
\end{minipage}
\begin{minipage}[b]{0.49\linewidth}
\centering
\caption{\label{tab:coeffs:9_4}$H^{(9)}_{g,4}$, $\rs=A_8^3$}\smallskip
\begin{tabular}{r|@{ }r@{ }r@{ }r@{ }r@{ }r@{ }r}
\toprule
$[g]$	&1A	&2A	&2B	&2C	&3A	&6A	\\
\midrule
	$\G_g$&		$1$&$1|4$&$2$&$2|2$&	${3|3}$&	${3|12}$\\	
\midrule
20& 2& -2& -2& 2& 2& -2\\56& 8& -8& 0& 0& 2& -2\\92& 10& -10& -2& 2& -2&2\\128& 18& -18& 2& -2& 0& 0\\164& 26& -26& -2& 2& 2& -2\\200& 42& -42& 2& -2& 0& 0\\236& 50& -50& -2& 2& 2& -2
\\272& 78& -78& 2& -2& 0& 0\\308& 100& -100& -4& 4& -2& 2\\344& 140& -140& 4& -4& 2& -2\\380&180& -180& -4& 4& 0& 0\\416& 244& -244& 4& -4& -2& 2\\452&302& -302& -6& 6& 2& -2\\488& 404& -404& 4& -4& 2& -2\\524&502& -502& -6& 6& -2& 2\\560& 648& -648& 8& -8& 0& 0\\
596&806& -806& -6& 6& 2& -2\\632& 1024& -1024& 8& -8& -2& 2\\668&1250& -1250& -10& 10& 2& -2\\704& 1574& -1574& 10& -10& 2& -2\\
740& 1916& -1916& -12& 12& -4& 4\\
776& 2372& -2372& 12& -12& 2& -2\\812& 2876& -2876& -12& 12& 2& -2\\848& 3530& -3530& 14& -14& -4& 4\\
884& 4240& -4240& -16& 16& 4& -4\\920& 5168& -5168& 16& -16& 2& -2\\
956& 6186& -6186& -18& 18& -6& 6\\992& 7460& -7460& 20& -20& 2& -2\\
1028& 8894& -8894& -22& 22& 2& -2\\1064& 10664& -10664& 24& -24& -4& 4\\1100& 12634& -12634& -26& 26& 4& -4\\
1136& 15070& -15070& 26& -26& 4& -4\\
1172& 17790& -17790& -30& 30& -6& 6\\1208& 21080& -21080& 32& -32& 2& -2\\1244& 24794& -24794& -34& 34& 2& -2\\
1280& 29250& -29250& 38& -38& -6& 6\\
1316& 34248& -34248& -40& 40&6& -6\\1352& 40234& -40234& 42& -42& 4& -4\\
\bottomrule
\end{tabular}
\end{minipage}
\end{table}

\begin{table}
\begin{minipage}[b]{0.49\linewidth}
\centering
\caption{\label{tab:coeffs:9_5}$H^{(9)}_{g,5}$, $\rs=A_8^3$}\smallskip
\begin{tabular}{r|@{ }r@{ }r@{ }r@{ }r@{ }r@{ }r}
\toprule
$[g]$	&1A	&2A	&2B	&2C	&3A	&6A	\\
\midrule
	$\G_g$&		$1$&$1|4$&$2$&$2|2$&	${3|3}$&	${3|12}$\\	
\midrule
11&4&4&0&0&-2&-2\\ 47&6&6&2&2&0&0\\ 83&12&12&0&0&0&0\\ 119&16&16&0&0&-2&-2\\ 155&24&24&0&0&0&0\\ 191&36&36&0&0&0&0\\ 227&52&52&0&0&-2&-2\\ 263&68&68&0&0&2&2\\ 299&96&96&0&0&0&0\\ 335&130&130&2&2&-2&-2\\ 371&168&168&0&0&0&0\\ 407&224&224&0&0&2&2\\ 443&292&292&0&0&-2&-2\\ 479&372&372&0&0&0&0\\ 515&480&480&0&0&0&0\\ 551&608&608&0&0&-4&-4\\ 587&764&764&0&0&2&2\\ 623&962&962&2&2&2&2\\ 659&1196&1196&0&0&-4&-4\\ 695&1478&1478&-2&-2&2&2\\ 731&1832&1832&0&0&2&2\\ 767&2248&2248&0&0&-2&-2\\ 803&2744&2744&0&0&2&2\\ 839&3348&3348&0&0&0&0\\ 875&4064&4064&0&0&-4&-4\\ 911&4910&4910&2&2&2&2\\ 947&5924&5924&0&0&2&2\\ 983&7116&7116&0&0&-6&-6\\ 1019&8516&8516&0&0&2&2\\ 1055&10184&10184&0&0&2&2\\ 1091&12132&12132&0&0&-6&-6\\ 1127&14404&14404&0&0&4&4\\ 1163&17084&17084&0&0&2&2\\ 1199&20202&20202&2&2&-6&-6\\ 1235&23824&23824&0&0&4&4\\ 1271&28054&28054&-2&-2&4&4\\ 1307&32956&32956&0&0&-8&-8\\ 1343&38626&38626&2&2&4&4\\
\bottomrule
\end{tabular}
\end{minipage}
\begin{minipage}[b]{0.49\linewidth}
\centering
\caption{\label{tab:coeffs:9_6}$H^{(9)}_{g,6}$, $\rs=A_8^3$}\smallskip
%
\end{minipage}
\end{table}

\begin{table}
\begin{minipage}[b]{0.49\linewidth}
\centering
\caption{\label{tab:coeffs:9_7}$H^{(9)}_{g,7}$, $\rs=A_8^3$}\smallskip
%
\end{minipage}
\end{table}

\clearpage

\begin{table}
\vspace{-12pt}
\subsection{Lambency 10}
\vspace{-12pt}
\begin{minipage}[t]{0.33\linewidth}
\centering
\caption{\label{tab:coeffs:10_1}$H^{(10)}_{g,1}$, $\rs=A_{9}^2D_6$}\smallskip
%
\end{minipage}
\begin{minipage}[t]{0.33\linewidth}
\centering
\caption{\label{tab:coeffs:10_2}$H^{(10)}_{g,2}$, $\rs=A_{9}^2D_6$}\smallskip
%
\end{minipage}
\begin{minipage}[t]{0.33\linewidth}
\centering
\caption{\label{tab:coeffs:10_3}$H^{(10)}_{g,3}$, $\rs=A_{9}^2D_6$}\smallskip
%
\end{minipage}
\end{table}

\begin{table}
\begin{minipage}[t]{0.33\linewidth}
\centering
\caption{\label{tab:coeffs:10_4}$H^{(10)}_{g,4}$, $\rs=A_{9}^2D_6$}\smallskip
%
\end{minipage}
\begin{minipage}[t]{0.33\linewidth}
\centering
\caption{\label{tab:coeffs:10_5}$H^{(10)}_{g,5}$, $\rs=A_{9}^2D_6$}\smallskip
%
\end{minipage}
\end{table}

\begin{table}
\begin{minipage}[t]{0.33\linewidth}
\centering
\caption{\label{tab:coeffs:10_7}$H^{(10)}_{g,7}$, $\rs=A_{9}^2D_6$}\smallskip
%
\end{minipage}
\end{table}

\clearpage

\begin{table}
\vspace{-16pt}
\subsection{Lambency 10+5}
\begin{minipage}[t]{0.49\linewidth}
\centering
\caption{\label{tab:coeffs:10+5_1}$H^{(10+5)}_{g,1}$, $\rs=D_6^4$}\smallskip
%
\end{minipage}
\begin{minipage}[t]{0.49\linewidth}
\centering
\caption{\label{tab:coeffs:10+5_3}$H^{(10+5)}_{g,3}$, $\rs=D_6^4$}\smallskip
%
\end{minipage}
\end{table}

\begin{table}
\vspace{-16pt}
\begin{minipage}[t]{0.49\linewidth}
\centering
\caption{\label{tab:coeffs:10+5_5}$H^{(10+5)}_{g,5}$, $\rs=D_6^4$}\smallskip
%
\end{minipage}
\end{table}

\clearpage

\begin{sidewaystable}
\subsection{Lambency 12}
\vspace{-12pt}
\centering
\caption{\label{tab:coeffs:12_1A}$H^{(12)}_{1A}$, $\rs=A_{11}D_7E_6$}\smallskip
\begin{tabular}{r|rrrrrrrrrrr}\toprule
$4k\ell,r$	&1	&2	&3	&4	&5	&6	&7	&8	&9	&10	&11	\\	\midrule
$0$	&-2	&	&	&	&	&	&	&	&	&	&	\\
$48$	&0	&2	&2	&2	&2	&2	&0	&0	&0	&0	&0	\\
$96$	&2	&2	&4	&4	&4	&4	&4	&2	&2	&0	&0	\\
$144$	&2	&4	&4	&6	&6	&6	&4	&4	&2	&2	&0	\\
$192$	&2	&6	&8	&8	&10	&8	&8	&6	&6	&2	&2	\\
$240$	&4	&6	&10	&12	&10	&12	&10	&8	&4	&4	&0	\\
$288$	&6	&10	&14	&16	&18	&16	&16	&12	&10	&6	&4	\\
\bottomrule
\end{tabular}
\end{sidewaystable}

\begin{sidewaystable}
\centering
\caption{$H^{(12)}_{2A}$, $\rs=A_{11}D_7E_6$}\smallskip
%
\end{sidewaystable}

\begin{table}
\vspace{-16pt}
\subsection{Lambency 12+4}
\centering
\caption{\label{tab:coeffs12+4_1}McKay--Thompson series $H^{(12+4)}_{g,1}=H^{\rs}_{g,1}$ for $\rs=E_6^4$}\smallskip
%
\end{table}

\begin{table}
\vspace{-16pt}
\centering
\caption{\label{tab:coeffs12+4_4}McKay--Thompson series $H^{(12+4)}_{g,4}=H^{\rs}_{g,4}$ for $\rs=E_6^4$}\smallskip
%
\end{table}

\begin{table}
\vspace{-16pt}
\centering
\caption{\label{tab:coeffs12+4_5}McKay--Thompson series $H^{(12+4)}_{g,5}=H^{\rs}_{g,5}$ for $\rs=E_6^4$}\smallskip
%
\end{table}

\clearpage

\begin{table}
\vspace{-12pt}
\subsection{Lambency 13}
\vspace{-12pt}
\begin{minipage}[t]{0.33\linewidth}
\centering
\caption{\label{tab:coeffs:13_1}$H^{(13)}_{g,1}$, $\rs=A_{12}^2$}\smallskip
%
\end{minipage}
\begin{minipage}[t]{0.33\linewidth}
\centering
\caption{\label{tab:coeffs:13_2}$H^{(13)}_{g,2}$, $\rs=A_{12}^2$}\smallskip
%
\end{minipage}
\begin{minipage}[t]{0.33\linewidth}
\centering
\caption{\label{tab:coeffs:13_3}$H^{(13)}_{g,3}$, $\rs=A_{12}^2$}\smallskip
%
\end{minipage}
\end{table}

\begin{table}
\begin{minipage}[t]{0.33\linewidth}
\centering
\caption{\label{tab:coeffs:13_4}$H^{(13)}_{g,4}$, $\rs=A_{12}^2$}\smallskip
%
\end{minipage}
\begin{minipage}[t]{0.33\linewidth}
\centering
\caption{\label{tab:coeffs:13_5}$H^{(13)}_{g,5}$, $\rs=A_{12}^2$}\smallskip
%
\end{minipage}
\end{table}

\begin{table}
\begin{minipage}[t]{0.33\linewidth}
\centering
\caption{\label{tab:coeffs:13_7}$H^{(13)}_{g,7}$, $\rs=A_{12}^2$}\smallskip
%
\end{minipage}
\end{table}

\begin{table}
\begin{minipage}[t]{0.33\linewidth}
\centering
\caption{\label{tab:coeffs:13_10}$H^{(13)}_{g,10}$, $\rs=A_{12}^2$}\smallskip
%
\end{minipage}
\end{table}

\clearpage

\begin{table}
\vspace{-12pt}
\subsection{Lambency 14+7}
\vspace{-12pt}
\begin{minipage}[t]{0.49\linewidth}
\centering
\caption{\label{tab:coeffs:14+7_1}$H^{(14+7)}_{g,1}$, $\rs=D_8^3$}\smallskip
%
\end{minipage}
\begin{minipage}[t]{0.49\linewidth}
\centering
\caption{\label{tab:coeffs:14+7_3}$H^{(14+7)}_{g,3}$, $\rs=D_8^3$}\smallskip
%
\end{minipage}
\end{table}

\begin{table}
\begin{minipage}[t]{0.49\linewidth}
\centering
\caption{\label{tab:coeffs:14+7_5}$H^{(14+7)}_{g,5}$, $\rs=D_8^3$}\smallskip
%
\end{minipage}
\begin{minipage}[t]{0.49\linewidth}
\centering
\caption{\label{tab:coeffs:14+7_7}$H^{(14+7)}_{g,7}$, $\rs=D_8^3$}\smallskip
%
\end{minipage}
\end{table}
\clearpage

\begin{sidewaystable}
\subsection{Lambency 16}
\vspace{-12pt}
\centering
\caption{\label{tab:coeffs:16_1A}$H^{(16)}_{1A}$, $\rs=A_{15}D_9$}\smallskip
%
\end{sidewaystable}

\begin{sidewaystable}
\centering
\caption{$H^{(16)}_{2A}$, $\rs=A_{15}D_9$}\smallskip
%
\end{sidewaystable}

\clearpage

\begin{sidewaystable}
\subsection{Lambency 18}
\vspace{-12pt}
\centering
\caption{\label{tab:coeffs:18_1A}$H^{(18)}_{1A}$, $\rs=A_{17}E_7$}\smallskip
%
\end{sidewaystable}

\begin{sidewaystable}
\centering
\caption{$H^{(18)}_{2A}$, $\rs=A_{17}E_7$}\smallskip
%
\end{sidewaystable}

\clearpage

\begin{table}
\vspace{-12pt}
\subsection{Lambency 18+9}
\vspace{-12pt}
\begin{minipage}[t]{0.33\linewidth}
\centering
\caption{\label{tab:coeffs:18+9_1}$H^{(18+9)}_{g,1}$, $\rs=D_{10}E_7^2$}\smallskip
%
\end{minipage}
\begin{minipage}[t]{0.33\linewidth}
\centering
\caption{\label{tab:coeffs:18+9_3}$H^{(18+9)}_{g,3}$, $\rs=D_{10}E_7^2$}\smallskip
%
\end{minipage}
\begin{minipage}[t]{0.33\linewidth}
\centering
\caption{\label{tab:coeffs:18+9_5}$H^{(18+9)}_{g,5}$, $\rs=D_{10}E_7^2$}\smallskip
%
\end{minipage}
\end{table}

\begin{table}
\begin{minipage}[t]{0.33\linewidth}
\centering
\caption{\label{tab:coeffs:18+9_7}$H^{(18+9)}_{g,7}$, $\rs=D_{10}E_7^2$}\smallskip
%
\end{minipage}
\begin{minipage}[t]{0.33\linewidth}
\centering
\caption{\label{tab:coeffs:18+9_9}$H^{(18+9)}_{g,9}$, $\rs=D_{10}E_7^2$}\smallskip
%
\end{minipage}
\end{table}

\clearpage

\begin{table}
\vspace{-12pt}
\subsection{Lambency 22+11}
\vspace{-12pt}
\begin{minipage}[t]{0.33\linewidth}
\centering
\caption{\label{tab:coeffs:22+11_1}$H^{(22+11)}_{g,1}$, $\rs=D_{12}^2$}\smallskip
%
\end{minipage}
\begin{minipage}[t]{0.33\linewidth}
\centering
\caption{\label{tab:coeffs:22+11_3}$H^{(22+11)}_{g,3}$, $\rs=D_{12}^2$}\smallskip
%
\end{minipage}
\begin{minipage}[t]{0.33\linewidth}
\centering
\caption{\label{tab:coeffs:22+11_5}$H^{(22+11)}_{g,5}$, $\rs=D_{12}^2$}\smallskip
%
\end{minipage}
\end{table}

\begin{table}
\begin{minipage}[t]{0.33\linewidth}
\centering
\caption{\label{tab:coeffs:22+11_7}$H^{(22+11)}_{g,7}$, $\rs=D_{12}^2$}\smallskip
%
\end{minipage}
\begin{minipage}[t]{0.33\linewidth}
\centering
\caption{\label{tab:coeffs:22+11_9}$H^{(22+11)}_{g,9}$, $\rs=D_{12}^2$}\smallskip
%
\end{minipage}
\end{table}

\clearpage

\begin{sidewaystable}
\subsection{Lambency 25}
\vspace{-12pt}
\centering
\caption{\label{tab:coeffs:25_1A}$H^{(25)}_{1A}$, $\rs=A_{24}$}\smallskip
\begin{small}
%%
\end{small}
\end{sidewaystable}

\begin{sidewaystable}
\centering
\caption{\label{tab:coeffs:25_2A}$H^{(25)}_{2A}$, $\rs=A_{24}$}\smallskip
\begin{small}
%
\end{small}
\end{sidewaystable}

\clearpage

\begin{table}
%\begin{sidewaystable}
\vspace{-12pt}
\subsection{Lambency 30+15}
\vspace{-12pt}
\centering
\caption{\label{tab:coeffs:30+15_1A}$H^{(30+15)}_{1A}$, $\rs=D_{16}E_8$}\smallskip
%\begin{small}
%
%\end{small}
%\end{sidewaystable}
\end{table}

\begin{table}
\vspace{-12pt}
\subsection{Lambency 30+6,10,15}
\begin{minipage}[t]{0.49\linewidth}
\centering
\caption{\label{tab:coeffs:30+6,10,15_1}$H^{(30+6,10,15)}_{g,1}$, $\rs=E_8^3$}\smallskip
%
\end{minipage}
\begin{minipage}[t]{0.49\linewidth}
\centering
\caption{\label{tab:coeffs:30+6,10,15_7}$H^{(30+6,10,15)}_{g,7}$, $\rs=E_8^3$}\smallskip
%
\end{minipage}
\end{table}

\clearpage

\begin{table}
%\begin{sidewaystable}
\vspace{-12pt}
\subsection{Lambency 46+23}
\vspace{-12pt}
\centering
\caption{\label{tab:coeffs:46+23_1A}$H^{(46+23)}_{1A}$, $\rs=D_{24}$}\smallskip
%\begin{small}
\begin{tabular}{r|rrrrrrrrrrrr}\toprule
$4k\ell,r$	&1	&3	&5	&7	&9	&11	&13	&15	&17	&19	&21	&23	\\	\midrule
$0$	&-2	&	&	&	&	&	&	&	&	&	&	&	\\
$184$	&2	&0	&0	&0	&2	&0	&2	&	&	&	&	&	\\
$368$	&0	&0	&0	&2	&0	&2	&0	&2	&2	&2	&	&	\\
$552$	&0	&0	&2	&0	&0	&2	&0	&2	&2	&0	&2	&2	\\
$736$	&0	&0	&0	&2	&0	&2	&2	&2	&0	&2	&2	&2	\\
$920$	&0	&0	&2	&0	&2	&0	&2	&0	&2	&2	&2	&0	\\
$1104$	&0	&2	&0	&0	&2	&2	&2	&2	&2	&2	&2	&2	\\
\bottomrule
\end{tabular}
%\end{small}
%\end{sidewaystable}
\end{table}

\clearpage

%------------------------------------------------------------------------%
\section{Decompositions}\label{sec:decompositions}
%------------------------------------------------------------------------%

As explained in \S\ref{sec:mckay} (see also \S\ref{sec:conj:mod}) our conjectural proposals for the umbral McKay--Thompson series $H^{X}_{g,r}(\t)=\sum_dc^{X}_{g,r}(d)q^{d}$ (cf. \S\S\ref{sec:mckay},\ref{sec:coeffs}) determine the $G^{X}$-modules $K^{X}_{r,d}$ up to isomorphism for $d>0$, at least for those values of $d$ for which we can identify all the Fourier coefficients $c^{X}_{g,r}(d)$. In this section we furnish tables of explicit decompositions of $K^{X}_{r,d}$ into irreducible representations for $G^{X}$, for the first few values of $d$. The coefficient $c^{X}_r(d)$ of $H^{X}_{g,r}$ is non-zero only when $d=n-r^2/4m$ for some integer $n\geq 0 $, where $m$ denotes the Coxeter number of $X$. For each of the tables in this section the rows are labelled by the values $4m d$, so that the entry in row $k$ and column $\chi_i$ indicates the multiplicity of the irreducible representation of $G^X$ with character $\chi_i$ (in the notation of 
\S\ref{sec:chars:irr}) appearing in the $G^{X}$-module $K^{X}_{r,k/4m}$. One can observe that these tables support Conjectures \ref{conj:conj:mod:Kell}, \ref{conj:conj:disc:dualpair} and \ref{conj:conj:disc:doub}, as well as Conjecture \ref{conj:conj:mod:factoring_through}.

As we have seen in \S\ref{sec:holes:gps}, for some Niemeier root systems, especially those with higher Coxeter numbers, the umbral groups $G^X$ are very small 
and the decompositions of the $G^X$-module into irreducibles are  extremely simple and do not need to be tabulated. For those cases we give instead the following expression for the module in terms of the umbral McKay--Thompson series $H^X_g$ whose low order coefficients are tabulated in \S\ref{sec:coeffs}.

For the Niemeier root system $X=A_9^2 D_6$ corresponding to  
$\ll=10$, we have $G^X \cong4$ and the decomposition is simply given by
\begin{align}
\sum_{\substack{d>0\\40d=-r^2 \text{ mod }~40}} q^d K^{(10)}_{r,d} &= \frac{\chi_3+\chi_4}{2} H^{(10)}_{1A,r}, \quad r {\text{ even}},\\ 
\sum_{\substack{d>0\\40d=-r^2 \text{ mod }~40}} q^d K^{(10)}_{r,d} &= \frac{\chi_1+\chi_2}{2} H^{(10)}_{1A,r} + \frac{\chi_1-\chi_2}{2} H^{(10)}_{4AB,r} , ~ r{\text{ odd}}.
 \end{align}

For the six Niemeier root systems 
\begin{gather}
	A_{11}D_7E_6,\; A_{15}D_9,\; A_{17}E_7,\; A_{24},\; D_{10}E_7^2,\;D_{12}^2,
\end{gather}
corresponding to $\ll\in\{12,16,18,25,18+9,22+11\}$ the associated umbral group is $G^X \cong 2$ and the decomposition of the representations is simply given by
\be
\sum_{\substack{d>0\\4m d=-r^2 \text{ mod }~4m}} q^d K^{X}_{r,d} =  \frac{\chi_1+\chi_2}{2} H^{X}_{1A,r} + \frac{\chi_1-\chi_2}{2} H^{X}_{2A,r} ,
\ee
where $m$ again denotes the Coxeter number of $X$. 
In particular, for the A-type cases above, begin $A_{11}D_7E_6$, $A_{15}D_9$, $A_{17}E_7$, and $A_{24}$, this reduces to
\begin{align}
\sum_{\substack{d>0\\4m d=-r^2 \text{ mod }~4m}} q^d K^{X}_{r,d} &=  {\chi_1} H^{X}_{1A,r}, \quad r{\text{ odd},}
\\
\sum_{\substack{d>0\\4m d=-r^2 \text{ mod }~4m}} q^d K^{X}_{r,d} &=  {\chi_2} H^{X}_{1A,r} , \quad r{\text{ even},}
\end{align}
where $m$ again denotes the Coxeter number of $X$. 

For the two Niemeier lattices $D_{16}E_8$ and $D_{24}$
the umbral group is $G^X \cong 1$ and the decomposition is completely trivial.  

\clearpage

 \begin{sidewaystable}

\subsection{Lambency 2}

\caption{Decomposition of $K^{(2)}_1=K^X_1$ for $X=A_1^{24}$}\vspace{.2cm}
 \begin{center}
 \begin{small}

\end{small}
\end{center}
\end{sidewaystable}

 \begin{table} 
\vspace{-16pt}
\subsection{Lambency 3}
\vspace{-16pt}
 \begin{center}
     \caption{  $K^{(3)}_1 = K^X_1$ for $X=A_2^{12}$ }
\vspace{.1cm}
 %
\end{center}
 
 \begin{center}                 
 \caption{  $K^{(3)}_2 = K^X_2$ for $X=A_2^{12}$ }    
\vspace{.1cm}
 %
                  \end{center}
   \end{table}

\begin{table}
\vspace{-16pt}
\subsection{Lambency 4}
\vspace{-16pt}

 \begin{center}
     \caption{  $K^{(4)}_1= K^X_1$ for $X=A_3^{8}$ }
    \vspace{0.1cm}
 %
\end{center}
 \begin{center}
     \caption{  $K^{(4)}_3= K^X_3$ for $X=A_3^{8}$}
    \vspace{0.1cm}

 %
\end{center}

 \begin{center}
  \caption{  $K^{(4)}_2= K^X_2$ for $X=A_3^{8}$}
    \vspace{0.1cm}

 %
\end{center}
\end{table}

\begin{table} 
%\vspace{-16pt}
\subsection{Lambency 5}
%\vspace{-16pt}
\begin{minipage}[t]{0.5\linewidth}
 \begin{center}
     \caption{  $K^{(5)}_1$, $X=A_4^{6}$}
    \vspace{0.1cm}

 %
\end{center}
 \begin{center}
     \caption{  $K^{(5)}_3$, $X=A_4^{6}$}
    \vspace{0.1cm}

 %
\end{center}
\end{minipage}
\begin{minipage}[t]{0.5\linewidth}
 \begin{center}
  \caption{  $K^{(5)}_2$, $X=A_4^{6}$}
    \vspace{0.1cm}

 %
\end{center}
 \begin{center}
  \caption{  $K^{(5)}_4$, $X=A_4^{6}$ }
    \vspace{0.1cm}

 %
\end{center}
\end{minipage}
\end{table}

\begin{table}   
%\vspace{-16pt}
\subsection{Lambency 6}   
% \vspace{-16pt}
 \begin{minipage}[t]{0.5\linewidth}
 
 \begin{center}
     \caption{  $K^{(6)}_1$, $X=A_5^{4} D_4$}
    \vspace{0.1cm}
 %
\end{center}
 \begin{center}
     \caption{  $K^{(6)}_3$, $X=A_5^{4} D_4$}
    \vspace{0.1cm}

 %
\end{center}
 \begin{center}
     \caption{  $K^{(6)}_5$, $X=A_5^{4} D_4$}
    \vspace{0.1cm}

 %
\end{center}
\end{minipage}
\begin{minipage}[t]{0.5\linewidth}
 \begin{center}
     \caption{  $K^{(6)}_2$, $X=A_5^{4} D_4$}
    \vspace{0.1cm}

 %
\end{center}
 \begin{center}
     \caption{  $K^{(6)}_4$, $X=A_5^{4} D_4$}
    \vspace{0.1cm}

  %
\end{center}

\end{minipage}\end{table}

\clearpage

\begin{table} 
\subsection{Lambency 6+3}  
\centering
     \caption{  $K^{(6+3)}_1$, $X=D_4^{6} $}
    \vspace{0.1cm}

 %
\centering
   \vspace{1cm}
     \caption{  $K^{(6+3)}_3$, $X=D_4^{6} $}\smallskip
  %
  \end{table}

 \newpage 
  
\begin{table} %\vspace{-2pt}  
\subsection{Lambency 7}   % \vspace{-16pt}
\begin{minipage}[t]{0.5\linewidth}
 \begin{center}
     \caption{  $K^{(7)}_1$, $X=A_6^{4} $}
    \vspace{0.1cm}

 %
\end{center}
 \begin{center}
     \caption{  $K^{(7)}_3$, $X=A_6^{4} $}
    \vspace{0.1cm}

 %
\end{center}
 \begin{center}
     \caption{  $K^{(7)}_5$, $X=A_6^{4} $}
    \vspace{0.1cm}

 %
\end{center}
\end{minipage}
\begin{minipage}[t]{0.5\linewidth}
 \begin{center}
     \caption{  $K^{(7)}_2$, $X=A_6^{4} $}
    \vspace{0.1cm}

 %
\end{center}
 \begin{center}
     \caption{  $K^{(7)}_4$, $X=A_6^{4} $}
    \vspace{0.1cm}

  %
\end{center}
\end{minipage}\end{table}

\begin{table}
%\vspace{-16pt}
\subsection{Lambency 8}
%\vspace{-16pt}
Consider the case $\ll=8$, $X=A_7^2D_5^2$. For the components with $r=0$ mod $2$ we have the module structure
\begin{gather}
%\sum_{d>0, ~32d =-4r^2 \text{ mod }32} q^d K_{2r,d}^{(8)} = \frac{1}{2}\chi_5 H_{2r}^{(8)}.
\sum_{\substack{d>0\\ 32d =-r^2 \text{ mod }32}} q^d K_{r,d}^{(8)} = \frac{1}{2}\chi_5 H_{r}^{(8)}.
\end{gather}

\begin{minipage}[t]{0.5\linewidth}
 \begin{center}
     \caption{  $K^{(8)}_1$, $X=A_7^{2} D_5^2 $}
    \vspace{0.1cm}
 %

\vspace{1cm}

\caption{  $K^{(8)}_5$, $X=A_7^{2} D_5^2 $}
    \vspace{0.1cm}
 %

\end{center}
\end{minipage}
\begin{minipage}[t]{0.5\linewidth}
 \begin{center}
   \caption{  $K^{(8)}_3$, $X=A_7^{2} D_5^2 $}
    \vspace{0.1cm}
 %

\vspace{1cm}

\caption{  $K^{(8)}_7$, $X=A_7^{2} D_5^2 $}
    \vspace{0.2cm}
 %

\end{center}
\end{minipage}
\end{table}

\begin{table}
%\vspace{-16pt}
\subsection{Lambency 9}
%\vspace{-16pt}
\begin{minipage}[t]{0.5\linewidth}
 \begin{center}
     \caption{  $K^{(9)}_1$, $X=A_8^{3}  $}\label{tab:dec:9_1}
    \vspace{0.1cm}
 %
\end{center}
 \begin{center}
     \caption{  $K^{(9)}_5$, $X=A_8^{3}  $}\label{tab:dec:9_5}
    \vspace{0.1cm}

 %
\end{center}
 \begin{center}
     \caption{  $K^{(9)}_7$, $X=A_8^{3}  $}\label{tab:dec:9_7}
    \vspace{0.1cm}

 %
\end{center}
\end{minipage}
\begin{minipage}[t]{0.5\linewidth}
 \begin{center}
     \caption{  $K^{(9)}_2$, $X=A_8^{3}  $}\label{tab:dec:9_2}
    \vspace{0.1cm}

 %
\end{center}
 \begin{center}
     \caption{  $K^{(9)}_4$, $X=A_8^{3}  $}\label{tab:dec:9_4}
    \vspace{0.1cm}

  %
\end{center}
\end{minipage}\end{table}

\begin{table} \begin{minipage}[t]{0.5\linewidth}
 \begin{center}
     \caption{  $K^{(9)}_3$, $X=A_8^{3}  $}\label{tab:dec:9_3}
    \vspace{0.1cm}

 %
\end{center}
\end{minipage}
\begin{minipage}[t]{0.5\linewidth}
 \begin{center}
     \caption{  $K^{(9)}_6$, $X=A_8^{3}  $}\label{tab:dec:9_6}
    \vspace{0.1cm}

 %
\end{center}
\end{minipage}\end{table}

\begin{table}

%\vspace{-16pt}
\subsection{Lambency 10+5}
%\vspace{-16pt}

\begin{minipage}[t]{0.5\linewidth}
 \begin{center}
     \caption{  $K^{(10+5)}_1$, $X=D_6^{4}  $}
    \vspace{0.1cm}

 %

    \vspace{1cm}
\caption{  $K^{(10+5)}_3$, $X=D_6^{4}  $}
    \vspace{0.1cm}

 %

\end{center}
\end{minipage}
\begin{minipage}[t]{0.5\linewidth}
 \begin{center}
   \caption{  $K^{(10+5)}_5$, $X=D_6^{4}  $}
    \vspace{0.1cm}

%
\end{center}
\end{minipage}
\end{table}

\begin{table} 
%\vspace{-16pt}
\subsection{Lambency 12+4}
%\vspace{-16pt}

\begin{minipage}[t]{0.6\linewidth}
 \begin{center}
     \caption{  $K^{(12+4)}_1$, $X=E_6^{4}  $}
    \vspace{0.1cm}

 %

\vspace{1cm}
\caption{  $K^{(12+4)}_5$, $X=E_6^{4}  $}
    \vspace{0.1cm}

 %

\end{center}
\end{minipage}
\begin{minipage}[t]{0.4\linewidth}
 \begin{center}
   \caption{  $K^{(12+4)}_4$, $X=E_6^{4}  $}
    \vspace{0.1cm}

%
\end{center}
\end{minipage}
\end{table}

\begin{table}
%\vspace{-16pt}
\subsection{Lambency 13}
%\vspace{-16pt}
\begin{minipage}[t]{0.5\linewidth}
 \begin{center}
     \caption{  $K^{(13)}_1$, $X=A_{12}^{2}$}
   % \vspace{0.2cm}
 %
\end{center}
 \begin{center}
     \caption{  $K^{(13)}_3$, $X=A_{12}^{2}$}
    \vspace{0.1cm}

 %
\end{center}
 \begin{center}
     \caption{  $K^{(13)}_5$, $X=A_{12}^{2}$}
    \vspace{0.1cm}
 %
\end{center}
\end{minipage}
\begin{minipage}[t]{0.5\linewidth}
 \begin{center}
     \caption{  $K^{(13)}_2$, $X=A_{12}^{2}$}
    %\vspace{0.2cm}
 %
\end{center}
 \begin{center}
     \caption{  $K^{(13)}_4$, $X=A_{12}^{2}$}
    \vspace{0.1cm}

  %
\end{center}
\end{minipage}\end{table}

\begin{table}
\begin{minipage}[t]{0.5\linewidth}
 \begin{center}
     \caption{  $K^{(13)}_7$, $X=A_{12}^{2}$}
   % \vspace{0.2cm}
 %
\end{center}
 \begin{center}
     \caption{  $K^{(13)}_{11}$, $X=A_{12}^{2}$}
    \vspace{0.2cm}
 %
\end{center}
\end{minipage}
\begin{minipage}[t]{0.5\linewidth}
 \begin{center}
     \caption{  $K^{(13)}_8$, $X=A_{12}^{2}$}
   % \vspace{0.2cm}
 %
\end{center}
 \begin{center}
     \caption{  $K^{(13)}_{10}$, $X=A_{12}^{2}$}
    \vspace{0.2cm}

  %
\end{center}
 \begin{center}
     \caption{  $K^{(13)}_{12}$, $X=A_{12}^{2}$}
    \vspace{0.1cm}

 %
\end{center}
\end{minipage}\end{table}
 
\newpage
\begin{table}[h!]

%\vspace{-16pt}
\subsection{Lambency 14+7}
%\vspace{-16pt}

\begin{minipage}[t]{0.5\linewidth}
 \begin{center}
     \caption{  $K^{(14+7)}_1$, $X=D_8^{3}$}
    %\vspace{0.2cm}

 %

\vspace{1cm}
\caption{  $K^{(14+7)}_3$, $X=D_8^{3}$}
    \vspace{0.1cm}

 %

\end{center}
\end{minipage}
\begin{minipage}[t]{0.4\linewidth}
 \begin{center}
   \caption{  $K^{(14+7)}_5$, $X=D_8^{3}$}
    \vspace{0.1cm}

%

\vspace{1cm}

\caption{  $K^{(14+7)}_7$, $X=D_8^{3}$}
    \vspace{0.1cm}

 %

\end{center}
\end{minipage}
\end{table}

\clearpage

\begin{table}
%\vspace{-16pt}
\subsection{Lambency 30+6,10,15}
%\vspace{-16pt}

\begin{minipage}[t]{0.5\linewidth}
 \begin{center}
     \caption{  $K^{(30+6,10,15)}_1$, $X=E_8^{3}$}
 %   \vspace{0.2cm}

 %

\end{center}
\end{minipage}
\begin{minipage}[t]{0.5\linewidth}
 \begin{center}
   \caption{  $K^{(30+6,10,15)}_7$, $X=E_8^{3}$}
  %  \vspace{0.2cm}

%

\end{center}
\end{minipage}
\end{table}

\clearpage

%------------------------------------------------------------------%
\addcontentsline{toc}{section}{References}
%\bibliographystyle{utphys}
%\bibliography{UMNL}

\begin{thebibliography}{10}

\bibitem{UM}
M.~C.~N. Cheng, J.~F.~R. Duncan, and J.~A. Harvey, ``{Umbral Moonshine},''
\href{http://arxiv.org/abs/1204.2779}{{\tt arXiv:1204.2779 [math.RT]}}.
%%CITATION = ARXIV:1204.2779;%%.

\bibitem{Nie_DefQdtFrm24}
H.-V. Niemeier, ``Definite quadratische {F}ormen der {D}imension {$24$} und
  {D}iskriminante {$1$},'' {\em J. Number Theory} {\bf 5} (1973)  142--178.

\bibitem{Lee_SphPkgs}
J.~Leech, ``{Notes on Sphere Packings},'' {\em Canad. J. Math.} {\bf 19} (1967)
   251--267.

\bibitem{Tho_FinGpsModFns}
J.~G. Thompson, ``Finite groups and modular functions,'' {\em Bull. London
  Math. Soc.} {\bf 11} (1979) no.~3, 347--351.

\bibitem{conway_norton}
J.~H. Conway and S.~P. Norton, ``{Monstrous Moonshine},'' {\em Bull. London
  Math. Soc.} {\bf 11} (1979)  308~339.

\bibitem{gannon}
T.~Gannon, {\em Moonshine beyond the monster. The bridge connecting algebra,
  modular forms and physics.}
\newblock Cambridge University Press, 2006.

\bibitem{Eguchi1987}
T.~Eguchi and A.~Taormina, ``Unitary representations of the {$N=4$}
  superconformal algebra,''
  \href{http://dx.doi.org/10.1016/0370-2693(87)91679-0}{{\em Phys. Lett. B}
  {\bf 196} (1987) no.~1, 75--81}.
  \url{http://dx.doi.org/10.1016/0370-2693(87)91679-0}.

\bibitem{Eguchi1988}
T.~Eguchi and A.~Taormina, ``Character formulas for the {$N=4$} superconformal
  algebra,'' \href{http://dx.doi.org/10.1016/0370-2693(88)90778-2}{{\em Phys.
  Lett. B} {\bf 200} (1988) no.~3, 315--322}.
  \url{http://dx.doi.org/10.1016/0370-2693(88)90778-2}.

\bibitem{Eguchi1989}
T.~Eguchi, H.~Ooguri, A.~Taormina, and S.-K. Yang, ``{Superconformal Algebras
  and String Compactification on Manifolds with $SU(N)$ Holonomy},''
\href{http://dx.doi.org/10.1016/0550-3213(89)90454-9}{{\em Nucl. Phys.} {\bf
  B315} (1989)  193}.
%%CITATION = NUPHA,B315,193;%%.

\bibitem{Eguchi2008}
T.~Eguchi and K.~Hikami, ``{Superconformal Algebras and Mock Theta
  Functions},''{\em J.Phys.A} {\bf 42:304010,2009} (Dec., 2008)  ,
  \href{http://arxiv.org/abs/0812.1151}{{\tt 0812.1151 [math-ph]}}.

\bibitem{Eguchi2010}
T.~Eguchi, H.~Ooguri, and Y.~Tachikawa, ``{Notes on the K3 Surface and the
  Mathieu group $M_{24}$},'' {\em Exper.Math.} {\bf 20} (2011)  91--96,
\href{http://arxiv.org/abs/1004.0956}{{\tt arXiv:1004.0956 [hep-th]}}.
%%CITATION = ARXIV:1004.0956;%%.

\bibitem{Cheng2010_1}
M.~C.~N. Cheng, ``{$K3$} {S}urfaces, {$N=4$} {D}yons, and the {M}athieu {G}roup
  {$M_{24}$},'' \href{http://arxiv.org/abs/1005.5415}{{\tt 1005.5415
  [hep-th]}}.

\bibitem{Gaberdiel2010}
M.~R. Gaberdiel, S.~Hohenegger, and R.~Volpato, ``{Mathieu twining characters
  for $K3$},'' \href{http://dx.doi.org/10.1007/JHEP09(2010)058}{{\em JHEP} {\bf
  1009} (2010)  058},
\href{http://arxiv.org/abs/1006.0221}{{\tt arXiv:1006.0221 [hep-th]}}.
%%CITATION = ARXIV:1006.0221;%%.

\bibitem{Gaberdiel2010a}
M.~R. Gaberdiel, S.~Hohenegger, and R.~Volpato, ``{Mathieu Moonshine in the
  elliptic genus of K3},''
  \href{http://dx.doi.org/10.1007/JHEP10(2010)062}{{\em JHEP} {\bf 1010} (2010)
   062},
\href{http://arxiv.org/abs/1008.3778}{{\tt arXiv:1008.3778 [hep-th]}}.
%%CITATION = ARXIV:1008.3778;%%.

\bibitem{Eguchi2010a}
T.~Eguchi and K.~Hikami, ``{Note on Twisted Elliptic Genus of K3 Surface},''
  \href{http://dx.doi.org/10.1016/j.physletb.2010.10.017}{{\em Phys.Lett.} {\bf
  B694} (2011)  446--455},
\href{http://arxiv.org/abs/1008.4924}{{\tt arXiv:1008.4924 [hep-th]}}.
%%CITATION = ARXIV:1008.4924;%%.

\bibitem{Gannon:2012ck}
T.~Gannon, ``{Much ado about Mathieu},''
\href{http://arxiv.org/abs/1211.5531}{{\tt arXiv:1211.5531 [math.RT]}}.
%%CITATION = ARXIV:1211.5531;%%.

\bibitem{CheDun_M24MckAutFrms}
M.~C.~N. Cheng and J.~F.~R. Duncan, ``{The Largest Mathieu Group and (Mock)
  Automorphic Forms},'' \href{http://arxiv.org/abs/1201.4140}{{\tt 1201.4140
  [math.RT]}}.

\bibitem{Gaberdiel:2012um}
M.~R. Gaberdiel and R.~Volpato, ``{Mathieu Moonshine and Orbifold $K3$'s},''
\href{http://arxiv.org/abs/1206.5143}{{\tt arXiv:1206.5143 [hep-th]}}.
%%CITATION = ARXIV:1206.5143;%%.

\bibitem{Creutzig2012}
T.~Creutzig, G.~Hoehn, and T.~Miezaki, ``{The McKay-Thompson series of Mathieu
  Moonshine modulo two},''
\href{http://arxiv.org/abs/1211.3703}{{\tt arXiv:1211.3703 [math.NT]}}.
%%CITATION = ARXIV:1303.3221;%%.

\bibitem{Gaberdiel:2012gf}
M.~R. Gaberdiel, D.~Persson, H.~Ronellenfitsch, and R.~Volpato, ``{Generalised
  Mathieu Moonshine},''
\href{http://arxiv.org/abs/1211.7074}{{\tt arXiv:1211.7074 [hep-th]}}.
%%CITATION = ARXIV:1211.7074;%%.

\bibitem{Gaberdiel:2013nya}
M.~R. Gaberdiel, D.~Persson, and R.~Volpato, ``{Generalised Moonshine and
  Holomorphic Orbifolds},''
\href{http://arxiv.org/abs/1302.5425}{{\tt arXiv:1302.5425 [hep-th]}}.
%%CITATION = ARXIV:1302.5425;%%.

\bibitem{Taormina:2013jza}
A.~Taormina and K.~Wendland, ``{Symmetry-surfing the Moduli Space of Kummer
  $K3$'s},''
\href{http://arxiv.org/abs/1303.2931}{{\tt arXiv:1303.2931 [hep-th]}}.
%%CITATION = ARXIV:1303.2931;%%.

\bibitem{Taormina:2013mda}
A.~Taormina and K.~Wendland, ``{A Twist in the $M_{24}$ Moonshine Story},''
\href{http://arxiv.org/abs/1303.3221}{{\tt arXiv:1303.3221 [hep-th]}}.
%%CITATION = ARXIV:1303.3221;%%.

\bibitem{Cappelli:1987xt}
A.~Cappelli, C.~Itzykson, and J.~Zuber, ``{The ADE Classification of Minimal
  and $A^{(1)}_1$ Conformal Invariant Theories},''
\href{http://dx.doi.org/10.1007/BF01221394}{{\em Commun.Math.Phys.} {\bf 113}
  (1987)  1}.
%%CITATION = CMPHA,113,1;%%.

\bibitem{MR1890629}
N.~Bourbaki, {\em Lie groups and {L}ie algebras. {C}hapters 4--6}.
\newblock Elements of Mathematics (Berlin). Springer-Verlag, Berlin, 2002.
\newblock Translated from the 1968 French original by Andrew Pressley.

\bibitem{MR0323842}
J.~E. Humphreys, {\em Introduction to {L}ie algebras and representation
  theory}.
\newblock Springer-Verlag, New York, 1972.
\newblock Graduate Texts in Mathematics, Vol. 9.

\bibitem{Con_ChrLeeLat}
J.~H. Conway, ``A characterisation of {L}eech's lattice,'' {\em Invent. Math.}
  {\bf 7} (1969)  137--142.

\bibitem{Lee_SphPkgHgrSpc}
J.~Leech, ``{Some Sphere Packings in Higher Space},'' {\em Canad. J. Math.}
  {\bf 16} (1964)  657--682.

\bibitem{ConParSlo_CvgRadLeeLat}
J.~H. Conway, R.~A. Parker, and N.~J.~A. Sloane, ``The covering radius of the
  {L}eech lattice,'' \href{http://dx.doi.org/10.1098/rspa.1982.0042}{{\em Proc.
  Roy. Soc. London Ser. A} {\bf 380} (1982) no.~1779, 261--290}.
  \url{http://dx.doi.org/10.1098/rspa.1982.0042}.

\bibitem{MR661720}
J.~H. Conway and N.~J.~A. Sloane, ``Twenty-three constructions for the {L}eech
  lattice,'' \href{http://dx.doi.org/10.1098/rspa.1982.0071}{{\em Proc. Roy.
  Soc. London Ser. A} {\bf 381} (1982) no.~1781, 275--283}.
  \url{http://dx.doi.org/10.1098/rspa.1982.0071}.

\bibitem{Shi_IntThyAutFns}
G.~Shimura, {\em Introduction to the arithmetic theory of automorphic
  functions}.
\newblock Publications of the Mathematical Society of Japan, No. 11. Iwanami
  Shoten, Publishers, Tokyo, 1971.
\newblock Kan{\^o} Memorial Lectures, No. 1.

\bibitem{Fer_Genus0prob}
C.~R. Ferenbaugh, ``The genus-zero problem for {$n\vert h$}-type groups,'' {\em
  Duke Math. J.} {\bf 72} (1993) no.~1, 31--63.

\bibitem{atlas}
J.~Conway, R.~Curtis, S.~Norton, R.~Parker, and R.~Wilson, {\em Atlas of finite
  groups. Maximal subgroups and ordinary characters for simple groups. With
  comput. assist. from J. G. Thackray.}
\newblock Oxford: Clarendon Press, 1985.

\bibitem{McKay_Corr}
J.~McKay, ``Graphs, singularities, and finite groups,'' in {\em The {S}anta
  {C}ruz {C}onference on {F}inite {G}roups ({U}niv. {C}alifornia, {S}anta
  {C}ruz, {C}alif., 1979)}, vol.~37 of {\em Proc. Sympos. Pure Math.},
  pp.~183--186.
\newblock Amer. Math. Soc., Providence, R.I., 1980.

\bibitem{Slo_SmpSngSmpAlgGps}
P.~Slodowy, {\em Simple Singularities and Simple Algebraic Groups}, vol.~815 of
  {\em Lecture Notes in Mathematics}.
\newblock Springer, Berlin, 1980.

\bibitem{GonVer_GeomCnstMcKCorr}
G.~Gonzalez-Sprinberg and J.-L. Verdier, ``Construction g\'eom\'etrique de la
  correspondance de {M}c{K}ay,'' {\em Ann. Sci. \'Ecole Norm. Sup. (4)} {\bf
  16} (1983) no.~3, 409--449 (1984).
  \url{http://www.numdam.org/item?id=ASENS_1983_4_16_3_409_0}.

\bibitem{BruFun_TwoGmtThtLfts}
J.~H. Bruinier and J.~Funke, ``On two geometric theta lifts,''
  \href{http://dx.doi.org/10.1215/S0012-7094-04-12513-8}{{\em Duke Math. J.}
  {\bf 125} (2004) no.~1, 45--90}.
  \url{http://dx.doi.org/10.1215/S0012-7094-04-12513-8}.

\bibitem{zwegers}
S.~Zwegers, {\em {Mock Theta Functions}}.
\newblock PhD thesis, Utrecht University, 2002.

\bibitem{zagier_mock}
D.~Zagier, ``Ramanujan's mock theta functions and their applications (after
  {Z}wegers and {O}no-{B}ringmann),'' {\em Ast\'erisque} (2009) no.~326, Exp.
  No. 986, vii--viii, 143--164 (2010). S{\'e}minaire Bourbaki. Vol. 2007/2008.

\bibitem{eichler_zagier}
M.~Eichler and D.~Zagier, {\em The theory of Jacobi forms}.
\newblock Birkh{\"a}user, 1985.

\bibitem{Dabholkar:2012nd}
A.~Dabholkar, S.~Murthy, and D.~Zagier, ``{Quantum Black Holes, Wall Crossing,
  and Mock Modular Forms},''
\href{http://arxiv.org/abs/1208.4074}{{\tt arXiv:1208.4074 [hep-th]}}.
%%CITATION = ARXIV:1208.4074;%%.

\bibitem{MR0332663}
G.~Shimura, ``On modular forms of half integral weight,'' {\em Ann. of Math.
  (2)} {\bf 97} (1973)  440--481.

\bibitem{Gepner:1986hr}
D.~Gepner and Z.-a. Qiu, ``{Modular Invariant Partition Functions for
  Parafermionic Field Theories},''
\href{http://dx.doi.org/10.1016/0550-3213(87)90348-8}{{\em Nucl.Phys.} {\bf
  B285} (1987)  423}.
%%CITATION = NUPHA,B285,423;%%.

\bibitem{MR958592}
N.-P. Skoruppa and D.~Zagier, ``Jacobi forms and a certain space of modular
  forms,'' \href{http://dx.doi.org/10.1007/BF01394347}{{\em Invent. Math.} {\bf
  94} (1988) no.~1, 113--146}. \url{http://dx.doi.org/10.1007/BF01394347}.

\bibitem{MR1074485}
N.-P. Skoruppa, ``Explicit formulas for the {F}ourier coefficients of {J}acobi
  and elliptic modular forms,''
  \href{http://dx.doi.org/10.1007/BF01233438}{{\em Invent. Math.} {\bf 102}
  (1990) no.~3, 501--520}. \url{http://dx.doi.org/10.1007/BF01233438}.

\bibitem{MR1116103}
N.-P. Skoruppa, ``Heegner cycles, modular forms and {J}acobi forms,'' {\em
  S\'em. Th\'eor. Nombres Bordeaux (2)} {\bf 3} (1991) no.~1, 93--116.
  \url{http://jtnb.cedram.org/item?id=JTNB_1991__3_1_93_0}.

\bibitem{Hoehn2007}
G.~Hoehn, ``{Selbstduale Vertexoperatorsuperalgebren und das Babymonster
  (Self-dual Vertex Operator Super Algebras and the Baby Monster)},''{\em
  Bonner Mathematische Schriften, Vol.} {\bf 286,} (June, 2007)
  1--85,Bonn1996, \href{http://arxiv.org/abs/0706.0236}{{\tt 0706.0236
  [math.QA]}}.

\bibitem{Witten2007}
E.~Witten, ``{Three-Dimensional Gravity Revisited},''
  \href{http://arxiv.org/abs/0706.3359}{{\tt 0706.3359 [hep-th]}}.

\bibitem{Sko_Thesis}
N.-P. Skoruppa, {\em {\"U}ber den {Z}usammenhang zwischen {J}acobiformen und
  {M}odulformen halbganzen {G}ewichts}.
\newblock Bonner Mathematische Schriften [Bonn Mathematical Publications], 159.
  Universit{\"a}t Bonn Mathematisches Institut, Bonn, 1985.
\newblock Dissertation, Rheinische Friedrich-Wilhelms-Universit{\"a}t, Bonn,
  1984.

\bibitem{Ramanujan_lost}
S.~Ramanujan, {\em The lost notebook and other unpublished papers}.
\newblock Springer-Verlag, Berlin, 1988.
\newblock With an introduction by George E. Andrews.

\bibitem{Gordon_Mcintosh}
B.~Gordon and R.~J. McIntosh, ``{Some Eighth Order Mock Theta Functions},''
  \href{http://dx.doi.org/10.1112/S0024610700008735}{{\em J. London Math. Soc.
  (2)} {\bf 62} (2000) no.~2, 321--335}.
  \url{http://dx.doi.org/10.1112/S0024610700008735}.

\bibitem{MR2317449}
R.~J. McIntosh, ``{Second Order Mock Theta Functions},''
  \href{http://dx.doi.org/10.4153/CMB-2007-028-9}{{\em Canad. Math. Bull.} {\bf
  50} (2007) no.~2, 284--290}. \url{http://dx.doi.org/10.4153/CMB-2007-028-9}.

\bibitem{Cheng2011}
M.~C.~N. Cheng and J.~F.~R. Duncan, ``{On Rademacher Sums, the Largest Mathieu
  Group, and the Holographic Modularity of Moonshine},''
  \href{http://arxiv.org/abs/1110.3859}{{\tt 1110.3859 [math.RT]}}.

\bibitem{2012arXiv1212.0906C}
M.~C.~N. {Cheng} and J.~F.~R. {Duncan}, ``{On the Discrete Groups of Mathieu
  Moonshine},''{\em ArXiv e-prints} (Dec., 2012)  ,
  \href{http://arxiv.org/abs/1212.0906}{{\tt arXiv:1212.0906 [math.NT]}}.

\bibitem{Ben_SchInd}
M.~Benard, ``Schur indexes of sporadic simple groups,''
  \href{http://dx.doi.org/10.1016/0021-8693(79)90177-7}{{\em J. Algebra} {\bf
  58} (1979) no.~2, 508--522}.
  \url{http://dx.doi.org/10.1016/0021-8693(79)90177-7}.

\bibitem{Mar_M12}
R.~S. Margolin, ``Representations of {$M_{12}$},''
  \href{http://dx.doi.org/10.1006/jabr.1993.1078}{{\em J. Algebra} {\bf 156}
  (1993) no.~2, 362--369}. \url{http://dx.doi.org/10.1006/jabr.1993.1078}.

\bibitem{2012arXiv1211.3703C}
T.~{Creutzig}, G.~{Hoehn}, and T.~{Miezaki}, ``{The McKay-Thompson series of
  Mathieu Moonshine modulo two},''{\em ArXiv e-prints} (Nov., 2012)  ,
  \href{http://arxiv.org/abs/1211.3703}{{\tt arXiv:1211.3703 [math.NT]}}.

\bibitem{DVV}
R.~Dijkgraaf, E.~Verlinde, and H.~Verlinde, ``{Counting Dyons in N=4 String
  Theory},'' {\em Nucl.Phys.B} {\bf 484:543-561,1997}
  (Nucl.Phys.B484:543-561,1997)  ,
  \href{http://arxiv.org/abs/hep-th/9607026}{{\tt hep-th/9607026}}.

\bibitem{Kawai:1995hy}
T.~Kawai, ``{$N=2$ Heterotic String Threshold Correction, $K3$ Surface and
  Generalized Kac-Moody Superalgebra},''
  \href{http://dx.doi.org/10.1016/0370-2693(96)00052-4}{{\em Phys.Lett.} {\bf
  B372} (1996)  59--64},
\href{http://arxiv.org/abs/hep-th/9512046}{{\tt arXiv:hep-th/9512046
  [hep-th]}}.
%%CITATION = HEP-TH/9512046;%%.

\bibitem{LopesCardoso:1996nc}
G.~Lopes~Cardoso, G.~Curio, and D.~Lust, ``{Perturbative couplings and modular
  forms in N=2 string models with a Wilson line},''
  \href{http://dx.doi.org/10.1016/S0550-3213(97)00047-3}{{\em Nucl.Phys.} {\bf
  B491} (1997)  147--183},
\href{http://arxiv.org/abs/hep-th/9608154}{{\tt arXiv:hep-th/9608154
  [hep-th]}}.
%%CITATION = HEP-TH/9608154;%%.

\bibitem{Raum}
M.~Raum, ``{$M_{24}$-twisted Product Expansions are Siegel Modular Forms},''
  \href{http://arxiv.org/abs/1208.3453}{{\tt arXiv:1208.3453 [math.NT]}}.

\bibitem{GriNik_K3SrfsLorKMAlgsMrrSym}
V.~A. Gritsenko and V.~V. Nikulin, ``{$K3$} surfaces, {L}orentzian
  {K}ac-{M}oody algebras and mirror symmetry,'' {\em Math. Res. Lett.} {\bf 3}
  (1996) no.~2, 211--229.

\bibitem{GriNik_ArthMrrSymCYMfds}
V.~A. Gritsenko and V.~V. Nikulin, ``{The Arithmetic Mirror Symmetry and
  {C}alabi-{Y}au manifolds},''
  \href{http://dx.doi.org/10.1007/s002200050769}{{\em Comm. Math. Phys.} {\bf
  210} (2000) no.~1, 1--11}. \url{http://dx.doi.org/10.1007/s002200050769}.

\bibitem{GriNik_AutFrmLorKMAlgs_I}
V.~A. Gritsenko and V.~V. Nikulin, ``Automorphic forms and {L}orentzian
  {K}ac-{M}oody algebras. {I},''
  \href{http://dx.doi.org/10.1142/S0129167X98000105}{{\em Internat. J. Math.}
  {\bf 9} (1998) no.~2, 153--199}.
  \url{http://dx.doi.org/10.1142/S0129167X98000105}.

\bibitem{GriNik_AutFrmLorKMAlgs_II}
V.~A. Gritsenko and V.~V. Nikulin, ``Automorphic forms and {L}orentzian
  {K}ac-{M}oody algebras. {II},''
  \href{http://dx.doi.org/10.1142/S0129167X98000117}{{\em Internat. J. Math.}
  {\bf 9} (1998) no.~2, 201--275}.
  \url{http://dx.doi.org/10.1142/S0129167X98000117}.

\bibitem{GrNik_SieAutFrmCorrLorKMAlgs}
V.~A. Gritsenko and V.~V. Nikulin, ``{Siegel Automorphic Form Corrections of
  Some {L}orentzian {K}ac-{M}oody {L}ie algebras},'' {\em Amer. J. Math.} {\bf
  119} (1997) no.~1, 181--224.
  \url{http://muse.jhu.edu/journals/american_journal_of_mathematics/v119/119.1gritsenko.pdf}.

\bibitem{Kondo_old}
S.~Kondo, ``{Niemeier lattices, Mathieu groups, and finite groups of symplectic
  automorphisms of $K3$ surfaces. With an appendix by Shigeru Mukai.},'' {\em
  Duke Math. J} {\bf 92, No.3} (1998)  593--603.

\bibitem{Nikulin:2011}
V.~V. Nikulin, ``{K\"ahlerian $K3$ Surfaces and Niemeier Lattices, I},''
  \href{http://arxiv.org/abs/1109.2879}{{\tt arXiv:1109.2879 [math.AG]}}.

\bibitem{Cheng:2013kpa}
M.~C. Cheng, X.~Dong, J.~Duncan, J.~Harvey, S.~Kachru, {\em et al.}, ``{Mathieu
  Moonshine and N=2 String Compactifications},''
\href{http://arxiv.org/abs/1306.4981}{{\tt arXiv:1306.4981 [hep-th]}}.
%%CITATION = ARXIV:1306.4981;%%.

\bibitem{aoki}
H.~Aoki and T.~Ibukiyama, ``{Simple graded rings of Siegel modular forms,
  differential operators and Borcherds products},'' {\em Int. J. Math.} {\bf
  16, No. 3} (2005)  249--279.

\bibitem{Gri_EllGenCYMnflds}
V.~Gritsenko, ``Elliptic genus of {C}alabi-{Y}au manifolds and {J}acobi and
  {S}iegel modular forms,'' {\em Algebra i Analiz} {\bf 11} (1999) no.~5,
  100--125.

\bibitem{GAP4}
The GAP~Group, {\em GAP -- Groups, Algorithms, and Programming, Version 4.4},
  2005.
\newblock \verb+(http://www.gap-system.org)+.

\end{thebibliography}
\providecommand{\href}[2]{#2}\begingroup\raggedright\endgroup

%------------------------------------------------------------------%
\end{document}